\documentclass{amsart}

\usepackage{graphicx}
\usepackage{xr}
\usepackage{morefloats,bm,caption,amsbsy,enumitem,amsmath,amsthm,
amssymb,mathtools,amsfonts,multirow,verbatim,tikz,tikz-cd,thmtools,thm-restate,chngcntr}
\usepackage[alphabetic]{amsrefs}
\usepackage[hidelinks,colorlinks=true,linkcolor=blue, citecolor=black,linktocpage=true]{hyperref}
\usetikzlibrary{matrix,arrows,decorations.pathmorphing}
\usepackage[top=1.25in, bottom=1in, left=1.25in, right=1.25in,marginpar=1in]{geometry}
\usepackage{url}
\counterwithin{figure}{section}

\newtheorem{innercustomthm}{Theorem}
\newenvironment{customthm}[1]
  {\renewcommand\theinnercustomthm{#1}\innercustomthm}
  {\endinnercustomthm}

\newtheorem{innercustomcor}{Corollary}
\newenvironment{customcor}[1]
  {\renewcommand\theinnercustomcor{#1}\innercustomcor}
  {\endinnercustomcor}

\newtheorem{thm}{Theorem}[section]
\newtheorem{prop}[thm]{Proposition}
\newtheorem{conj}[thm]{Conjecture}
\newtheorem{cor}[thm]{Corollary}
\newtheorem{lem}[thm]{Lemma}

\theoremstyle{definition}
\newtheorem{define}[thm]{Definition}

\theoremstyle{remark}
\newtheorem{rem}[thm]{Remark}

\newcommand{\ve}[1]{\boldsymbol{\mathbf{#1}}}
\newcommand{\R}{\mathbb{R}}

\newcommand{\Z}{\mathbb{Z}}
\newcommand{\N}{\mathbb{N}}

\renewcommand{\d}{\partial}
\renewcommand{\subset}{\subseteq}
\renewcommand{\supset}{\supseteq}
\renewcommand{\tilde}{\widetilde}
\renewcommand{\bar}{\overline}
\renewcommand{\Hat}{\widehat}
\newcommand{\iso}{\cong}

\DeclareMathOperator{\cl}{{cl}}

\DeclareMathOperator{\Crit}{{Crit}}

\DeclareMathOperator{\Diff}{{Diff}}

\DeclareMathOperator{\GL}{{GL}}
\DeclareMathOperator{\gr}{{gr}}

\DeclareMathOperator{\id}{{id}}

\DeclareMathOperator{\Int}{{int}}

\DeclareMathOperator{\im}{{im}}

\DeclareMathOperator{\Mat}{{Mat}}
\DeclareMathOperator{\Max}{{Max}}
\DeclareMathOperator{\MCG}{{MCG}}

\DeclareMathOperator{\Mor}{{Mor}}

\DeclareMathOperator{\nice}{{nice}}

\DeclareMathOperator{\Spin}{{Spin}}

\DeclareMathOperator{\Span}{{Span}}
\DeclareMathOperator{\Sym}{{Sym}}

\DeclareMathOperator{\Coor}{{Coor}}

\newcommand{\bF}{\mathbb{F}}

\newcommand{\bK}{\mathbb{K}}
\newcommand{\bL}{\mathbb{L}}

\newcommand{\bS}{\mathbb{S}}
\newcommand{\bT}{\mathbb{T}}
\newcommand{\bU}{\mathbb{U}}

\newcommand{\cA}{\mathcal{A}}
\newcommand{\cB}{\mathcal{B}}
\newcommand{\cC}{\mathcal{C}}
\newcommand{\cD}{\mathcal{D}}

\newcommand{\cF}{\mathcal{F}}
\newcommand{\cG}{\mathcal{G}}
\newcommand{\cH}{\mathcal{H}}

\newcommand{\cK}{\mathcal{K}}
\newcommand{\cL}{\mathcal{L}}
\newcommand{\cM}{\mathcal{M}}

\newcommand{\cR}{\mathcal{R}}

\newcommand{\cT}{\mathcal{T}}

\newcommand{\cV}{\mathcal{V}}
\newcommand{\cW}{\mathcal{W}}

\newcommand{\frs}{\mathfrak{s}}

\newcommand{\cCFL}{\mathcal{C\!F\!L}}
\newcommand{\cHFL}{\mathcal{H\!F\! L}}
\newcommand{\CF}{\mathit{CF}}
\newcommand{\HF}{\mathit{HF}}

\newcommand{\CFL}{\mathit{CFL}}
\newcommand{\HFL}{\mathit{HFL}}

\newcommand{\PD}{\mathit{PD}}

\newcommand{\xs}{\ve{x}}
\newcommand{\ys}{\ve{y}}
\newcommand{\zs}{\ve{z}}
\newcommand{\ws}{\ve{w}}
\newcommand{\os}{\ve{o}}

\newcommand{\as}{\ve{\alpha}}
\newcommand{\bs}{\ve{\beta}}
\newcommand{\gs}{\ve{\gamma}}

\newcommand{\taus}{\ve{\tau}}

\newcommand{\SO}{\mathit{SO}}
\renewcommand{\GL}{\mathit{GL}}

\newcommand{\from}{\leftarrow}

\newcommand{\bmP}{{\bm{P}}}

\title{Link cobordisms and functoriality in link Floer homology}
\author{Ian Zemke}
\address{Department of Mathematics\\Princeton University\\  Princeton, NJ 08544, USA}
\email{izemke@math.princeton.edu}
\thanks{This research was supported by NSF grant DMS-1703685}

\begin{document}

\begin{abstract}We construct cobordism maps on link Floer homology associated to decorated link cobordisms.  The maps are defined on a curved chain homotopy type invariant. We describe the construction, and prove invariance. We also make a comparison with the graph TQFT for Heegaard Floer homology.
\end{abstract}

\maketitle

\tableofcontents

\section{Introduction}

 Ozsv\'{a}th and Szab\'{o} constructed a powerful set of algebraic invariants of 3- and 4-manifolds, called Heegaard Floer homology \cite{OSDisks} \cite{OSTriangles}. Heegaard Floer homology fits into the framework of a $(3+1)$-dimensional TQFT. Also defined by Ozsv\'{a}th and Szab\'{o} \cite{OSKnots} and independently by Rasmussen \cite{RasmussenKnots}, there is a refinement of Heegaard Floer homology for knots embedded in 3-manifolds, called knot Floer homology. There is a generalization of the knot invariant to links, due to Ozsv\'{a}th and Szab\'{o} \cite{OSLinks}, called link Floer homology.

 To a 4-manifold $W$ with properly embedded surface $\Sigma\subset W$, the boundary $\d \Sigma$ is a link in $\d W$. In light of the TQFT structure of the Heegaard Floer 3- and 4-manifold invariants, it is natural to expect functorial maps between the link Floer homologies of the ends of a link cobordism $(W,\Sigma)$.  An important step in this direction is the work of Juh\'{a}sz \cite{JCob}, who constructed cobordism maps for decorated link cobordisms on the hat version of link Floer homology. One limitation of Juh\'{a}sz's construction is that it uses the contact gluing map due to Honda, Kazez and Mati\'{c} \cite{HKMTQFT}, which limits the construction to $\Hat{\HFL}$, and not the full link Floer complex.
 
 In this paper we provide a construction of a $(3+1)$-dimensional TQFT for a general version of the link Floer complex, and prove invariance of the construction. The link invariant we consider in this paper is a curved variation of the link Floer complexes described by Ozsv\'{a}th and Szab\'{o} \cite{OSLinks}.
 
Given a finite set $\bmP=\{p_1,\dots, p_n\}$, we define
\[
\cR^-_{\bmP}:=\bF_2[X_{p_1},\dots, X_{p_n}],
\]
the polynomial ring freely generated by the variables $X_{p_1},\dots, X_{p_n}$.

We will consider links $L$ embedded in a 3-manifold, equipped with two disjoint collections of basepoints, $\ws$ and $\zs$, which are contained in $L$. We will write $\bL=(L,\ve{w},\ve{z})$ for a link together with two collections of basepoints. We will write $\bL^\sigma$ for the multi-based link $\bL$, equipped with a map
\[
\sigma\colon \ws\cup \zs\to \bmP.
\] 
 We will construct a free and finitely generated $\cR_{\bmP}^-$-module
\[
\cCFL^-(Y,\bL^\sigma,\frs),
\]
which also depends on a choice of $\Spin^c$ structure $\frs$ on $Y$. The module $\cCFL^-(Y,\bL^\sigma,\frs)$ has a natural filtration by $\Z^{\bmP}$. We call the map $\sigma$ a \emph{coloring} of the link $\bL$.

  We note also that the construction of $\cCFL^-(Y,\bL^\sigma,\frs)$ requires some auxiliary choices (a Heegaard diagram and a choice of almost complex structure),  so a precise definition of the invariant  requires some additional formalism. See Section~\ref{sec:complexes} for a precise definition.

The module $\cCFL^-(Y,\bL^\sigma,\frs)$ comes with a distinguished endomorphism $\d$, obtained by counting pseudo-holomorphic curves, which satisfies
\[
\d^2=\omega_{\bL,\sigma} \cdot \id,
\]
for a scalar $\omega_{\bL,\sigma}\in \cR^-_{\bmP}$.

 Modules with an endomorphism squaring to the action of a scalar in the ground ring are often referred to as \emph{matrix factorizations}, and appear in the context of Khovanov-Rozansky homology \cite{KhoRoz}. We will call a module with such an endomorphism a \emph{curved chain complex}.

In this paper, a \emph{decorated link cobordism} 
\[
(W,\cF^\sigma)\colon (Y_1,\bL_1)\to (Y_2,\bL_2)
\]
 between two multi-based links consists of a pair  $(W,\cF^\sigma)$ where $W$ is a compact 4-manifold with $\d W=-Y_1\sqcup Y_2$, and $\cF^\sigma$ is a \emph{colored, oriented surface with divides}. More explicitly, $\cF^\sigma$ consists of a properly embedded,  oriented surface $\Sigma$ in $W$, with a properly embedded collection of dividing arcs $\cA\subset \Sigma$, as well as a map $\sigma\colon C(\Sigma\setminus \cA)\to \bmP$, where $C(\Sigma\setminus \cA)$ denotes the set of connected components of $\Sigma\setminus \cA$. We further assume that 
\begin{enumerate}
\item  $\d\Sigma=-L_1\sqcup L_2$.
\item $\cA$ divides $\Sigma$ into two subsurfaces, $\Sigma_{\ws}$ and $\Sigma_{\zs}$, which meet along $\cA$. Furthermore $\Sigma_{\ws}$ contains  the $\ws$-basepoints of $\bL_1$ and $\bL_2$, and $\Sigma_{\zs}$ contains  the $\zs$-basepoints.
\end{enumerate}

A pictorial description of a decorated link cobordism is shown in Figure~\ref{fig::32}.

 \begin{figure}[ht!]
\centering
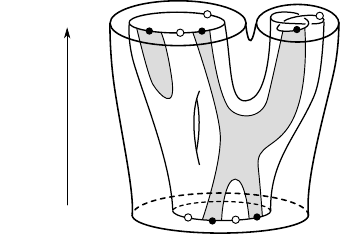
\caption{\textbf{A decorated link cobordism $(W,\cF^\sigma)$}.  The solid dots denote $\ve{w}$-basepoints and the open dots denote $\ve{z}$-basepoints. The shaded region is $\Sigma_{\ws}$, and the unshaded region is $\Sigma_{\zs}$.\label{fig::32}}
\end{figure}

\subsection{Invariants of decorated link cobordisms}

Given a decorated link cobordism 
\[
(W,\cF^\sigma)\colon (Y_1,\bL_1)\to (Y_2,\bL_2)
\]
 and a $\Spin^c$ structure $\frs\in \Spin^c(W)$, our paper centers on constructing a diffeomorphism invariant of $(W,\cF^\sigma)$, which takes the form of a cobordism map
\[
F_{W,\cF^{\sigma},\frs}\colon \cCFL^-(Y_1,\bL_1^{\sigma_1},\frs_1)\to \cCFL^-(Y_2,\bL_2^{\sigma_2},\frs_2),
\] 
where $\frs_i=\frs|_{Y_i}$ and $\sigma_i=\sigma|_{\bL_i}$.

To construct the link cobordism maps, we describe maps on the link Floer complexes which correspond to the following types of elementary cobordisms:
\begin{enumerate}
\item 4-dimensional handle cobordisms for handles attached to the complement of the link.
\item Cylindrical link cobordisms with simple dividing sets.
\item Saddle cobordisms for band surgery on a link in a 3-manifold.
\end{enumerate}
We define the cobordism maps for a general link cobordism $(W,\cF^\sigma)$ by decomposing $(W,\cF^\sigma)$ into a sequence of such elementary cobordisms, and taking the composition of the maps for each elementary piece.

Our paper centers on proving two main properties of the cobordism maps: \emph{diffeomorphism invariance} and \emph{functoriality}.

Proving diffeomorphism invariance of the cobordism maps amounts to showing that the cobordism maps are independent of the decomposition of $(W,\cF)$ into elementary link cobordisms. To this end, we state the following theorem:

\begin{customthm}{A}\label{thm:A} If  $(W,\cF^\sigma)\colon (Y_1,\bL_1)\to (Y_2,\bL_2)$ is a decorated link cobordism, equipped with a $\Spin^c$ structure $\frs\in \Spin^c(W)$, then the map  $F_{W,\cF^{\sigma},\frs}$
constructed in this paper is a $\Z^{\bmP}$-filtered, $\cR_{\bmP}^-$-equivariant chain map, which is a diffeomorphism invariant, up to $\Z^{\bm{P}}$-filtered, $\cR_{\bmP}^-$-equivariant chain homotopy.
\end{customthm}

If $\bL$ is a multi-based link in $Y$, with coloring $\sigma$, there is an identity decorated link cobordism
\[
(W_{\id},\cF_{\id}^{\sigma})\colon (Y,\bL)\to (Y,\bL),
\]
obtained by decorating $([0,1]\times Y, [0,1]\times L)$ with a vertical dividing set (see Definition~\eqref{def:identitylinkcob} for a precise definition).

Proving functoriality of the cobordism maps amounts to showing two things: that the identity link cobordism induces the identity map, and that a version of the composition law is satisfied.

We note that the fact that $(W_{\id},\cF_{\id}^{\sigma})$ induces the identity map is essentially immediate from the definition. Indeed we can give $(W_{\id},\cF_{\id}^{\sigma})$ a decomposition into elementary link cobordisms consisting of a single elementary link cobordism, namely $(W_{\id},\cF_{\id}^{\sigma})$ itself. The induced map is defined to be the identity.

As a second step towards proving functoriality, we prove the following version of the composition law:
\begin{customthm}{B}\label{thm:B}Suppose that $(W,\cF^\sigma)\colon (Y_1,\bL_1)\to (Y_2,\bL_2)$ is a decorated link cobordism which decomposes as
\[
(W,\cF^\sigma)=(W_2,\cF_2^{\sigma_2})\cup (W_1,\cF_1^{\sigma_1}),
 \] where each $(W_i,\cF_i^{\sigma_i})$ is a  decorated link cobordism, and $\sigma_i=\sigma|_{\cF_i}$. If $\frs_1$ and $\frs_2$ are $\Spin^c$ structures on $W_1$ and $W_2$, respectively, then
\[
F_{W_2,\cF_2^{\sigma_2},\frs_2}\circ F_{W_1,\cF_1^{\sigma_1},\frs_1}\simeq \sum_{\substack{\frs\in \Spin^c(W)\\ \frs|_{W_i}=\frs_i}} F_{W,\cF^{\sigma},\frs}.
\]
\end{customthm}

\subsection{Further remarks on the cobordism maps}

We now make several brief remarks about the TQFT described in this paper.

Firstly, we note that our  TQFT is somewhat more flexible than the original (3+1)-dimensional TQFT described by Ozsv\'{a}th and Szab\'{o} \cite{OSDisks} \cite{OSTriangles} with respect to connectedness of 3- and 4-manifolds. In this paper, we allow 3- and 4-manifolds to be disconnected or empty. Our requirements for 3-manifolds containing multi-based links $(Y,\bL)$ are simply that each connected component of $Y$ contains a component of $\bL$, and that each connected component of $\bL$ contains at least two basepoints. Similarly we require that each connected component of a 4-manifold contains a component of the link cobordism surface, and that each connected component of the link cobordism surface contains a dividing arc.

There are two elementary dividing sets on $[0,1]\times L\subset [0,1]\times Y$ which are particularly important to our paper. They are shown in Figure~\ref{fig::112}. The induced maps are the \emph{quasi-stabilization maps}, which we denote by $S_{w,z}^+$, $S_{w,z}^-$, $T_{w,z}^+$ and $T_{w,z}^-$. It is important to note that, up to isotopy, any dividing set on $[0,1]\times L$ can be built by stacking a sequence of such pieces on top of each other. Hence all of our cobordism invariants for decorations of cylindrical link cobordisms can be encoded into  algebraic relations involving the quasi-stabilization maps; See Section~\ref{sec:mapswhichappear} for further details.

\begin{figure}[ht!]
\centering
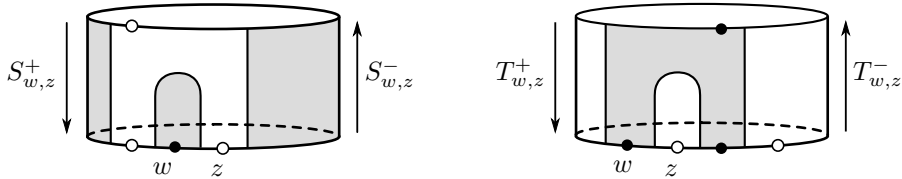
\caption{\textbf{Two simple dividing sets on $[0,1]\times L\subset [0,1]\times Y$.} In this paper, we will write $S_{w,z}^+,$ $S_{w,z}^-$, $T_{w,z}^+$ and $T_{w,z}^-$ for the induced maps, and refer to these maps as the \emph{quasi-stabilization} maps.\label{fig::112}}
\end{figure}

Finally, we remark that there is an additional relation which appears frequently in the TQFT. Suppose $(W,\cF_1^{\sigma_1})$, $(W,\cF_2^{\sigma_2})$ and $(W,\cF_3^{\sigma_3})$ are decorated link cobordisms which all share the same underlying, undecorated link cobordism $(W,\Sigma)$, and whose decorations agree outside of a disk $D^2$ contained in $\Sigma$. Suppose further that on $D^2$, the three decorations satisfy the configuration shown in Figure~\ref{fig::113}. Then
 \begin{equation}
F_{W,\cF_1^{\sigma_1},\frs}+F_{W,\cF_2^{\sigma_2},\frs}+F_{W,\cF_3^{\sigma_3},\frs}\simeq 0. \label{eq:THEbypassrelation}
 \end{equation}
Inspired by an analogous picture from contact geometry,  we  refer to  the relation in Equation~\eqref{eq:THEbypassrelation} as the \emph{bypass relation}. Since the bypass relation appears frequently, there are many ways of proving it. See, e.g., Lemma~\ref{lem:PhiPsicommutators}.

\begin{figure}[ht!]
\centering
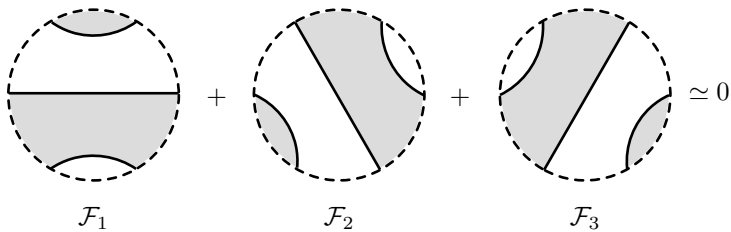
\caption{\textbf{The bypass relation.} If three dividing sets on a decorated link cobordism agree outside of a disk, and inside the disk they fit into the triad shown, then the sum of their invariants is chain homotopic to zero. Many relations between naturally occurring maps on the link Floer complexes can be interpreted in terms of this relation. \label{fig::113}}
\end{figure}

\subsection{Algebraic reduction to the graph TQFT}

As an application, we compare the link cobordism maps defined in this paper to the graph cobordism maps defined in \cite{ZemGraphTQFT} for cobordisms $(W,\Gamma)$ with embedded ribbon graphs between closed, multi-pointed 3-manifolds. For the purposes of this section, we restrict to colorings with exactly two colors, and we further assume that all the $\ws$-basepoints and regions are assigned the variable $U$, and all the $\zs$-basepoints and regions are assigned the variable $V$. We omit the coloring $\sigma$ from the notation in this section.

\begin{define}A \emph{ribbon graph cobordism} $(W,\Gamma)\colon (Y_1,\ve{w}_1)\to (Y_2,\ve{w}_2)$ between two multi-based 3-manifolds is a pair $(W,\Gamma)$ such that the following are satisfied:
\begin{enumerate}
\item $W$ is a cobordism from $Y_1$ to $Y_2$.
\item $\Gamma$ is an embedded graph in $W$ such that $\Gamma\cap \d W=\ve{w}_1\cup \ve{w}_2$ and each basepoint of $\ve{w}_1$ and $\ve{w}_2$ has valence 1 in $\Gamma$.
\item Each vertex of $\Gamma$ has valence at least 1.
\item $\Gamma$ is decorated with a ribbon structure, i.e.,  a choice of cyclic ordering of the edges adjacent to each vertex.
\end{enumerate}
\end{define}

In \cite{ZemGraphTQFT}, the author proves that there are $\Spin^c$ functorial maps associated to ribbon graph cobordisms $(W,\Gamma)$ equipped with 4-dimensional $\Spin^c$ structures.

 If $(Y,\bL)$ is a multi-based link in $Y$, then assuming that we are coloring the complexes with only two variables $U$ and $V$, as above, there are natural chain isomorphisms
\[
\cCFL^-(Y,\bL,\frs)\otimes_{\bF_2[U,V]}\bF_2[U,V]/(V-1)\iso \CF^-(Y,\ve{w},\frs)
\] and
\[
\cCFL^-(Y,\bL,\frs)\otimes_{\bF_2[U,V]}\bF_2[U,V]/(U-1)\iso\CF^-(Y,\ve{z},\frs-\PD[L]).
\]
 The map $F_{W,\cF,\frs}$  associated to the link cobordism $(W,\cF)$ thus induces two maps on $\CF^-$, for which we write
\[
F_{W,\cF,\frs}|_{U=1}\qquad \text{and} \qquad F_{W,\cF,\frs}|_{V=1}.
\]

Given a decorated link cobordism $(W,\cF)\colon (Y_1,\bL_1)\to (Y_2,\bL_2)$, we will describe a way to construct ribbon graphs 
\[
\Gamma(\Sigma_{\ve{w}})\subset \Sigma_{\ve{w}} \qquad \text{and} \qquad \Gamma(\Sigma_{\ve{z}})\subset \Sigma_{\ve{z}}\] 
which we will call \textit{ribbon 1-skeletons} of the surfaces $\Sigma_{\ve{w}}$ and $\Sigma_{\ve{z}}$, respectively (see Definition~\ref{def:ribbongraphcore} for a precise description). We prove the following:

\begin{customthm}{C}\label{thm:C} Given a decorated link cobordism $(W,\cF)$  with subsurfaces $\Sigma_{\ve{w}}$ and $\Sigma_{\ve{z}}$ and choices of ribbon 1-skeletons $\Gamma(\Sigma_{\ve{w}})$ and $\Gamma(\Sigma_{\ve{z}})$,  the reductions satisfy
\[
F_{W,\cF,\frs}|_{V=1}\simeq F_{W,\Gamma(\Sigma_{\ve{w}}),\frs}^B\qquad \text{and}\qquad F_{W,\cF,\frs}|_{U=1}\simeq F_{W,\Gamma(\Sigma_{\ve{z}}),\frs-\PD[\Sigma]}^A.
\] 
\end{customthm}

In the above theorem, the maps $F_{W,\Gamma,\frs}^A$ are the graph cobordism maps defined in \cite{ZemGraphTQFT}. The maps $F_{W,\Gamma,\frs}^B$ are minor variations, which we describe in Section~\ref{sec:graphTQFTs}. The two versions of the graph cobordism maps (the $A$-maps and the $B$-maps) satisfy the relation
\[
F_{W,\Gamma,\frs}^B\simeq F_{W,\bar{\Gamma},\frs}^A,
\] 
where $\bar{\Gamma}$ denotes the ribbon graph $\Gamma$ with cyclic orderings reversed.

In \cite{ZemGraphTQFT}, the graph cobordism maps are shown to be invariant under isotopies of the graph $\Gamma$ which fix $\Gamma\cap \d W$. Many of the algebraic relations suggest that the maps are invariant under more moves than just isotopy, such as adding trivial strands to the graph, or sliding edges across vertices (while respecting the cyclic orderings). In Definition~\ref{def:ribbonequivalent} we provide a notion of \textit{ribbon-equivalence} between ribbon graphs, which is weaker than isotopy. As corollaries to Theorem~\ref{thm:C}, we obtain  the following:

\begin{customcor}{D}\label{cor:D}If $(W,\Gamma)$ and $(W,\Gamma')$ are ribbon-equivalent ribbon graph cobordisms, then
\[
F_{W,\Gamma,\frs}^A\simeq F_{W,\Gamma',\frs}^A \qquad \text{and} \qquad F_{W,\Gamma,\frs}^B\simeq F_{W,\Gamma',\frs}^B.\] 
\end{customcor}
\begin{customcor}{E}\label{cor:E}
 The $V=1$ reduction of $F_{W,\cF,\frs}$,
\[
F_{W,\cF,\frs}|_{V=1}\colon \CF^-(Y_1,\ve{w}_1,\frs|_{Y_1})\to\CF^-(Y_2,\ve{w}_2,\frs|_{Y_2}),
\] depends only on the ribbon surface $\Sigma_{\ve{w}}$ and the $\Spin^c$ structure $\frs$. The $U=1$ reduction of $F_{W,\cF,\frs}$, 
\[
F_{W,\cF,\frs}|_{U=1}\colon \CF^-(Y_1,\ve{z}_1,\frs-\PD[L_1])\to\CF^-(Y_2,\ve{z}_2,\frs-\PD[L_2]),
\]
 depends only on the ribbon surface $\Sigma_{\ve{z}}$, and the $\Spin^c$ structure $\frs-\PD[\Sigma]$.
\end{customcor}

\begin{rem}Corollary~\ref{cor:E} can be used to give an alternate construction and proof of invariance of the graph cobordism maps from \cite{ZemGraphTQFT}; see Section~\ref{sec:alternateconstructionofgraphs}.
\end{rem}

\subsection{Comparison to other link cobordism map constructions}  There are already several notable examples of link cobordism map constructions in the literature. The most relevant to this paper is the construction by Juh\'{a}sz \cite{JCob} for the hat flavor of link Floer homology, denoted $\Hat{\HFL}$. If $\bL$ is a multi-pointed link in $Y$, the hat version of the link Floer complex, denoted $\Hat{\CFL}(Y,\bL)$, is obtained by formally setting all of the variables in $\cR_{\bmP}^-$ equal to 0, and then summing over all $\Spin^c$ structures. Taking homology, we obtain $\Hat{\HFL}(Y,\bL)$. We note that on the hat version of link Floer homology, our maps are independent of the coloring $\sigma$. Similarly, Juh\'{a}sz's definition of a decorated link cobordism is equivalent to ours, if we ignore the coloring. In light of this, we make the following conjecture:

\begin{conj}\label{conj:TQFTsagree}The maps $\Hat{F}_{W,\cF} :=\sum_{\frs\in \Spin^c(W)} \Hat{F}_{W,\cF,\frs}$ defined in this paper for the hat flavor coincide with the maps $F_{W,\cF}$ defined in \cite{JCob} using sutured Floer homology.\footnote{Since the original posting of this paper, this conjecture has been proven by Juh\'{a}sz and the author \cite{JuhaszZemkeContactHandles}.}
\end{conj}

In \cite{JMComputeCobordismMaps}, Juh\'{a}sz and Marengon provide a partial computation of many elementary link cobordism maps from Juh\'{a}sz's sutured TQFT, which could be helpful in proving Conjecture~\ref{conj:TQFTsagree}. We will not pursue the conjecture in this paper. We note that the maps induced by Juh\'{a}sz's TQFT for knot concordances are studied in \cite{JMConcordance}.

There are additional partial constructions of link cobordism maps in link Floer homology. For example, Sarkar defines maps associated to surfaces in $[0,1]\times S^3$ in order to provide an alternate proof of the Milnor conjecture using grid diagrams \cite{SarkarTau}.  Ozsv\'{a}th, Stipsicz and Szab\'{o} use a grid diagrammatic construction of link cobordism maps to analyze the unoriented four-ball genus of knots and links \cite{OSSUnorientedGenus}. We note that neither Sarkar nor Ozsv\'{a}th--Stipsicz--Szab\'{o} prove invariance of their constructions. It would be interesting to know whether a construction and proof of invariance could be done entirely through grid diagrams. A set of grid movie moves for decorated surfaces was described in  \cite{GrahamMovieMoves}, though verification of the invariance of the grid homology maps under these moves has not been completed. Nevertheless, we state the following conjecture:

\begin{conj}\label{conj:griddiagrams} The maps described in terms of grid diagrams  by Sarkar \cite{SarkarTau} and Ozsv\'{a}th--Stipsicz--Szab\'{o}  \cite{OSSUnorientedGenus} are equal to the maps defined in this paper for elementary link cobordisms inside of $[0,1]\times S^3$.
\end{conj}

There is an additional construction due to Alishahi and Eftekhary \cite{AEtangle} of Floer homology for multi-based tangles, generalizing knot Floer homology. Alishahi and Eftekhary give a construction of  cobordism maps for cobordisms between tangles, and prove invariance. When restricted to links, their construction is similar to the one in this paper, and features maps associated to saddles which are defined with the same Heegaard triples as the ones in this paper. Nonetheless, the TQFT structure is not (obviously) the same as the one in this paper, though it is likely that there is a concrete relationship.

\subsection{Further developments}

 We now mention a few basic properties of the cobordism maps which have been proven in subsequent papers. The link cobordism maps constructed in this paper satisfy some simple Maslov and Alexander grading change formulas \cite{ZemAbsoluteGradings}*{Theorem~1.4}. Furthermore, using the grading change formula, the author computes the maps for closed surfaces with simple dividing sets when $b_1=0$ and $b_2^+=0$ \cite{ZemAbsoluteGradings}*{Theorem~1.7}.  
 
The link cobordism maps also satisfy a version of conjugation invariance \cite{ZemConnectedSums}*{Theorem~1.3} similar to the version satisfied by the standard 4-manifold invariants \cite{OSTriangles}*{Theorem~3.6}. Using this, and some further properties of the link cobordism maps, the author explores the relation between the cobordism maps and the K\"{u}nneth theorem for connected sums, and in particular proves a connected sum formula for involutive knot Floer homology \cite{ZemConnectedSums}*{Theorem~1.1}.

\subsection{Conventions and notation}

\label{subsec:convention}
Throughout the paper, we will use the outward normal first orientation convention, which is standard in gauge theory. That is, if $M$ is an oriented manifold and $p\in \d M$, with an outward normal vector $N$ at $p$, we say an ordered basis $(v_1,\dots, v_n)$ of $T_p \d M$ is oriented if and only if $(N,v_1,\dots, v_n)$ is an oriented basis for $T_p M$. With this convention, the identity cobordism is $([0,1]\times Y,[0,1]\times L)$.

Unless otherwise specified, we will write ``chain homotopic'' and use the symbol $\simeq$ to mean ``chain homotopic through filtered and equivariant chain homotopies.''

Finally, whenever we list several consecutive basepoints on a link component, we will always write them ordered from right to left (we will remind the reader of this convention). For example, if we write $(w,z,w',z')$, we mean the $w$ follows $z$, which follows $w'$, which follows $z'$. We extend this notation to quasi-stabilization maps and basepoint moving maps. If we write $S_{w,z}^+$, the implication is that $w$ immediately follows $z$, whereas if we write $S_{z',w'}^+$, then the implication is that $z'$ immediately follows $w'$. 

\subsection{Organization}

In Section~\ref{subsec:category}, we describe  the categories of decorated link cobordisms and curved chain complexes. In Section~\ref{sec:complexes}, we define the link Floer complexes we consider in this paper, and prove some basic properties. In Section~\ref{sec:mapswhichappear}, we define the quasi-stabilization maps, as well as several other maps which naturally appear on the link Floer complexes, and relate the quasi-stabilization maps to some basepoint moving diffeomorphism maps on link Floer homology. In Section~\ref{sec:handleattachmentmapsawayfromL}, we describe the 4-dimensional handle maps on the link Floer complex. In Section~\ref{sec:bandmaps}, we define the cobordism maps for saddle cobordisms. In Section~\ref{sec:compoundhandles}, we describe some useful compositions of the maps from the previous sections which are associated to several simple link cobordisms, such as a birth or death link cobordism, or a 4-dimensional 1-handle containing a saddle. In Sections~\ref{sec:relationsI} and \ref{sec:furtherrelations} we prove some more technical relations about the various maps defined previously in the paper. In Sections~\ref{sec:modelcobs-parametrizedcobs} and \ref{sec:morsetheoyr-weakequivalences}, we prove the necessary versions of Morse and Cerf theory  which we will use to prove invariance of the cobordism maps. In Section~\ref{sec:construction-and-invariance}, we prove that the maps described in this paper are invariants of decorated link cobordisms. In Section~\ref{sec:compositionlaw} we prove the composition law. In Section~\ref{sec:graphTQFTs}, we show that the link cobordism maps naturally reduce to the graph cobordism maps.

\subsection{Acknowledgments} The author would like to thank Akram Alishahi, Haofei Fan, Kristen Hendricks, Jen Hom, Andr\'{a}s Juh\'{a}sz, Robert Lipshitz, Ciprian Manolescu, and Marco Marengon for helpful conversations and suggestions. The author would also like to thank Sucharit Sarkar and Andy Manion for suggesting to work in the setting of curved chain complexes. Finally, the author thanks the anonymous referee for pointing out several errors and making many helpful suggestions.

\section{The categories of decorated link cobordisms and curved chain complexes}

\label{subsec:category}

In this section we define the category of decorated link cobordisms which features in this paper.

\begin{define}\label{def:multibasedlink}A \emph{3-manifold with a multi-based link} is a pair $(Y,\bL)$ such that
\begin{enumerate}
\item  $Y$ is a closed, oriented 3-manifold,
\item  $\bL=(L,\ve{w},\ve{z})$ is a triple consisting of an oriented, embedded link $L\subset Y$, with disjoint collections of basepoints $\ve{w}$ and $\ve{z}$ on $L$, and
\item each component of $L$ has at least two basepoints, and the basepoints alternate between those in $\ve{w}$ and those in $\ve{z}$ as one traverses a component of $L$.
  \end{enumerate}
\end{define}

\begin{define} A \emph{surface with divides}  is a pair $\cF=(\Sigma, \cA)$ such that,
\begin{enumerate}
\item  $\Sigma$ is a compact, oriented surface,
\item  $\cA\subset \Sigma$ is a properly embedded 1-manifold, and
\item the components of $\Sigma\setminus \cA$ are partitioned into two sub-surfaces, $\Sigma_{\ve{w}}$ and $\Sigma_{\ve{z}}$, which meet along $\cA$.
\end{enumerate}
\end{define}

To achieve well-defined, functorial maps in our link TQFT, we need to consider colorings of the basepoints and the surfaces with divides.

\begin{define} A \emph{coloring} of a multi-based link $\bL=(L,\ve{w},\ve{z})$ is a map $\sigma\colon \ws\cup \zs\to \bmP$, where $\bmP$ is a finite set. A \emph{coloring} of a surface with divides $\cF=(\Sigma, \cA)$ is a map $\sigma\colon C(\Sigma\setminus \cA)\to \bmP$, where $\bmP$ is a finite set, and $C(\Sigma\setminus \cA)$ denotes the set of components of $\Sigma\setminus \cA$.
\end{define}

We think of $\bmP$ as the set of colors. We will write $\bL^\sigma$ for a multi-based link equipped with a coloring $\sigma\colon \ws\cup \zs\to \bmP$. Similarly we will write $\cF^\sigma$ for a surface with divides, equipped with a coloring $\sigma$. Note that the codomain $\bmP$ of $\sigma$ is part of the data of a coloring.

Modeled on Juhasz's definition of a decorated link cobordism \cite{JCob}*{Definition~4.5}, we now define a notion of morphism between 3-manifolds with multi-based links which we consider in this paper:

\begin{define}If $(Y_1,\bL_1)=(Y_1,(L_1,\ws_1,\zs_1))$ and $(Y_2,\bL_2)=(Y_2,(L_2,\ws_2,\zs_2))$ are two 3-manifolds with multi-based links, a \emph{decorated link cobordism} from $(Y_1,\bL_1)$ to $(Y_2,\bL_2)$  consists of a pair  $(W,\cF^\sigma)$ such that
\begin{enumerate}
\item $W$ is a compact 4-manifold with $\d W=-Y_1\sqcup Y_2$,
\item  $\cF=(\Sigma, \cA)$ is a surface with divides, equipped with the coloring  $\sigma$,
\item $\Sigma\subset W$ is a properly embedded, compact, oriented surface with $\Sigma\cap \d W=-L_1\sqcup L_2$,
\item each component of $L_i\setminus \cA$ contains exactly one basepoint of $\ve{w}_i\cup \ve{z}_i$, and
\item $\ws_1\cup \ws_2\subset \Sigma_{\ws}$ and $\zs_1\cup \zs_2 \subset\Sigma_{\zs}.$ 
\end{enumerate}
\end{define}

We define the identity decorated link cobordism as follows:
\begin{define}\label{def:identitylinkcob} If $Y$ is a 3-manifold containing the multi-based link $\bL=(L,\ws,\zs)$, equipped with the coloring $\sigma$, the \emph{identity decorated link cobordism} 
\[
(W_{\id}, \cF_{\id}^{\sigma})\colon (Y,\bL)\to (Y,\bL)
\]
is constructed by setting $W_{\id}:=[0,1]\times Y$ and $\cF_{\id}:=([0,1]\times L,[0,1]\times \ve{p})$, where $\ve{p}\subset L\setminus (\ws\cup \zs)$ is a collection of points consisting of exactly one point per component of $L\setminus (\ws\cup \zs)$. The coloring of $\cF_{\id}$ (which by abuse of notation we are denoting $\sigma$) is the unique coloring which restricts to $\sigma$ on $\bL$.
\end{define}

\begin{rem}To properly define a link cobordism category, one must work with cobordisms which are equipped with parametrizations of their boundary components. More precisely, a  link cobordism in the category of decorated link cobordisms consists of a 4-manifold with an embedded surface $(W,\cF)$ with boundary equal to $(-Y_1',-\bL_1')\sqcup (Y_2',\bL_2')$, together with choices of orientation preserving diffeomorphisms $\phi_i\colon (Y_i',\bL_i')\to (Y_i,\bL_i)$. Such a definition is necessary to precisely define a category of cobordisms, since otherwise the cobordism category lacks identity morphisms, for example. We will never explicitly deal with cobordisms at this level, since the notation becomes cumbersome.
\end{rem}

If $\bL=(L,\ws,\zs)$ is a multi-based link with $\ws=\{w_1,\dots, w_n\}$ and $\zs=\{z_1,\dots, z_n\}$, we define the ring
\[
\bF_2[U_{\ws},V_{\zs}]:=\bF_2[U_{w_1},\dots, U_{w_n}, V_1,\dots, V_{z_n}].
\] We define the   \emph{curvature constant} $\omega_{\bL}\in \bF_2[U_{\ws},V_{\zs}]$ via the formula
\[\omega_{\bL}:=\sum_{K\in C(L)} \left(U_{w_{K,1}}V_{z_{K,1}}+V_{z_{K,1}}U_{w_{K,2}}+U_{w_{K,2}}V_{z_{K,2}}+\cdots +V_{z_{K,n_K}}U_{w_{K,1}}
\right),\] where $C(L)$ denotes the set of components of $L$, and $w_{K,1},$ $z_{K,1},$ $w_{K,2},$ \dots , $z_{K,n_K}$ are the basepoints on a link component $K$, in the order that they appear. 

 If $\bmP=\{p_1,\dots, p_n\}$ is a set of colors, we define the rings
\[\cR^-_{\bmP}:=\bF_2[X_{p_1},\dots, X_{p_n}]\qquad \text{and} \qquad \cR^\infty_{\bmP}:=\bF_2[X_{p_1},\dots, X_{p_n}, X_{p_1}^{-1},\dots, X_{p_n}^{-1}].\]

 If $\sigma$ is a coloring of $\bL$  which has codomain $\bmP$, then $\cR^-_{\bmP}$ is a module over $\bF_2[U_{\ws}, V_{\zs}]$, and we define the colored curvature constant $\omega_{\bL,\sigma}\in \cR^-_{\bmP}$ to be the element
\[\omega_{\bL,\sigma}:= \omega_{\bL}\cdot 1_{\cR^-_{\bmP}}\]
\[=\sum_{K\in C(L)} \left(X_{\sigma(w_{K,1})}X_{\sigma(z_{K,1})}+X_{\sigma(z_{K,1})}X_{\sigma(w_{K,2})}+X_{\sigma(w_{K,2})}X_{\sigma(z_{K,2})}+\cdots +X_{\sigma(z_{K,n_K})}X_{\sigma(w_{K,1})}
\right).\]

If two colored, multi-based links are cobordant via a decorated link cobordism, then they have the same curvature constant:

\begin{lem}\label{lem:cobordantlinkshavesamecoloring}Suppose that $\cF^\sigma$ is a colored surface with divides, with boundary consisting of disjoint union of the multi-based links $-\bL_1$ and $\bL_2$. If $\sigma_i$ denotes the induced coloring of $\bL_i$, then
\[\omega_{\bL_1,\sigma_1}=\omega_{\bL_2,\sigma_2}\in \cR^-_{\bmP}.\]
\end{lem}
\begin{proof}Since no confusion can arise, let us write $\sigma$ for both $\sigma_1$ and $\sigma_2$. Define the set of points $\ve{p}_i\subset L_i\setminus (\ws_i\cup \zs_i)$ as 
\[\ve{p}_i:= L_i\cap \cA.\] For each point $p\in \ve{p}_i$, let us write $w(p)$ and $z(p)$ for the two basepoints immediately adjacent to $p$, on $L_i$. Note that
\[\omega_{\bL_i,\sigma}=\sum_{p\in \ve{p}_i} X_{\sigma(w(p))}X_{\sigma(z(p))}.\] If $a$ is an arc of $\cA$, with boundary points $p$ and $p'$, then $\sigma(w(p))=\sigma(w(p'))$ and $\sigma(z(p))=\sigma(z(p'))$, so that
\[X_{\sigma(w(p))}X_{\sigma(z(p))}=X_{\sigma(w(p'))} X_{\sigma(z(p'))}.\] Hence
\[\omega_{\bL_1,\sigma}+\omega_{\bL_2,\sigma}=\sum_{a \in C(\cA)} \sum_{p\in \d a} X_{\sigma(w(p))}X_{\sigma(z(p))}=0\in \cR^-_{\bmP},\] where $C(\cA)$ denotes the set of components of $\cA$.
\end{proof}

\subsection{The category of curved chain complexes} 

Suppose that $\cR$ is a ring of characteristic two. In this section we describe the category of curved chain complexes over $\cR$.

\begin{define} A \emph{curved chain complex over $\cR$} is an $\cR$-module $C$, together with an endomorphism $\d\colon C\to C$ such that
\[
\d^2=\omega\cdot\id_C,
\]
 for some $\omega\in \cR$. The constant $\omega$ is called the \emph{curvature constant}.
\end{define}

We now establish some terminology (which is standard in the context of genuine chain complexes):

\begin{define} We say  $F\colon (C_1,\d_1)\to (C_2,\d_2)$ is a \emph{morphism of curved chain complexes} over $\cR$ if $F:C_1\to C_2$ is an $\cR$-module homomorphism. We say a morphism $F\colon (C_1,\d_1)\to (C_2,\d_2)$ is a \emph{homomorphism of curved chain complexes} if $F$ is also a \emph{chain map}, i.e.,
\[F\circ \d_1=\d_2\circ F.\] An \emph{$\cR$-equivariant chain homotopy} between two morphisms  $F$, $G\colon (C_1,\d_1)\to (C_2,\d_2)$ is an $\cR$-module homomorphism $H\colon C\to C'$ such that $F+G=\d_2\circ H+H\circ\d_1.$
\end{define}

We note that there are not many interesting homomorphisms between curved chain complexes with different curvatures:

\begin{lem}\label{lem:notalotofcurvedmaps} Suppose that $(C_1,\d_1)$ and $(C_2,\d_2)$ are two curved chain complexes which are free over a domain $\cR$, with curvature constants $\omega_1$ and $\omega_2$, respectively, and suppose that $F\colon C_1\to C_2$ is a homomorphism of curved chain complexes over $\cR$. If $\omega_1\neq \omega_2$, then $F=0$.
\end{lem}
\begin{proof} Since $F$ is a homomorphism of curved chain complexes, we have $\d_2\circ F=F\circ \d_1$. Precomposing with $\d_1$, and using the fact that $F$ is a chain map, we see that
\[\d_2^2\circ F=F\circ \d_1^2,\] so that $(\omega_1+\omega_2)\cdot F=0$. Since $C_1$ and $C_2$ are free modules over the domain $\cR$, it follows that $F$ is the zero map. 
\end{proof}

\begin{rem} In light of Lemma~\ref{lem:notalotofcurvedmaps}, one should not expect a TQFT to give meaningful chain maps between curved chain complexes with different curvatures. Lemma~\ref{lem:cobordantlinkshavesamecoloring} implies that if two colored links are cobordant via a decorated link cobordism, then their curvature constants are the same. In Lemma~\ref{lem:del^2=0}, below, we identify the constant $\omega_{\bL}$ with the curvatures of the link Floer chain complexes $\cCFL^-(Y,\bL,\frs)$.
\end{rem}

Given two curved chain complexes $\cC_1=(C_1,\d_1)$ and $\cC_2=(C_2,\d_2)$ over $\cR$ (of arbitrary curvatures) we can consider the \emph{curved complex of morphisms}
\[\Mor_{\cR}(\cC_1,\cC_2),\] consisting of $\cR$-module homomorphisms from $C_1$ to $C_2$. A differential on $\Mor_{\cR}(\cC_1,\cC_2)$  can be defined by the formula
\[\d_{\Mor}(f)=\d_2\circ f+f\circ \d_1.\] 

\begin{lem}\label{lem:morcomplex}If $\cC_1$ and $\cC_2$ are two curved chain complexes over $\cR$, with curvatures $\omega_1$ and $\omega_2$, then $\Mor_{\cR}(\cC_1,\cC_2)$ is a curved chain complex with curvature $\omega_1+\omega_2$.
\end{lem}
\begin{proof}One simply computes 
\[(\d_{\Mor}\circ \d_{\Mor})(f)=\d_2^2\circ f+\d_2\circ f\circ \d_1+\d_2\circ f\circ \d_1+f\circ \d_1^2=(\omega_1+\omega_2)\cdot f.\]
\end{proof}

An important algebraic structure on the link Floer chain complexes is the filtration. The standard definition of a filtered chain complex is suitable for our needs:

\begin{define}A $\Z^n$ filtration on a curved chain complex $(C,\d)$ over $\cR$ is a collection of subgroups $\cG_I\subset C$ ranging over $I\in \Z^n$ such that $\d (\cG_I)\subset \cG_I$ and such that if $I\le J$ (i.e., $I(i)\le J(i)$ for each $i\in \{1,\dots, n\}$) then $\cG_I\supset \cG_J$. 
\end{define}

Suppose $\cC_1=(C_1,\d_1)$ and $\cC_2=(C_2,\d_2)$ are two curved chain complexes over $\cR$, which are also both filtered over $\Z^n$. Write $\cG_I(C_i)$ for the subgroups of $C_i$ associated to their filtrations. There is an induced filtration over $\Z^n$ of the morphism complex $\Mor_{\cR}(\cC_1,\cC_2)$. If $I\in \Z^n$ is a multi-index, then a filtration on the morphism complex can be defined by the formula
\[\cG_I(\Mor_{\cR}(\cC_1,\cC_2)):=\left\{f\in \Mor_{\cR}(\cC_1,\cC_2): f(\cG_J(C_1))\subset \cG_{I+J}(C_2) \text{ for all } J\in \Z^n\right\}.\]

Note that by Lemma~\ref{lem:morcomplex}, if $\cC_1$ and $\cC_2$ are curved chain complexes with the same curvatures, then $\Mor_{\cR}(\cC_1,\cC_2)$ is a genuine chain complex. A homomorphism of curved chain complexes from $\cC_1$ to $\cC_2$ is the same as a cycle in the morphism complex $\Mor_{\cR}(\cC_1,\cC_2)$. A filtered map from $\cC_1$ to $\cC_2$ is the same as an element in $\cG_{\ve{0}}(\Mor_{\cR}(\cC_1,\cC_2))$, where $\ve{0}\in \Z^n$ denotes the zero vector. A filtered chain map, defined up to filtered, $\cR$-equivariant chain homotopies, is the same as an element in the homology group
\[H_*(\cG_{\ve{0}}(\Mor_{\cR}(\cC_1,\cC_2))).\]

\begin{rem}The link Floer complexes we consider in this paper will in general have non-zero curvature. However,  by choosing a coloring $(\sigma,\bmP)$ of the links and decorated surfaces appropriately, one can ensure that all the link Floer chain complexes  are genuine chain complexes (i.e. have curvature constant $0\in \cR$). For example, in the case that there are only two variables, one for the $\ws$-basepoints and one for the $\zs$-basepoints, the curvature constant is zero. This situation is sufficient for most of our applications. Nonetheless, many properties of the maps appearing in the link Floer TQFT are more natural to consider in the context of curved chain complexes.
\end{rem}

\subsection{Transitive systems of curved chain complexes}
We note that we cannot quite work on just the level of chain complexes, since the modules $\cCFL^-(Y,\bL^\sigma,\frs)$ require several choices of auxiliary data. To define the invariants $\cCFL^-(Y,\bL^\sigma,\frs)$ precisely, we need the following formalism:

\begin{define}\label{def:transitivesystemoverC} If $\cC$ is an category and $A$ is a set, a \emph{transitive system over $\cC$ indexed by $A$} is a map $C\colon A\to \cC$, together with distinguished morphisms 
\[\Phi_{a\to a'}\colon C(a)\to C(a')\] for $a,a'\in A$ such that
\begin{itemize}
\item $\Phi_{a\to a}=\id$;
\item $\Phi_{a'\to a''}\circ \Phi_{a\to a'}=\Phi_{a\to a''}$.
\end{itemize}
If $C_1$ and $C_2$ are two transitive systems over $\cC$, indexed by $A_1$ and $A_2$ respectively, then a \emph{morphism of transitive systems} over $\cC$ is a collection of morphisms $F_{a_1,a_2}$, indexed by pairs $(a_1,a_2)\in A_1\times A_2$, such that
\[F_{a_1,a_2}\colon C_1(a_1)\to C_2(a_2)\] and the following diagram commutes for any two pairs $(a_1,a_2),(a_1',a_2')\in A_1\times A_2$:
\[\begin{tikzcd}
C_1(a_1)\arrow{r}{F_{a_1,a_2}}\arrow{d}{\Phi_{a_1\to a_1'}}& C_2(a_2)\arrow{d}{\Phi_{a_2\to a_2'}}\\
C_1(a_1')\arrow{r}{F_{a_1',a_2'}}& C_2(a_2')
\end{tikzcd}.\]

\end{define}
Note that a morphism between transitive systems is determined by a single one of the morphisms $F_{a_1,a_2}$.

\section{The link Floer complexes and some basic properties}\label{sec:complexes}

In this section, we define the curved link Floer chain complexes $\cCFL^\infty(Y,\bL^\sigma,\frs)$ which feature in this paper.

\subsection{Heegaard diagrams for multi-based links}

We use the following standard definition of a Heegaard diagram of a multi-based link:

\begin{define}\label{def:Heegaarddiagramlink}If $\bL=(L,\ve{w},\ve{z})$ is a multi-based link in $Y^3$, an \emph{embedded Heegaard diagram} of $(Y,\bL)$ is a tuple  $(\Sigma, \ve{\alpha},\ve{\beta},\ve{w},\ve{z})$ such that the following hold:
\begin{enumerate}
\item $\Sigma$ is an oriented, embedded surface in $Y$, which divides $Y$ into two handlebodies, $U_{\as}$ and $U_{\bs}$.
\item $\Sigma$ is oriented as $\d U_{\as}=-\d U_{\bs}$.
\item $\Sigma\cap L=\ve{w}\cup \ve{z}$ and  $L$ intersects $\Sigma$ positively at the $\zs$-basepoints and negatively at the $\ws$-basepoints.
\item \label{HD:cond3}$\as$ consists of $g(\Sigma)+|\ws|-1$ pairwise disjoint, simple closed curves which bound a set of $g(\Sigma)+|\ws|-1$ pairwise disjoint compressing disks $\cD_{\as}\subset U_{\as}\setminus L$. The set $U_{\as}\setminus \cD_{\as}$ consists of $|\ws|$ 3-balls, each of which contains a single $\ws$-basepoint and a single $\zs$-basepoint. An analogous condition holds for the $\bs$ curves.
\item \label{HD:cond7}The arcs $L\cap U_{\as}$  can be isotoped within $U_{\as}\setminus \cD_{\as}$,  relative to their endpoints, into $\Sigma$. The arcs $L\cap U_{\bs}$ satisfy an analogous condition.
\end{enumerate}

\end{define}

We remark that Condition~\eqref{HD:cond7} provides a way of visualizing the link. First, one connects the $\zs$-basepoints to the $\ws$-basepoints in $\Sigma\setminus \as$, and then pushes the interior of these arcs into $U_{\as}$. Then one connects the $\ws$-basepoints to the $\zs$-basepoints in $\Sigma\setminus \bs$, and pushes the interiors of these arcs into $U_{\bs}$. Connecting up these two collections of arcs yields the link $L$ (up to isotopy).

Writing $n$ for the quantity $|\ws|=|\zs|$, there are  two embedded tori in $\Sym^{g+n-1}(\Sigma)$, namely
\[\bT_{\as}=\alpha_1\times \cdots \times \alpha_{g+n-1}\quad \text{and} \quad \bT_{\bs}=\beta_1\times \cdots \times \beta_{g+n-1}.\]

Given a Heegaard diagram $(\Sigma,\as,\bs,\ws,\zs)$, and an intersection point $\xs\in \bT_{\as}\cap \bT_{\bs}$, we define the set $\Pi_{\xs}^{\ws}$ to be the set of homotopy classes  $\phi\in \pi_2(\xs,\xs)$ which have $n_{w}(\phi)=0$ for all $w\in \ws$. There is a map
\[
H\colon \Pi_{\xs}^{\ws}\to H_2(Y;\Z),
\]
obtained by coning off a 2-chain $\phi$ using the $\as$ and $\bs$ curves. Throughout the paper, we will restrict attention to diagrams which satisfy the following admissibility condition (which is a slight adaptation of \cite{OSDisks}*{Definition~4.10}):

\begin{define} A diagram $\cH=(\Sigma,\as,\bs,\ws,\zs)$ is \emph{strongly $\frs$-admissible} if there is an intersection point $\xs\in \bT_{\as}\cap \bT_{\bs}$ such that $\frs_{\ws}(\xs)=\frs$ and for any class $\phi\in \Pi_{\xs}^{\ws}$ with
\[\langle c_1(\frs), H(\phi)\rangle=2N>0\] the class $\phi$ has a multiplicity which is at least $N$.
\end{define}

The above admissibility condition implies that for a fixed $\Spin^c$ structure, the set of homotopy classes with Maslov index $j$ which have nonnegative domains is finite  \cite{OSDisks}*{Lemma~4.14} (see also \cite{ZemGraphTQFT}*{Lemma~4.9} for an adaptation of their argument for multi-pointed diagrams).

\subsection{3-dimensional \texorpdfstring{$\Spin^c$}{Spinc} structures}

As with the Heegaard Floer invariants of closed 3-manifolds \cite{OSDisks}, we use the Turaev interpretation of $\Spin^c$ structures on closed 3-manifolds \cite{TuraevTorsionSpinc}. More precisely,  we define $\Spin^c(Y)$ to be the collection of non-vanishing vector fields on $Y$ modulo the equivalence relation of being isotopic on the complement of a collection of 3-balls. The set $\Spin^c(Y)$ is an affine space over $H^1(Y;\Z)$.

Ozsv\'{a}th and Szab\'{o} describe a way of associating a 3-dimensional $\Spin^c$ structure to an intersection point \cite{OSDisks}. If $(\Sigma, \ve{\alpha},\ve{\beta},\ve{w})$ is an $n$-pointed Heegaard diagram for $Y$, then there is a map
\[\frs_{\ve{w}}\colon \bT_\alpha\cap \bT_\beta\to \Spin^c(Y).\] If $\ve{x}\in \bT_\alpha\cap \bT_\beta$ is an intersection point, then $\frs_{\ve{w}}(\ve{x})$ is constructed as follows: One first picks a Morse function which induces the diagram $(\Sigma, \ve{\alpha},\ve{\beta})$ and considers the gradient vector field. Note that such a Morse function will have $|\ws|$ local minima, and $|\ws|$ local maxima. Each point in $\ve{x}$ determines a flowline connecting an index 1 and index 2 critical point of the Morse function. The basepoints $\ws$ each determine a flowline from an index 0 critical point to an index 3 critical point. After removing regular neighborhoods of these flowlines, we obtained a non-vanishing vector field on the complement of this neighborhood. We can extend this to a non-vanishing vector field on all of $Y$ since the flowlines we removed connected critical points with indices of opposite parity.

 The map $\frs_{\ve{w}}$ depends on the choice of basepoints, $\ve{w}$, in our Heegaard diagram. The dependence on the basepoints can be summarized as follows:

\begin{lem}\label{lem:changeSpincstructure}Suppose $(\Sigma, \ve{\alpha},\ve{\beta},\ve{w},\ve{z})$ is a diagram for the multi-based link $(Y,L,\ve{w},\ve{z})$. If $\ve{x}\in \bT_\alpha\cap \bT_\beta$, then
\[\frs_{\ve{w}}(\ve{x})-\frs_{\ve{z}}(\ve{x})=\PD[L].\]
\end{lem}
\begin{proof} This follows from reasoning identical to the proof of \cite{OSDisks}*{Lemma~2.19}. Alternatively, $\Spin^c(Y)$ is described as the set of homology classes of nonvanishing vector fields on $Y$. Given a closed path $\gamma\in H_1(Y;\Z)$, the action of $\gamma\in H_1(Y;\Z)$, can be described by a procedure sometimes referred to as \text{Reeb surgery} (cf. \cite{Nicolasecu}*{Section~3.2}). Given the explicit description from \cite{OSDisks} of the vector fields $\frs_{\ve{w}}$ and $\frs_{\ve{z}}$, it is easy to see that they differ by Reeb surgery on the link $L$.
\end{proof}

It follows from the previous lemma that if the total class of a link is trivial in $H_1(Y;\Z)$, then the two maps $\frs_{\ve{w}}$ and $\frs_{\zs}$ agree.  Note that we could instead have defined our link Floer complexes to be generated by intersection points $\xs$ with $\frs_{\zs}(\xs)=\frs$. For a link whose total class is non-zero, this does not change the resulting link Floer complexes. For links whose total class is not null-homologous, the resulting complexes may be different. An instructive example is $Y=S^1\times S^2$ with $K=S^1\times \{pt\}$.

\subsection{The uncolored link Floer complexes}

We now describe the transitive systems of curved chain complexes $\cCFL^\infty(Y,\bL,\frs)$ associated to a 3-manifold with a multi-based link. The construction is an adaptation of the construction of link Floer homology from \cite{OSLinks}.

Suppose that $\bL=(L,\ws,\zs)$ is a multi-based link in $Y$. Write $\bF_2[U_{\ws},V_{\zs},U_{\ws}^{-1},V_{\zs}^{-1}]$ denote for the ring obtained from $\bF_2[U_{\ws},V_{\zs}]$ by inverting the variables $U_{w_1},\dots, U_{w_n}$, $V_{z_1}, \dots, V_{z_n}$.

Suppose that $\cH=(\Sigma,\as,\bs,\ws,\zs)$ is a diagram for $(Y,\bL)$. As a module, we define
\[
\cCFL^{-}(\cH,\frs)
\] 
to be the free $\bF_2[U_{\ve{w}},V_{\ve{z}}]$-module generated by the intersection points $\ve{x}\in \bT_{\as}\cap \bT_{\bs}$ with $\frs_{\ws}(\ve{x})=\frs$. Over $\bF_2$, the complex $\cCFL^-(\cH,\frs)$ is generated by the elements
\[
U_{\ve{w}}^I V_{\ve{z}}^J\cdot \ve{x},
\]
 with $\frs_{\ve{w}}(\ve{x})=\frs$ and nonnegative multi-indices $I\in \Z^{\ws}$ and $J\in \Z^{\zs}$. The module 
\[
\cCFL^\infty(\cH,\frs)
\] 
is defined similarly over the ring $\bF_2[U_{\ws}, V_{\zs}, U_{\ws}^{-1}, V_{\zs}^{-1}]$ and is generated over $\bF_2$ by elements $U_{\ve{w}}^I V_{\ve{z}}^J$ with arbitrary multi-indices $I$ and $J$.

The modules $\cCFL^-(\cH,\frs)$ and $\cCFL^\infty(\cH,\frs)$ have natural filtrations by $\Z^{\ws}\oplus \Z^{\zs}$, which are given by powers of the variables. More explicitly, if $I\in \Z^{\ws}$ and $J\in \Z^{\zs}$, there is a subset $\cG_{I,J}\subset \cCFL^\infty(\cH,\frs)$ generated over $\bF_2$ by tuples $U_{\ws}^{I'} V_{\zs}^{J'}\cdot \xs$ with $I'\ge I$ and $J'\ge J$.

We note that Ozsv\'{a}th and Szab\'{o}'s construction of a relative homological grading from \cite{OSDisks} on the closed 3-manifold invariants yields two gradings on $\cCFL^-(Y,\bL,\frs)$. If $\phi\in \pi_2(\xs,\ys)$, we can define two gradings on $\Hat{\CFL}(Y,\bL,\frs)$ via the formulas
\begin{equation}
\gr_{\ve{w}}(\ve{x},\ve{y})=\mu(\phi)-2\sum_{w\in \ve{w}}n_{w}(\phi),\qquad \text{and} \qquad \gr_{\ve{z}}(\ve{x},\ve{y})=\mu(\phi)-2\sum_{z\in \ve{z}}n_{z}(\phi).
\label{eq:defgradings}
\end{equation}  If $c_1(\frs)$ is torsion, then $\gr_{\ve{w}}$ is independent of the disk $\phi$, using the formula for the Maslov index of a periodic class \cite{OSDisks}*{Theorem~4.9}. Analogously, using Lemma~\ref{lem:changeSpincstructure}, if $c_1(\frs-\PD[L])$ is torsion, then $\gr_{\ve{z}}$ will be independent of the disk $\phi$. We can extend these gradings to all of $\cCFL^\infty(Y,\bL,\frs)$ by declaring the $U_{\ve{w}}$ and $V_{\zs}$ variables to be $-2$ and $0$ graded (resp.) with respect to $\gr_{\ws}$. Similarly we declare the $U_{\ws}$ and $V_{\zs}$ variables to be $0$ and $-2$ graded with respect to $\gr_{\zs}$.

Given an appropriately generic path of almost complex structure $(J_s)_{s\in [0,1]}$ on $\Sym^{n+g-1}(\Sigma)$,   we define an endomorphism 
\[
\d\colon\cCFL^-(\cH,\frs)\to \cCFL^-(\cH,\frs)
\]
 by counting Maslov index 1 holomorphic strips via the formula
\[
\d(\ve{x})=\sum_{\ve{y}\in \bT_\alpha\cap \bT_\beta}\sum_{\substack{\phi\in \pi_2(\ve{x},\ve{y})\\ \mu(\phi)=1}} \# \Hat{\cM}(\phi) U_{\ve{w}}^{n_{\ve{w}}(\phi)} V_{\ve{z}}^{n_{\ve{z}}(\phi)}\cdot \ve{y},
\]
extended linearly over $\cR_{\bmP}^-$. 
Since pseudo-holomorphic disks must have nonnegative domains, the differential respects the filtration over $\Z^{\ws}\oplus \Z^{\zs}$, given by powers of the variables, defined above.

 The map $\d$ does not square to zero, in general. Instead  $(\cCFL^-(\cH,\frs), \d)$ is a curved chain complex:

\begin{lem}[\cite{ZemQuasi}*{Lemma~2.1}]\label{lem:del^2=0}The map $\d\colon \cCFL^-(\cH,\frs)\to \cCFL^-(\cH,\frs)$ satisfies
\[
\d^2(\ve{x})=\omega_{\bL}\cdot \id,
\] where
\[
\omega_{\bL}=\sum_{K\in C(L)} \left(U_{w_{K,1}}V_{z_{K,1}}+V_{z_{K,1}}U_{w_{K,2}}+U_{w_{K,2}}V_{z_{K,2}}+\cdots +V_{z_{K,n_K}}U_{w_{K,1}}\right)\cdot \ve{x}.
\]
 Here $w_{K,1},$ $z_{K,1},$ \dots, $w_{K,n_K},$  $z_{K,n_K}$ are the basepoints on the link component $K$, in the order that they appear on $K$.
\end{lem}

An important property of Heegaard Floer homology is that it is independent of the choice of Heegaard diagram. For our purposes, we will need the following naturality result:

\begin{prop}Given two strongly $\frs$-admissible diagrams $\cH$ and $\cH'$ for $(Y,\bL)$, there is a $\Z^{\ws}\oplus \Z^{\zs}$-filtered, $\bF_2[U_{\ws}, V_{\zs}]$-equivariant chain homotopy equivalence
\[\Phi_{\cH\to \cH'}\colon \cCFL^-(\cH,\frs)\to \cCFL^-(\cH',\frs),\] which is well-defined up to filtered, equivariant chain homotopies.
\end{prop}
\begin{proof}In the setting of the knot and link invariants \cite{OSKnots} \cite{OSLinks}, Ozsv\'{a}th and Szab\'{o} constructed maps associated to changes of the almost complex structure, isotopies and handleslides of the $\as$ and $\bs$ curves, stabilizations of the Heegaard surface, as well as isotopies of the Heegaard surface relative to $L$. They showed that these maps were quasi-isomorphisms. If $\cH$ and $\cH'$ are arbitrary diagrams, one can construct a transition map $\Phi_{\cH\to \cH'}$ as the composition of the maps associated to a sequence of elementary Heegaard moves connecting $\cH$ and $\cH'$. Juh\'{a}sz and Thurston \cite{JTNaturality} showed that the maps $\Phi_{\cH\to \cH'}$ are independent, up to chain homotopy, on the choice of intermediate diagrams. We note that neither Ozsv\'{a}th--Szab\'{o} nor Juh\'{a}sz--Thurston explicitly showed that these transition maps are chain homotopy equivalences, as opposed to quasi-isomorphisms.  To see that the maps are chain homotopy equivalences, we refer the reader to \cite{HMInvolutive}*{Proposition~2.3}, whose proof works equally well in our setting.
\end{proof}

In light of the above naturality result, we define $\cCFL^\infty(Y,\bL,\frs)$ to be the transitive system over $\cC$  indexed by $A$ (cf. Definition~\ref{def:transitivesystemoverC}), where $A$ is the set of all strongly $\frs$-admissible diagrams for $(Y,\bL)$ and $\cC$ is the category of $\Z^{\ve{w}}\oplus\Z^{\ve{z}}$-filtered, $\bF_2[U_{\ve{w}},V_{\ve{z}}]$-equivariant curved chain complexes. We call $\cCFL^\infty(Y,\bL,\frs)$ the \emph{transitive chain homotopy type invariant}.

Finally, we note that for some of the more technical arguments of the paper, we will use Lipshitz's cylindrical reformulation of Heegaard Floer homology \cite{LipshitzCylindrical}. Instead of counting holomorphic curves from the complex disk into $\Sym^{n+g-1}(\Sigma)$, Lipshitz's version of Heegaard Floer homology counts pseudo-holomorphic curves of higher genus which map into $\Sigma\times [0,1]\times \R$. The equivalence of Lipshitz's reformulation with the Ozsv\'{a}th--Szab\'{o} construction follows (at least morally) from the \emph{tautological correspondence} between holomorphic disks mapping into $\Sym^{n+g-1}(\Sigma)$ and holomorphic curves mapping into $\Sigma\times [0,1]\times \R$, whose projection onto $[0,1]\times \R$ are $(n+g-1)$-fold branched covering maps. We refer the reader to \cite{LipshitzCylindrical} for more technical details about this approach.

\subsection{Coloring the link Floer complexes}
 Suppose that
\[\sigma\colon  \ws\cup \zs\to \bmP\] is a coloring of $\bL=(L,\ws,\zs)$. The coloring $\sigma$ induces an $\bF_2[U_{\ws}, V_{\zs}]$-module structure on $\cR_{\bmP}^-$. We define the colored complexes
\[\cCFL^-(\cH,\sigma,\frs):=\cCFL^-(\cH,\frs)\otimes_{\bF_2[U_{\ve{w}},V_{\ve{z}}]} \cR_{\bmP}^-,\] and
\[\cCFL^\infty(\cH,\sigma,\frs):=\cCFL^\infty(\cH,\frs)\otimes_{\bF_2[U_{\ve{w}},U_{\ws}^{-1},V_{\ve{z}}, V_{\zs}^{-1}]} \cR_{\bmP}^\infty.\]

The complexes $\cCFL^\infty(\cH,\sigma,\frs)$ fit together to form a transitive system over $\cC$, indexed by $A$, where $\cC$ is the category of $\Z^{\bmP}$-filtered, curved chain complexes over $\cR_{\bmP}^\infty$, and $A$ is the set of strongly $\frs$-admissible Heegaard diagrams for $(Y,\bL)$. Note that the filtration on $\cCFL^\infty(\cH,\sigma,\frs)$ is given by filtering by powers of the variables in $\cR_{\bmP}^\infty$.

We  note that including colorings is essential for functoriality. Without colorings, we do not have a fixed category to work in. Also, most of the maps we define for decorations on the surface are not chain maps until we color the complexes.

Although we will primarily work with curved chain complexes, we will occasionally need to work on the level of homology. If $\bL^\sigma$ is a colored link with $\omega_{\bL,\sigma}=0\in \bF_2[U_{\bmP}]$ (or equivalently $\d^2=0$ on $\cCFL^\circ(Y,\bL^\sigma,\frs)$), we define
\[\cHFL^\circ(Y,\bL^\sigma,\frs):=H_*(\cCFL^\circ(Y,\bL^\sigma,\frs)).\]

\subsection{Distinguished elements of \texorpdfstring{$\cHFL^-$}{HFL} for multi-based unlinks in \texorpdfstring{$(S^1\times S^2)^{\# n}$}{\# g (S1 x S2)}}
\label{sec:computemodulesunknots}

Suppose $Y$ is a 3-manifold with  link $\bL=(L,\ws,\zs),$ and $\sigma$ is a coloring of $\bL$ such that $\sigma(\ws)\cap \sigma(\zs)=\varnothing$. Furthermore, suppose that the total class of $L$ in $H_1(Y;\Z)$ is null-homologous, and that $\frs$ is a torsion $\Spin^c$ structure on $Y$. In this situation, the relative gradings $\gr_{\ws}$ and $\gr_{\zs}$ described in Equation~\eqref{eq:defgradings} are both well-defined on the  complex $\cCFL^-(Y,\bL^\sigma,\frs)$. If $\ve{o}\in \{\ve{w},\ve{z}\}$ we will let $\Max_{\gr_{\ve{o}}}(\cHFL^-(Y, \bL^\sigma, \frs))$ denote the subset of $\cHFL^-(Y,\bL^\sigma,\frs)$ of maximal $\gr_{\ve{o}}$ grading.

\begin{lem}\label{lem:topdegreeunlink1}Suppose $\bU$ is an unlink in $(S^1\times S^2)^{\# n}$, and each component has exactly two basepoints.  There  are isomorphisms
\[\Max_{\gr_{\ws}}\left(\cHFL^-((S^1\times S^2)^{\# n},\bU, \frs_0)\right)\iso \bF_2[V_{\ve{z}}],\] and
\[\Max_{\gr_{\zs}}\left(\cHFL^-((S^1\times S^2)^{\# n},\bU, \frs_0)\right)\iso \bF_2[U_{\ve{w}}].\] Furthermore, 
\[\Max_{\gr_{\ws}}\left(\cHFL^-((S^1\times S^2)^{\# n},\bU, \frs_0)\right)\cap \Max_{\gr_{\zs}}\left(\cHFL^-((S^1\times S^2)^{\# n},\bU, \frs_0)\right)\] is a 1-dimensional vector space, spanned by an element $\Theta$, which generates $\Max_{\gr_{\ws}}\left(\cHFL^-\right)$ as an $\bF_2[V_{\zs}]$-module, and generates $\Max_{\gr_{\zs}}\left(\cHFL^-\right)$ as an $\bF_2[U_{\ws}]$-module.
\end{lem}
\begin{proof} We can pick a diagram $\cH=(\Sigma,\as,\bs,\ws,\zs)$ for $((S^1\times S^2)^{\# n}, \bU)$ where the $\bs$ curves are all small Hamiltonian translates of the $\as$ curves, and the basepoints in $\ws$ and $\zs$ come in pairs, which are immediately adjacent to each other on the Heegaard diagram, and not contained in any of the  bigons bounded by a pair of  $\as$ and $\bs$ curves.

On the diagram $\cH$, one has $\gr_{\ws}(\xs,\ys)=\gr_{\zs}(\xs,\ys)$ for any pair of intersection points $\xs,\ys\in \bT_{\as}\cap \bT_{\bs}$, since the $\ws$ and $\zs$-basepoints come in adjacent pairs, so the formulas in Equation~\eqref{eq:defgradings} coincide. Furthermore, there is a canonical top degree intersection point $\Theta\in \bT_{\as}\cap \bT_{\bs}$.

Furthermore, it is easy to see that
\[\Max_{\gr_{\ws}}\left(\cCFL^-(\cH, \frs_0)\right)\iso \bF_2[V_{\ve{z}}],\] and
\[\Max_{\gr_{\zs}}\left(\cCFL^-(\cH, \frs_0)\right)\iso \bF_2[U_{\ve{w}}].\] Furthermore, 
\[\Max_{\gr_{\ws}}\left(\cCFL^-(\cH ,\frs_0)\right)\cap \Max_{\gr_{\zs}}(\cCFL^-(\cH, \frs_0))=\Span_{\bF_2} \langle\Theta\rangle.\]

To show the lemma statement, it is thus sufficient to show that $\d (\Theta)=0$ and that no non-trivial sum of terms of the form $V_{\ve{z}}^J\cdot \Theta$ is a cycle, and that no non-trivial sum of terms of the form $U_{\ws}^J \cdot \Theta$ is a cycle.

To see that $ \Theta$ is a cycle, one simply analyzes the proof of \cite{OSDisks}*{Lemma~9.1}. Using the relative Maslov grading, one sees that if $\phi\in \pi_2(\Theta,\ve{y})$ is a Maslov index 1 disk, then
\[2\sum_{w\in \ve{w}} n_{w}(\phi)\le 0,\] hence any such disk admitting holomorphic representatives cannot cross any of the $\ve{w}$-basepoints. It is easy to see that on the diagram $\cH$, the only disks in any $\pi_2(\Theta,\ys)$ which do not cross over any $\ve{w}$-basepoints come in pairs, and hence cancel when counted by the differential. Thus $\d (\Theta)=0$.

To see that no non-trivial sum of terms of the form $ V_{\ve{z}}^J\cdot \Theta$ is a boundary, we note that any disk $\phi\in \pi_2(\ve{y},\Theta)$ satisfies
\[\mu(\phi)-2\sum_{w\in \ve{w}} n_{w}(\phi)\le 0,\] by the relative homological grading formula, and hence
\[2\sum_{w\in \ve{w}} n_w (\phi)\ge 1,\] implying that if $\phi$ has a holomorphic representative which is counted by the differential, it must have nonzero multiplicities on the $\ve{w}$-basepoints. In particular $\phi$ must increase the $U_{\ve{w}}$-filtration, and hence all cycles live in the subset of $\cCFL^-$ of $U_{\ve{w}}$-filtration at least 1, and in particular nothing which is a sum of elements of the form $V_{\ve{z}}^J\cdot \Theta$ is a boundary. 

An identical argument shows that no non-trivial sum of elements of the form $U_{\ws}^J\cdot \Theta$ is a boundary.
\end{proof}

We now consider the case that all components   of the unlink $\bU\subset (S^1\times S^2)^{\# g}$ except for one have exactly two basepoints, and that the remaining component of $\bU$ has exactly four basepoints. These will appear when we define maps associated to saddle cobordisms.

\begin{lem}\label{lem:topdegreeunlink2}Suppose that $\bU$ is an unlink in $(S^1\times S^2)^{\# n}$ such that all components of $\bU$ except for one have exactly two basepoints, and one component of $\bU$ has exactly four basepoints. Let $U$ denote the component of $\bU$ with four basepoints, and let  $w$, $w',$ $z$ and $z'$ denote the 4-basepoints. Define the set of colors $\bmP_0:=\ve{w}\cup (\ve{z}/(z\sim z'))$ where $\ve{z}/(z\sim z')$ denotes the set of $\zs$-basepoints modulo the relation that $z\sim z'$, and define the coloring $\sigma_0\colon \ws\cup \zs\to \bmP_0$ to be the natural map. Then there an isomorphism
\[\Max_{\gr_{\ws}}\left(\cHFL^-((S^1\times S^2)^{\# n}, \bU^{\sigma_0},\frs_0)\right)\iso \bF_2[V_{\ve{z}}]/(V_z-V_{z'}).\] In particular there is a well-defined generator $\Theta^{\ws}\in \Max_{\gr_{\ws}}(\cHFL^-((S^1\times S^2)^{\# n}, \bU^{\sigma_0},\frs_0))$.

Similarly if $\bmP_0':=(\ve{w}/(w \sim w'))\cup \ve{z}$ and $\sigma_0':\ws\cup \zs\to \bmP_0'$ is the natural map, then 
\[\Max_{\gr_{\zs}}\left(\cHFL^-((S^1\times S^2)^{\# n}, \bU^{\sigma_0'},\frs_0)\right)\iso \bF_2[U_{\ve{w}}]/(U_w-U_{w'}).\] In particular there is a well-defined generator $\Theta^{\zs}\in  \Max_{\gr_{\zs}}\left(\cHFL^-((S^1\times S^2)^{\# n}, \bU^{\sigma_0'},\frs_0)\right)$.
\end{lem}

\begin{proof} We will prove the claim about $\Max_{\gr_{\ws}}(\cHFL^-)$, since the result about $\Max_{\gr_{\zs}}(\cHFL^-)$ follows symmetrically. The proof is similar to the proof of Lemma~\ref{lem:topdegreeunlink1}. We will construct a convenient diagram $\cH$ for $((S^1\times S^2)^{\# n}, \bU)$. We start with a diagram $\cH_0=(S^2,\alpha_0,\beta_0,w,w',z,z')$ for $(S^3,U,w,w',z,z')$, obtained taking the Heegaard surface $S^2\subset S^3$, and picking an $\alpha_0$ and $\beta_0$ curve which intersect in two points.

 There are four components of $S^2\setminus (\alpha_0\cup \beta_0)$, all of which are bigons. Each bigon contains one of the basepoints $w$, $w',$ $z$ and $z'$. We can form a diagram $\cH$ for $((S^1\times S^2)^{\# n}, \bU)$ by attaching a genus $n$ surface $\Sigma_n$ to $S^2$ in the bigon on $S^2$ which contains the basepoint $w$. On $\Sigma_n$, we can pick additional attaching curves $\as$ and $\bs$ which are small Hamiltonian translates of each other, intersecting in pairs of points. We can place the additional basepoints on $\Sigma_n$ so that they come in pairs of $\ws$ and $\zs$-basepoints, which are immediately adjacent to each other. An example is shown in Figure~\ref{fig::26}.

There is an intersection point $\Theta^{\ws}$ of maximal $\gr_{\ws}$ grading on $\cH$. The same argument that we used in the proof of Lemma~\ref{lem:topdegreeunlink1} shows that sums of elements of the form $\Theta^{\ws}\cdot V_{\ve{z}}^J$ generate the subset of $\cCFL^-(\cH,\frs)$ of top $\gr_{\ve{w}}$ grading. Also as before, we note that none of these elements are boundaries. 

However, it is important to note that on the uncolored module $\cCFL^-(\cH,\frs_0)$, the element $\Theta^{\ws}$ is not a cycle. Instead $\d (\Theta^{\ws})=(V_z+V_{z'})\cdot \Theta^{\zs}$, as shown in Figure~\ref{fig::26}. On the other hand, the coloring $\sigma_0$ identifies $V_z$ and $V_{z'}$, so $\Theta^{\ws}$ becomes a cycle in the colored chain complex. Hence \[\Max_{\gr_{\ws}}\left(\cHFL^-(\cH,\sigma_0,\frs_0)\right)\iso \Max_{\gr_{\ws}}\left(\cCFL^-(\cH, \sigma_0,\frs)\right)\iso \Max_{\gr_{\ws}}\left(\cCFL^-(\cH,\frs)\right)/(V_{z}-V_{z'}),\] which is clearly isomorphic to $\bF_2[V_{\zs}]/(V_z-V_{z'})$. 
\end{proof}

\begin{figure}[ht!]
\centering
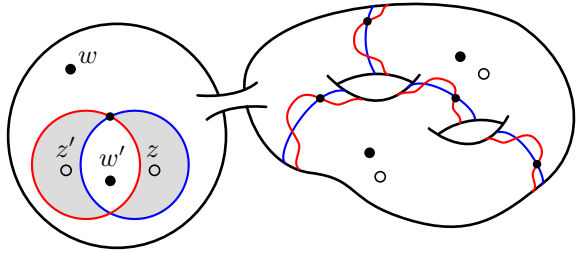
\caption{A diagram $\cH$ for $((S^1\times S^2)^{\# n}, \bU)$ in Lemma~\ref{lem:topdegreeunlink2}. The intersection point $\Theta^{\ws}$ is labeled with solid dots.  The two shaded bigons are disks contributing to the relation $\d(\Theta^{\ws})=(V_{z}+V_{z'})\cdot\Theta^{\zs}$.\label{fig::26}}
\end{figure}

\section{Quasi-stabilization, basepoint actions, and basepoint moving maps} \label{sec:mapswhichappear}

In this section, we describe some important maps which appear in the graph TQFT.

\subsection{Quasi-stabilization maps for adding and removing basepoints}
\label{sec:quasi-stabilization}
In this section, we describe the \emph{quasi-stabilization} operation, initially described in \cite{MOIntSurg}, which is a procedure for adding or removing an adjacent pair of basepoints from a link component. Some basic properties of the quasi-stabilization operation (such as the fact that the quasi-stabilization maps are chain maps, and induce natural chain maps between link Floer complexes) are proven in \cite{ZemQuasi}. 

Suppose that $w$ and $z$ are new basepoints for a link $\bL=(L,\ws,\zs)$, contained in a single component of $L\setminus (\ws\cup \zs)$, such that $w$ immediately follows $z$ with respect to the links orientation. Suppose that $\sigma\colon \ws\cup \zs\to \bmP$ is a coloring and that $\sigma'\colon \ws\cup \zs\cup \{w,z\}\to \bmP$ is a coloring which extends $\sigma$. Suppose also that $z$ is given the same color as the other $\zs$-basepoint adjacent to $w$. We will describe filtered equivariant chain maps
\[S_{w,z}^+\colon \cCFL^-(Y,\bL^\sigma,\frs)\to \cCFL^-(Y,(\bL_{w,z}^+)^{\sigma'},\frs)\] and 
\[S_{w,z}^-\colon \cCFL^-(Y,(\bL_{w,z}^+)^{\sigma'},\frs)\to \cCFL^-(Y,\bL^\sigma,\frs),\] which are well-defined on the transitive chain homotopy type invariants. Here $\bL_{w,z}^+$ denotes the multi-based link $(L,\ws\cup \{w\}, \zs\cup \{z\})$. If instead $z$ follows $w$, there are maps $S_{z,w}^+$ and $S_{z,w}^-$, which are defined analogously. We will call these the \emph{type-$S$} quasi-stabilization maps.

\begin{rem}  Our notation for the quasi-stabilization follows the right-to-left convention described in Section~\ref{subsec:convention}. If we write $S_{w,z}^\circ$, the implication is that $w$ immediately follows $z$. If we write $S_{z',w'}^\circ$, the implication is that $z'$ immediately follows $w'$.
\end{rem}

The construction of the type-$S$ quasi-stabilization map is not symmetric between the $\ws$ and the $\zs$-basepoints. Instead, by reversing the roles of $\ws$ and $\zs$-basepoints in the construction, we will construct \emph{type-$T$} quasi-stabilization maps $T_{w,z}^+$ and $T_{w,z}^-$ as well. If  $\sigma\colon \ws\cup \zs\to \bmP$ is a coloring which is extended by $\sigma'\colon \ws\cup \zs\cup \{w,z\}\to\bmP$  such that $w$ is given the same color as the other $\ws$ basepoint adjacent to $z$, we will describe filtered, equivariant chain maps
\[T_{w,z}^+\colon \cCFL^-(Y,\bL^\sigma,\frs)\to \cCFL^-(Y,(\bL_{w,z}^+)^{\sigma'},\frs)\] and 
\[T_{w,z}^-\colon \cCFL^-(Y,(\bL_{w,z}^+)^{\sigma'},\frs)\to \cCFL^-(Y,\bL^\sigma,\frs).\]

We now provide the description of both the type-$S$ and type-$T$ quasi-stabilization maps on the level of intersection points and Heegaard diagrams.

Suppose we are given a Heegaard diagram $\cH=(\Sigma, \ve{\alpha},\ve{\beta},\ve{w},\ve{z})$ for $(Y,\bL)$, and $w$ and $z$ are new basepoints on $\bL$, which are contained in a single component of $L\setminus (\ws\cup \zs)$. Furthermore, suppose that   $w$ comes after $z$. Write $w'$ and $z'$ for the two basepoints of $\bL$ which are adjacent to $w$ and $z$. Since $w$ appears immediately after $z$ on $L$, if follows from Definition~\ref{def:Heegaarddiagramlink} that the component of $L\setminus (\ws\cup \zs)$ which contains $w$ and $z$ is contained in the $U_{\as}$ handlebody of $Y\setminus \Sigma$.

  Let $A$ denote the component of $\Sigma\setminus \ve{\alpha}$ containing $w'$ and $z'$. We pick a point $p\in A\setminus (\ve{\alpha}\cup \ve{\beta}\cup \ve{w}\cup \ve{z})$ and a choice of new $\alpha_s$ curve in $A$ passing through $p$ which cuts $A$ into two pieces, one of which contains $w'$ and one of which contains $z'$. Note that $\alpha_s$ can in general intersect many different $\bs$ curves.

Let $\cH_0=(D^2,\alpha_0,\beta_0,w,z)$ be a piece of a Heegaard diagram, where $D^2$ is a disk, $\alpha_0$ is an arc with both endpoints on $\d D^2$, which cuts $D^2$ into two pieces, one of which contains $w$ and one of which contains $z$. Furthermore $\beta_0$ is a closed curve which intersects $\alpha_0$ twice and bounds a disk containing the basepoints $w$ and $z$. We then paste $\cH_0$ onto $\cH$ at $p$, so that $\alpha_0$ is lined up with $\alpha_s$. This operation is called the \emph{special connected sum} in \cite{MOIntSurg}. An example of a quasi-stabilization is shown in Figure~\ref{fig::31}.

\begin{figure}[ht!]
\centering
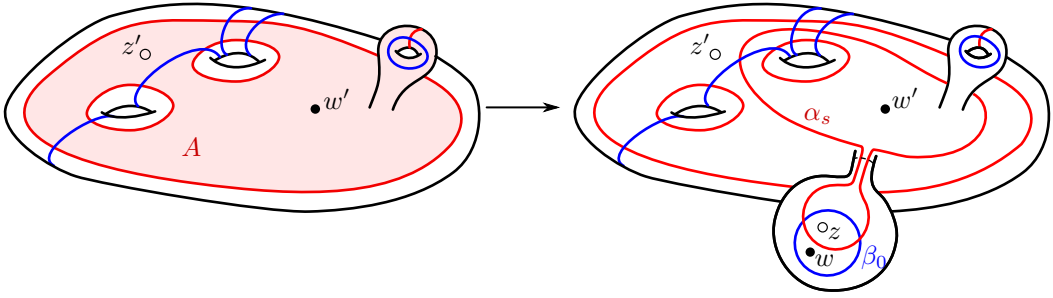
\caption{\textbf{Quasi-stabilizing a diagram.} On the left, the region in the unstabilized diagram $A\subset \Sigma\setminus \as$ which contains $w'$ and $z'$ is shown. On the right, a quasi-stabilization is shown which adds the two basepoints $w$ and $z$ between $w'$ and $z'$.\label{fig::31}}
\end{figure}

Note that $\alpha_s\cap \beta_0$ consists of two points, which are distinguished by the gradings $\gr_{\ws}$ and $\gr_{\zs}$. We will write $\theta^{\ve{w}}$ for the higher $\gr_{\ws}$-graded intersection point and $\xi^{\ws}$ for the lower. Somewhat redundantly, we will write $\theta^{\ve{z}}$ for the higher $\gr_{\zs}$-graded intersection point, and $\xi^{\zs}$ for the lower. By inspection, we have that
\[\theta^{\ve{w}}=\xi^{\zs} \qquad \text{and} \qquad \theta^{\zs}=\xi^{\ws}.\]

The type-$S$ quasi-stabilization maps are defined by the formulas
\[
S_{w,z}^+(\ve{x})=\ve{x}\otimes \theta^{\ws}
\] and
\[
S_{w,z}^-(\ve{x}\otimes \theta^{\ws})=0\qquad \text{and} \qquad S_{w,z}^-(\ve{x}\otimes \xi^{\ve{w}})=\ve{x},
\] extended $\cR_{\bmP}^-$-equivariantly. Similarly, we define the type-$T$ quasi-stabilization maps via the formulas
\[
T_{w,z}^+(\ve{x})=\ve{x}\otimes \theta^{\ve{z}}
\] 
and
\[
T_{w,z}^-(\ve{x}\otimes \theta^{\ve{z}})=0\qquad \text{and} \qquad T_{w,z}^-(\ve{x}\otimes \xi^{\ve{z}})=\ve{x},
\]
extended equivariantly over $\cR_{\bmP}^-$.

If $w$ and $z$ are two new adjacent basepoints, but $z$ follows $w$, then a slight adaptation of the above construction yields maps $S_{z,w}^+,$ $S_{z,w}^-,$ $T_{z,w}^+$ and $T_{z,w}^-$. In this situation, if we let $w'$ and $z'$ denote the two basepoints of $\bL$ which are adjacent to $w$ and $z$, then $w'$ and $z'$ are contained in a single component $B\subset \Sigma\setminus \bs$. The quasi-stabilization maps are then defined in this case by picking a point $p\in B\setminus (\as\cup \bs\cup \{w',z'\})$ and a simple closed curve $\beta_s\subset B\setminus \bs$ which intersects $p$ and separates $w'$ from $z'$. The quasi-stabilization construction is then performed by inserting a new $\alpha_0$ curve intersecting $\beta_s$ twice, and bounding the new basepoints $w$ and $z$. The same formulas as before are used to define the maps $S_{z,w}^+,$ $S_{z,w}^-,$ $T_{z,w}^+$ and $T_{z,w}^-$.

\begin{rem}The  quasi-stabilization maps are only defined when $w$ and $z$ are not the only basepoints on their link component, since link Floer homology is only defined when all components of a link have at least two basepoints.
\end{rem}

The first step towards analyzing the quasi-stabilization maps is to understand when they are chain maps. In \cite{ZemQuasi}, the following is proven using a neck-stretching argument:

\begin{prop}[\cite{ZemQuasi}*{Proposition~5.3}]\label{prop:quasi-stabilizeddifferential}Suppose $\cH$ is a diagram for $(Y,\bL)$, and let $\cH^+$ be a diagram for $(Y,\bL_{w,z}^+)$, obtained by quasi-stabilizing $\cH$ by adding $w$ and $z$ between the basepoints $w'$ and $z'$ on $\bL$. As modules, there is an isomorphism
\[
\cCFL^-(\cH^+,\frs)\iso \cCFL^-(\cH,\frs)\otimes_{\bF_2} \langle \theta^{\ws}, \xi^{\ws}\rangle\otimes_{\bF_2} \bF_2[U_w,V_z],
\] where $\langle \theta^{\ws},\xi^{\ws}\rangle$ denotes the 2 dimensional vector space over $\bF_2$ spanned by $\theta^{\ws}$ and $\xi^{\ws}$. Furthermore, for sufficiently stretched almost complex structures, there is an identification of differentials
\[\d_{\cH^+}=\begin{pmatrix}\d_{\cH}& U_w+U_{w'}\\
V_{z}+V_{z'}& \d_{\cH}
\end{pmatrix}.\] The above matrix is in terms of the generators $\theta^{\ve{w}}$ and $\xi^{\ws}$, with $\theta^{\ve{w}}$ corresponding to the first row and column, and $\xi^{\ve{w}}$ corresponding to the second row and column.
\end{prop}

\begin{cor}Suppose that $\bL=(L,\ws,\zs)$ is a multi-based link in $Y$, and $w$ and $z$ are two new basepoints added between $w'\in \ws$ and $z'\in \zs$. If $\sigma'\colon \ws\cup \zs\cup \{w,z\}\to \bmP$ is a coloring, then
\begin{enumerate}
\item $S_{w,z}^+$ and $S_{w,z}^-$ are chain maps if and only if $\sigma'(z)=\sigma'(z')$;
\item $T_{w,z}^+$ and $T_{w,z}^-$ are chain maps if and only if $\sigma'(w)=\sigma'(w')$.
\end{enumerate}
\end{cor}

We have the following naturality result, concerning the quasi-stabilization maps:

\begin{thm}[\cite{ZemQuasi}*{Theorem~A}]\label{thm:quasisarenatural} If $\cH_1$ and $\cH_2$ are two diagrams for $(Y,\bL)$ and $\cH_1^+$ and $\cH_2^+$ are two diagrams for $(Y,\bL_{w,z}^+)$, obtained by quasi-stabilizing $\cH$ with the basepoints $w$ and $z$, then the following diagram commutes up to chain homotopy:
\[
\begin{tikzcd}
\cCFL^-(\cH_1,\sigma,\frs)\arrow{d}{S_{w,z}^+}\arrow{r}{\Phi_{\cH_1\to \cH_2}} &\cCFL^-(\cH_2,\sigma,\frs)\arrow{d}{S_{w,z}^+}\\
\cCFL^-(\cH_1^+,\sigma',\frs)\arrow{r}{\Phi_{\cH_1^+\to \cH_2^+}}& \cCFL^-(\cH_2^+,\sigma',\frs)
\end{tikzcd}
\]
 Here $\sigma$ and $\sigma'$ are colorings of $\bL$ and $\bL_{w,z}^+$, respectively, such that 
 $\sigma=\sigma'|_{\ws\cup \zs}$ and $z$ is given the same coloring as the other $\zs$ basepoint adjacent to $z$. Analogous statements hold for the maps $S_{w,z}^-,$ $T_{w,z}^+$, and $T_{w,z}^-$.
\end{thm}

We note that in \cite{ZemQuasi}, only the maps $S_{w,z}^+$ and $S_{w,z}^-$ were considered. Furthermore, they were considered on a version of the link Floer complex denoted $\CFL^\circ_{UV}$, obtained by setting all of the $V_{\zs}$ variables on a link component $K$ equal to a single variable $V_K$. Nonetheless, the proof of \cite{ZemQuasi}*{Theorem~A} adapts with only minor notational changes to prove Theorem~\ref{thm:quasisarenatural}.

\subsection{The basepoint actions }
\label{subsec:relations1}

Two  maps which naturally appear in the link Floer TQFT are the maps $\Phi_w$ and $\Psi_z$. We define these via the formulas
\[\Phi_w(\ve{x})=U_w^{-1} \sum_{\ve{y}\in \bT_\alpha\cap \bT_\beta} \sum_{\substack{\phi\in \pi_2(\ve{x},\ve{y})\\ \mu(\phi)=1}} n_w(\phi)\# \Hat{\cM}(\phi)U_{\ve{w}}^{n_{\ve{w}}(\phi)}V_{\ve{z}}^{n_{\ve{z}}(\phi)}\cdot \ve{y},\] and
\[\Psi_z(\ve{x})=V_z^{-1} \sum_{\ve{y}\in \bT_\alpha\cap \bT_\beta} \sum_{\substack{\phi\in \pi_2(\ve{x},\ve{y})\\ \mu(\phi)=1}} n_z(\phi)\# \Hat{\cM}(\phi)U_{\ve{w}}^{n_{\ve{w}}(\phi)}V_{\ve{z}}^{n_{\ve{z}}(\phi)}\cdot \ve{y}.\]
Note that on the uncolored chain complexes, these can be thought of as ``formal derivatives'' of the differential with respect to the variables $U_w$ or $V_z$. Since the maps $d/d U_w$ and $d/ d V_z$ are not $\bF_2[U_{\ws},V_{\zs}]$-equivariant, the interpretation of $\Phi_w$ and $\Psi_z$ as derivatives of the differential may disappear once we tensor the complexes with $\cR_{\bmP}^-$ to form the colored complexes.

It is not hard to see that these maps commute with the change of diagrams maps, up to chain homotopy, and induce well-defined maps on the transitive chain homotopy type invariants. A proof of this fact is written down in \cite{ZemQuasi}*{Lemma~3.2}. The maps $\Phi_w$ and $\Psi_z$ will turn out to be the cobordism maps of  $([0,1]\times Y ,[0,1]\times  L)$, with dividing sets shown in Figure~\ref{fig::116}. 

\begin{figure}[ht!]
\centering
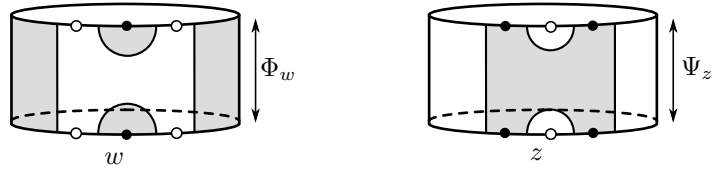
\caption{\textbf{Decorated link cobordisms for the maps $\Phi_w$ and $\Psi_z$.} The underlying link cobordism is $([0,1]\times L, [0,1]\times Y)$.\label{fig::116}}
\end{figure}

 We now summarize several important properties of the maps $\Phi_w$ and $\Psi_z$. By differentiating the expression $\d^2=\omega_{\bL}\cdot \id$ from Lemma~\ref{lem:del^2=0} once, with respect to a variable in $\bF_2[U_{\ws},V_{\zs}]$, we obtain the following:

\begin{lem}[\cite{ZemQuasi}*{Lemma~3.1}]\label{lem:PhiPsichainhomotopies}If $z$ is adjacent to $w$ and $w'$ then
\[\d\Psi_z+\Psi_z \d=U_w+U_{w'}.\] If $w$ is adjacent to $z$ and $z'$ then we have that
\[\d \Phi_w+\Phi_w\d=V_z+V_{z'}.\]
\end{lem}

In particular, the maps $\Phi_w$ and $\Psi_z$ are not chain maps on the uncolored complexes. In fact, $\Psi_z$ is a chain map if and only if the variables $U_w$ and $U_{w'}$ are identified, and $\Phi_w$ is a chain map if and only if the variables $V_z$ and $V_{z'}$ are identified. Notice that this corresponds exactly to the condition that the coloring of the basepoints is compatible with a coloring of the decorated surfaces from Figure~\ref{fig::116}.

The following is easily proven by differentiating the relations in Lemma~\ref{lem:PhiPsichainhomotopies} with respect to one of the $U_{\ws}$ or $V_{\zs}$ variables:

\begin{lem}[\cite{ZemQuasi}*{Lemmas~9.1 and 9.2}]\label{lem:PhiPsicommutators}If $w,$ $w',$ $z$ and $z'$ are basepoints on $\bL$, we have
\[
\Phi_w\Phi_{w'}+\Phi_{w'}\Phi_w\simeq 0,
\] and
\[
\Psi_{z}\Psi_{z'}+\Psi_{z'}\Psi_z\simeq 0.
\] We also have
\[
\Phi_w \Psi_z+\Psi_z\Phi_w+ N(w,z)\cdot \id\simeq 0,
\] 
where $N(w,z)\in \bF_2$ is the number of components of $L\setminus (\ve{w}\cup \ve{z})$ which have both $w$ and $z$ as boundary.
\end{lem}

Interpretations of the relations from Lemma~\ref{lem:PhiPsichainhomotopies} in terms of dividing sets on cylindrical link cobordisms are shown in Figure~\ref{fig::87}.

\begin{figure}[ht!]
\centering
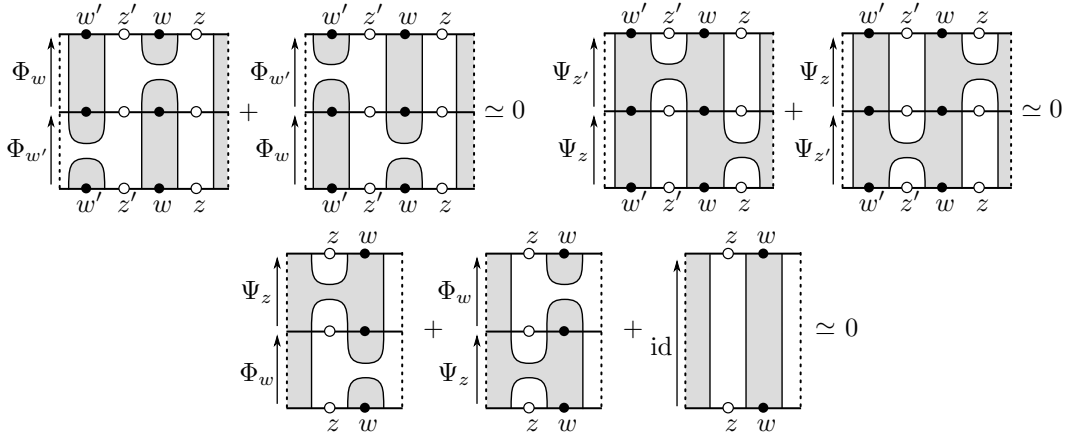
\caption{\textbf{Interpretations of the relations from Lemma~\ref{lem:PhiPsicommutators} in terms of dividing sets on $[0,1]\times L$.} The top left is $\Phi_{w}\Phi_{w'}+\Phi_{w'}\Phi_w\simeq 0$. The top right is $\Psi_{z}\Psi_{z'}+\Psi_{z'}\Psi_{z}\simeq 0$. The bottom is $\Phi_w \Psi_z+\Psi_z\Phi_w+ \id\simeq 0$ when $N(w,z)=1$. The top two relations represent equivalent dividing sets. The three dividing sets in the bottom relation form a bypass triple. \label{fig::87}}
\end{figure}

\begin{rem}it is often convenient to think of the maps $\Phi_w$ as being $(d/d U_w)(\d)$. Note that there is a distinction between $\Phi_w=(d/d U_w)(\d)$ and $(d/d U_w) \circ \d$, i.e. between applying $d/d U_w$ to the entries of a matrix representing $\d$, and composing the map $\d$ with the map $d/d U_w$. The difference between the two operations is reflected by the Leibniz rule for differentiating products. Hence on the uncolored complex $\cCFL^-(\cH,\frs)$, the map $\Phi_w$ satisfies
\[
\Phi_w=\frac{d}{d U_w} \circ \d+\d\circ \frac{d}{d U_w}.
\] 
Similarly on the uncolored complex, we have
\[
\Psi_z=\frac{d}{d V_z}\circ \d+\d\circ \frac{d}{ d V_z}.
\]
The chain homotopies $d/d U_w$ and $d/d V_z$ do in fact preserve the subcomplex $\cCFL^-(\cH,\frs)$, though they are neither $\bF_2[U_{\ws},V_{\zs}]$-equivariant nor $\Z^{\ws}\oplus \Z^{\zs}$-filtered. Hence they will not in general induce well-defined maps after we color the complexes by tensoring with $\cR_{\bmP}^-$. On the colored complexes, the maps $\Phi_w$ and $\Psi_z$ are often not chain homotopic to zero, even non-equivariantly. Compare \cite{ZemGraphTQFT}*{Corollary~14.11} for the analogous algebraic result for the closed 3-manifold invariants.
\end{rem}

The maps $\Phi_w$ and $\Psi_z$ are homotopy differentials (cf. \cite{ZemQuasi}*{Lemma~9.7}):

\begin{lem}\label{lem:Phi^2Psi^2simeq0}The endomorphisms $\Phi_w$ and $\Psi_z$ satisfy 
\[\Phi_w^2\simeq 0 \qquad \text{and}\qquad \Psi_z^2\simeq 0,\] through filtered, equivariant chain homotopies.
\end{lem}
\begin{proof}The two relations are an algebraic consequence of our expression for the curvature constant $\omega_{\bL}\in \bF_2[U_{\ws},V_{\zs}]$ from Lemma~\ref{lem:del^2=0}. We will show that $\Phi_w^2\simeq 0$; the proof that $\Psi_z^2\simeq 0$ is analogous. Write $\d=\sum_{k=0}^\infty \d_k U_w^k$, where $\d_i$ is thought of as an $n\times n$ matrix (where $n$ is the number of intersection points representing $\frs$) with entries in $\bF_2[U_{\ws}, V_{\zs}]$, which  do not involve $U_w$. Using this notation, we have
\[\Phi_w=\sum_{k=0}^\infty k \d_k U_w^{k-1}\qquad \text{and}\qquad  \d^2=\sum_{k\ge 0} \sum_{n+m=k} \d_n\d_m U_w^{n+m}.\] Applying  Lemma~\ref{lem:del^2=0} for any $k\ge 2$, we have that 
\begin{equation}\sum_{m+n=k}\d_n\d_m =0.\label{eq:kge2termsPhi^2=0}\end{equation} Define
\[H:=\sum_{k=0}^\infty \frac{k(k-1)}{2} \d_k U_w^{k-2}.\] One now simply computes that
\[\Phi_w^2+(\d H+H\d)=\sum_{m,n\ge 0} \frac{(m+n)(m+n-1)}{2} \d_n\d_m U_w^{n+m-2}=\sum_{k=0}^\infty \frac{k(k-1)}{2}U_w^{k-2}\sum_{n+m=k} \d_n\d_m ,\] which vanishes, because the $k=0$ and $k=1$ terms vanish trivially and the terms for $k\ge 2$ vanish by Equation \eqref{eq:kge2termsPhi^2=0}.
\end{proof}

The dividing sets on $[0,1]\times L$ corresponding to $\Phi_w^2$ and $\Psi_z^2$ both contain a null-homotopic dividing curve, bounding a disk. They are shown in Figure~\ref{fig::88}.

\begin{figure}[ht!]
\centering
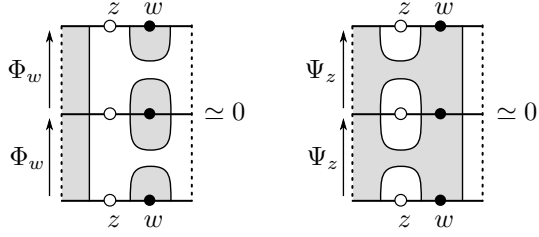
\caption{\textbf{The dividing sets corresponding to $\Phi_w^2$ (left) and $\Psi_z^2$ (right) on $[0,1]\times L$.} Both $\Phi_w^2$ and $\Phi_z^2$ are chain homotopic to zero. The corresponding dividing sets both have a null-homotopic closed curve, bounding a disk. \label{fig::88}}
\end{figure}

\subsection{Relations between the basepoint actions and the quasi-stabilization maps}
\label{subsec:relations2}

In this section, we prove some relations between the quasi-stabilization maps and the maps $\Phi_w$ and $\Psi_z$.

As a first relation, we show that under favorable assumptions on the coloring $\bL$, the maps $S^{\circ}_{w,z}$ and $T^{\circ}_{w,z}$ can be related to each other using the maps $\Phi_w$ and $\Psi_z$:

\begin{lem}\label{lem:olddefsforTw,z} Suppose that $\bL=(L,\ve{w},\ve{z})$ is a multi-based link and $w$ and $z$ are new basepoints, in a single component of $L\setminus (\ve{w}\cup \ve{z})$, and write $w'\in \ws$ and $z'\in \zs$ for the two basepoints adjacent to $w$ and $z$. If $w$ and $w'$ are given the same color, and $z$ and $z'$ are given the same color, then
\[
T_{w,z}^{+}\simeq \Psi_z  S_{w,z}^+\qquad \text{ and }\qquad T_{w,z}^-\simeq S_{w,z}^-\Psi_z.
\]
 Similarly
\[
S_{w,z}^+\simeq \Phi_w  T_{w,z}^+\qquad \text{ and }\qquad S_{w,z}^-\simeq T_{w,z}^-\Phi_w.
\]
\end{lem}
\begin{proof} We consider the formula for the quasi-stabilized differential in Proposition~\ref{prop:quasi-stabilizeddifferential}. Using that result, any disk which is counted by $\Phi_w$ has domain equal to the bigon going over $w$ in the quasi-stabilized region, and similarly any disk counted by $\Psi_z$ has domain equal to the bigon going over $z$ in the quasi-stabilized region. Hence, using the matrix notation from Proposition~\ref{prop:quasi-stabilizeddifferential}, we have that
\[\Phi_w=\begin{pmatrix}0& 
\id\\
0& 0
\end{pmatrix} \qquad \text{and} \qquad \Psi_z=\begin{pmatrix}0& 0\\
\id & 0
\end{pmatrix}.\] In terms of generators, the map $\Psi_z$ satisfies $\Psi_z(\ve{x}\times \xi^{\ws})=0$ and $\Psi_z(\ve{x}\times \theta^{\ws})=\ve{x}\times \xi^{\ws}$, and $\Phi_w$ satisfies a similar relation. Hence
\[(\Psi_zS_{w,z}^+)(\ve{x})=\Psi_z(\ve{x}\times \theta^{\ws})=\ve{x}\times \xi^{\ws}=\ve{x}\times \theta^{\zs}=T_{w,z}^+(\ve{x}).\] The other relations follow similarly.
\end{proof}

The dividing set interpretations of the relations from Lemma~\ref{lem:olddefsforTw,z} are shown in Figure~\ref{fig::89}.

\begin{figure}[ht!]
\centering
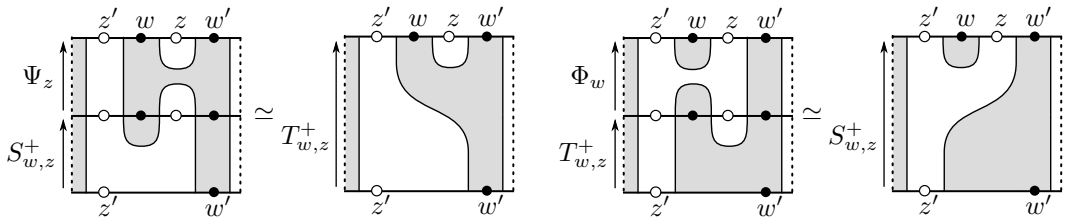
\caption{\textbf{The dividing sets on $[0,1]\times L\subset [0,1]\times Y$ corresponding to the relations from Lemma~\ref{lem:olddefsforTw,z}.}  \label{fig::89}}
\end{figure}

As one might expect given the interpretation of the maps in terms of dividing set, the endomorphisms $\Phi_w$ and $\Psi_z$ can each be defined as a compositions of a quasi-destabilization followed by a quasi-stabilization:

\begin{lem}[\cite{ZemQuasi}*{Lemma~9.3}]\label{lem:S^+S^-=Phi} Suppose that $(z,w,z')$ is a triple of consecutive and distinct basepoints on a component of $\bL$, ordered right to left, and that $\sigma$ is a coloring of $\bL$ with $\sigma(z)=\sigma(z')$. Then
\[\Phi_w\simeq S_{w,z}^+S_{w,z}^-\simeq S_{z',w}^+S_{z',w}^-,\] as endomorphisms of $\cCFL^-(\cH,\sigma,\frs)$.
\end{lem}
\begin{proof}By analyzing the quasi-stabilized differential appearing in Proposition~\ref{prop:quasi-stabilizeddifferential}, we see that $\Phi_w$ satisfies $\Phi_w(\ve{x}\times \theta^{\ws})=0$ and $\Phi_w(\ve{x}\times \xi^{\ws})=\ve{x}\times \theta^{\ws}$. On the other hand, this is the same as the composition $S_{w,z}^+ S_{w,z}^-$, by definition. The formula $\Phi_w\simeq S_{z',w}^+ S_{z',w}^-$ is proven similarly.
\end{proof}

Symmetrically, we have the following:

\begin{lem}\label{lem:T^+T^-=Psi} Suppose that $(w,z,w')$ is a triple of consecutive and distinct basepoints on a link component of $\bL$, ordered right to left, and suppose that $\sigma$ is a coloring of $\bL$ such that $\sigma(w)=\sigma(w')$. Then
\[\Psi_z\simeq T_{w,z}^+  T_{w,z}^-\simeq T_{z,w'}^+T_{z,w'}^-\] as endomorphisms of $\cCFL^-(\cH,\sigma,\frs)$.
\end{lem}
\begin{proof}The proof is similar to the proof of Lemma~\ref{lem:S^+S^-=Phi}.
\end{proof}

The dividing set interpretation of the relations in Lemmas~\ref{lem:S^+S^-=Phi}~and~\ref{lem:T^+T^-=Psi} is illustrated in Figure~\ref{fig::90}.

\begin{figure}[ht!]
\centering
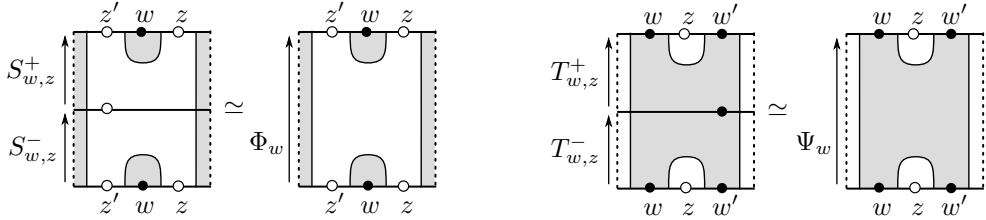
\caption{\textbf{The dividing sets on $L\times [0,1]\subset Y\times [0,1]$ corresponding to the relations from Lemmas~\ref{lem:S^+S^-=Phi}~and~\ref{lem:T^+T^-=Psi}.}  \label{fig::90}}
\end{figure}

Additionally we have the following:

\begin{lem}\label{lem:addtrivialstrandII}The quasi-stabilization maps satisfy
\[S_{w,z}^-  S_{w,z}^+\simeq 0, \qquad T_{w,z}^-   T_{w,z}^+\simeq 0,\qquad T_{w,z}^-  S_{w,z}^+\simeq \id,\qquad S_{w,z}^-  T_{w,z}^+\simeq \id,\] and
\[S_{w,z}^+   T_{w,z}^-+T_{w,z}^+  S_{w,z}^-+\id\simeq 0.\]
\end{lem}
\begin{proof}All of these relations follow immediately from the formulas used to define the quasi-stabilization maps.\end{proof}

The dividing set interpretation of the relations from Lemma~\ref{lem:addtrivialstrandII} are shown in Figure~\ref{fig::91}.

\begin{figure}[ht!]
\centering
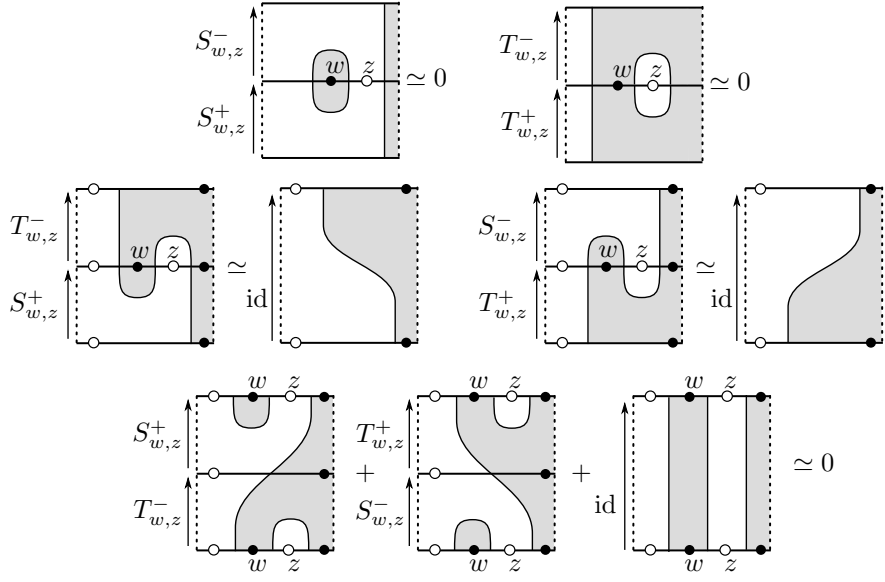
\caption{\textbf{The dividing sets on $[0,1]\times L\subset [0,1]\times Y$ corresponding to the relations from Lemma~\ref{lem:addtrivialstrandII}.}  \label{fig::91}}
\end{figure}

We now consider the question of commuting the maps $\Phi_w$ and $\Psi_z$ with the quasi-stabilization maps. Versions of most of these relations can be found in \cite{ZemQuasi}*{Section~9}.

Suppose that $A$ is a subarc of $L$, with ends on two $\ws$-basepoints. We do not exclude the case that $A$ consists of an entire link component of $L$, or that $A$ consists of a single basepoint in $\ws$ (though in this latter case we think of $A$ as being the empty arc). We define
\[\Psi_A:=\sum_{z\in A\cap \zs} \Psi_z.\] Using Lemma~\ref{lem:PhiPsichainhomotopies}, we see that $\Psi_A$ is a chain map if and only if the endpoints of $A$ share the same color.

\begin{lem}\label{lem:repacakgedSw,zPsicommutator}Suppose that $(w,z)$ is a pair of new, adjacent basepoints on $\bL$, and $A$ is a subarc of $L$, with endpoints on two $\ws$-basepoints which are given the same color. Then the following relations are satisfied:
\[\Psi_A   S_{w,z}^{\circ}+S_{w,z}^{\circ}  \Psi_A\simeq 0\qquad \text{and} \qquad \Psi_A   T_{w,z}^{\circ}+T_{w,z}^{\circ} \Psi_A\simeq 0,\] for $\circ\in \{+,-\}$.
\end{lem}
\begin{proof} We focus on the claim for $S_{w,z}^+$. Let $\bmP=\ve{w}\cup \ve{z}\cup \{w,z\}$, and let $\sigma\colon \ve{w}\cup \ve{z}\to \bmP$ denote the natural inclusion, and let $\sigma'\colon \ve{w}\cup \ve{z}\cup \{w, z\}\to \bmP$ denote the identity. Let $\cH$ denote a diagram for $(L,\ve{w},\ve{z})$, and let $\cH^+$ denote a quasi-stabilization of $\cH$ constructed by adding the basepoints $w$ and $z$. We consider the  complexes $(\cCFL^-(\cH, \sigma, \frs),\d_{\cH})$ and $(\cCFL^-(\cH^+,\sigma',\frs),\d_{\cH^+})$. Note that $\Psi_A$ is not a chain map on these complexes. We will establish the stated relation on the uncolored complexes, and it will persist to the colored complexes for a coloring which makes $\Psi_A$ a chain map.

Using the computation of the quasi-stabilized differential from Proposition~\ref{prop:quasi-stabilizeddifferential} as well as the formulas for the map $S_{w,z}^+$, we can write
\begin{equation}\d_{\cH^{+}}  S_{w,z}^{+}+ S_{w,z}^{+}  \d_{\cH}+(V_z+V_{z'})\cdot F=0,\label{eq:Swz,walmostchainmap}\end{equation} where $F$ is the map $F(\ve{x})=\ve{x}\times \theta^-_{\ws}$, extended linearly over $\cR_{\bmP}^-$,  and $z'$ denotes the $\zs$-basepoint adjacent to $w$. We view Equation~\eqref{eq:Swz,walmostchainmap} as a matrix equation involving matrices over the intersection points with coefficients in $\cR_{\bmP}^-$. Define
\[D_A:=\sum_{s\in A\cap (\zs\cup \{z\})} \frac{d}{d V_s}.\] Since $\tfrac{d}{d V_z} (\d_{\cH})=0$, we note that $D_A( \d_{\cH})$ and $D_A(\d_{\cH^+})$ can each be identified with the matrices for $\Psi_A$ on the complexes $\cCFL^-(\cH,\sigma,\frs)$ and $\cCFL^-(\cH,\sigma',\frs)$, respectively. We also note that on the quasi-stabilized link, it is either the case that both $z$ and $z'$ are in $A$, or neither are in $A$, since $w$ is not an endpoint of $A$. Hence $D_A(V_z+V_{z'})=0$. Using the Leibniz rule and the fact that $D_A(F)=0$, we see that
\[D_A\left((V_z+V_{z'})\cdot F\right)=0.\] Furthermore, we note that $D_A(S_{w,z}^+)=0$, by inspection of the formula for $S_{w,z}^+$. Hence, applying $D_A$ to Equation~\eqref{eq:Swz,walmostchainmap} and using the Leibniz rule shows that 
\[\Psi_A  S_{w,z}^{+}+S_{w,z}^{+}  \Psi_A= 0.\] The statements involving $S_{w,z}^-$, $T_{w,z}^{+}$ and $T_{w,z}^-$ follow from similar lines of reasoning.
\end{proof}

The analogous statement for the maps $\Phi_w$ also holds. If $A$ is an arc on $L$ connecting two $\zs$-basepoints, we  define
\[\Phi_A:=\sum_{w\in A\cap \ws} \Phi_w.\] The map $\Phi_A$ is a chain map if and only if the two $\zs$-basepoints on the ends of $A$ share the same color. Analogously to Lemma~\ref{lem:repacakgedSw,zPsicommutator}, we have the following:

\begin{lem}\label{lem:PhiPsiandS^pmcommutator}Suppose that $A$ is an arc on $L$ between two $\zs$-basepoints of $(L,\ve{w},\ve{z})$, and $(w,z)$ is a pair of new adjacent basepoints on $L$. Then
\[S_{w,z}^{\circ}   \Phi_A+\Phi_A   S_{w,z}^{\circ}\simeq 0 \qquad \text{and} \qquad T_{w,z}^{\circ}  \Phi_A+\Phi_A  T_{w,z}^{\circ}\simeq 0,\] for $\circ\in \{+,-\}$.
\end{lem}
The proof of Lemma~\ref{lem:PhiPsiandS^pmcommutator} follows the same strategy as the proof of Lemma~\ref{lem:repacakgedSw,zPsicommutator}. We now highlight several useful special cases of the previous lemmas:

\begin{lem}\label{lem:Spm-Psicommutatornotadjacent} Suppose that $(L,\ws,\zs)$ is a multi-based link in $Y$ and  $(w,z)$ is an adjacent pair of new basepoints on $L$. The following hold:
\begin{enumerate}
\item\label{rel:1:lem:Spm-Psicommutatornotadjacent} If $w'\in \ws$ is a basepoint (in particular $w'\neq w$) then 
\[
S_{w,z}^\circ   \Phi_{w'}+\Phi_{w'}  S_{w,z}^\circ\simeq 0.
\]
\item\label{rel:2:lem:Spm-Psicommutatornotadjacent} If $z'$ is not adjacent to $w$, then
\[
S_{w,z}^{\circ}   \Psi_{z'}+\Psi_{z'}  S_{w,z}^\circ\simeq 0.
\]

\item \label{rel:3:lem:Spm-Psicommutatornotadjacent} If $z'\in \zs$ is adjacent to $w$ (and $z'\neq z$), then
\[
S_{w,z}^+  \Psi_{z'}\simeq (\Psi_z+\Psi_{z'})  S_{w,z}^+\qquad \text{and} \qquad S_{w,z}^- (\Psi_z+\Psi_{z'})\simeq \Psi_{z'}  S_{w,z}^-.
\]
\end{enumerate}
\end{lem}

\begin{proof} We will prove all the stated relations by applying Lemmas~\ref{lem:repacakgedSw,zPsicommutator} and~\ref{lem:PhiPsiandS^pmcommutator} for  appropriately chosen subarcs $A$ of $L$ . For example, to prove the relation in part~\eqref{rel:2:lem:Spm-Psicommutatornotadjacent}, we let $A$ be an arc between the two $\ws$-basepoints adjacent to $z'$, so that $\Psi_A=\Psi_{z'}$ on both the stabilized and unstabilized complexes. Similarly for the relation in part~\eqref{rel:3:lem:Spm-Psicommutatornotadjacent}, we let $A$ be the arc on $L$ between the two basepoints in $\ws$ (on the unstabilized link) which are adjacent to $z'$.

The relations in part~\eqref{rel:1:lem:Spm-Psicommutatornotadjacent} require a slight additional argument. If $z$ is not adjacent to $w'$, then picking $A$ to be arc between the two basepoints in $\zs$ which are adjacent to $w'$, we see that $\Phi_A=\Phi_{w'}$ on both the stabilized and unstabilized complexes. If $z$ is adjacent to $w'$, then we instead have $\Phi_A=\Phi_{w'}$ on the unstabilized complex, and $\Phi_A=\Phi_{w'}+\Phi_{w}$ on the stabilized complex. Hence Lemma~\ref{lem:PhiPsiandS^pmcommutator} implies 
\[
S_{w,z}^+ \Phi_{w'}\simeq (\Phi_{w'}+\Phi_w) S_{w,z}^+\qquad  \text{and}\qquad  S_{w,z}^-(\Phi_{w'}+\Phi_w)\simeq \Phi_{w'} S_{w,z}^-.
\] Part~\eqref{rel:1:lem:Spm-Psicommutatornotadjacent} follows from the above two equations once we note that $\Phi_{w}S_{w,z}^+\simeq 0$ and $S_{w,z}^-\Phi_w\simeq 0$, since $\Phi_{w}\simeq S_{w,z}^+S_{w,z}^-$ and $S_{w,z}^-S_{w,z}^+\simeq 0$, by Lemmas~\ref{lem:T^+T^-=Psi}~and~\ref{lem:addtrivialstrandII}, respectively.
\end{proof}

Interpretations of some of the relations from Lemma~\ref{lem:Spm-Psicommutatornotadjacent} in terms of dividing sets are shown in Figure~\ref{fig::94}.

\begin{figure}[ht!]
\centering
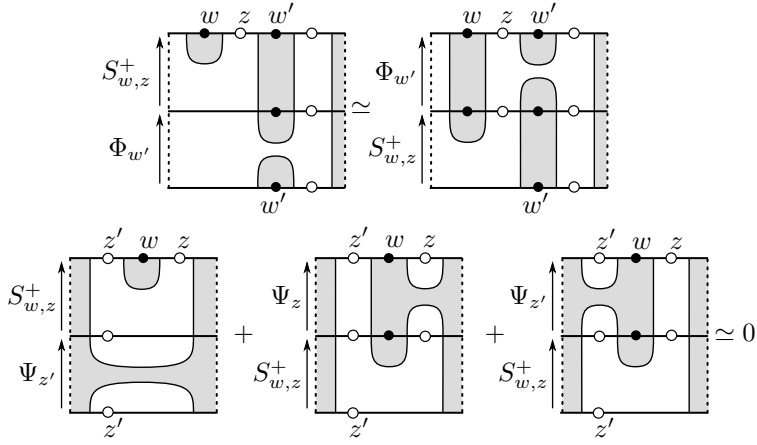
\caption{\textbf{The dividing sets on $[0,1]\times L\subset [0,1]\times  Y$ corresponding to the relations from Lemma~\ref{lem:Spm-Psicommutatornotadjacent}.} The top relation corresponds to diffeomorphic sets of dividing arcs, while the bottom three dividing sets form a bypass triple.  \label{fig::94}}
\end{figure}

A convenient reformulation of some of the previous relations is the following:

\begin{lem}\label{lem:oldaddtrivialstrand}Suppose that $(z',w,z,w')$ are four consecutive basepoints appearing on $\bL$ and suppose that $\sigma$ is a coloring of $\bL$ so that $\sigma(w)=\sigma(w')$ and $\sigma(z)=\sigma(z')$. Then
\[S_{w,z}^- \Psi_z S_{w,z}^+\simeq \id \simeq S_{w,z}^- \Psi_{z'} S_{w,z}^+\qquad \text{and}\qquad T_{w,z}^- \Phi_w T^+_{w,z}\simeq \id\simeq T_{w,z}^- \Phi_{w'} T_{w,z}^+.\]
\end{lem}

\begin{proof} We will focus on the relations $S_{w,z}^- \Psi_z S_{w,z}^+\simeq S_{w,z}^- \Psi_{z'} S_{w,z}^+\simeq  \id$, since the other relations follow from a symmetrical argument replacing the $\ws$ and $\zs$-basepoints.

From Lemma~\ref{lem:addtrivialstrandII} we know that $S_{w,z}^-T_{w,z}^+\simeq \id$. Lemma~\ref{lem:olddefsforTw,z} implies that $T_{w,z}^+\simeq \Psi_z S_{w,z}^+$. Combining these two relations yields  $S_{w,z}^- \Psi_z S_{w,z}^+\simeq \id$.

Lemma~\ref{lem:Spm-Psicommutatornotadjacent} implies that 
\[S_{w,z}^-\Psi_z S_{w,z}^+\simeq S_{w,z}^-\Psi_{z'} S_{w,z}^++S_{w,z}^-S_{w,z}^+\Psi_{z'},\] however $S_{w,z}^-S_{w,z}^+\simeq 0$ by Lemma~\ref{lem:addtrivialstrandII}, and hence we obtain \[S_{w,z}^-\Psi_z S_{w,z}^+\simeq S_{w,z}^-\Psi_{z'} S_{w,z}^+,\] completing the proof.
\end{proof}

Dividing set interpretations of the relations from Lemma~\ref{lem:oldaddtrivialstrand} are shown in Figure~\ref{fig::95}.

\begin{figure}[ht!]
\centering
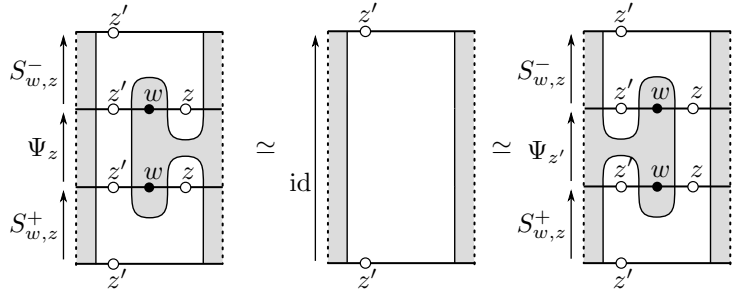
\caption{\textbf{The dividing sets on $[0,1]\times L\subset [0,1]\times Y$ corresponding to some of the relations from Lemma~\ref{lem:oldaddtrivialstrand}.} \label{fig::95}}
\end{figure}

In \cite{ZemQuasi}, we showed the following commutation result about quasi-stabilization maps:

\begin{prop}[\cite{ZemQuasi}*{Theorem~8.1}]\label{thm:quasistabcommute}If $(w,z)$ and $(w',z')$ are two disjoint pairs of basepoints, with $w$ following $z$ and $w'$ following $z'$, then
\[S_{w,z}^{\circ} S_{w',z'}^{\circ'}\simeq S_{w',z'}^{\circ'} S_{w,z}^{\circ},\] for any choice of $\circ,\circ'\in \{+,-\}$. Commutation still holds if both of the pairs of basepoints $(w,z)$ and $(w',z')$ appear with the opposite ordering (i.e. if $z$ follows $w$ and $z'$ follows $w'$, and we use the maps $S_{z,w}^\circ$ and $S_{z',w'}^{\circ'}$, instead).
\end{prop}

The proof of Proposition~\ref{thm:quasistabcommute} can be found in \cite{ZemQuasi}*{Section~8}. We will sketch most of the necessary details in the proof of a similar result, Proposition~\ref{prop:TScommute}, below. The proof in \cite{ZemQuasi} carries over without major change if we replace both $S$-quasi-stabilizations by $T$-quasi-stabilizations:

\begin{prop}\label{lem:z-quasistabcommute}If $(w,z)$ and $(w',z')$ are disjoint pairs of basepoints, with $w$ following $z$ and $w'$ following $z'$, then
\[
T_{w,z}^{\circ} T_{w',z'}^{\circ'}\simeq T_{w',z'}^{\circ'} T_{w,z}^{\circ},
\]
 for any choice of $\circ,\circ'\in \{+,-\}$. Commutation still holds if both of the pairs of basepoints $(w,z)$ and $(w',z')$ appears with the opposite ordering (i.e. if $z$ follows $w$ and $z'$ follows $w'$, and we use the maps $T_{z,w}^\circ$ and $T_{z',w'}^{\circ'}$, instead).
\end{prop}

Propositions~\ref{thm:quasistabcommute} and~\ref{lem:z-quasistabcommute} should be expected from the interpretation of the maps in terms of dividing sets. The dividing sets for two type-$S$ quasi-stabilizations can always be moved past each other. Similarly, the dividing sets for two type-$T$ quasi-stabilizations can always be moved past each other. Examples are shown in Figure~\ref{fig::92}.

\begin{figure}[ht!]
\centering
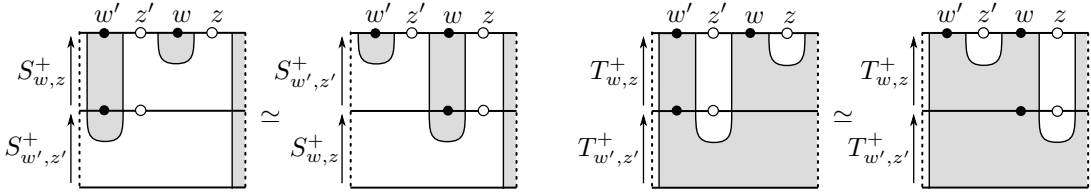
\caption{\textbf{The relations from Propositions~\ref{thm:quasistabcommute}~and~\ref{lem:z-quasistabcommute} interpreted in terms of dividing sets on $[0,1]\times L \subset [0,1]\times Y.$} Two type-$S$ quasi-stabilizations can always be commuted past each other, and similarly to type-$T$ quasi-stabilizations can always be commuted past each other.  \label{fig::92}}
\end{figure}

The previous two results do not tell us whether $T_{w,z}^{\circ}$ and $S_{w',z'}^{\circ'}$ commute.  In fact, we will later see that they do not always commute (see Lemma~\ref{lem:bypasstriplefromquasistabilization}). If one considers the interpretation of the maps in terms of dividing sets, it is easy to see that the dividing sets for $S_{w,z}^{\circ}$ and $T_{w',z'}^{\circ'}$ will commute if $w$ is not adjacent to $z'$. Indeed this graphical interpretation is reflected by the holomorphic geometry:

\begin{prop}\label{prop:TScommute} Suppose $(w,z)$ are $(w',z')$ are disjoint pairs of adjacent basepoints, with $w$ following $z$ and $w'$ following $z'$. If $w'$ is not adjacent to $z$, then
\[T_{w,z}^{\circ}  S_{w',z'}^{\circ'}\simeq S_{w',z'}^{\circ'} T_{w,z}^{\circ},\] for any choice of $\circ,\circ'\in \{+,-\}$. The same result holds if $z$ follows $w$ and $z'$ follows $w'$, with  $T_{w,z}^{\circ}$ changed to $T_{z,w}^{\circ}$ and $S_{w',z'}^{\circ'}$ changed to $S_{z',w'}^{\circ'}$, as long as $z$ and $w'$ are not adjacent.
\end{prop}

\begin{proof}We will only consider the case when $w$ follows $z$ and $w'$ follows $z'$. The argument follows from a small modification of the argument used to show that the type-$S$ quasi-stabilization maps commute with each other \cite{ZemQuasi}*{Section~8}. 

We start with a diagram $\cH=(\Sigma, \ve{\alpha},\ve{\beta},\ve{w},\ve{z})$ for the unstabilized link $\bL=(L,\ve{w},\ve{z})$.  We form the doubly quasi-stabilized diagram $\cH^{++}$ by inserting the two subdiagrams shown in Figure~\ref{fig::93} at two points $p$ and $p'$ in $\Sigma\setminus \as$.

\begin{figure}[ht!]
\centering
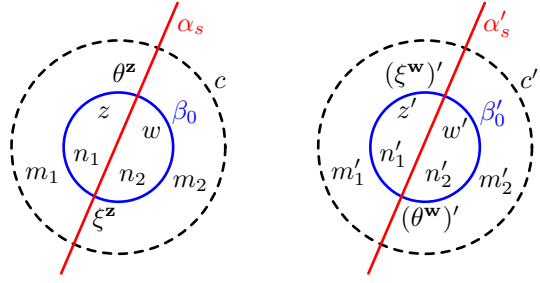
\caption{\textbf{We obtain $\cH^{++}$ (shown) by quasi-stabilizing $\cH$ at $p,p'\in \Sigma\setminus \as$.} Multiplicities $m_1,$ $m_2$, $n_1$, $n_2$, $m_1',$ $m_2'$, $n_1'$ and $n_2'$ are shown.\label{fig::93}}
\end{figure}

 The diagrams $\cH$, $\cH^{++}$, and the two intermediate, singly quasi-stabilized diagrams can be used to compute all of the quasi-stabilization maps. Although the formulas for the two quasi-stabilization maps look like they commute, it is important to  note that we place very specific requirements on the almost complex structures which can be used to compute the quasi-stabilization maps. To compute $S_{w',z'}^{\circ'}$, the almost complex structure must be stretched on the curve $c'$ shown in Figure~\ref{fig::93}. To compute $T_{w,z}^{\circ}$, one must stretch along the curve $c$. To compute the commutator of the two maps, we must compute a change of almost complex structure map associated to changing the relative lengths of the two necks. If $J_s$ is an almost complex structure on $\Sigma\times [0,1]\times \R$ which is split in a cylindrical neighborhood of $p$ and $p'$, and $\ve{T}=(T,T')$ is a pair of positive real numbers, we will write $J_s(\ve{T})$ for an almost complex structure which has had necks of length $T$ and $T'$ inserted along the curves $c$ and $c'$, respectively.
  
  By the definition of the quasi-stabilization maps, an almost complex structure $J_s(T,T')$ can be used to compute $S_{w',z'}^{+}$ if the change of almost complex structure map $\Phi_{J_s(T,T_1')\to J_s(T,T_2')}$ preserves the element $\ve{x}\times (\theta^{\ws})'$ whenever $T_1',T_2'\ge T'$. Similarly $J_s(T,T')$ can be used to compute $S_{w',z'}^-$ if the map $\Phi_{J_s(T,T'_1)\to J_s(T,T'_2)}$ sends $\ve{x}\times (\xi^{\ws})'$ to an element of the form $\ve{x}\times (\xi^{\ws})'+ \sum_{\ys}c_{\xs,\ys} \cdot \ys\times (\theta^{\ws})'$, for some collection of $c_{\xs,\ys}\in \cR_{\bmP}^-$, whenever $T_1',T_2'\ge T'$.  Similar criteria apply for the maps $T_{w,z}^+$ and $T_{w,z}^-$.
 
Proving commutation of the maps $S_{w',z'}^{\circ'}$ and $T_{w,z}^{\circ}$ thus amounts to analyzing  the change of almost complex structure map $\Phi_{J_s(\ve{T}_1)\to J_s(\ve{T}_2)}$. The map $\Phi_{J_s(\ve{T}_1)\to J_s(\ve{T}_2)}$ can be computed by counting index 0 holomorphic strips in $\Sigma\times [0,1]\times \R$ for a non-cylindrical almost complex structure which agrees with $J_s(\ve{T}_1)$ on $\Sigma\times [0,1]\times (-\infty,-1]$ and agrees with $J_s(\ve{T}_2)$ on $\Sigma\times [0,1]\times [1,\infty)$. We will use a neck-stretching argument.

We let $\{\ve{T}_{1}^i\}_{i\in \N}$ and $\{\ve{T}_{2}^i\}_{i\in \N}$ denote two sequences of pairs of neck lengths, all whose components approach $+\infty$. Let $\tilde{J}^i$ be a non-cylindrical almost complex structure interpolating $J_s(\ve{T}_1^i)$ and $J_s(\ve{T}_2^i)$. We can assume that the almost complex manifold $(\Sigma\times [0,1]\times \R, \tilde{J}^i)$ contains the almost complex submanifold $(\Sigma\setminus N_i(\{p,p'\})\times [0,1]\times \R, J_s)$ where $N_i(\{p,p'\})$ is a nested sequence of regular neighborhoods of the set $\{p,p'\}$ whose intersection over $i\in \N$ is $\{p,p'\}$. Also $J_s$ is a fixed cylindrical almost complex structure on $\Sigma\times [0,1]\times \R$.

Given a sequence of $\tilde{J}^i$-holomorphic disks $u_i$, representing a Maslov index 0 homology class $\phi$ on $\cH^{++}$, we can extract a broken collection of limiting holomorphic curves on $(\Sigma\times [0,1]\times \R, J_s)$. As in the proof of \cite{MOIntSurg}*{Proposition~6.2}, we can arrange the limiting curves on $(\Sigma\times [0,1]\times \R,J_s)$ into a broken holomorphic disk $U_0$ on $(\Sigma,\as,\bs)$, which has no boundary components mapping to $\alpha_s$ or $\alpha_s'$, as well as a collection of boundary degenerations $\cA$, which has boundary on $\as\cup \alpha_s\cup \alpha_s'$. We will write $\phi_0$ for the total homology class of $U_0$.

There are two cases to consider:
\begin{enumerate}
\item\label{case:pp'diffcomponent} $p$ and $p'$ are in different components, $A_p$ and $A_{p'}$,  of $\Sigma\setminus \as$.
\item\label{case:pp'samecomponent} $p$ and $p'$ are in the same component, $A_{p,p'}$, of $\Sigma\setminus \as$.
\end{enumerate}

Consider Case \eqref{case:pp'diffcomponent} first. This occurs exactly when $(w,z)$ and $(w',z')$ are not adjacent on $\bL$. In this case, following the proof of \cite{ZemQuasi}*{Lemma~8.3}, the Maslov index $\mu(\phi)$ satisfies the formula
\[\mu(\phi)=\mu(\phi_0)+n_1(\phi)+n_2(\phi)+n_1'(\phi)+n_2'(\phi)\]
\[+m_1(\cA)+m_2(\cA)+m_1'(\cA)+m_2'(\cA)+2\sum_{\substack{\cD\in C(\Sigma\setminus \as)\\ \cD\neq A_p,A_{p'}}} n_{\cD}(\cA).\] 
  Since $\phi_0$ has the broken holomorphic representative $U_0$, we know that $\mu(\phi_0)\ge 0$, by transversality. Since all of the other summands are nonnegative and $\mu(\phi)=0$ by assumption, we know that all terms must be zero. It is easy to see that this implies that $\phi$ is a constant homology class. On the other hand, a constant homology class always has a unique $\tilde{J}^i$-holomorphic representatives. It follows that if $\ve{T}_1$ and $\ve{T}_2$ are two pairs of neck lengths, all of whose components are sufficiently large, then $\Phi_{J_s(\ve{T}_1)\to J_s(\ve{T}_2)}$ is the identity map, on the level of intersection points. In particular, it follows that $S_{w',z'}^{\circ'}$ and $T_{w,z}^{\circ}$ commute if $(w,z)$ and $(w',z')$ are not adjacent.

Consider now Case \eqref{case:pp'samecomponent}, which is slightly more subtle. This case occurs when the pairs $(w,z)$ and $(w',z')$ are adjacent on $\bL$. By hypotheses on the non-adjacency of $z$ and $w'$,  the four basepoints must appear with ordering $(w',z',w,z)$, read right to left.  

In this case, the curves $\alpha_s$ and $\alpha_s'$ divide $A_{p,p'}\subset \Sigma\setminus \as$ into three connected components, $A_1$, $A_2$ and $A_3$. Let us write $A_1$ for the component that contains $z$, $A_2$ for the component that contains $w$ and $z'$, and $A_3$ for the component that contains $w'$. In the proof of \cite{ZemQuasi}*{Lemma~8.3} the Maslov index of $\phi$ is computed as
\begin{equation}
\mu(\phi)=\mu(\phi_0)+n_1(\phi)+n_2(\phi)+n_1'(\phi)+n_2'(\phi)+m_1(\cA)+m_2'(\cA)+2\sum_{\substack{\cD\in C(\Sigma\setminus \as)\\ \cD\neq A_{p,p'}}} n_{\cD}(\cA).\label{eq:Maslovindexpp'samecomp}
\end{equation} As in Case~\eqref{case:pp'diffcomponent}, all the summands in the above equation must vanish. Unlike in Case~\eqref{case:pp'diffcomponent}, this does not force the disk $\phi$ to be constant. Instead, there remains the possibility that all of the summands in Equation~\eqref{eq:Maslovindexpp'samecomp} are zero, but that we have
\[
m_2(\phi)=m_2(\cA)=m_1'(\cA)=m_1'(\phi)=1.
\] The homology class can be described as a bigon in each of the quasi-stabilization regions glued to a boundary degeneration with domain $A_2$. This is easily seen to be a class in $\pi_2(\ve{x}\times \xi^{\zs}\times (\xi^{\ws})', \ve{x}\times \theta^{\zs}\times (\theta^{\ws})')$. The only other classes which can contribute are the constant classes, which are always counted.  Hence it follows that if we write $F=\Phi_{J_s(\ve{T}_1)\to J_s(\ve{T}_2)}$ then
\begin{equation}
\begin{split}
F(\ve{x}\times \xi^{\ve{z}}\times (\theta^{\ws})')&=\ve{x}\times \xi^{\ve{z}}\times (\theta^{\ve{w}})',\\
 F(\ve{x}\times \xi^{\ve{z}}\times (\xi^{\ve{w}})')&=\ve{x}\times \xi^{\ve{z}}\times (\xi^{\ws})'+C\cdot \ve{x}\times \theta^{\ve{z}}\times(\theta^{\ve{w}})',\\
F(\ve{x}\times \theta^{\ve{z}}\times (\theta^{\ve{w}})')&=\ve{x}\times \theta^{\ve{z}}\times(\theta^{\ve{w}})', \\
F(\ve{x}\times\theta^{\ve{z}}\times(\xi^{\ws})')&=\ve{x}\times\theta^{\ve{z}}\times(\xi^{\ve{w}})',
\end{split}
\label{eq:changeofacstructurecomp}
\end{equation}
for some $C\in \bF_2$ (which is not independent of $\ve{T}_1$ and $\ve{T}_2$). Note that the map $F=\Phi_{J_s(\ve{T}_1)\to J_s(\ve{T}_2)}$ is $\cR_{\bmP}^-$-equivariant, so the above formulas extend equivariantly over $\cR_{\bmP}^-$.

We now explain why Equation~\eqref{eq:changeofacstructurecomp} implies that $S_{w',z'}^{\circ'}$ and $T_{w,z}^{\circ}$ commute (in the case that the basepoints are ordered $(w',z',w,z)$, read right to left).

To see that the relation $S_{w',z'}^+ T_{w,z}^+ \simeq T_{w,z}^+  S_{w',z'}^+$ follows from Equation~\eqref{eq:changeofacstructurecomp}, we pick neck lengths $0\ll T_1\ll T_1'$, so that $J_s(T_1,T_1')$ can compute the composition $S_{w',z'}^+ T_{w,z}^+$ and we pick neck lengths $0\ll T_2'\ll T_2$, so that $J_s(T_2,T_2')$ can compute the composition $T_{w,z}^+   S_{w',z'}^+$. Using the description of $F$ in Equation~\eqref{eq:changeofacstructurecomp}, we have
\[
(F  S_{w',z'}^+  T_{w,z}^+)(\ve{x})=\ve{x}\times \theta^{\zs}\times (\theta^{\ws})'
\] 
while also
\[
(T_{w,z}^+  S_{w',z'}^+)(\ve{x})=\ve{x}\times \theta^{\zs}\times (\theta^{\ws})',
\] 
implying that $S_{w',z'}^+  T_{w,z}^+\simeq T_{w,z}^+  S_{w',z'}^+$ on the level of transitive systems of chain complexes.

We now show that the relation $S_{w',z'}^+   T_{w,z}^-\simeq T_{w,z}^-   S_{w',z'}^+$ follows from Equation~\eqref{eq:changeofacstructurecomp}.  We note that to compute $S_{w',z'}^+   T_{w,z}^-$ we do not need to compute any change of almost complex structures. However to compute $T_{w,z}^-   S_{w',z'}^+$  we must insert a change of almost complex structure map between $T_{w,z}^-$ and $S_{w',z'}^+$ on $\cH^{++}$. We pick neck lengths $T_1,$ $T_1',$ $T_2$ and $T_2'$ with the same relative sizes as above and we compute that
\[
(S_{w',z'}^+  T_{w,z}^-)(\ve{x}\times \xi^{\zs})=\ve{x}\times (\theta^{\ws})' \quad \text{and}\quad (S_{w',z'}^+  T_{w,z}^-)(\ve{x}\times (\theta^{\zs})')=0,
\]
 while
\[
(T_{w,z}^-  F   S_{w',z'}^+)(\ve{x}\times \xi^{\zs})=\ve{x}\times (\theta^{\ws})'\quad 
\text{and}\quad (T_{w,z}^-  F   S_{w',z'}^+)(\ve{x}\times \theta^{\zs})=0,
\]
 so we see that $T_{w,z}^-  S_{w',z'}^+=S_{w',z'}^+  T_{w,z}^+$.

One proves the relations $T_{w,z}^+  S_{w',z'}^-\simeq S_{w',z'}^-  T_{w,z}^+$ and $T_{w,z}^-  S_{w',z'}^-\simeq S_{w',z'}^-  T_{w,z}^-$ in a similar manner, though we leave these last two computations to the reader.  
\end{proof}

The strategy used in the previous proof fails to show, for example, that $S_{w',z'}^+$ and $T_{w,z}^-$ commute if $(w,z,w',z')$ are four basepoints on a link (ordered right to left). Using the interpretation of the quasi-stabilization maps in terms of the dividing sets, we note that we should not expect these maps to commute, as the dividing sets for the two compositions are different. The left two dividing sets of Figure~\ref{fig::76} represent the compositions $T^-_{w,z}  S_{w',z'}^+$ and $S_{w',z'}^+  T_{w,z}^-$, and are not isotopic. Instead, the dividing sets satisfy a bypass triple with a third dividing set. This is reflected algebraically by the following lemma, which we will not use elsewhere in the paper.

\begin{figure}[ht!]
\centering
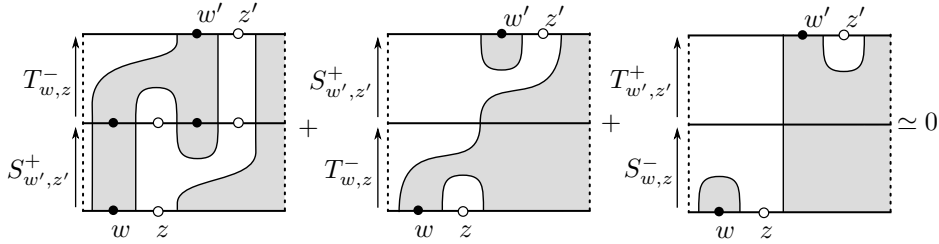
\caption{\textbf{The dividing set  interpretation of the relations from Lemma~\ref{lem:bypasstriplefromquasistabilization}.} This is an example when two quasi-stabilization maps fail to commute. The three dividing sets form a bypass triple. \label{fig::76}}
\end{figure}

\begin{lem}\label{lem:bypasstriplefromquasistabilization}Suppose that $\bL=(L,\ws,\zs)$ is a multi-based link in $Y$ and $(w,z,w' ,z')$ is a quadruple of new basepoints, which are consecutive on $\bL$ and not contained in $\ws\cup \zs$. Then
\[
T_{w,z}^- S_{w',z'}^++S_{w',z'}^+ T_{w,z}^-\simeq T_{w',z'}^+S_{w,z}^-,
\] as maps from $\cCFL^-(Y,(\bL_{w,z}^+)^{\sigma},\frs)$ to  $\cCFL^-(Y,(\bL_{w',z'}^+)^{\sigma'},\frs)$, where $\bL_{w,z}^+$ and $\bL_{w',z'}^+$ are obtained by adding $(w,z)$ or $(w',z')$ to $\bL$, respectively, and $\sigma$ and $\sigma'$ are colorings such that $\sigma|_{\bL}=\sigma'|_{\bL}$ and all of the quasi-stabilization maps are defined.
\end{lem}

\begin{proof}We start with the relation
\[
T^+_{w,z}S_{w,z}^-+S_{w,z}^+T_{w,z}^-+ \id\simeq 0,
\]
 from Lemma~\ref{lem:addtrivialstrandII} and then multiply on the left by $T_{w,z}^-S_{w',z'}^+$ to get
\[
T_{w,z}^-S_{w',z'}^+T^+_{w,z}S_{w,z}^-+T_{w,z}^-S_{w',z'}^+S_{w,z}^+T_{w,z}^-+ T_{w,z}^-S_{w',z'}^+\simeq 0.
\]
 We now can manipulate the terms to see that
\begin{align*}0&\simeq T_{w,z}^-S_{w',z'}^+T^+_{w,z}S_{w,z}^-+T_{w,z}^-S_{w',z'}^+S_{w,z}^+T_{w,z}^-+ T_{w,z}^-S_{w',z'}^+\\
&\simeq  T_{w,z}^-S_{w',z'}^+T^+_{w,z}S_{w,z}^-+T_{w,z}^-S_{w,z}^+S_{w',z'}^+T_{w,z}^-+ T_{w,z}^-S_{w',z'}^+
&&\text{(Proposition~\ref{thm:quasistabcommute})}\\
&\simeq  T_{w,z}^-S_{w',z'}^+T^+_{w,z}S_{w,z}^-+S_{w',z'}^+T_{w,z}^-+ T_{w,z}^-S_{w',z'}^+
&&\text{(Lemma~\ref{lem:addtrivialstrandII})}\\
&\simeq  T_{w,z}^-\Phi_{w'}T_{w',z'}^+T^+_{w,z}S_{w,z}^-+S_{w',z'}^+T_{w,z}^-+ T_{w,z}^-S_{w',z'}^+
&&\text{(Lemma~\ref{lem:olddefsforTw,z})}\\
&\simeq T_{w,z}^-\Phi_{w'}T^+_{w,z}T_{w',z'}^+S_{w,z}^-+S_{w',z'}^+T_{w,z}^-+ T_{w,z}^-S_{w',z'}^+
&&\text{(Proposition~\ref{lem:z-quasistabcommute})}\\
&\simeq T_{w',z'}^+S_{w,z}^-+S_{w',z'}^+T_{w,z}^-+ T_{w,z}^-S_{w',z'}^+
&&(\text{Lemma~\ref{lem:oldaddtrivialstrand}}),
\end{align*}
completing the proof.
\end{proof}

\subsection{Quasi-stabilization and basepoint moving maps}
\label{sec:quasi-stab-and-identifications}

In this section, we describe some useful formulas involving the quasi-stabilization maps and the diffeomorphism maps on link Floer homology induced by moving basepoints. The relations in this section generalize the relations for moving basepoints from \cite{ZemQuasi}.

We recall that we constructed quasi-stabilization maps for adding a new pair of adjacent basepoints $(w,z)$. The exact construction differed depending on the relative ordering of $w$ and $z$. If $w$ followed $z$, we write $S_{w,z}^{\circ}$ and $T_{w,z}^{\circ}$. If instead $z$ follows $w$, then we write  $S_{z,w}^{\circ}$ and $T_{z,w}^\circ$. Analyzing our proposed interpretation in terms of dividing sets, one should expect the two constructions to be related by a diffeomorphism map. In this section, we prove a precise relation.

Suppose that $\bL=(L,\ve{w},\ve{z}_0\cup \{z'\})$ is a multi-based link. Suppose that $(w,z)$ is a new pair of adjacent basepoints, which are contained in a single component of $L\setminus (\ve{w}\cup \ve{z}_0\cup \{z'\})$, such that $z'$ is adjacent to $w$ and immediately follows $w$ with respect to the links orientation. Let $\tau^{z'\to z}$ be a diffeomorphism
\begin{equation}
\tau^{z'\to z}\colon (Y,L,\ve{w},\ve{z}_0\cup \{z'\})\to (Y,L,\ve{w}, \zs_0\cup \{z\})\label{eq:rhodefinitionStype}
\end{equation}  such that $\tau^{z'\to z}(z')=z$, and $\tau^{z'\to z}$ is the identity outside of a neighborhood of the arc between $z'$ and $z$. Define the diffeomorphism $\tau^{z'\from z}$ to be the inverse of $\tau^{z'\to z}$.

\begin{rem} Our notation for the basepoint moving maps follows the right-to-left convention described in Section~\ref{subsec:convention}. Hence the notation $\tau^{z'\to z}$ and $\tau^{z'\from z}$ encodes the fact that $z'$ immediately follows $z$ with respect to the link's orientation.
\end{rem}

\begin{lem}\label{lem:mapsforI+agree} There are chain homotopies 
\[
S_{w,z}^+\simeq S_{z',w}^+\tau_*^{z'\to z} \qquad \text{and}\qquad S_{w,z}^-\simeq \tau_*^{z'\from z} S_{z',w}^-.
\]
\end{lem}

\begin{proof} We will focus on the relation $S_{w,z}^+\simeq S_{z',w}^+\tau_*^{z'\to z}$, since the other relation follows from a similar argument.

We consider the diagram shown in Figure~\ref{fig::56}, which can be used to compute both $S_{w,z}^+$ and $S_{z',w}^+$, though not \emph{a priori} with the same almost complex structure. To compute the map $S_{w,z}^+$, we need to stretch along the circle $c_{\beta}$, encircling $\beta_0$, whereas  to compute $S_{z',w}^+$ we need to stretch along the curve $c_{\alpha}$, encircling the curve $\alpha_0$. The argument will proceed by analyzing a change of almost complex structure map between an almost complex structure stretched along $c$ and $c_\alpha$ and an almost complex structure stretched along $c$ and $c_\beta$.

\begin{figure}[ht!]
\centering
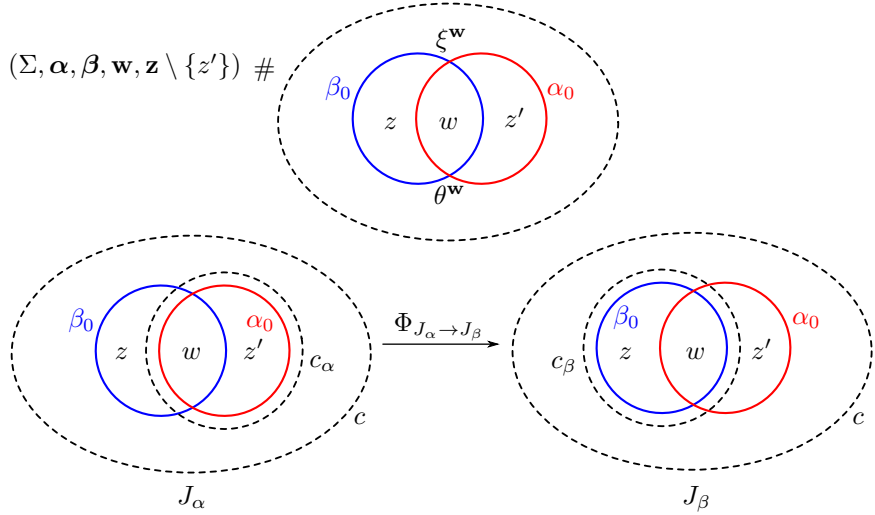
\caption{\textbf{The diagram we consider in Lemma~\ref{lem:mapsforI+agree}.}
On the bottom row, we indicate the curves $c$, $c_{\alpha}$ and $c_{\beta}$ along which the almost complex structures $J_\alpha$ and $J_\beta$ are stretched.\label{fig::56}}
\end{figure}

We pick a  cylindrical almost complex structure $J_s$ on $\Sigma\times [0,1]\times \R$ which is split in a neighborhood of $z'$ on $\Sigma$. We construct the cylindrical almost complex structure $J_{\alpha}$ on $\Sigma\times [0,1]\times \R$ by cutting out a neighborhood of $z'$, and gluing in an almost complex structure on $D^2\times [0,1]\times \R$, for a disk $D^2$ containing the diagram shown in Figure~\ref{fig::56}. Furthermore, we assume that $J_{\alpha}$ has long necks inserted along $c$ and $c_{\alpha}$. Similarly we construct an almost complex structure $J_\beta$ which is stretched along $c$ and $c_{\beta}$.

We make the following two subclaims:

\begin{enumerate}[leftmargin=22mm, ref= \arabic*, label=\textrm{Subclaim} (\arabic*):]
\item \label{subclaim1} \emph{If $J_{\alpha}$ and $J_{\alpha}'$ are two almost complex structures stretched along $c$ and $c_{\alpha}$, then
\[\Phi_{J_\alpha\to J_{\alpha}'}=\begin{pmatrix} \id&*\\
0& \id\end{pmatrix},\] as long as all neck lengths are sufficiently long (but irrespective of the relative lengths). The matrix is written in terms of the notation from Proposition~\ref{prop:quasi-stabilizeddifferential}, with $\theta^{\ws}$ as the  first row and column and $\xi^{\ws}$ as the second. A similar result holds for $J_{\beta}$.}
\item\label{subclaim2} \emph{If $J_\alpha$ is sufficiently stretched along $c$ and $c_{\alpha}$ and $J_{\beta}$ is sufficiently stretched along $c$ and $c_{\beta}$ then
\[\Phi_{J_{\alpha}\to J_{\beta}}=\begin{pmatrix} \id&*\\
0& \id\end{pmatrix}.\]}
\end{enumerate}

We note that a more detailed analysis would allow one to strengthen Subclaim~\eqref{subclaim1}, by proving that the $*$ component of $\Phi_{J_{\alpha}\to J_{\alpha}'}$ can be taken to be $0$, however we will not have need for this result. The $*$ component in $\Phi_{J_{\alpha}\to J_{\beta}}$ cannot be taken to be 0 in general, however.

The proofs of both Subclaims~\eqref{subclaim1} and \eqref{subclaim2} follow the same line of reasoning, so we focus on Subclaim~\eqref{subclaim2}. The argument is similar to the proof of Proposition~\ref{prop:TScommute}. Suppose we take sequences of almost complex structures $\{J_{\alpha,i}\}_{i\in \N}$ and $\{J_{\beta,i}\}_{i\in \N}$, such that the neck lengths of $J_{\alpha,i}$ along $c$ and $c_{\alpha}$ approach $+\infty$, and the neck lengths of $J_{\beta,i}$ along $c$ and $c_{\beta}$ approach $+\infty$ as well. We can pick a sequence  non-cylindrical almost complex structures $\tilde{J}_i$, interpolating $J_{\alpha,i}$ and $J_{\beta,i}$, such that the $(\Sigma\times [0,1]\times \R, \tilde{J}_i)$ contains the almost complex submanifold
\[((\Sigma\setminus N_i)\times [0,1]\times \R, J_s),\] for a nested sequence of open neighborhoods $N_i\subset \Sigma$ such that $\bigcap_{i\in \N} N_i=\{p\}$, for some point $p\in \Sigma$. Also, importantly, we can do this for a fixed \emph{cylindrical} almost complex structure $J_s$ on $\Sigma\times [0,1]\times \R$. We will compute the map $\Phi_{J_{\alpha,i}\to J_{\beta,i}}$ by counting index 0 holomorphic disks with the almost complex structure $\tilde{J}_i$.

Given a class $\phi=\phi_\Sigma\# \phi_0\in \pi_2(\xs\times x,\ys\times y)$ with $\phi_{\Sigma}\in \pi_2(\xs,\ys)$ a class on $(\Sigma,\as,\bs)$ and $\phi_0\in \pi_2(x,y)$ a class on $(S^2,\alpha_0,\beta_0)$, the Maslov index of $\phi$ is easily computed to be
\begin{equation}\mu(\phi)=\mu(\phi_\Sigma)+\gr_{\ws}(x,y)+2n_{w}(\phi_0),\label{eq:MaslovindexStypestab}\end{equation} where $\gr_{\ws}(x,y)$ is the drop in $\gr_{\ws}$ grading from $x$ to $y$.

Given a sequence of $\tilde{J}_i$-holomorphic curves $u_i$ representing a class $\phi=\phi_\Sigma\# \phi_0$, we can extract a limit to a broken curve $U_\Sigma$ representing $\phi_\Sigma$ (note that technically a limiting collection of curves would also contain curves on $(S^2,\alpha_0,\beta_0)$ though such curves are not important for our present argument). In particular, the class $\phi_\Sigma$ would have a broken holomorphic representative, and hence $\mu(\phi_{\Sigma})\ge 0$. Hence,  Equation \eqref{eq:MaslovindexStypestab} implies that if $\gr_{\ws}(x,y)>0$, then $\mu(\phi_{\Sigma})\le -1$, so there are no holomorphic representatives. This implies that the lower left entry of the matrix for $\Phi_{J_\alpha\to J_{\beta}}$ is 0. Similarly if $\gr_{\ws}(x,y)=0$, then $\mu(\phi_{\Sigma})\le 0$. By transversality and the existence of a broken holomorphic representative of $\phi_{\Sigma}$, we conclude that $\mu(\phi_{\Sigma})=0$, so $\phi_{\Sigma}$ is the constant class. Furthermore, $n_w(\phi_0)=0$. It is easy to see that this implies that $\phi$ must be a constant class. Conversely, the constant homology classes always have representatives for $\tilde{J}_i$, which are counted by $\Phi_{J_\alpha\to J_{\beta}}$. Hence the diagonal entries of the change of almost complex structure map are identified with the identity map.

The argument establishing the form of $\Phi_{J_\alpha\to J_{\alpha}'}$ is essentially the same.

Next, we consider the diffeomorphism map $\tau_*^{z'\to z}$, on the unstabilized complexes. Let $J_s'$ denote an almost complex structure obtained by stretching $J_s$ along $c$ (note that the stretching is not important for analyzing the diffeomorphism map $\tau_*^{z'\to z}$, but will be the almost complex structure we will later use). We claim that on the unstabilized diagram, the diffeomorphism map 
\[\tau_*^{z'\to z}\colon  \cCFL^-_{J_s'}(\Sigma,\as,\bs,\ws,\zs_0\cup \{z'\})\to \cCFL^-_{J_s'}(\Sigma,\as,\bs,\ws,\zs_0\cup \{z\})\] takes the form
\begin{equation}\tau_*^{z'\to z}(\ve{x})=\ve{x},\label{eq:rhodiffeomorphismmap}\end{equation} extended equivariantly over the ring $\bF_2[U_{\ws}, V_{\zs}]$. The diffeomorphism map $\tau_*^{z'\to z}$ is the map induced by naturality, and hence is the composition of a tautological map from $\cCFL^-_{J_s'}(\Sigma,\as,\bs,\ws,\zs_0\cup \{z'\})$ to $\cCFL^-_{\tau_*^{z'\to z}J_s'}(\Sigma,\as,\bs,\ws,\zs_0\cup \{z\})$ (which by definition takes the form in Equation~\eqref{eq:rhodiffeomorphismmap}), and the change of almost complex structure map $\Phi_{\tau_*^{z'\to z} J_s'\to J_s'}$. To establish this, let us write $\tau_t$ for an isotopy of $\Sigma$ (defined over all $t\in \R$) with $\tau_t=\id$ for $t\le 0$ and $\tau_t=\tau^{z'\to z}$ for $t\ge 1$, which is supported in a neighborhood of a small path from $z'$ to $z$.  We define a self-diffeomorphism of $\Sigma\times [0,1]\times \R$ by the formula $P(x,s,t)=(\tau_{-t}(x),s,t)$, and we consider the non-cylindrical almost complex structure $\tilde{J}:=P_* J_s'$, which interpolates $\tau_*^{z'\to z}J_s'$ for $t\in (-\infty, 0]$ and $ J_s'$ for $t\ge 1$. Furthermore, since $P$ fixes the cylinders $(\as\cup \bs)\times\{0,1\}\times \R$, there is a bijection between the index 0 holomorphic disks counted by the change of almost complex structure map $\Phi_{\tau_*^{z'\to z} J_s'\to J_s'}$ and the set of index 0 $J_s'$-holomorphic disks. However $J_s'$ is a cylindrical almost  complex structure, so by transversality, the set of index 0 $J_s'$-holomorphic disks consists of only the constant disks. Hence $\Phi_{\tau_*^{z'\to z} J_s'\to J_s'}$ also only counts the constant homology classes, and is thus the identity map on intersection points.

Finally, proving the lemma statement is just a matter of putting the pieces together. Subclaim~\eqref{subclaim1} shows that $J_\alpha$ and $J_{\beta}$ can be used to compute the quasi-stabilization maps since additional stretching on $c_\alpha$ and $c_{\beta}$ preserves the images of the maps $S_{w,z}^+$ and $S_{z',w}^+$, respectively. On the other hand, using Subclaim~\eqref{subclaim2}, Equation~\eqref{eq:rhodiffeomorphismmap}, and the definitions of the quasi-stabilization maps, we directly compute that the following diagram commutes:
\[\begin{tikzcd}\cCFL^-_{J'_s}(\Sigma,\as,\bs,\ws,\zs_0\cup \{z'\})\arrow{d}{\tau_*^{z'\to z}}\arrow{r}{S_{w,z}^+} &\cCFL^-_{J_\alpha}(\Sigma,\as\cup \{\alpha_0\},\bs\cup \{\beta_0\},\ws\cup \{w\}, \zs_0\cup \{z,z'\})\arrow{d}{\Phi_{J_{\alpha}\to J_{\beta}}}\\
\cCFL^-_{J_s'}(\Sigma,\as,\bs,\ws,\zs_0\cup \{z\})\arrow{r}{S_{z',w}^+}& \cCFL^-_{J_{\beta}}(\Sigma,\as\cup \{\alpha_0\},\bs\cup \{\beta_0\},\ws\cup \{w\},\zs_0\cup \{z,z'\})
\end{tikzcd}\] completing the proof.
\end{proof}

The interpretation of Lemma~\ref{lem:mapsforI+agree} in terms of dividing sets is shown in Figure~\ref{fig::97}.

\begin{figure}[ht!]
\centering
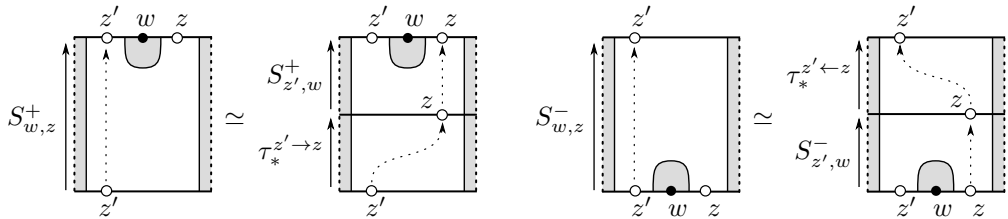
\caption{\textbf{The dividing set interpretation of the relations from  Lemma~\ref{lem:mapsforI+agree}.}
The dashed arrows in $\Sigma_{\zs}$ indicate which basepoints are identified by the diffeomorphism and quasi-stabilization maps. Note that the dashed arrows are \emph{not} part of the decoration of the surfaces appearing in our TQFT.\label{fig::97}}
\end{figure}

We now state the analog of Lemma~\ref{lem:mapsforI+agree} for the type-$T$ quasi-stabilization maps. Suppose that $\bL=(L,\ve{w}_0\cup \{w'\},\ve{z})$ is a multi-based link, and that $(w,z)$ is a pair of new basepoints in a component of $L\setminus (\ws_0\cup \{w'\}\cup\zs)$,  $z$ is adjacent to $w'$, and the three basepoints are ordered $(w,z,w')$, read right to left. 
  There is a  diffeomorphism
\[
\tau^{w\from  w'}\colon (Y,L,\ve{w}_0\cup \{w'\}, \ve{z})\to (Y,L,\ve{w}_0\cup \{w\}, \ve{z}),
\]
 well-defined up to isotopy, which satisfies
\[
\tau^{w\from w'}(w')=w
\] 
and  which is the identity outside of a neighborhood of the arc from $w$ to $w'$. Define the diffeomorphism $\tau^{w\to  w'}$ to be  the inverse of $\tau^{w\from w'}$.
Analogously to Lemma~\ref{lem:mapsforI+agree}, we have the following:

\begin{lem}\label{lem:mapsforII+agree} There are chain  homotopies
\[T_{w,z}^+\simeq T_{z,w'}^+\tau_*^{ w\from w'} \qquad \text{and} \qquad T_{w,z}^-\simeq \tau^{w\to w'}_*T_{z,w'}^-.\]
\end{lem}
\begin{proof} The proof follows by switching the roles of the $\ws$ and $\zs$-basepoints in the proof of Lemma~\ref{lem:mapsforI+agree} . 
\end{proof}

The interpretation of Lemma~\ref{lem:mapsforII+agree} in terms of surfaces with divides is shown in Figure~\ref{fig::100}.

\begin{figure}[ht!]
\centering
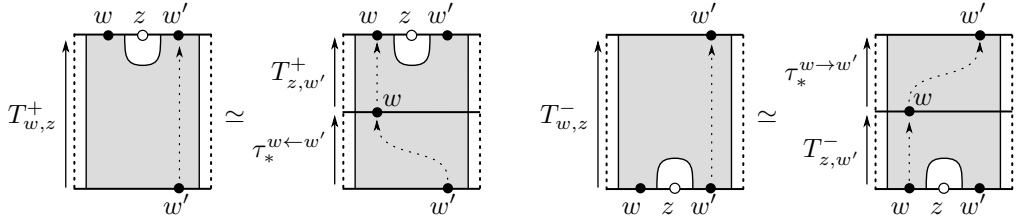
\caption{\textbf{The interpretation of the relations from Lemma~\ref{lem:mapsforII+agree} in terms of dividing sets.} The dashed arrows show how the basepoints are identified using the maps, but are not part of the data of the TQFT.\label{fig::100}}
\end{figure}

We now extend Lemmas~\ref{lem:mapsforI+agree}~and~\ref{lem:mapsforII+agree} by combining them with some previous results:

\begin{lem}\label{lem:movingbasepointsstep1}\label{lem:movingbasepointswithoutcolors} Suppose that $(L,\ws,\zs_0\cup \{z\})$ is a multi-based link, and $(z',w)$ are two new basepoints in a single component of $L\setminus (\ws\cup \zs_0\cup \{z\})$, such that $(z',w,z)$ form a consecutive triple of basepoints (ordered right to left). The map induced by the diffeomorphism $\tau^{ z'\from z}\colon (Y,L,\ws,\zs_0\cup \{z\})\to (Y,L,\ws,\zs_0\cup \{z'\})$ satisfies
\[
\tau_*^{z'\from z}\simeq T_{w,z}^-S_{z',w}^+\simeq S_{w,z}^-T_{z',w}^+.
\]
\end{lem}
\begin{proof} We start with the relation
\begin{equation}
S_{w,z}^+\tau_*^{z'\from z}\simeq S_{z',w}^+,\label{eq:basepointmoving1}
\end{equation} which is obtained by rearranging Lemma~\ref{lem:mapsforI+agree}. We compose Equation~\eqref{eq:basepointmoving1} with $T_{w,z}^-$ on the left to get
\[
T_{w,z}^- S_{w,z}^+ \tau_*^{z'\from z}\simeq T_{w,z}^- S_{z',w}^+.
\]
 Noting that $T_{w,z}^-S_{w,z}^+\simeq \id$ by Lemma~\ref{lem:addtrivialstrandII}, we obtain the relation $\tau_*^{z'\from z}\simeq T_{w,z}^- S_{z',w}^+$.
 
 To obtain the relation $T_{w,z}^-S_{z',w}^+\simeq S_{w,z}^-T_{z',w}^+$, we perform the following manipulation:
 \begin{align*} T_{w,z}^-S_{z',w}^+&\simeq S_{w,z}^-\Psi_z S_{z',w}^+&& \text{(Lemma~\ref{lem:olddefsforTw,z})}\\
 &\simeq S_{w,z}^-\Psi_{z'}S_{z',w}^++\Psi_{z'} S_{w,z}^-S_{z',w}^+&& \text{(Lemma~\ref{lem:Spm-Psicommutatornotadjacent})}\\
 &\simeq S_{w,z}^-T_{z',w}^++\Psi_{z'}S_{w,z}^-S_{z',w}^+ && \text{(Lemma~\ref{lem:olddefsforTw,z})}.
 \end{align*}
 Note that if we can show that $\Psi_{z'}S_{w,z}^-S_{z',w}^+\simeq 0$ we will be done. To establish this, we note that $\Psi_{z'}S_{w,z}^-S_{z',w}^+\simeq \Psi_{z'} S_{w,z}^- S_{w,z}^+\tau_*^{z'\from z}$ by Lemma~\ref{lem:mapsforI+agree}, and $S_{w,z}^-S_{w,z}^+\simeq 0$ by Lemma~\ref{lem:addtrivialstrandII}.
\end{proof}

The relations from Lemma~\ref{lem:movingbasepointsstep1} are illustrated in Figure~\ref{fig::101}.

\begin{figure}[ht!]
\centering
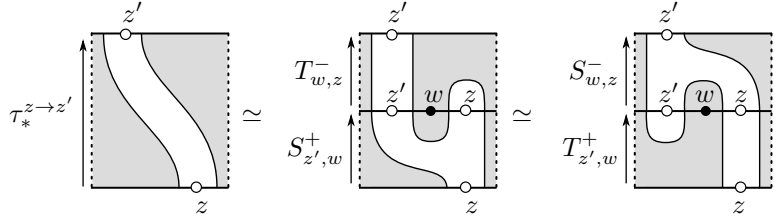
\caption{\textbf{The dividing set interpretation of the relations from Lemma~\ref{lem:movingbasepointsstep1}.}\label{fig::101}}
\end{figure}

Analogously to Lemma~\ref{lem:movingbasepointsstep1}, we have the following result about moving a single $\ws$ basepoint:

\begin{lem}\label{lem:movingbasepointsstep2} Suppose that $(L,\ws_0\cup \{w\},\zs)$ is a multi-based link, and $(z,w')$ are a new pair of basepoints, which are contained in a single component of $L\setminus (\ws_0\cup \{w\}\cup \zs)$ such that $w$ immediately follows $z$.  The map associated to the diffeomorphism $\tau^{w\to w'}\colon (Y,L,\ws_0\cup \{w\},\zs)\to (Y,L,\ws_0\cup \{w'\}, \zs)$ satisfies the relations
\[\tau^{w\to w'}_*\simeq S_{w,z}^-T_{z,w'}^+\simeq T_{w,z}^-S_{z,w'}^+.\]
\end{lem}

\begin{proof} One simply switches the roles of $\ws$ and $\zs$ in the proof of Lemma~\ref{lem:movingbasepointsstep1}.
\end{proof}

We now give some further expressions for the diffeomorphism maps induced by moving basepoints, which will be useful throughout the paper. Suppose that $\bL=(L,\ve{w}_0\cup \{w\},\ve{z}_0\cup \{z\})$ is a multi-based link, such that $w$ and $z$ are adjacent, and $w$ follows $z$. Suppose that $(w',z')$ are two new basepoints, contained in the component of $L\setminus (\ws_0\cup \zs_0\cup \{w,z\})$ which immediately follows $w$.  Write $\tau^{(w',z')\from (w,z) }$ for a diffeomorphism
\begin{equation}
\tau^{(w',z')\from (w,z)}\colon  (Y,L,\ws_0\cup \{w\},\zs\cup \{z\})\to (Y,L,\ws_0 \cup \{w'\}, \zs_0\cup \{z'\}),\label{eq:tauwzdefinition}
\end{equation}
which moves $(w,z)$ to $(w',z')$, but is fixed outside of a subinterval of $L$ containing $(w',z',w,z)$.

\begin{lem}\label{lem:newversionbasepointmovingmaps}Suppose that $\bL=(L,\ve{w}_0\cup \{w\},\ve{z}_0\cup \{z\})$ and $(w',z')$ are as above, with $(w',z',w,z)$ forming a tuple of consecutive basepoints (ordered right to left). Writing $\tau^{(w',z')\from (w,z) }$ for the diffeomorphism from Equation~\eqref{eq:tauwzdefinition}, we have
\[
\tau^{(w',z')\from (w,z)}_*\simeq S_{w,z}^-T_{w',z'}^+ \simeq S_{w,z}^-\Psi_{z'} S_{w',z'}^+\simeq T_{w,z}^-\Phi_w T_{w',z'}^+.
\]
If instead the basepoints are ordered $(w,z,w',z')$ (read right to left), and $\tau^{(w,z)\to (w',z')}$ is the diffeomorphism constructed by moving $(w,z)$ to $(w',z')$, analogous to the previous situation, then
\[
\tau_*^{(w,z)\to (w',z')}\simeq T_{w,z}^- S_{w',z'}^+\simeq S_{w,z}^- \Psi_z S_{w',z'}^+\simeq T_{w,z}^- \Phi_{w'} T_{w',z'}^+.
\]
\end{lem}

\begin{proof}Suppose first that the basepoints are ordered $(w',z',w,z)$, read right to left. First note that we can decompose the diffeomorphism $\tau^{(w',z')\from (w,z)}$ as the composition $\tau^{z'\from z}\circ \tau^{w'\from w}$. We compute
\begin{align*}\tau^{(w',z')\from (w,z)}_*& \simeq \tau^{ z'\from z}_*\tau^{ w'\from w}_*&& \\
& \simeq S^-_{w,z} T^+_{z',w} \tau^{ w'\from w}_*&& \text{(Lemma~\ref{lem:movingbasepointsstep1})}\\
&\simeq S^-_{w,z}T^+_{w',z'}&& \text{(Lemma~\ref{lem:mapsforII+agree})}.
\end{align*}

Note that on the intermediate link Floer complex with all four basepoints $(w,z,w',z')$, we give $w$ and $w'$ the same color, and $z$ and $z'$ the same color. Applying Lemma~\ref{lem:olddefsforTw,z}, we see that
\[
S_{w,z}^-T_{w',z'}^+\simeq S_{w,z}^- \Psi_{z'} S_{w',z'}^+\qquad \text{and} \qquad S_{w,z}^- T_{w',z'}^+\simeq T_{w,z}^- \Phi_w T_{w',z'}^+,
\] completing the proof in the case that the basepoints are ordered $(w',z',w,z)$, read right to left.

The proof of the relations in the case that the basepoints are ordered $(w,z,w',z')$, read right to left, is a simple modification. 
\end{proof}

The dividing set interpretation of some of the relations from Lemma~\ref{lem:newversionbasepointmovingmaps} is shown in Figure~\ref{fig::102}.

\begin{figure}[ht!]
\centering
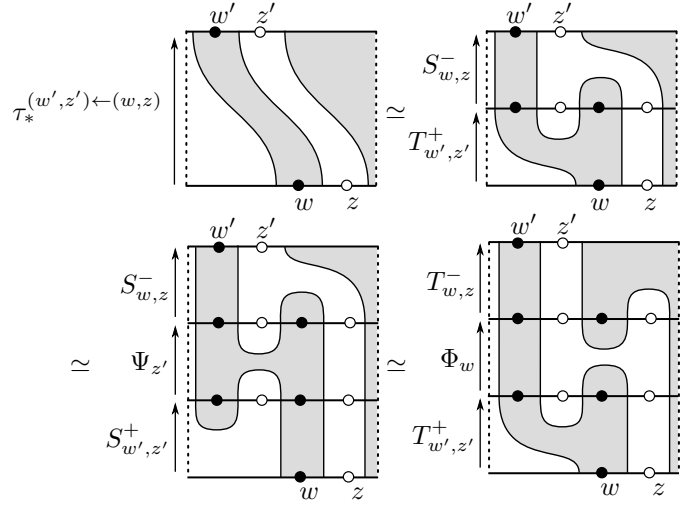
\caption{\textbf{The dividing set interpretation of the relations from Lemma~\ref{lem:newversionbasepointmovingmaps}.}\label{fig::102}}
\end{figure}

\subsection{Quasi-stabilization and Heegaard triples}

In this section, we state a useful result about quasi-stabilization and triangle maps, from \cite{MOIntSurg}.

Given a Heegaard triple $(\Sigma,\as,\bs,\gs)$, there is a 4-manifold $X_{\as,\bs,\gs}$, described by Ozsv\'{a}th and Szab\'{o} \cite{OSTriangles}*{Section~2.2}. Let $\Delta$ denote a 2-cell, viewed as a triangle with boundary edges labeled $e_{\as}$, $e_{\bs}$ and $e_{\gs}$. If $\taus\in \{\as,\bs,\gs\}$, write $U_{\taus}$ for the 3-dimensional handlebody which has boundary $-\Sigma$ and has $\taus$ as a collection of compressing disks. The 4-manifold $X_{\as,\bs,\gs}$ is defined as
\[
X_{\as,\bs,\gs}:=\left((\Delta\times \Sigma)\sqcup (e_{\as}\times U_{\as})\sqcup (e_{\bs} \times U_{\bs})\sqcup (e_{\gs}\times U_{\gs})\right) /{\sim},
\]
where $\sim$ denotes gluing $e_{\taus}\times U_{\taus}$ to $\Delta\times \Sigma$ in the natural way.

Given a Heegaard triple $(\Sigma, \ve{\alpha},\ve{\beta},\ve{\gamma},\ve{w},\ve{z})$, as well as a curve $\alpha_s\subset \Sigma\setminus \as$ and a distinguished point $p\in \alpha_s\setminus (\bs\cup \gs)$, we can construct a quasi-stabilized Heegaard triple $\cT^+_{\alpha_s,p}$, as shown in Figure~\ref{fig::16quasi}.

 Note that the 4-manifolds $X_{\as,\bs,\gs}$ and $X_{\as\cup \{\alpha_s\}, \bs\cup \{\beta_0\}, \gs\cup \{\gamma_0\}}$ are canonically diffeomorphic, since both are constructed by gluing $\Sigma\times \Delta$ together with the handlebodies $U_{\as}\times e_{\as},$ $U_{\bs}\times e_{\bs}$ and $U_{\gs}\times e_{\gs}$, and the Heegaard surface and the 3-dimensional handlebodies $U_{\as},$ $U_{\bs}$ and $U_{\gs}$  are unchanged when we quasi-stabilize. Furthermore, it is straightforward to see that if $\psi\in \pi_2(\xs,\ys,\zs)$ is a homology class of triangles on the triple $(\Sigma,\as,\bs,\gs)$ and $\psi_0$ is a triangle on the subdiagram $(S^2,\alpha_s,\beta_0,\gamma_0)$ which has the same multiplicity on both sides of $\alpha_s$ at the connected sum point, then $\frs_{\ws}(\psi\# \psi_0)=\frs_{\ws}(\psi)$ (note however that not all homology classes of triangles on $\cT^+_{p,\alpha_s}$ can be written as $\psi\#\psi_0$ for such triangles $\psi$ and $\psi_0$).

If $J_s$ is an almost complex structure on $\Sigma\times[0,1]\times \R$ which is split in a neighborhood of $p$, we will write $J_s(T)$  for an almost complex structure obtained from $J_s$ by inserting a sufficiently large neck along the curve $c$, shown in Figure~\ref{fig::16quasi}.

We will write (abusing notation slightly)
\[
\alpha_s\cap \beta_0=\{\theta^{\ws},\theta^{\zs}\},\qquad \alpha_s\cap \gamma_0=\{\theta^{\ws}, \theta^{\zs}\}\qquad \text{and} \qquad \beta_0\cap \gamma_0=\{\theta^+,\theta^-\},
\]
In the above expressions, if $\os\in \{\ws,\zs\}$, then we write $\theta^{\os}$ for the top $\gr_{\os}$-graded intersection point when the designation as top and bottom degree intersection point depends on the choice of grading. We write $\theta^+$ for the top graded intersection point (when the designation is independent of the choice of grading).

   \begin{figure}[ht!]
\centering
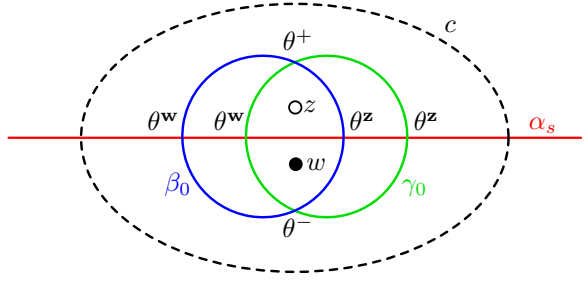
\caption{\textbf{The quasi-stabilized triple $\cT_{\alpha_s,p}^+$ from Proposition~\ref{prop:singlealphaquasistabtriangle}}}.\label{fig::16quasi}
\end{figure}

We state the following:

\begin{prop}[\cite{MOIntSurg}*{Proposition~5.2}]\label{prop:singlealphaquasistabtriangle} Suppose that $\cT=(\Sigma,\ve{\alpha},\ve{\beta},\ve{\gamma},\ve{w},\ve{z})$ is a strongly $\frs$-admissible triple and suppose that $\alpha_s$ is a new $\ve{\alpha}$-curve, passing through the point $p\in \Sigma$. Let $\cT_{\alpha_s,p}^+$ denote the quasi-stabilized Heegaard triple, as in Figure~\ref{fig::16quasi}. For sufficiently large $T$, we have
\[F_{\cT_{\alpha_s,p}^+, J_s(T),\frs}(\ve{x}\times \theta^{\ws}, \ve{y}\times \theta^+)=F_{\cT,J_s,\frs}(\ve{x},\ve{y})\otimes \theta^{\ws} \qquad \text{and}\]
\[F_{\cT_{\alpha_s,p}^+, J_s(T),\frs}(\ve{x}\times \theta^{\zs}, \ve{y}\times \theta^+)=F_{\cT,J_s,\frs}(\ve{x},\ve{y})\otimes \theta^{\zs}.\] 
\end{prop}

\begin{rem}Proposition~\ref{prop:singlealphaquasistabtriangle} was used \cite{ZemQuasi} in order to prove that the quasi-stabilization maps are natural. It implies that the quasi-stabilization maps $S_{w,z}^{\circ}$ and $T_{w,z}^\circ$ commute with handleslides and isotopies of the $\bs$ curves on the unstabilized diagram. However Proposition~\ref{prop:singlealphaquasistabtriangle} is not sufficient to show that the quasi-stabilization maps commute with moves of the $\as$ curves, or are independent of the choice of $\alpha_s$ curve. To fully show naturality of the quasi-stabillization maps, one needs an additional triangle count \cite{ZemQuasi}*{Theorem~6.5}, which has a non-trivial requirement on the unstabilized Heegaard triple $(\Sigma,\as,\bs,\gs)$.
\end{rem}

\section{Maps for 4-dimensional handles}
\label{sec:handleattachmentmapsawayfromL}
In this section we define maps for 4-dimensional handles attached in $Y\setminus L$. The maps we describe in this section are modifications of the maps from \cite{OSTriangles}, though there are some notable differences.

\subsection{Framed spheres and handle attachment cobordisms}
\label{subsec:modelhandlecobsawayfromL}

\begin{define}A \emph{framed $k$-sphere} in $Y\setminus L$ is a map 
\[
\bS\colon S^k\times D^{3-k}\to Y\setminus L
\] which is an embedding.
\end{define}

Given a framed $k$-sphere $\bS$ in $Y\setminus L$, we define the \emph{surgery} of $\bS$ to be the 3-manifold
\[
Y(\bS):=(Y\setminus \Int(\im (\bS)))\cup (D^{k+1}\times S^{3-k-1}).
\]
Note that the link $L\subset Y$ induces a link $L\subset Y(\bS)$. We define the \emph{trace} of $\bS$ to be the 4-manifold
\[W(Y,\bS):=([0,1]\times Y)\cup H_{k+1}\] where 
\[H_{k+1}:=D^{k+1}\times D^{3-k}\] is attached to $\{1\}\times \im (\bS)\subset \{1\}\times Y$. We define the \emph{trace link cobordism} of $\bS$ to be the undecorated link cobordism
\[\cW(Y,L,\bS):=(W(Y,\bS),[0,1]\times L)\colon (Y,L)\to (Y(\bS), L).
\]

If $\bS$ is instead a collection of framed  spheres in $Y\setminus L$, with pairwise disjoint images, we can still form the surgered manifold $Y(\bS)$ and the trace link cobordism $\cW(Y,L,\bS)$, by performing the surgeries or handle attachments simultaneously.

Note that although $S^{-1}\times D^{4}$ is the empty set, it is useful to distinguish between a $(-1)$-dimensional sphere $\bS^{-1}$, and the  framed link with no components $\bS_\varnothing$. We define the trace of $\bS_{\varnothing}$ to be $[0,1]\times Y$, while we define the trace of a $(-1)$-dimensional sphere $\bS^{-1}$ to be $([0,1]\times Y)\sqcup D^4$.

\subsection{0-handle and 4-handle maps}
\label{sec:0/4-handlemaps}

We now describe handle attachment maps for 0-handles and 4-handles, which we think of as the maps for framed $(-1)$- and $3$-dimensional spheres in $Y\setminus L$. In terms of decorated link cobordisms, the 0-handle map will be the cobordism map for a decorated link cobordism $(W,\cF)$ with $\cF=(\Sigma,\cA)$, where $W$ has a Morse function $f$ such that there is a point $p\in \Int (W)$ which is the unique critical point of $f$, $f|_{\Sigma}$ and $f|_{\cA}$, and has index 0 for all three. The 4-handle map will correspond to a similar cobordism, but turned upside.

Suppose that $(Y,\bL)$ is a  3-manifold containing a multi-based link, and let $(S^3,\bU)$ denote a copy of $S^3$ containing a doubly based unknot. Write $\bL=(L,\ws,\zs)$ and $\bU=(U,w,z)$, and let $(\Sigma,\as,\bs,\ws,\zs)$ be a diagram for $(Y,\bL)$. Let $(S^2,w,z)$ denote a diagram for $(S^3,\bU)$ which has no $\as$- or $\bs$-curves. Suppose that $\sigma\colon \ws\cup \zs\to \bmP$ is a coloring, and $\sigma'\colon \ws\cup \zs\cup \{w,z\}\to \bmP$ is a coloring which extends $\sigma$.

We now describe the 0-handle map
\[
F_{Y,\bL,\bS^{-1},\hat{\frs}}\colon \cCFL^-(Y,\bL^\sigma,\frs)\to  \cCFL^-(Y\sqcup S^3, (\bL\sqcup \bU)^{\sigma'},\frs\sqcup \frs_0).
\]
 Here $\hat{\frs}$ is the unique $\Spin^c$ structure on $[0,1]\times Y\cup D^4$ which extends $\frs$.

 Noting that the set of intersection points on the diagrams $(\Sigma,\as,\bs)$ and $(\Sigma\sqcup S^2,\as,\bs)$ are equal, we define the map
\[
F_{Y,\bL,\bS^{-1},\hat{\frs}}\colon \cCFL^-(\Sigma,\as,\bs,\ws,\zs,\sigma,\frs\sqcup \frs_0)\to \cCFL^-(\Sigma\sqcup S^2,\as,\bs,\ws,\zs,\sigma',\frs\sqcup \frs_0)
\] via the  formula
\begin{equation}
F_{Y,\bL,\bS^{-1},\hat{\frs}}(\ve{x})=\ve{x}.
\label{eq:0-handlemap}
\end{equation} 
The map $F_{Y,\bL,\bS^{-1},\hat{\frs}}$  is a chain map and commutes with the change of diagrams maps for $Y$.

Dually, a framed 3-sphere in a 3-manifold consists of a distinguished component of $Y$ which is identified with $S^3$. If this copy of $S^3$ contains a doubly based unknot $\bU$, we can define the 4-handle map
\[
F_{Y,\bL,S^3,\hat{\frs}}\colon \cCFL^-(Y\sqcup S^3,(\bL\sqcup \bU)^{\sigma'},\frs\sqcup \frs_0)\to \cCFL^-(Y, \bL^{\sigma},\frs)
\]
using the inverse of the formula for the 0-handle in Equation~\eqref{eq:0-handlemap}.

\subsection{1-handle and 3-handle maps} 
\label{sec:1--handlemaps}

We now describe the 1-handle and 3-handle maps. The constructions we present are similar but not identical to the 1-handle and 3-handle maps defined in \cite{OSTriangles}. See also the construction of 1-handle and 3-handle maps from \cite{JCob}. If $(Y,\bL)$ is a 3-manifold with a multi-based link, a 1-handle is the trace cobordism of a framed 0-sphere in $Y$. Similarly a 3-handle is the trace cobordism of a framed 2-sphere.

 Suppose $\bS^i$ is a framed $i$-sphere in $Y\setminus L$, for $i\in \{0,2\}$. Given a $\Spin^c$ structure $\frs$ on $Y$, there is a unique $\Spin^c$ structure $\hat{\frs}$ on the trace $W(Y,\bS^i)$ which extends $\frs$. We let $\frs'$ denote the restriction of $\hat{\frs}$ to $Y(\bS^i)$. In this section, we describe 1- and 3-handle maps
\[F_{Y,\bL,\bS^i,\hat{\frs}}\colon  \cCFL^-(Y,\bL^\sigma,\frs)\to \cCFL^-(Y(\bS^i),\bL^\sigma,\frs').\] 

Let us describe the 1-handle map first. We start with a diagram $\cH=(\Sigma,\as,\bs,\ws,\zs)$ such that the image of $\bS^0$ intersects $\Sigma$ in two disks $D_1$ and $D_2$ which are in the complement of $\as\cup \bs\cup \ws\cup \zs$. We can form a diagram $\cH'=(\Sigma',\as',\bs',\ws,\zs)$ for $(Y(\bS^0),\bL)$ by defining
\[
\Sigma':=(\Sigma\setminus (D_1\cup D_2))\cup A,
\]
where $A$ is an annulus, and defining
\[
\as':=\as\cup \{\alpha_0\}\qquad \text{and} \qquad \bs':=\bs\cup \{\beta_0\},
\] where $\alpha_0$ and $\beta_0$ are two homologically essential curves in $A$ which intersect in exactly two points. We will write $\alpha_0\cap \beta_0=\{\theta^+,\theta^-\}.$ The two points $\theta^+$ and $\theta^-$ are distinguished by the Maslov grading (note that both $\gr_{\ws}$ and $\gr_{\zs}$ induce the same decomposition of $\alpha_0\cap \beta_0$ into higher and lower intersection points). The 1-handle map
\[
F_{Y,\bL,\bS^0,\hat{\frs}}\colon \cCFL^-(\cH, \sigma,\frs)\to \cCFL^-(\cH',\sigma,\frs')
\] is defined via the formula
\[F_{Y,\bL,\bS^0,\hat{\frs}}(\ve{x})=\ve{x}\times \theta^+,\] extended equivariantly over $\cR_{\bmP}^-$.

In the opposite direction, the 3-handle map for a framed 2-sphere $\bS^2$ is defined by picking a Heegaard diagram $\cH'=(\Sigma',\as\cup \{\alpha_0\},\bs\cup \{\beta_0\},\ws,\zs)$ such that $\Sigma'$ intersects $\bS^2$ in an annulus $A$, which contains both the curves  $\alpha_0$ and $\beta_0$, which are homologically essential in $A$, and intersect in two points, $\theta^+$ and $\theta^-$. Furthermore $A$ intersects no curves in $\as$ or $\bs$. We let $\Sigma$ denote the surface obtained by removing $A$ from $\Sigma'$ and filling in the two boundary components with disks. Write $\cH=(\Sigma,\as,\bs,\ws,\zs)$. If $\hat{\frs}\in \Spin^c(W(Y,\bS^2))$, write $\frs$ and $\frs'$ for the restrictions of $\hat{\frs}$ to $Y$ and $Y(\bS^2)$, respectively. The 3-handle map
\[
F_{Y,\bL,\bS^2,\hat{\frs}}\colon \cCFL^-(\cH',\sigma,\frs)\to \cCFL^-(\cH,\sigma,\frs')
\]
is defined via the formula
\[
F_{Y,\bL,\bS^2,\hat{\frs}}(\ve{x}\times \theta^+)=0\qquad \text{and} \qquad F_{Y,\bL,\bS^2,\hat{\frs}}(\ve{x}\times \theta^-)=\ve{x},
\]
extended equivariantly over $\cR_{\bmP}^-$.

\begin{prop}\label{prop:1/3handlemapwelldefined}The 1-handle and 3-handle maps are chain maps for almost complex structures which have been sufficiently stretched on the two boundary components of the annulus $A$. Furthermore, the 1- and 3-handle maps are well-defined in the sense that they commute with the change of diagrams maps, up to filtered, equivariant chain homotopy.
\end{prop}

Proposition~\ref{prop:1/3handlemapwelldefined} can be proven by adapting the proofs from \cite{ZemGraphTQFT}*{Section~8},  which carry over to the setting of link Floer homology without major change. We will state an important holomorphic triangle map computation which is proven in \cite{ZemGraphTQFT} and used to prove well-definedness of the 1-handle and 3-handle maps. If $\cT=(\Sigma,\as,\bs,\gs,\ws,\zs)$ is
Heegaard triple with two points $p_1,p_2\in \Sigma\setminus (\as\cup \bs\cup \gs\cup \ws\cup \zs)$ we construct a new Heegaard triple $\cT^+=(\Sigma',\as\cup \{\alpha_0\}, \bs\cup \{\beta_0\}, \gs\cup \{\gamma_0\}, \ws,\zs)$, as follows. We remove neighborhoods of the two points $p_1$ and $p_2$, and connect the resulting boundary components with an annulus $A$ to form the surface $\Sigma'$. In the  annulus $A$, we add three homologically essential curves $\alpha_0,$ $\beta_0$ and $\gamma_0$. We assume that the intersection of subpair of $\{\alpha_0,\beta_0,\gamma_0\}$ consists of exactly two points. We will write (abusing notation slightly)
\[
\xi_0\cap \zeta_0=\{\theta^+,\theta^-\}
\]
whenever $\xi_0$ and $\zeta_0$ are distinct elements of $\{\alpha_0,\beta_0,\gamma_0\}$. The diagram $\cT^+$ is shown in Figure~\ref{fig::114}.

\begin{figure}[ht!]
\centering
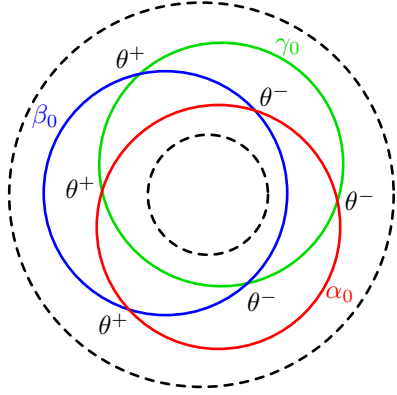
\caption{\textbf{The Heegaard triple $\cT^+$ used to show well-definedness of the 1- and 3-handle maps.} The almost complex structure is stretched along the two dashed curves.\label{fig::114}}
\end{figure}

It is a straightforward exercise using several Mayer-Vietoris exact sequences to show that there is a canonical isomorphism
\[
\Spin^c(X_{\as,\bs,\gs})\iso \Spin^c(X_{\as\cup \{\alpha_0\},\bs\cup \{\beta_0\}, \gs\cup \{\gamma_0\}}).
\]
If $\frs\in \Spin^c(X_{\as,\bs,\gs})$, we will write $\hat{\frs}$ for the corresponding element of $\Spin^c(X_{\as\cup \{\alpha_0\},\bs\cup \{\beta_0\}, \gs\cup \{\gamma_0\}})$.

Suppose $J$ is an almost complex structure on $\Sigma\times \Delta$, which is split on a neighborhood of $\{p_1,p_2\}\times\Delta$. If $\ve{T}=(T_1,T_2)$ is a pair of neck lengths,  write $J(\ve{T})$ for an almost complex structure on $\Sigma'\times \Delta$ constructed from $J$ with necks of length $T_1$ and $T_2$ inserted along the boundaries of the annulus $A$.

We state the following triangle count:
 \begin{prop}[\cite{ZemGraphTQFT}*{Theorem~8.8}] \label{prop:1-handletrianglecount} Suppose that $\cT=(\Sigma,\as,\bs,\gs,\ws,\zs)$ is a Heegaard triple with two chosen points $p_1,p_2\in \Sigma\setminus (\as\cup \bs\cup\gs\cup \ws\cup \zs)$. Let $\cT^+=(\Sigma',\as\cup \{\alpha_0\}, \bs\cup \{\beta_0\}, \gs\cup \{\gamma_0\}, \ws,\zs)$ denote the triple described above. If $J$ is an almost complex structure on $\Sigma\times \Delta$, and $J(\ve{T})$ is the almost complex structure on $\Sigma'\times \Delta$ for a pair $\ve{T}=(T_1,T_2)$ of neck lengths, then whenever both components of $\ve{T}$ are sufficiently large, we have that
 \begin{align*}F_{\cT^+, \hat{\frs}, J(\ve{T})}(\xs\times \theta^+, \ys\times \theta^+)&=F_{\cT,\frs,J}(\xs,\ys)\otimes \theta^+\\
 F_{\cT^+, \hat{\frs}, J(\ve{T})}(\xs\times \theta^+, \ys\times \theta^-)&=F_{\cT,\frs,J}(\xs,\ys)\otimes \theta^- +F_0(\xs,\ys)\otimes \theta^+\\
  F_{\cT^+, \hat{\frs}, J(\ve{T})}(\xs\times \theta^-, \ys\times \theta^+)&=F_{\cT,\frs,J}(\xs,\ys)\otimes \theta^- +G_0(\xs,\ys)\otimes \theta^+,
 \end{align*}
 where $F_0$ and $G_0$ are two maps $\cCFL^-(\as,\bs)\otimes \cCFL^-(\bs,\gs)\to \cCFL^-(\as,\gs)$ (which are not independent of $J$ and $\ve{T}$).
 \end{prop}

We now show that the 1-handle and 3-handle maps commute with each other. We note that if $\bS$ and $\bS'$ are disjoint framed spheres in $Y$, then there are diffeomorphisms
\[
W(Y(\bS),\bS')\circ W(Y,\bS)\iso W(Y,\bS\cup \bS')\iso W(Y(\bS), \bS')\circ W(Y,\bS),
\]
which are well defined up to isotopy (see Remark~\ref{rem:canonicaldiffeomorphism}).

\begin{lem}\label{lem:1-3-handlemapscommute} Suppose that $\bS$ and $\bS'$ are two  framed spheres, which are pairwise disjoint, and which  are each either 0-dimensional or 2-dimensional. If $\frs\in \Spin^c(W(Y,\bS,\bS'))$, then
\[F_{Y(\bS),\bL,\bS',\frs|_{W(Y(\bS),\bS')}}F_{Y,\bL,\bS,\frs|_{W(Y,\bS)}}\simeq F_{Y(\bS'),\bL,\bS,\frs|_{W(Y(\bS'),\bS)}}F_{Y,\bL,\bS',\frs|_{W(Y,\bS')}}.\]
\end{lem}

\begin{proof}The formulas for the maps look like they commute, though one needs to pay attention to the almost complex structures used to compute the two compositions, since it is not \emph{a priori} obvious that a single almost complex structure can be chosen to compute both. Our argument will be similar to the proof of Lemma~\ref{prop:TScommute}. We will show the result in the case that $\bS$ and $\bS'$ are both framed 0-spheres, since this is the simplest case notationally. The claim when one or both of $\bS$ and $\bS'$ is a 2-sphere follows from the small adaptation of the argument we present.

Let $\cH=(\Sigma,\as,\bs,\ws,\zs)$ be a diagram for $Y$, such that $\Sigma$ intersects each component of the images of $\bS$ and $\bS'$ along a single disk, which is disjoint from $\as\cup \bs\cup \ws\cup \zs$. Let $\cH^{++}=(\Sigma'',\as\cup \{\alpha_0,\alpha_0'\}, \bs\cup \{\beta_0,\beta_0'\}, \ws,\zs)$ denote the diagram for $Y(\bS,\bS')$ obtained by surgering the diagram $\cH$. There are two distinguished annuli, $A$ and $A'$, as well as four curves $c_1$, $c_2$, $c_1',$ and $c_2'$ (the boundaries of $A$ and $A'$) along which we will stretch the almost complex structure. The subregions $A$ and $A'$ of the diagram $\cH^{++}$ are shown in Figure~\ref{fig::107}.

\begin{figure}[ht!]
\centering
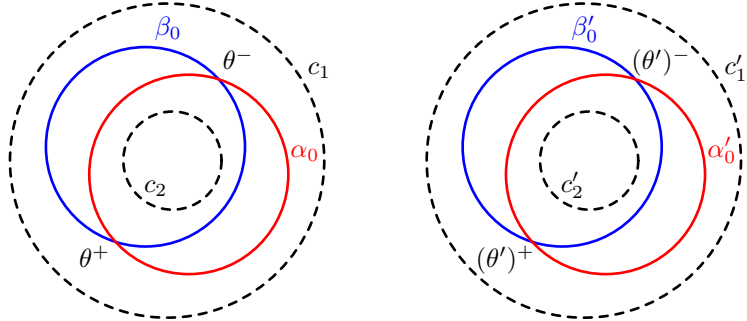
\caption{\textbf{The subregions $A$ and $A'$ of the diagram $\cH^{++}$.} The almost complex structures are stretched along the dashed curves labeled $c_1,$ $c_2$, $c_1
'$ and $c_2'$.\label{fig::107}}
\end{figure}

The maps for surgery on $\bS$ require the almost complex structure be stretched along $c_1$ and $c_2$, while the maps for surgery on $\bS'$ require the almost complex structure be stretched along $c_1'$ and $c_2'$.

Suppose that $\ve{T}=(T_1,T_2,T_1',T_2')$ is a 4-tuple of neck lengths. Let $J_s$ be a cylindrical almost complex structure on $\Sigma\times [0,1]\times \R$ and let $J_s(\ve{T})$ be an almost complex structure on $\Sigma''\times [0,1]\times \R$  which has had necks of length $T_1,$ $T_2,$ $T_1'$ and $T_2'$ inserted along $c_1,$ $c_2$, $c_1'$ and $c_2'$, respectively. We will show that if $\ve{T}_1$ and $\ve{T}_2$ are two 4-tuples of neck lengths and all eight components are sufficiently large, then
\begin{equation}
\Phi_{J_s(\ve{T}_1)\to J_s(\ve{T}_2)}(\ve{x}\times \theta^+\times (\theta')^+))=\ve{x}\times \theta^+\times (\theta')^+.\label{eq:changeofcsstructure1-handle}
\end{equation} Importantly, we will prove that Equation~\eqref{eq:changeofcsstructure1-handle} holds without any assumption about the relative sizes of the 8 neck lengths, just that all are sufficiently large.

The map $\Phi_{J_s(\ve{T}_1)\to J_s(\ve{T}_2)}$ counts holomorphic disks of index 0 with an almost complex structure on $\Sigma''\times [0,1]\times \R$ which interpolates $J_s(\ve{T}_1)$ and $J_s(\ve{T}_2)$. 

To prove Equation~\eqref{eq:changeofcsstructure1-handle}, we take two sequences $\ve{T}_1^i$ and $\ve{T}_2^i$ of 4-tuples of neck lengths, all of whose components approach $+\infty$, as well as sequence of dynamic almost complex structures $\Hat{J}^i$, interpolating $J_s(\ve{T}_1^i)$ and $J_s(\ve{T}^i_2)$.  We assume that $(\Sigma''\times [0,1]\times \R,\Hat{J}^i)$ contains the almost complex submanifold $((\Sigma\setminus N_i)\times [0,1]\times \R, J_s)$, where $N_i$ consists of the disjoint union of four disks in $\Sigma$. Furthermore, we can arrange it so that $N_{i+1}\subset N_i$ and $\bigcap_{i\in \N} N_i$ consists of four points in $\Sigma$.

Suppose that $u_i$ is a sequence of $\Hat{J}^i$-holomorphic curves, which represent a Maslov index 0 class 
\[
\phi\# \phi_0\# \phi_0'\in \pi_2(\ve{x}\times \theta^+\times (\theta')^+, \ve{y}\times y\times y'),
\] where $y\in \alpha_0\cap \beta_0$ and $y'\in \alpha_0'\cap \beta_0'$. A straightforward Maslov index computation using the formula from \cite{LipshitzCylindrical}*{Corollary~4.3}  shows that
\begin{equation}
\mu(\phi\# \phi_0\# \phi_0')=\mu(\phi)+\gr(\theta^+,y)+\gr((\theta')^+,y').
\label{eq:Maslovindexformula}
\end{equation}

As in Lemma~\ref{prop:TScommute},  from the sequence $u_i$ of $\Hat{J}^i$-holomorphic representatives  of $\phi\#\phi_0\#\phi_0'$, after taking a subsequence, we can extract a limiting curve $u$ on $\Sigma\times [0,1]\times \R$ which represents the homology class $\phi$,  so $\mu(\phi)\ge 0$ by transversality. In particular, the expression on the right hand side of Equation~\eqref{eq:Maslovindexformula} is a sum of nonnegative terms, so each must be zero if $\phi\# \phi_0\#\phi_0'$ has Maslov index 0. Hence
\[\mu(\phi)=0,\qquad y=\theta^+\qquad \text{and} \qquad y'=(\theta')^+.\] By transversality, since $u$ has Maslov index 0, it must represent the constant homology class. It is straightforward to see that this implies that $\phi\# \phi_0\# \phi_0'$ is itself the constant homology class. On the other hand, the constant homology class always has a unique $\Hat{J}^i$-holomorphic representative. Equation~\eqref{eq:changeofcsstructure1-handle} follows, and hence the 1-handle maps for $\bS$ and $\bS'$ commute.
\end{proof}

\subsection{2-handle maps}

Suppose $\bL=(L,\ws,\zs)$ is a multi-based link in $Y$ and $\bS^1$ is a framed 1-dimensional link in $Y\setminus L$. We now describe maps $F_{Y,\bL,\bS^1,\frs}$ associated to a framed link $\bS^1$ and $\Spin^c$ structure $\frs\in \Spin^c(W(Y,\bS^1))$.

Let $\bL_\alpha$ denote the closure of the components of $L\setminus (\ws\cup \zs)$ which are oriented to start in $\ws$ and end in $\zs$. Let $\bL_\beta$ denote the closure of the components of $L\setminus (\zs\cup \ws)$ which start in $\zs$ and end in $\ws$.

\begin{define}An \emph{$\alpha$-bouquet}, $\cB^\alpha$, for a framed link $\bS^1$ in $Y\setminus L$ is a collection of arcs which connect each of the components of $\bS^1$ to one of the components of $\bL_\alpha$. We assume that there is exactly one arc per component of $\bS^1$, and each arc has one endpoint on $\bS^1$ and one endpoint in $\bL_\alpha$. A \emph{$\beta$-bouquet}, $\cB^\beta$, for $\bS^1$ is defined analogously, but with endpoints on $\bL_\beta$ instead.
\end{define}

Given a bouquet $\cB$ for a framed link $\bS^1\subset Y\setminus L$, we  construct the sutured manifold $Y(\cB)$ obtained by removing a regular neighborhood of $L\cup \cB$, and adding sutures corresponding to the basepoints $\ve{w}$ and $\ve{z}$.

Adapting \cite{OSTriangles}*{Definition~4.2} and \cite{JCob}*{Definition~6.3}, we make the following definition:

\begin{define}\label{def:subordinatebetabouquet}Suppose $\bL=(L,\ws,\zs)$ is a multi-based link in $Y$ and $\bS^1$ is a framed 1-dimensional link in $Y\setminus L$. We say that a triple $(\Sigma,\alpha_1,\dots, \alpha_n,\beta_1,\dots, \beta_n,\beta_1',\dots, \beta'_n,\ve{w},\ve{z})$ is \emph{subordinate} to a $\beta$-bouquet $\cB^\beta$ for $\bS^1$ if the following are satisfied: 
\begin{enumerate}
\item After removing neighborhoods of the $\ws$ and $\zs$-basepoints, the diagram 
\[
(\Sigma, \alpha_1,\dots, \alpha_n, \beta_{k+1},\dots, \beta_n, \ve{w},\ve{z})
\] becomes a sutured diagram for $Y(\cB^\beta)$.
\item  The curves $\beta_{1},\dots, \beta_k$ are each a meridian of a different component of $\bS^1$. In particular 
\[
(\Sigma, \alpha_1,\dots, \alpha_n, \beta_{1},\dots, \beta_n, \ve{w},\ve{z})
\] is a diagram for $(Y,\bL)$.
\item The curves $\beta'_{k+1},\dots ,\beta'_{n}$ are each small Hamiltonian translates of the curves $\beta_{k+1},\dots, \beta_n$ respectively, with $|\beta_i\cap \beta'_j|=2\delta_{ij}$.
\item For $i=1,\dots, k$, the curve $\beta_i'$ is induced by the framing of the corresponding link component that $\beta_i$ is a meridian of.
\end{enumerate}
\end{define}

If $(\Sigma, \as,\bs,\bs',\ws,\zs)$ is subordinate to a $\beta$-bouquet, then $(\Sigma, \bs,\bs',\ve{w},\ve{z})$ represents an unlink in $(S^1\times S^2)^{\#(g(\Sigma)-|\bS^1|)}$ with $|\ws|$ components, each of which has exactly two basepoints. As such, regardless of the coloring, $\cHFL^-(\Sigma, \bs,\bs',\ve{w},\ve{z},\sigma,\frs_0)$ admits a distinguished element $\Theta_{\bs,\bs'}^+$ by Lemma~\ref{lem:topdegreeunlink1}.

The 4-manifold $X_{\as,\bs,\bs'}$ becomes diffeomorphic to $W(Y,\bS^1)$ once we fill in $Y_{\bs,\bs'}\subset X_{\as,\bs,\bs'}$ with 3- and 4-handles (see \cite{OSTriangles}*{Proposition 4.3}). As such, given $\frs\in \Spin^c(W(Y,\bS^1))$, we can restrict $\frs$ to $\Spin^c(X_{\as,\bs,\bs'})$. With this in mind, given an $\frs\in \Spin^c(W(Y,\bS^1))$, we define the $\beta$-subordinate 2-handle map
\[F_{Y,\bL,\bS^1,\frs}^\beta\colon\cCFL^-(Y,\bL^\sigma,\frs|_{Y})\to \cCFL^-(Y(\bS^1),\bL^\sigma,\frs|_{Y(\bS^1)})\] via the formula
\begin{equation}
F_{Y,\bL,\bS^1,\frs}^\beta(\ve{x}):= F_{\as,\bs,\bs',\frs}(\ve{x}\otimes\Theta_{\bs,\bs'}^+)=\sum_{\ys\in \bT_{\as}\cap \bT_{\bs'}}\sum_{\substack{\psi\in \pi_2(\ve{x},\Theta_{\bs,\bs'}^+,\ve{y})\\ \frs_{\ve{w}}(\psi)=\frs\\ \mu(\psi)=0}} \# \cM(\psi)U_{\ve{w}}^{n_{\ve{w}}(\phi)} V_{\ve{z}}^{n_{\ve{z}}(\phi)} \cdot \ve{y}.\label{eq:2-handlemap}
\end{equation}

\begin{prop}\label{prop:2-handlemapwelldefined}The map $F_{Y,\bL,\bS^1,\frs}^\beta$ described in Equation~\eqref{eq:2-handlemap} is independent of $\beta$-bouquet $\cB$ and the Heegaard triple subordinate to $\cB$, on the level of transitive systems of chain complexes.
\end{prop}

Proposition~\ref{prop:2-handlemapwelldefined} can be proven by adapting the proof of \cite{OSTriangles}*{Theorem~4.4} (see also \cite{JCob}*{Theorem~6.9}), which carries over to this setting without major change.

 We note that one slight difference to the argument from \cite{OSTriangles} is that we are working on the level of chain complexes, instead of the level of homology. The element $\Theta_{\bs,\bs'}^+\in \cHFL^-(\Sigma,\bs,\bs',\ws,\zs,\sigma,\frs_0)$ is only a well-defined homology class, not a well-defined cycle in the chain complex $\cCFL^-(\Sigma,\bs,\bs',\ws,\zs,\sigma,\frs_0)$. Nonetheless,  adding a boundary in $\cCFL^-(\Sigma,\bs,\bs',\ws,\zs,\sigma,\frs_0)$ to $\Theta_{\bs,\bs'}^+$ only changes the 2-handle map in Equation~\eqref{eq:2-handlemap} by a filtered, equivariant chain homotopy, by the associativity relations for the holomorphic triangle maps.

By adapting Definition~\ref{def:subordinatebetabouquet}, one defines what it means for a triple  $(\Sigma, \as',\as,\bs,\ve{w},\ve{z})$ to be subordinate to an $\alpha$-bouquet $\cB^\alpha$. For such a triple, the homology group $\cHFL^-(\Sigma, \as',\as,\ve{w},\ve{z},\sigma,\frs_0)$ will have a distinguished element $\Theta_{\as',\as}^+$  and we define the $\alpha$-subordinate 2-handle map via the formula
\[
F_{Y,\bL,\bS^1,\frs}^\alpha(\ve{x}):=F_{\as',\as,\bs,\frs}(\Theta_{\as',\as}^+\otimes \ve{x}).
\]
 As with the case of the $\alpha$-subordinate 2-handle maps, the argument from \cite{OSTriangles}*{Theorem~4.4} goes through without major modification to show that $F_{Y,\bL,\bS^1,\frs}^\alpha$ is independent of the choice of $\alpha$-bouquet and triple subordinate to it.

\begin{rem}Both the $\alpha$- and $\beta$-subordinate 2-handle maps count triangles such that $\frs_{\ws}=\frs$. This is an inherent asymmetry of the 2-handle maps between the $\ws$ and $\zs$-basepoints. This asymmetry leads to the Alexander grading change formula in \cite{ZemAbsoluteGradings}.\end{rem}

We now wish to show that $F_{Y,\bL,\bS^1,\frs}^\alpha$ and $F_{Y,\bL,\bS^1,\frs}^\beta$ are chain homotopic. The proof given by Ozsv\'{a}th and Szab\'{o} in \cite{OSTriangles}*{Lemma 5.2} for the standard cobordism maps carries over without major change. We state it as the lemma and outline their argument briefly:

\begin{lem}\label{lem:alphasurgerymaps=betasurgerymaps}If $\bS$ is a framed 1-dimensional link in $Y\setminus L$ and $\frs\in \Spin^c(W(Y,\bS))$, then the maps $F_{Y,\bL,\bS,\frs}^\alpha$ and $F_{Y,\bL,\bS,\frs}^\beta$ are chain homotopic on the level of transitive systems of chain complexes.
\end{lem}
\begin{proof}[Proof sketch] Let $\bS'$ be an isotopic copy of $\bS$, in the direction of the framing of $\bS$. One finds a Heegaard quadruple $(\Sigma, \as',\as,\bs,\bs',\ve{w},\ve{z})$ such that $(\Sigma, \as',\as,\bs,\ve{w},\ve{z})$ is triple subordinate to an $\alpha$-bouquet of $\bS'$, and $(\Sigma, \as,\bs,\bs',\ve{w},\ve{z})$ is a triple subordinate to a $\beta$-bouquet of $\bS$. After  surgering $Y$ along $\bS$, the framed link $\bS'\subset Y(\bS)$ becomes a collection of 0-framed unknots. Similarly after surgering $Y$ on $\bS'$, the link $\bS\subset Y(\bS')$ also becomes a collection of 0-framed unknots. The associativity relations shows that, 
\[F^\alpha_{Y(\bS),\bL, \bS',\frs_0} F^\beta_{Y,\bL,\bS,\frs}\simeq F_{Y(\bS'),\bL, \bS,\frs_0}^\beta F_{Y,\bL,\bS',\frs}^\alpha.\] Here we are writing $\frs$ for the original $\Spin^c$ structure on $W(Y,\bS)$ and the corresponding $\Spin^c$ structure on $W(Y,\bS')$. Similarly we are writing $\frs_0$ for the $\Spin^c$ structures on $W(Y(\bS),\bS')$ and $W(Y(\bS'),\bS)$ which extends over the 3-handles which cancel $\bS'$ or $\bS$ (resp.).  The second map in each of the compositions can be canceled by a composition of 3-handle maps, implying that $F_{Y,\bL,\bS,\frs}^\beta$ and $F_{Y,\bL,\bS',\frs}^\alpha$ are related by post-composition by the diffeomorphism map associated to cancellations of 4-dimensional handles.
\end{proof}

Since the triangle maps are graded over 4-dimensional $\Spin^c$ structures, and there is not a canonical way to compose $\Spin^c$ structures, they instead satisfy a $\Spin^c$ composition law. We have the following:

\begin{prop}\label{lem:compositionlawforlinks}If $\bS$ and $\bS'$ are two disjoint framed 1-dimensional links in $Y\setminus L$, then
\[F_{Y(\bS),\bL,\bS',\frs_2} F_{Y,\bL,\bS,\frs_1}\simeq\sum_{\substack{\frs\in \Spin^c(W(Y,\bS\cup\bS'))\\ \frs|_{W(Y,\bS)}=\frs_1, \frs|_{W(Y(\bS), \bS')}=\frs_2}} F_{Y,\bL,\bS\cup \bS',\frs}.\]
\end{prop}
\begin{proof}This follows from the associativity of the triangle maps \cite{OSDisks}*{Theorem~8.16}.
\end{proof}

\section{Maps for band surgery} 
\label{sec:bandmaps}
In this section, we construct maps associated to performing band surgery on a link in a 3-manifold. A \emph{band} $B$ for a link $L$ is an embedded copy of $[-1,1]\times [-1,1]$ in $Y$ such that $B\cap L=\{-1,1\}\times [-1,1]$. In terms of link cobordisms, a band corresponds to a saddle cobordism with underlying 4-manifold $[0,1]\times Y$. The surgered link $L(B)$ is defined as the union 
\[
L(B):= (L\setminus B)\cup ([-1,1]\times \{-1,1\}).
\]

  If $L$ is an oriented link, an \emph{oriented band} $B$ is a band such that boundary orientation on $\d B$ is the opposite of the orientation of $L$ on $B\cap L$.  If $L$ is an oriented link and $B$ is an oriented band for $L$, then $L(B)$ has a natural orientation induced by $L$. We will assume all bands are oriented.

In order to define the maps, we need to assume that the ends of the bands are disjoint from the basepoints, and satisfy one of two configurations with respect to the basepoints. We will refer to bands which have one of these two favorable configurations  as \emph{$\alpha$-bands} or \emph{$\beta$-bands} (see Definition~\ref{def:alpha/betaband}). If $B$ has either of these configurations, then we can define
\[
\bL(B):=(L(B),\ws,\zs),
\]
which is a valid multi-based link in the sense of Definition~\ref{def:multibasedlink}.

 If $B$ is either an $\alpha$-band or a $\beta$-band, we will define two maps
\[
F_{B}^{\ws}, F_{B}^{\zs}\colon  \cCFL^-(Y,\bL^{\sigma}, \frs)\to \cCFL^-(Y,\bL^{\sigma}(B),\frs).
\] The maps $F_{B}^{\ws}$ and $F_{B}^{\zs}$ correspond to decorated link cobordisms of the form $([0,1]\times Y,\cF)$  where $\cF$ is one of the saddle cobordisms shown in Figure~\ref{fig::103}.
 
  As one might expect, there are some requirements on the coloring $\sigma$ which ensure that the band maps are well-defined chain maps. Analogously to the quasi-stabilization maps, the requirement on the coloring corresponds nicely with the dividing set on the decorated link cobordism we associate to the band maps.

\begin{figure}[ht!]
\centering
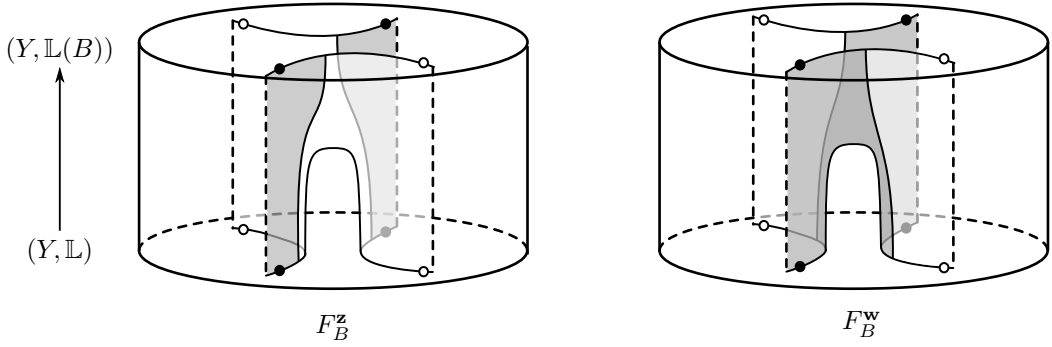
\caption{\textbf{Decorated link cobordisms corresponding to $F_{B}^{\zs}$ and $F_{B}^{\ws}$.} The underlying 4-manifold is $[0,1]\times Y$.\label{fig::103}}
\end{figure}

We note that $F_B^{\ws}$ and $F_{B}^{\zs}$ represent genuinely different decorated link cobordisms. On the other hand, there is a simple relation between the maps for $\alpha$-bands and the maps for the $\beta$-bands, which we prove in Section~\ref{sec:bandmapsandbasepointmovingmaps}.

Finally, we remark that the maps in this section use the same Heegaard triples as the maps from \cite{AEtangle}.

\subsection{Heegaard triples and bands}

In order to define maps for saddles, we need to restrict to bands which satisfy one of the following conditions:
 
\begin{define}\label{def:alpha/betaband}We say that a band $B^\alpha$ of an oriented multi-based link $\bL=(L,\ve{w},\ve{z})$ is an \emph{$\alpha$-band} if the ends of $B^\alpha$ occur in components of $L\setminus (\ve{w}\cup \ve{z})$ which go from $\ve{w}$-basepoints to $\ve{z}$-basepoints. We say that a band $B^\beta$ is a \emph{$\beta$-band} if the ends of $B^\beta$ occur in components of $L\setminus (\ve{w}\cup \ve{z})$ going from $\ve{z}$-basepoints to $\ve{w}$-basepoints. For a band $B$ of either type, we always assume that the ends of $B$ lie in distinct components of $L\setminus (\ve{w}\cup \ve{z})$.
\end{define}

If $B^\alpha$ is an $\alpha$-band for $\bL$, then we can turn $Y\setminus (N(L\cup B^\alpha))$ into a sutured manifold by adding a meridional suture for each basepoint of $\ws\cup \zs$. Write $L_0$ for the two components of $L\setminus (\ws\cup\zs)$ which contain an end of $B$. We define the subset $H\subset N(L\cup B^\alpha)$ to be points in $N(L\cup B^\alpha)$ which live over the subset $L_0\cup B^\alpha$. In particular $\d H\cap N(L\cup B^\alpha)$ is a 4-punctured sphere, and each puncture of $S$ corresponds to one of the basepoints adjacent to an end of $B^\alpha$.

\begin{define}\label{def:triplesubbetaband}If $B^\alpha$ is an $\alpha$-band for $\bL=(L,\ws,\zs)$, we say that the Heegaard triple 
\[
(\Sigma, \alpha_1',\dots, \alpha_n',\alpha_1,\dots, \alpha_n,\beta_1,\dots, \beta_n,\ve{w},\ve{z})
\] is \emph{subordinate} to $B^\alpha$  if the following hold:
\begin{enumerate}
\item After removing neighborhoods of the $\ws$ and $\zs$-basepoints from the diagram \[
(\Sigma, \alpha_1,\dots, \alpha_{n-1},\beta_1,\dots, \beta_n,\ve{w}, \ve{z})
\]
 we obtain a sutured Heegaard diagram for the sutured manifold $Y\setminus N(L\cup B^\alpha)$ (with meridional sutures induced by the basepoints).
\item $\alpha'_1,\dots, \alpha'_{n-1}$ are small Hamiltonian isotopies of the curves $\alpha_1,\dots, \alpha_{n-1}$ with $|\alpha_i'\cap \alpha_j|=2\delta_{ij},$ for $i,j\in \{1,\dots, n-1\}$, where $\delta_{ij}$ denotes the Kronecker delta.
\item Let $L_0$, $S$ and $H$ be as defined in the previous paragraph. Let $A\subset S$ be the closed curve which is specified up to isotopy by the property that $A$ bounds a disk in $H$ which separates the two components of $L_0$, and let $A'\subset S$ denote the closed curve which is specified up to isotopy by the property that $A'$ bounds a disk in $H$ which separates the two components of $L_0(B)$. Then $\alpha_n$ is a curve which is obtained by projecting $A$ onto $\Sigma\setminus (\alpha_1\cup \cdots\cup \alpha_{n-1})$ and $\alpha_n'$ is the curve which is obtained by projecting $A'$ onto $\Sigma\setminus(\alpha_1'\cup \cdots \cup \alpha_{n-1}')$.
\end{enumerate}
\end{define}

   A schematic of Definition~\ref{def:triplesubbetaband} is shown in Figure~\ref{fig::117}.

\begin{figure}[ht!]
\centering
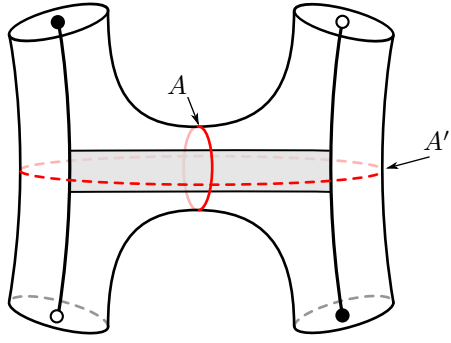
\caption{\textbf{The region $H$ and the curves $A$ and $A'$ in $\d (Y\setminus N(L_0\cup B^\alpha))$ which induce $\alpha_n$ and $\alpha_n'$.} The curve $A$ is a solid red circle, and $A'$ is a dashed red circle. The band $B^\alpha$ is shown shaded in gray. The curve $\alpha_n$ is obtained from $A$, and $\alpha_n'$ is obtained from $A'$. \label{fig::117}}
\end{figure}

Note that if $(\Sigma, \alpha_1',\dots, \alpha_n',\alpha_1,\dots, \alpha_n,\beta_1,\dots, \beta_n,\ve{w},\ve{z})$ is subordinate to $B^\alpha$, then the diagram $(\Sigma,\alpha_1,\dots, \alpha_n,\beta_1,\dots, \beta_n,\ws,\zs)$ is a diagram for $(Y,\bL)$ and $(\Sigma, \alpha_1',\dots, \alpha_n',\beta_1,\dots, \beta_n,\ws,\zs)$ is a diagram for $(Y,\bL(B^\alpha))$.

It is well known that any two sutured Heegaard diagrams for a sutured manifold can be connected by a simple set of Heegaard moves  \cite{JTNaturality}*{Proposition 2.37}. Applying this fact to the sutured manifold $Y\setminus N(L\cup B^\alpha)$, we obtain the following:

\begin{lem}\label{lem:movesbetweenalphabandtriples}Given an $\alpha$-band $B^\alpha$, any two triples subordinate to $B^\alpha$ can be connected by a sequence of the following moves
\begin{enumerate}
\item Isotopies $\Sigma_t$ of the Heegaard surface  inside of $Y$ such that $\Sigma_t\cap (L\cup B)=\ws\cup \zs$ for all $t$.
\item Isotopies or handleslides amongst the $\ve{\beta}$-curves.
\item Isotopies or handleslides amongst the $\alpha_1,\dots, \alpha_{n-1}$ curves followed by the corresponding handleslide or isotopy of the corresponding curve in $\alpha_1',\dots, \alpha_{n-1}'$.
\item Isotopies or handleslides of the curve $\alpha_n$ across other $\ve{\alpha}$-curves.
\item Isotopies or handleslides of the curve $\alpha_n'$ across other $\ve{\alpha}'$-curves.
\item  Stabilizations of the Heegaard triple, i.e., taking the internal connected sum of $\Sigma$ in $Y$ with a torus $\bT^2\subset Y$ such that there is a 3-ball $B^3\subset Y$ containing $\bT^2$ with $B^3\cap (\Sigma\cup L\cup B)=\varnothing$. On $\bT^2$ we put three curves, $\alpha_0'$, $\alpha_0$ and $\beta_0$, with $|\alpha_0\cap \beta_0|=1$ and $\alpha_0'$ a small Hamiltonian isotopy of $\alpha_0$ with $|\alpha_0'\cap \alpha_0|=2$.  
\end{enumerate}
\end{lem}

Definition~\ref{def:triplesubbetaband} can be easily modified to provide a definition of a Heegaard triple $(\Sigma,\as,\bs,\bs')$ being subordinate to a $\beta$-band. Similarly, Lemma~\ref{lem:movesbetweenalphabandtriples} can be easily modified to give a set of moves between any two Heegaard triples which are subordinate to a $\beta$-band.

If $(\Sigma,\as',\as,\bs,\ws,\zs)$ is a Heegaard triple subordinate to a band, we will define the band maps by counting holomorphic triangles. To do this, it is important to understand the chain complex $\cCFL^-(\Sigma,\as',\as,\ws,\zs)$. To this end, we prove the following:

\begin{lem}\label{lem:bandtriplehas4pointedunknot}If $(\Sigma, \ve{\alpha}',\as,\bs,\ve{w},\ve{z})$ is a Heegaard triple subordinate to an $\alpha$-band $B^\alpha\subset Y$, then $(\Sigma, \ve{\alpha}',\as,\ve{w},\ve{z})$ is a diagram for an unlink in (a disjoint union of copies of) $(S^1\times S^2)^{\# g(\Sigma)}$, and every component of the unlink has exactly two basepoints, except one component, which has four.
\end{lem}

\begin{proof}Let $(Y_{\as',\as},\bL_{\as',\as})$ denote the 3-manifold and link induced by the diagram $(\Sigma,\as',\as,\ws,\zs)$. Write $L_{\as',\as}$ for the underlying link of $\bL_{\as',\as}$.  Note that for $i\in \{1,\dots, n-1\}$ the pair $(\alpha_i',\alpha_i)$ determines an embedded 2-sphere $S_i$ in $Y_{\as',\as}$ which does not intersect $L_{\as',\as}$. By surgering out $S_i$, we can reduce to the case that $\Sigma$ consists of $|L_{\as',\as}|$ genus 0 components. There are $|L_{\as',\as}|-1$ of these components which have no $\ve{\alpha}'$ or $\ve{\as}$ curves, but have a single pair of $\ve{w}$ and $\zs$-basepoints, thus determining an unknotted component of $\bL_{\as',\as}$ in $Y_{\as',\as}$ with exactly two basepoints. There will be one component of the resulting Heegaard diagram which contains the curves $\alpha_n'$ and $\alpha_n$. A model diagram shows that for this link, $\alpha_n'$ and $\alpha_n$ can be taken to intersect twice on a genus zero Heegaard surface such that each domain is a bigon containing a single basepoint, which is a diagram for an unknot with four basepoints.
\end{proof}

\subsection{Type-$\mathbf{z}$ band maps}

We now construct the band maps for $\ws$ bands, for both $\alpha$ bands and $\beta$ bands. Suppose first that $B^\alpha$ is an $\alpha$-band and write $w_1,$ $z_1$, $w_2$ and $z_2$ for the basepoints adjacent to the ends of $B^\alpha$. Let $\sigma\colon \ws\cup \zs\to \bmP$ be a coloring. Assuming 
\[
\sigma(z_1)=\sigma(z_2),
\] we will define a map
\[
F_{B^\alpha}^{\ve{z}}\colon \cCFL^-(Y,\bL^\sigma,\frs)\to \cCFL^-(Y,\bL(B^\alpha)^\sigma,\frs).
\] Notice that the restriction on the coloring $\sigma$ is exactly that the coloring on $\bL$ is induced by a coloring on the decorated link cobordism from Figure~\ref{fig::103}.

 We take a Heegaard triple $(\Sigma,\as',\as,\bs,\ws,\zs)$ which is subordinate to $B^\alpha$. Define the coloring $\sigma_0\colon \ws\cup \zs\to (\ws\cup \zs)/(z_1\sim z_2)$. By Lemmas~\ref{lem:bandtriplehas4pointedunknot}~and~\ref{lem:topdegreeunlink2}, there is a distinguished element 
 \[
 \Theta^{\ws}_0\in \cHFL^-(\Sigma,\as',\as,\ws,\zs,\sigma_0,\frs_0),
 \] which is determined by the property that $\Theta^{\ws}_0$ generates
 \[
 \Max_{\gr_{\ws}}\left(\cHFL^-(\Sigma,\as',\as,\ws,\zs,\sigma_0, \frs_0)\right)\iso \bF_2[V_{\zs}]/(V_{z_1}+V_{z_2})
 \]
 as an $\bF_2[V_{\zs}]$-module. We  define $\Theta^{\ws}\in \cHFL^-(\Sigma,\as',\as,\ws,\zs,\sigma,\frs_0)$ to be the image of $\Theta^{\ws}_0$ under the natural map
 \begin{equation}\cHFL^-(\Sigma,\as',\as,\ws,\zs,\sigma_0,\frs_0)\to \cHFL^-(\Sigma,\as',\as,\ws,\zs,\sigma,\frs_0).\label{eq:naturalmapfromstrictercoloring}\end{equation} Note that the existence of the map from Equation~\eqref{eq:naturalmapfromstrictercoloring} uses the assumption that $\sigma(z_1)=\sigma(z_2)$.

The 4-manifold $X_{\as',\as,\bs}$ is diffeomorphic to $[0,1]\times Y$ with a neighborhood of a 1-complex removed. Hence the $\Spin^c$ structure $\frs\in \Spin^c(Y)$ uniquely determines a $\Spin^c$ structure on $X_{\as',\as,\bs}$, which restricts to the torsion $\Spin^c$ structure $\frs_0\in \Spin^c(Y_{\as',\as})$. We also write $\frs$ for the 4-dimensional $\Spin^c$ structure determined by $\frs\in \Spin^c(Y)$.

We  now define the type-$\zs$ band map for $B^\alpha$ to be
\[
F_{B^\alpha}^{\zs}(-):=F_{\as',\as,\bs,\frs}(\Theta^{\ws}\otimes -)\colon \cCFL^-(\Sigma,\as,\bs,\ws,\zs,\sigma,\frs)\to \cCFL^-(\Sigma,\as',\bs,\ws,\zs,\sigma,\frs),
\]
by counting holomorphic triangles representing index 0 homology classes $\psi$ with $\frs_{\ws}(\psi)=\frs$.

Note that $\Theta^{\ws}$ is only a well-defined homology class, not a well-defined cycle in the chain complex $\cCFL^-(\Sigma,\as',\as,\ws,\zs,\sigma,\frs_0)$. Nonetheless, adding a boundary $\d \eta \in \cCFL^-(\Sigma,\as',\as)$ to $\Theta^{\ws}$ has the effect of changing $F_{B^\alpha}^{\zs}$ by a filtered, equivariant chain homotopy, using the associativity relations for the triangle maps.

The definition of the map for a $\beta$-band is analogous. If $B^\beta$ is a $\beta$-band, we pick a triple $(\Sigma,\as,\bs,\bs',\ws,\zs)$ subordinate to $B^\beta$, and consider the triangle map $F_{\as,\bs,\bs',\frs}$. As with the case of $\alpha$-bands, there is a distinguished element $\Theta^{\ws}\in \cHFL^-(\Sigma,\bs,\bs,\ws,\zs,\sigma,\frs)$, as long as $\sigma(z_1)=\sigma(z_2)$. The $\zs$-band map for $B^\beta$ is then defined by the formula
\[F_{B^{\beta}}^{\zs}(-):=F_{\as,\bs,\bs',\frs}(-\otimes \Theta^{\ws}).\]

 We now claim that these induce well-defined maps on the level of transitive chain homotopy types:

\begin{lem}The map $F_{B}^{\ve{z}}$ is  well-defined up to equivariant, filtered chain homotopy.
\end{lem}
\begin{proof}We just need to check independence from each of the moves appearing Lemma~\ref{lem:movesbetweenalphabandtriples}. Each of these moves is considered in the maps associated to surgery on framed 1-dimensional links embedded in 3-manifolds, and can be addressed using the same argument as~\cite{OSTriangles}*{Proposition~4.6}.
\end{proof}

We also make the following remark:

\begin{rem}\label{rem:canwind}Since the element $\Theta^{\ws}$ is well-defined on the level of homology, it is not necessary to work with Heegaard triples which satisfy Definition~\ref{def:triplesubbetaband} exactly as  it is stated. Rather it is sufficient to work with triples which are related to the ones described in Definition~\ref{def:triplesubbetaband} by a sequence of handleslides and isotopies of the attaching curves.
\end{rem}

\subsection{Type-$\mathbf{w}$ band maps}

If $B$ is a band (either $\alpha$- or $\beta$-), the $\ws$-type band maps are defined similarly the $\zs$-type band maps in the previous section.

Suppose first $B^\alpha$ is an $\alpha$-band for the link $\bL$ in $Y$, and $w_1,$ $w_2$, $z_1$ and $z_2$ are the basepoints adjacent to the ends of $B^\alpha$. Suppose $\sigma$ is a coloring such that
\[\sigma(w_1)=\sigma(w_2).\] Under the above assumption, we will describe the type-$\ws$ band map, $F_{B^\alpha}^{\ws}$. 

As with the type-$\zs$ map, this restriction on the coloring $\sigma$ is equivalent to the requirement that the coloring on $\bL$ is induced by a coloring on the decorated link cobordism for $F_{B^{\alpha}}^{\zs}$ from Figure~\ref{fig::103}. Let $(\Sigma,\as',\as,\bs,\ws,\zs)$ be a triple subordinate to $B^\alpha$. We define a coloring of $\bL$ via the formula $\sigma_0'\colon \ws\cup \zs\to (\ws\cup \zs)/(w_1\sim w_2)$. There is a distinguished element $\Theta^{\zs}_0\in \cHFL^-(\Sigma,\as',\as,\ws,\zs,\sigma_0',\frs_0)$ which generates
\[\Max_{\gr_{\zs}}\left(\cHFL^-(\Sigma,\as',\as,\ws,\zs,\sigma_0',\frs_0)\right)\] as an $\bF_2[U_{\ws}]$-module. Since $\sigma(w_1)=\sigma(w_2)$, we can define the element $\Theta^{\zs}\in \cHFL^-(\Sigma,\as',\as,\ws,\zs,\sigma,\frs_0)$ as the image of $\Theta^{\zs}_0$.

The type-$\ws$ band map for $B^\alpha$ is defined by the formula
\[
F_{B^\alpha}^{\ws}:=F_{\as',\as,\bs,\frs}(\Theta^{\zs}\otimes -)\colon \cCFL^-(\Sigma,\as,\bs,\ws,\zs,\sigma,\frs)\to \cCFL^-(\Sigma,\as',\bs,\ws,\zs,\sigma,\frs)
.\]

The type-$\ws$ band maps for $\beta$-bands are defined analogously, using triples $(\Sigma,\as,\bs,\bs',\ws,\zs)$ subordinate to $\beta$-bands.

The same argument used for the type-$\ve{z}$ band maps can be used to show that the type-$\ve{w}$ band maps yield well-defined maps on the level of transitive chain homotopy types.

\section{Two compound handle attachment maps}
\label{sec:compoundhandles}
In this section, we describe two compound handle maps, which we define as compositions of the  maps from Sections~\ref{sec:handleattachmentmapsawayfromL}~and~\ref{sec:bandmaps}. The first type of compound map are the birth and death cobordism maps, which will be useful later when we prove several facts about the band maps (notably Proposition~\ref{prop:alphabandmapsarebetabandmaps}). The next map we define is the compound 1-handle/band map, for a framed 0-sphere along a link $L$, which will be useful later when we prove invariance of the cobordism maps.

\subsection{Birth/death cobordism maps}
\label{subsec:diskstabilization}

Suppose that $\bL=(L,\ve{w},\ve{z})$ is a multi-based link in $Y^3$ and $\bU=(U,w,z)$ is an doubly based unknot in $Y$, and $D$ is a choice of embedded disk in $Y\setminus L$ which has boundary $U$. Write $\bL\cup \bU$ for the link $(L\cup U, \ve{w}\cup \{w\},\ve{z}\cup \{z\})$. Suppose $\sigma$ and $\sigma'$ are colorings of $\bL$ and $\bL\cup \bU$ (resp.) such that $\sigma'$ restricts to $\sigma$. We  define maps
\[\cB^{+}_{\bU,D}\colon \cCFL^{\circ}(Y,\bL^{\sigma},\frs)\to \cCFL^{\circ}(Y,(\bL\cup \bU)^{\sigma'},\frs)\] and
\[\cD^{-}_{\bU,D}\colon \cCFL^{\circ}(Y,(\bL\cup \bU)^{\sigma'},\frs)\to \cCFL^{\circ}(Y,\bL^\sigma,\frs)\] which are well-defined on the level of transitive chain homotopy types. We will describe the maps on the level of intersection points, and state the effect of adding a doubly based unknot on the chain complex $\cCFL^-$. We will not explicitly prove invariance of the maps $\cB^+_{\bU,D}$ and $\cD^-_{\bU,D}$, since our definition of the map $\cB_{\bU,D}^+$ can be alternatively written as a composition of a 0-handle map (which adds the unknot $\bU$ inside of a disjoint copy of $S^3$) as well as a 1-handle map. Similarly $\cD_{\bU,D}^-$ can be written as a composition of a 3-handle map and a 4-handle map. Invariance of the maps $\cB_{\bU,D}^{+}$ and $\cD_{\bU,D}^-$ hence follows from invariance of the 0-, 1-, 3- and 4-handle maps,  which we proved earlier.

We pick a regular neighborhood $N(D)$ of $D$ which does not intersect $L$  and consider only diagrams $\cH=(\Sigma, \ve{\alpha},\ve{\beta},\ve{w},\ve{z})$ such that 
\begin{enumerate}
\item $N(D)\cap \Sigma$ is a disk;
\item $N(D)\cap (\ve{\alpha}\cup\ve{\beta}\cup \ve{w}\cup \ve{z})=\varnothing$.
\end{enumerate}
In $N(D)\cap \Sigma$, we insert the subdiagram shown in  Figure~\ref{fig::6} into $\Sigma$, to form the diagram
\[
\Hat{\cH}:=(\Sigma,\as\cup \{\alpha_0\}, \bs\cup \{\beta_0\}, \ws\cup \{w\}, \zs\cup \{z\}).
\]
  The two points in $\alpha_0\cap \beta_0$ are distinguished by the relative grading, and the designation of top and bottom generators is the same for $\gr_{\ve{w}}$ and $\gr_{\ve{z}}$.

\begin{figure}[ht!]
\centering
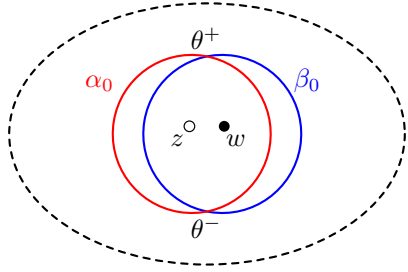
\caption{\textbf{The Heegaard diagram after a birth cobordism, adding a doubly based unknot.} This is the subdiagram which is added to $N(D)\cap \Sigma$ to form the diagram $\Hat{\cH}$.\label{fig::6}}
\end{figure}

  Let $J_s$ be an almost complex structure on $\Sigma\times [0,1]\times \R$ which is split on $(\Sigma\cap N(D))\times [0,1]\times \R$. And let $J_s(T)$ be an almost complex structure obtained by inserting a neck of length $T$ along the circle $\d (\Sigma\cap N(D))$, which is shown as a dashed circle in Figure~\ref{fig::6}. We have the following:
\begin{lem}\label{lem::differentialdiskstabilization}Suppose that $\cH=(\Sigma, \ve{\alpha},\ve{\beta},\ve{w},\ve{z})$ is a Heegaard diagram for $(Y,\bL)$, such that $\Sigma\cap N(D)$ is a disk which is disjoint from $\as\cup \bs\cup \ws\cup \zs$.  For sufficiently large $T$, 
there is an identification of
\[\d_{\Hat{\cH}, J_s(T)}=\begin{pmatrix}\d_\cH & *\\
0& \d_\cH
\end{pmatrix},\] with the matrix decomposition induced by writing $\theta^+$ as the first row and column and $\theta^-$ as the second. 
\end{lem}

The relevant counts of holomorphic disks in the above lemma are established in \cite{OSLinks}*{Proposition~6.5}.

\begin{rem} The entry $*$ in the expression for the differential can be identified by analyzing the proof of \cite{OSLinks}*{Proposition~6.5}. Write $A$ for the component of $\Sigma\setminus \as$ which contains $N(D)\cap \Sigma$ and write $B$ for the component of $\Sigma\setminus \ve{\beta}$ which contains $N(D)\cap  \Sigma$. Let $w_A$ and $z_A$ (resp. $w_B$ and $z_B$) denote the $\ws$ and $\zs$-basepoints in $A$ (resp. $B$). If the connected sum point in the stabilized region is chosen sufficiently close to the $\alpha_0$ curve, we can identify the component labeled  $*$ with the action of  $U_wV_z+U_{w_A}V_{z_A}$. For a connected sum point sufficient close to the $\beta_0$ curve, we can identify $*$ with the action of $U_wV_z+U_{w_B}V_{z_B}$.
\end{rem}
 
  We define the birth map
\[
\cB_{\bU,D}^+\colon \cCFL_{J_s}^\circ(\cH,\sigma,\frs)\to \cCFL_{J_s(T)}^\circ(\Hat{\cH}, \sigma',\frs),
\]
 and the death map $\cD_{\bU,D}^-$, defined in the opposite direction, via the formulas 
\[
\cB_{\bU,D}^+(\ve{x})= \ve{x}\times \theta^+,
\]
 \[
 \cD_{\bU,D}^-(\ve{x}\times \theta^-)= \ve{x}\qquad \text{and} \qquad \cD_{\bU,D}^-(\ve{x}\times \theta^+)=0.
 \]
  For large enough $T$, the maps $\cB_{\bU,D}^+$ and $\cD_{\bU,D}^-$ are chain maps by Lemma~\ref{lem::differentialdiskstabilization}. The map $\cB_{\bU,D}^+$ is a composition of a 0-handle and a 1-handle map. Similarly, the map $\cD_{\bU,D}^-$ is a composition of a 3-handle and 4-handle map. Hence $\cB_{\bU,D}^+$ and $\cD_{\bU,D}^-$ are well-defined on the level of transitive chain homotopy types.

\subsection{Surgeries and traces of framed 0-spheres along $L$}
\label{sec:modelcompounthandlebandcobordisms}

It will be convenient for our proof of invariance to consider framed 0-spheres $\bS^0$ which are centered at a pair of points contained in the link $L$. In this section, we describe traces and surgeries of such framed 0-spheres. In the following section, we will define cobordism maps associated to  framed 0-spheres along $L$. We make the following definition:

\begin{define}\label{def:framedzerospherealongL} Suppose $L$ is an oriented link in $Y$. We say that a smooth embedding $\bS^0\colon S^0\times D^3\hookrightarrow Y$ is a \emph{framed 0-sphere along $L$} if the following hold:
\begin{enumerate}
\item\label{cond:framed0sphereL1} If we write $D^3$ as $\{(y,w,z): y^2+w^2+z^2\le 1\}$, then $(\bS^0)^{-1}(L)=S^0\times \{(y,0,0): -1\le y\le 1\}$.
\item\label{cond:framed0sphereL2} Each of $\bS^0(S^0\times \{(1,0,0)\})$ and $\bS^0(S^0\times \{(-1,0,0)\})$ contains an incoming point of $L$, and an outgoing point of $L$, according to the orientation of $L$.
\end{enumerate}

\end{define}

We now describe the surgeries and traces of a framed 0-sphere which occurs along $L$, extending the notation from Section~\ref{subsec:modelhandlecobsawayfromL}. If $\bS^0$ is a framed 0-sphere along $L\subset Y$, we define a link  $L(\bS^0)\subset Y(\bS^0)$ via the equation
\[L(\bS^0):=(L\setminus \Int \im (\bS^0))\cup ([-1,1]\times \{\pm 1,0,0\}).\] Condition~\eqref{cond:framed0sphereL1} of Definition~\ref{def:framedzerospherealongL} implies that the link $L(\bS^0)$ is a closed 1-manifold in $Y(\bS^0)$, while Condition~\eqref{cond:framed0sphereL2} implies that an orientation of $L$ uniquely determines an orientation on $L(\bS^0)$. 

We now define the trace of $\bS^0$, as well as the induced trace link cobordism. It is convenient to define it slightly differently than in Section~\ref{subsec:modelhandlecobsawayfromL}.  We define the 4-manifold
\[W(Y,\bS^0):=([0,1]\times Y)\cup (D^1\times D^3)\cup ([1,2]\times Y(\bS^0)).\] 

There is an oriented surface $\Sigma(L,\bS^0)\subset W(Y,\bS^0)$ defined as
\[\Sigma(L,\bS^0):=([0,1]\times L)\cup B\cup ([1,2]\times L(\bS^0))\] where $B\subset [-1,1]\times D^3$ is the subset 
 $B:=[-1,1]\times \{(y,w,z): -1\le y\le 1, z=w=0\}$.  Note that $B$ is canonically identified with $[-1,1]\times [-1,1]$. 
 
 We now define the \emph{trace link cobordism} of $\bS^0$ to be
\[\cW(Y,L,\bS^0):=(W(Y,\bS^0), \Sigma(L,\bS^0))\colon (Y,L)\to (Y(\bS^0), L(\bS^0)).\]

 \subsection{Construction of the maps for framed 0-spheres along $L$}\label{subsec:constructioncompound1-handlemaps} 

Analogously to the situation of the band maps, to define maps for framed 0-spheres along $L$, we will need to restrict to framed 0-spheres which satisfy several additional requirements:

\begin{define} Suppose $\bS^0$ is a framed 0-sphere along $L$. We say that $\bS^0$ is of \emph{$\alpha$-type} if the two components of $\im(\bS^0)$ are contained in distinct $\alpha$-subarcs of $L\setminus (\ws\cup \zs)$ (i.e. they are contained in two distinct components of $U_{\as}\cap (L\setminus (\ws\cup \zs))$). We say that $\bS^0$ is of \emph{$\beta$-type} if the two components of $\im(\bS^0)$ are contained in two distinct $\beta$-subarcs of $L\setminus (\ws\cup \zs)$.
\end{define}

Under the above assumptions,  $\bL(\bS^0):=(L(\bS^0),\ws,\zs)$ is a valid multi-based link. If $\frs\in \Spin^c(Y)$, then $\frs$ admits a unique extension  to $W(Y,\bS^0)$, which we denote by $\hat{\frs}$. We let $\frs'$ denote the restriction of $\hat{\frs}$ to $Y(\bS^0)$. A coloring $\sigma$ of $\bL$,
 induces a coloring of $\bL(\bS^0)$, for which we also write $\sigma$. For a framed $0$-sphere which is of either $\alpha$-type or $\beta$-type, we will define a type-$\ve{z}$ compound 1-handle/band map
 \[
 F^{\ve{z}}_{Y,\bL, \bS^0, \hat{\frs}}\colon \cCFL^-(Y,\bL^\sigma,\frs)\to \cCFL^-(Y(\bS^0),(\bL(\bS^0))^{\sigma},\frs'),
 \] 
 as well as a type-$\ve{w}$ compound handle/band map
 \[
 F^{\ve{w}}_{Y,\bL, \bS^0, \hat{\frs}}\colon \cCFL^-(Y,(\bL)^{\sigma},\frs)\to \cCFL^-(Y(\bS^0),(\bL(\bS^0))^{\sigma},\frs').
 \]

For the type-$\ve{o}$ maps (where $\ve{o}\in \{\ve{w},\ve{z}\}$), we require that the two $\ve{o}$-basepoints adjacent to the components of $\bS^0$ have the same color. 

 Given a framed 0-sphere $\bS^0$ along $L$, which is of $\alpha$-type or $\beta$-type, we first pick a vector $\theta\in S^1=\{(0,w,z): w^2+z^2=1\}$. Such a vector determines a half disk $H_i$ in each $\{p_i\}\times D^3$ (the set of points with spherical coordinates which have angle $\theta$), and an arc $a$ on $S^2$ (the boundary of the half disk). There is a natural choice of rectangle $R_\theta:=[-1,1]\times a$ contained in the 1-handle region $[-1,1]\times S^2\subset Y(\bS^0)$. The half disks $H_i\subset \im (\bS^0)$, as well as the rectangle $R_\theta\subset [-1,1]\times S^2$ are shown in Figure~\ref{fig::53}.

  \begin{figure}[ht!]
\centering
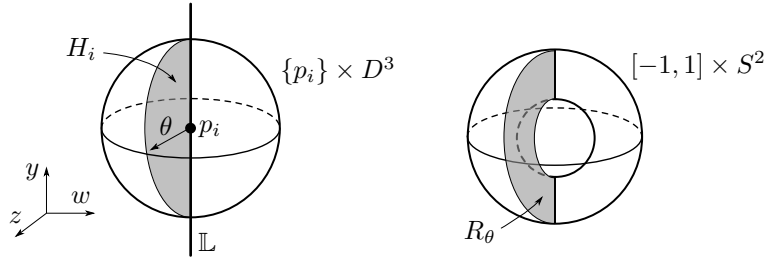
\caption{\textbf{Auxiliary data to construct the compound 1-handle/band map.} On the left is the vector $\theta\in S^1$, as well as a half disk $H_i$, the link $\bL$ and the balls $\{p_i\}\times D^3$. On the right is the 1-handle region $[-1,1]\times S^2\subset Y(\bS^0)$, and the rectangle $R_\theta$  with boundary on $\{-1,1\}\times S^2\cup [-1,1]\times \{(\pm 1,0,0)\}.$\label{fig::53}}
\end{figure}
 
Extend the half disks $H_i$ to slightly larger half disks $\bar{H}_1$ and $\bar{H}_2$ which extend just past the balls $\{p_i\}\times D^3$. The half disks $\bar{H}_1$ and $\bar{H}_2$ determine a diffeomorphism, $\psi^\theta$ of $Y$ which is is supported in a neighborhood of $\im (\bS^0)$ and pushes $L$ across the half disks $\bar{H}_i$  to a link $L_\theta$, which is disjoint from $\im(\bS^0)$.  The extensions $\bar{H}_i$ of the half disks $H_i$ depend on a contractible set of choices.

  Inside of $Y(\bS^0)$, there is a natural band
  \[
  B_\theta=\left((\bar{H}_1\cup\bar{H}_2)\setminus \im (\bS^0)\right)\cup R_\theta
  \] 
  from $\bL_\theta$ to $\bL(\bS^0)$.

  For $\ve{o}\in \{\ve{w},\ve{z}\}$, we define the type-$\ve{o}$ compound 1-handle/band map as the composition
 \[
 F_{Y,\bL,\bS^0,\hat{\frs}}^{\ve{o}}:=F_{B_\theta}^{\ve{o}} F_{Y,\bL_\theta,\bS^0,\hat{\frs}}\psi_*^\theta,
 \] 
 where $F_{Y,\bL_\theta,\bS^0,\hat{\frs}}$ denotes the 1-handle map from Section~\ref{sec:1--handlemaps} and $F_{B_\theta}^{\ve{o}}$ denotes the band map from Section~\ref{sec:bandmaps}.

 \begin{lem}Suppose $\bS^0$ is a framed zero sphere along $L$ which is of $\alpha$-type or $\beta$-type and $\ve{o}\in \{\ws,\zs\}$. The compound 1-handle/band map $F^{\ve{o}}_{Y,\bL,\bS^0,\hat{\frs}}$  is independent of the choice of $\theta\in S^1$, as well as the extensions $\bar{H}_i$.
 \end{lem}
 
 \begin{proof} Fix $\ve{o}\in \{\ve{w},\ve{z}\}$. The extended half disks $\bar{H}_i$ depend on a contractible set of choices (a vector field pointing out of the boundary of $\bS^0$, and a Riemannian metric), and hence it is easy to see that the maps do not depend on the extensions $\bar{H}_i$.
  
We now consider dependence on $\theta\in S^1$.  Let $\theta$ and $\theta'$ be two elements in $S^1$, and let $\psi^{\theta}$ and $\psi^{\theta'}$ be the two associated diffeomorphisms, and let $B_{\theta}$ and $B_{\theta'}$ be the two bands. We pick a self-diffeomorphism $\Psi$ of $Y$ which rotates $\bL_\theta$ to $\bL_{\theta'}$, which is supported in a neighborhood of $\bS^0$, but is constant on $\bS^0$. Note that since $\Psi$ is constant on $\bS^0$, there is an induced self-diffeomorphism $\Psi^{\bS^0}$ of $Y(\bS^0)$.

 Note that $\Psi^{\bS^0}$ sends the band $B_\theta$ to a band which is isotopic relative its boundary to $B_{\theta'}$. As a consequence, we conclude that
 \begin{equation}
 F_{B_{\theta'}}^{\ve{o}}\Psi^{\bS^0}_*\simeq \Psi^{\bS^0}_*F_{B_{\theta}}^{\ve{o}}\,,
 \label{eq:compband3}
 \end{equation} by diffeomorphism invariance of the band maps.

Furthermore, since $\Psi$ fixes $\bS^0$, we have
 \begin{equation}
 \Psi^{\bS^0}_* F_{Y,\bL_{\theta},\bS^0,\hat{\frs}}\simeq F_{Y,\bL_{\theta'},\bS^0,\hat{\frs}} \Psi_*.\label{eq:compband1}
  \end{equation}
   Note too that $\Psi\circ \psi^\theta$ and $ \psi^{\theta'}$ are isotopic relative to $\bL$, so we conclude that
 \begin{equation}
 \Psi_* \psi^\theta_*\simeq \psi^{\theta'}_*.
 \label{eq:compband2}
 \end{equation} 
 Finally, we note that $\Psi^{\bS^0}$ is isotopic to $\id_{Y(\bS^0)}$ relative to $\bL(\bS^0)$, so
 \begin{equation}
 \Psi^{\bS^0}_*\simeq \id
 \label{eq:compband4}
 \end{equation}
 on $\cCFL^-(Y(\bS^0), (\bL(\bS^0))^\sigma, \frs')$. Hence
 
 \begin{align*}F_{B_\theta}^{\ve{o}} F_{Y,\bL_\theta,\bS^0, \hat{\frs}} \psi^\theta_*&\simeq (\Psi^{\bS^0})^{-1}_* F_{B_{\theta'}}^{\ve{o}}\Psi^{\bS^0}_*F_{Y,\bL_\theta,\bS^0, \hat{\frs}} \psi^\theta_*&&\text{(Equation \eqref{eq:compband3})}\\
 &\simeq (\Psi^{\bS^0})^{-1}_* F_{B_{\theta'}}^{\ve{o}} F_{Y,\bL_{\theta'},\bS^0, \hat{\frs}}\Psi_* \psi_*^\theta&& \text{(Equation \eqref{eq:compband1})}\\
 &\simeq(\Psi^{\bS^0})^{-1}_* F_{B_{\theta'}}^{\ve{o}} F_{Y,\bL_{\theta'},\bS^0, \hat{\frs}} \psi_*^{\theta'}&&\text{(Equation \eqref{eq:compband2})}\\
 &\simeq F_{B_{\theta'}}^{\ve{o}} F_{Y,\bL_{\theta'},\bS^0, \hat{\frs}} \psi_*^{\theta'},&&\text{(Equation \eqref{eq:compband4})}
 \end{align*}
 completing the proof.
 \end{proof}

\section{Basic relations involving the handle maps, band maps and quasi-stabilization maps}
\label{sec:relationsI}

In this section, we prove some  basic relations involving the handle attachment maps, the band maps, and the quasi-stabilization maps.

\subsection{Basic relations between handle attachments, quasi-stabilizations and basepoint actions}

We first show that the 1-handle and 3-handle maps always commute with the quasi-stabilization maps:

\begin{lem}\label{lem:1-handlesquasistabcommute}Suppose $\bL$ is a multi-based link in $Y$, and $\bL_{w,z}^+$ is the multi-based link obtained by adding two new, adjacent basepoints $(w,z)$ to $\bL$. If $\bS$ is a framed 0-sphere or 2-sphere in $Y\setminus L$, then
\[F_{Y,\bL^+_{w,z}, \bS, \hat{\frs}}S_{w,z}^+\simeq S_{w,z}^+ F_{Y,\bL, \bS,\hat{\frs}} \qquad \text{and} \qquad F_{Y,\bL^+_{w,z}, \bS, \hat{\frs}}T_{w,z}^+\simeq T_{w,z}^+ F_{Y,\bL, \bS,\hat{\frs}}.\]
The maps $S_{w,z}^-$ and $T_{w,z}^-$ satisfy the analogous relations with $F_{Y,\bL^+_{w,z},\bS,\hat{\frs}}$ and $F_{Y,\bL,\bS,\hat{\frs}}$. 
\end{lem}

\begin{proof}The proof is similar to the proof that the 1-handle maps commute with each other (Lemma~\ref{lem:1-3-handlemapscommute}) and the proof that quasi-stabilization maps commute with each other (Proposition~\ref{prop:TScommute}). As in those situations, the formulas for the two maps look like they commute, but we need to verify that a single almost complex structure can be chosen  to compute both compositions.

We start with the case that $\bS=\bS^0$ is a framed 0-sphere. Furthermore, we first show the relation involving the positive quasi-stabilization maps $S_{w,z}^+$ and $T_{w,z}^+$. We want to show the following:
\begin{equation}
F_{Y,\bL_{w,z}^+,\bS^0,\hat{\frs}} S_{w,z}^+\simeq S_{w,z}^+F_{Y,\bL,\bS^0,\hat{\frs}}\qquad \text{and}  \qquad F_{Y,\bL_{w,z}^+,\bS^0,\hat{\frs}} T_{w,z}^+\simeq T_{w,z}^+F_{Y,\bL,\bS^0,\hat{\frs}}.
\label{eq:1handlequasicommute}
\end{equation}

  Let $(\Sigma, \ve{\alpha},\ve{\beta},\ve{w},\ve{z})$ be a diagram for $(Y,\bL)$ such that $\Sigma$ intersects $\bS^0$ in two disks which are disjoint from $\as\cup \bs\cup \ws\cup \zs$. Let $\alpha_0$ and $\beta_0$ denote the new curves in the annulus region of the 1-handle, and $\alpha_s'$ and $\beta_0'$ denote the two new curves from quasi-stabilization. These are shown in Figure~\ref{fig::28}.
 
\begin{figure}[ht!]
\centering
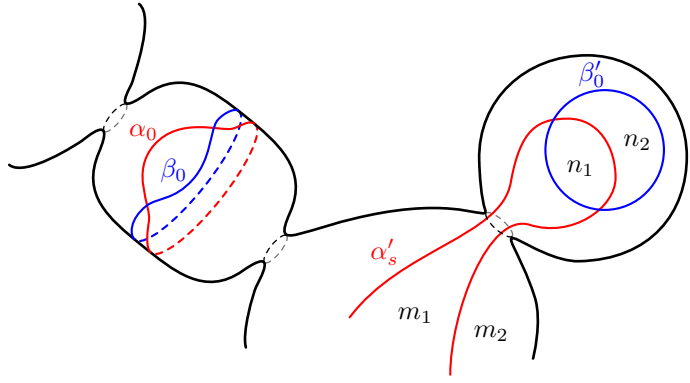
\caption{\textbf{The diagram after attaching a 1-handle and quasi-stabilizing.} This is considered in Lemma~\ref{lem:1-handlesquasistabcommute}. We stretch the almost complex structure along the three dashed curves. Multiplicities $m_1,$ $m_2,$ $n_1,$ and $n_2$ are also labeled.\label{fig::28}}
\end{figure} 
 
Suppose that $J_s$ is a cylindrical almost complex structure on $\Sigma\times [0,1]\times \R$, which is split on  $(\im (\bS^0)\cap \Sigma) \cup N(p)\times [0,1]\times \R$ where $p\in \Sigma$ is the point at which we perform the quasi-stabilization operation. Let $\ve{T}=(T_1,T_2,T')$ denote a triple of positive real numbers. Let $\Sigma'$ denote the Heegaard surface we obtain after attaching a 1-handle along $\bS^0$ and quasi-stabilizing at $p$. We view $\Sigma'$ as being obtained by removing $\im (\bS^0)$ as well as a neighborhood of $p$, and then gluing in an annulus and a disk to the resulting boundary components. We construct an almost complex structure $J_s(\ve{T})$ on $\Sigma'\times [0,1]\times \R$, which agrees with $J_s$ on $\Sigma\setminus (\im (\bS)\cup N(p))\times [0,1]\times \R$, and has necks of length $T_1,$ $T_2$ and $T'$ inserted along the connected sum tubes.

 To show Equation~\eqref{eq:1handlequasicommute}, we need to show that if $\ve{T}_1$ and $\ve{T}_2$ are two 3-tuples of neck lengths, with all components sufficiently large (but with no assumption on the relative sizes of the neck lengths), then
\begin{equation}
\Phi_{J_s(\ve{T}_1)\to J_s(\ve{T}_2)}(\ve{x}\times \theta^+\times \theta^{\os})=\ve{x}\times \theta^+\times \theta^{\os},\label{eq:changeofcxstructure1-handlequasi}
\end{equation} for $\os\in \{\ws,\zs\}$. Here $\theta^+\in \alpha_0\cap \beta_0$ denotes the top degree intersection point and $\theta^{\os}\in \alpha_s'\cap  \beta_0'$ denotes the top $\gr_{\os}$-graded intersection point.

To establish Equation~\eqref{eq:changeofcxstructure1-handlequasi}, we take two sequences $\{\ve{T}_1^i\}_{i\in \N}$ and $\{\ve{T}_2^i\}_{i\in \N}$ of 3-tuples, all of whose components approach $+\infty$ as $i\to \infty$. We also construct a sequence of non-cylindrical almost complex structures $\{\tilde{J}^i\}_{i\in \N}$ which agree with $J_s(\ve{T}^i)$ on $\Sigma\times [0,1]\times (-\infty,-1]$ and agree with $J_s(\ve{T}_2^i)$ on $\Sigma\times [0,1]\times [1,\infty)$. As in Proposition~\ref{prop:TScommute} and Lemma~\ref{lem:1-3-handlemapscommute}, we assume that $(\Sigma'\times [0,1]\times \R, \tilde{J}^i)$ contains the almost complex submanifold $(\Sigma\setminus N_i\times [0,1]\times \R,J_s)$ where  $N_i\subset \Sigma$ denotes a nested sequence of subsets, each consisting of three disks, each contained in a component of $N(p)\cup \im(\bS^0)\cap \Sigma$, such that $\bigcap_{i\in \N} N_i$ consists of three points.

Suppose we are given a sequence $\tilde{J}^i$-holomorphic disks $u_i$, which represent a fixed class $\phi\in \pi_2(\ve{x}\times \theta\times \theta^{\os}, \ve{y}\times y\times y')$, where $\theta\in \{\theta^+,\theta^-\}$,  $y\in \alpha_0\cap \beta_0$ and $y'\in \alpha_s'\cap \beta_0'$. We can extract a broken limit curve consisting of a broken holomorphic disk $U_\Sigma$ on $(\Sigma,\as,\bs)$, an $\as\cup \{\alpha_s'\}$-boundary degeneration $\cA$, and a broken holomorphic disk $U_0$ on $(S^2,\alpha_0',\beta_0')$, where $\alpha_0'$ denotes the restriction of $\alpha_s'$ to the quasi-stabilized region. Write $\phi_\Sigma$ and $\phi_0$ for the total homology classes of $U_\Sigma$ and $U_0$. The index computation from \cite{ZemQuasi}*{Lemma~5.3} (compare the proofs of Proposition~\ref{prop:TScommute} and Lemma~\ref{lem:1-3-handlemapscommute}, above) adapts to our present situation to yield that
\begin{equation}
\mu(\phi)=\mu(\phi_\Sigma)+\gr(\theta,y)+n_1(\phi_0)+n_2(\phi_0)+m_1(\cA)+m_2(\cA)+2\sum_{\substack{\cD\in C(\Sigma\setminus \as)\\ \cD\neq A_p}} n_{\cD}(\cA).\label{eq:Maslovindex1-handlequasi}
\end{equation}
In the above expression, we are writing $A_p$ for the component of $\Sigma\setminus \as$ which contains the point $p$.

To establish Equation~\eqref{eq:changeofcxstructure1-handlequasi}, we need to consider only classes $\phi$ which are in $\pi_2(\xs\times \theta^+, \theta^{\os}\times \ys\times y\times y')$. Since $\phi_{\Sigma}$ admits a broken holomorphic representative, by transversality we know that $\mu(\phi_\Sigma)\ge 0$. The rest of the terms in Equation~\eqref{eq:Maslovindex1-handlequasi} are clearly nonnegative, and hence we conclude all must be zero if $\mu(\phi)=0$. It is straightforward to see that this implies that $\phi$ is a constant class. On the other hand, the constant homology classes always have a unique $\tilde{J}^i$-holomorphic representative. Equation~\eqref{eq:changeofcxstructure1-handlequasi} follows, and hence the relation from Equation~\eqref{eq:1handlequasicommute} also follows.

Note that the commutation of the 1-handle map with the minus quasi-stabilization maps also follows immediately from Equation~\eqref{eq:changeofcxstructure1-handlequasi}. Hence
\begin{equation} F_{Y,\bL,\bS^0,\hat{\frs}}  S_{w,z}^-\simeq S_{w,z}^- F_{Y,\bL_{w,z}^+,\bS^0,\hat{\frs}} \qquad \text{and} \qquad  T_{w,z}^- F_{Y,\bL_{w,z}^+,\bS^0,\hat{\frs}}\simeq F_{Y,\bL,\bS^0,\hat{\frs}} T_{w,z}^-.
\label{eq:1handlequasicommuteminus}
\end{equation}

We now consider the case that $\bS$ is a framed 2-sphere, and we wish to commute either a positive or negative quasi-stabilization with the 3-handle map, i.e.
\begin{equation}
 F_{Y,\bL,\bS^2,\hat{\frs}}  S_{w,z}^+\simeq S_{w,z}^+ F_{Y,\bL_{w,z}^+,\bS^2,\hat{\frs}} \qquad \text{and} \qquad  T_{w,z}^+ F_{Y,\bL_{w,z}^+,\bS^0,\hat{\frs}}\simeq F_{Y,\bL,\bS^2,\hat{\frs}} T_{w,z}^+,\label{eq:3handlequasicommute}
\end{equation}
and similarly for the negative quasi-stabilization maps. To show Equation~\eqref{eq:3handlequasicommute}, we now show that whenever $\ve{T}_1$ and $\ve{T}_2$ have sufficiently large components, we have the equality
\begin{equation}
\Phi_{J_s(\ve{T}_1)\to J_s(\ve{T}_2)}(\ve{x}\times \theta^-\times \theta^{\os})=\ve{x}\times \theta^-\times \theta^{\os}+\sum_{\ve{o}\in \{\ws,\zs\}} \sum_{\ys\in \bT_{\as}\cap \bT_{\bs}} c_{\xs,\ys}^{\os}\cdot\ve{y}\times \theta^+\times \theta^{\os},\label{eq:changeofcxstructure1-handlequasi2}
\end{equation}
where $\os\in \{\ws,\zs\}$ and $c_{\xs,\ys}^{\os}$ are elements of $\cR_{\bmP}^-$ which are not necessarily independent of $\ve{T}_1$ and $\ve{T}_2$. To establish this fact, we note that we simply adapt the above procedure when $\bS$ was a framed 0-sphere. In our present case, we need to show that if $\phi\in \pi_2(\xs\times \theta^-\times\theta^{\os}, \ys\times \theta^-\times y')$ is a Maslov index zero class which admits holomorphic representatives for $\tilde{J}^i$ when $i$ is arbitrary large, then $\phi$ is in fact the constant holomorphic class. This follows by analyzing the index formula from Equation~\eqref{eq:Maslovindex1-handlequasi}. As before, if $\phi\in \pi_2(\xs\times \theta^-\times \theta^{\os}, \ys\times \theta^-\times y')$ admits $\tilde{J}^i$-holomorphic representatives for arbitrarily large $i$, then all the terms in Equation~\eqref{eq:Maslovindex1-handlequasi} must be zero, and this clearly implies that $\phi$ is a constant class. Since the constant classes are always counted by $\Phi_{J_s(\ve{T}_1\to J_s(\ve{T}_2)}$, Equation~\eqref{eq:changeofcxstructure1-handlequasi2} follows.

 Combining Equations~\eqref{eq:changeofcxstructure1-handlequasi} and  \eqref{eq:changeofcxstructure1-handlequasi2} implies Equation~\eqref{eq:3handlequasicommute}, completing the proof.
\end{proof}

We now show that the 2-handle maps commute with the quasi-stabilization maps:

\begin{lem}\label{lem:2-handlemapsquasistabcommute}Suppose $\bL$ is a multi-based link in $Y$, and let $\bL^+_{w,z}$ be the link obtained by adding the adjacent pair of basepoints $(w,z)$ to $\bL$. If $\bS^1\subset Y\setminus L$ is a framed 1-dimensional link, and $\frs\in \Spin^c(W(Y,\bS))$, then
\[S_{w,z}^{+} F_{Y,\bL, \bS^1,\frs}\simeq F_{Y,\bL^+_{w,z},\bS^1,\frs} S_{w,z}^{+} \qquad \text{and}\qquad T_{w,z}^{+} F_{Y,\bL, \bS^1,\frs}\simeq F_{Y,\bL_{w,z}^+,\bS^1,\frs} T_{w,z}^{+}.\] The analogous relation holds with the maps $S_{w,z}^-$ and $T_{w,z}^-$.
\end{lem}
\begin{proof}We pick a Heegaard triple $(\Sigma,\as,\bs,\bs',\ws,\zs)$ which is subordinate to a $\beta$-bouquet for $\bS^1$ in $(Y,\bL)$. A Heegaard triple subordinate to a $\beta$-bouquet for $\bS^1$ in $(Y,\bL_{w,z}^+)$ can be constructed by quasi-stabilizing the triple $(\Sigma,\as,\bs,\bs',\ws,\zs)$ as in Figure~\ref{fig::16quasi}. The result then follows from the triangle counts from Proposition~\ref{prop:singlealphaquasistabtriangle}.
\end{proof}

Finally, it is convenient to have the following fact stated:

\begin{lem}\label{lem:PhiPsicommutewithhandles} Suppose that $Y$ is a 3-manifold containing the multi-based link $\bL=(L,\ws,\zs)$ and $\bS\subset Y\setminus L$ is a framed 0-sphere, 2-sphere or 1-dimensional link. If $w\in \ws$ and $z\in \zs$, then
\[F_{Y, \bL, \bS,\frs} \Phi_w\simeq \Phi_w F_{Y,\bL,\bS,\frs}\qquad \text{and} \qquad F_{Y,\bL,\bS,\frs}\Psi_z\simeq \Psi_z F_{Y,\bL,\bS,\frs}.\]
\end{lem}
\begin{proof}The map $F_{Y,\bL,\bS,\frs}$ is a chain map on the uncolored complexes. Hence if we differentiate the expression
\[\d F_{Y,\bL,\bS,\frs}+ F_{Y,\bL,\bS,\frs} \d=0\] with respect to $U_w$, and apply the Leibniz rule, we obtain the relation
\[\Phi_wF_{Y,\bL,\bS,\frs}+F_{Y,\bL,\bS,\frs}\Phi_w\simeq 0.\] A similar argument works for $\Psi_z$.
\end{proof}

\begin{lem}\label{lem:bandmapsandsurgerymapscommute}Suppose that $B$ is a band for the link $\bL$ in $Y$, and $\bS^1$ is a framed 1-dimensional link in $Y$, which is disjoint from $L\cup B$ and $\frs\in \Spin^c(W(Y,\bS))$. If $\ve{o}\in \{\ve{w},\ve{z}\}$, then
\[F_{Y,\bL(B),\bS^1,\frs}F_B^{\ve{o}}\simeq F_B^{\ve{o}} F_{Y,\bL,\bS^1,\frs}.\]
\end{lem}
\begin{proof} Using Proposition~\ref{prop:alphabandmapsarebetabandmaps} or \ref{prop:alphabandmapsarebetabandmapstypew}, we may assume that $B$ is an $\alpha$-band.  We can pick a Heegaard quadruple $(\Sigma, \ve{\alpha}',\ve{\alpha},\ve{\beta},\ve{\beta}',\ve{w},\ve{z})$ such that the following holds:
\begin{enumerate}
\item $(\Sigma, \ve{\alpha},\ve{\beta},\ve{w},\ve{z})$ is Heegaard diagram for $(Y,\bL)$.
\item $(\Sigma, \ve{\alpha}',\ve{\alpha},\ve{\beta},\ve{w},\ve{z})$ is a triple subordinate to the band $B\subset Y$.
\item $(\Sigma,\ve{\alpha},\ve{\beta},\ve{\beta}',\ve{w},\ve{z})$ is a triple subordinate to a $\beta$-bouquet for $\bS^1$ in $(Y,\bL)$.
\item $(\Sigma,\ve{\alpha}',\ve{\alpha},\ve{\beta}',\ve{w},\ve{z})$ is a triple subordinate to the band $B$ in $Y(\bS^1)$.
\item $(\Sigma, \ve{\alpha}',\ve{\beta},\ve{\beta}',\ve{w},\ve{z})$ is a triple subordinate to a $\beta$-bouquet for $\bS^1$ in $(Y,\bL(B))$.
\end{enumerate}

Let $\Theta_{\as',\as}^{\ve{o}}\in \cHFL^-(\Sigma, \ve{\alpha}',\ve{\alpha},\sigma)$ denote the distinguished elements, for $\os\in \{\ws,\zs\}$, from Lemma~\ref{lem:topdegreeunlink2}. Similarly let $\Theta_{\beta\beta'}^+$ be the distinguished element of $\cHFL^-(\Sigma, \ve{\beta},\ve{\beta}',\sigma)$ from Lemma~\ref{lem:topdegreeunlink1}. The associativity relations yield the relation
\[F_{\as',\as,\bs',\frs|_{Y(\bS)}}(\Theta_{\as',\as}^{\ve{o}}, F_{\as,\bs,\bs',\frs}(-, \Theta_{\bs,\bs'}^+))\simeq F_{\as',\bs,\bs',\frs}(F_{\as',\as,\bs,\frs|_{Y}}(\Theta_{\as,'\as}^{\ve{o}},-), \Theta_{\bs,\bs'}^+),\] showing that
\[F_{Y,\bL(B),\bS^1,\frs}^\beta F_B^{\ve{o}}\simeq F_B^{\ve{o}} F_{Y,\bL,\bS^1,\frs}^\beta.\] Noting that $F_{Y,\bL,\bS_1,\frs}\simeq F_{Y,\bL,\bS^1,\frs}^\alpha\simeq F_{Y,\bL,\bS^1,\frs}^\beta$ by Lemma~\ref{lem:alphasurgerymaps=betasurgerymaps}, the proof is complete.
\end{proof}

\subsection{Quasi-stabilizations and critical point cancellations}
\label{sec:quasisandbirthdeath}
In this section, we show that the quasi-stabilization maps, birth and death cobordism maps, and band maps satisfy a set of relations which we can interpret as showing invariance of the link cobordism maps under index 1/2 critical point cancellations of the surface a the link cobordism. Notice that if we ignore the dividing set, a birth cobordism adds a disk to the surface of a link cobordism. The birth cobordism can be canceled by attaching a band, leaving a link cobordism which is diffeomorphic to the identity link cobordism, as an undecorated link cobordism. However as a decorated link cobordism, the resulting link cobordism will not be the identity. Instead, we are left with the dividing set corresponding to a quasi-stabilization map. The manipulation is shown in Figure~\ref{fig::111}.

\begin{figure}[ht!]
\centering
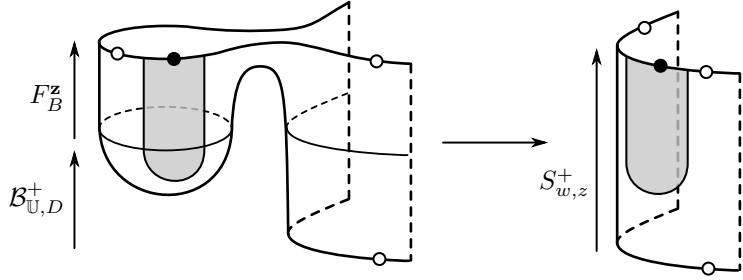
\caption{\textbf{A birth cobordism followed by a saddle is diffeomorphic to a quasi-stabilization.} This is reflected by the relation in Proposition~\ref{prop:bandsanddiskstab}.\label{fig::111}}
\end{figure}

Suppose that $\bL=(L,\ws,\zs)$ is a multi-based link in $Y$ and $\bU=(U,w,z)$ is an unknot with two basepoints, spanned by a 2-disk $D\subset Y$ with $D\cap L=\varnothing$. Suppose that we attach a band $B$ to $L\cup U$, which has one foot on $U$ and one foot on $L$, and is disjoint from $\ws\cup \zs\cup \{w,z\}$. The surface $B\cup D$, yields an isotopically unique diffeomorphism 
\[\phi\colon (Y,\bL)\to (Y,(\bL\cup \bU)(B)),\] which is the identity outside of a neighborhood of $B\cup D$. Let $\bar{B}$ denote the dual band, attached to $(Y,\bL\cup \bU)(B)$, yielding the link $\bL\cup \bU$. We prove the following:

\begin{prop}\label{prop:bandsanddiskstab}Suppose that $(Y,\bL)$ is a multi-based link and $\bU=(U,w,z)$ be a doubly based unknot with a Seifert disk $D$ which is disjoint from $\bL$, and let $B$ and $\bar{B}$ be the bands described in the previous paragraph. If $B$ is an $\alpha$-band then we have
\[F_{B}^{\ve{z}} \cB_{\bU,D}^+\simeq S_{w,z}^+\phi_*\qquad\text{and} \qquad F_B^{\ve{w}} \cB_{\bU,D}^+\simeq T_{w,z}^+ \phi_*.\] Similarly,
\[ \cD_{\bU,D}^-F_{\bar{B}}^{\ve{z}}\simeq S_{w,z}^-\phi_*\qquad\text{and} \qquad \cD_{\bU,D}^-F_{\bar{B}}^{\ve{z}}\simeq T_{w,z}^-\phi_*.\]

If $B$ is a $\beta$-band, we have
\[F_{B}^{\ve{z}} \cB_{\bU,D}^+\simeq S_{z,w}^+ \phi_*\qquad\text{and} \qquad F_B^{\ve{w}} \cB_{\bU,D}^+\simeq T_{z,w}^+ \phi_*,\] and
\[ \cD_{\bU,D}^-F_{\bar{B}}^{\ve{w}}\simeq T_{z,w}^-\phi_*\qquad\text{and} \qquad \cD_{\bU,D}^-F_{\bar{B}}^{\ve{w}}\simeq T_{z,w}^-\phi_*.\]
\end{prop}

Proposition~\ref{prop:bandsanddiskstab} will follow from two counts of holomorphic triangles. If $\cT=(\Sigma,\ve{\alpha}', \ve{\alpha},\ve{\beta},\ve{w},\ve{z})$ is a Heegaard triple with a distinguished basepoint $z'\in \zs$,  we will stabilize $\cT$ near $z'$ to form the Heegaard triple
\[\cT^+:=(\Sigma,\as'\cup \{\alpha_0'\},\as\cup \{\alpha_0\}, \bs\cup \{\beta_0\},\ws\cup \{w\}, \zs\cup \{z\}),\]
 shown in Figure~\ref{fig::11}. We note that $X_{\as',\as,\bs}$ and $X_{\as'\cup \{\alpha_0'\}, \as\cup \{\alpha_0\},\bs\cup \{\beta_0\}}$ are canonically diffeomorphic. Furthermore, if $\psi$ is a homology class of triangles on $\cT$ and $\psi_0$ is a homology class of triangles on the diagram $(S^2,\alpha_0',\alpha_0,\beta_0)$ and $\psi$ and $\psi_0$ have the same multiplicity in the connected sum region, then it is straightforward to see that
\[
\frs_{\ws\cup \{w\}}(\psi\# \psi_0)=\frs_{\ws}(\psi),
\]
after identifying $X_{\as',\as,\bs}$ and $X_{\as'\cup \{\alpha_0'\}, \as\cup \{\alpha_0\},\bs\cup \{\beta_0\}}$.

 We will write (abusing notation slightly)  
\[\alpha_0'\cap \alpha_0=\{\theta^{\ws}, \xi^{\ws}\}, \qquad \alpha_0\cap \beta_0=\{\theta^+,\theta^-\}\qquad \text{and} \qquad \alpha_0'\cap \beta_0=\{\theta^{\ws},\xi^{\ws}\}.\] For $\alpha_0'\cap \alpha_0$ and $\alpha'_0\cap \beta_0$,  we write $\theta^{\ws}$ for the higher degree intersection point with respect to $\gr_{\ws}$, and  $\xi^{\ws}$ for the lower degree intersection point. For $\alpha_0\cap \beta$,  we write $\theta^+$ and $\theta^-$ for the higher and lower degree intersection points, since the designation is the same for $\gr_{\ws}$ and $\gr_{\zs}$.

\begin{figure}[ht!]
\centering
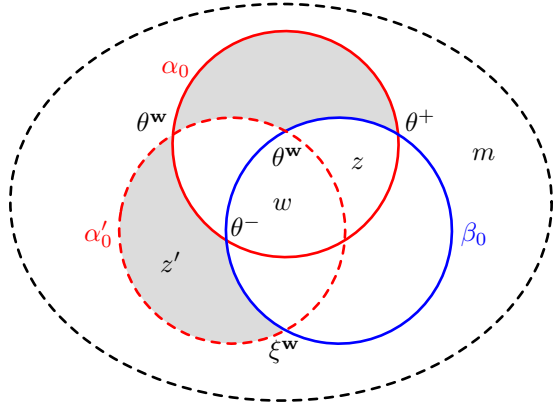
\caption{\textbf{The stabilized region of the Heegaard triple $\cT^+$ for merging a doubly based unknot with $\bL$.} This triple is considered in Lemma~\ref{lem:trianglemapforalphabandmap1}. The triple $\cT^+$ can be used to compute the $\alpha$-band map for a band which merges an unknot with two basepoints onto the link component containing $z'$. The two shaded regions are each examples of homology classes of triangles, which could be counted. The letter $m$ denotes the multiplicity in the connected sum region.\label{fig::11}}
\end{figure}

We recall that for the $\zs$-band map to be defined, the variables $V_z$ and $V_{z'}$ must be identified. In the following lemma, we identify both with a single variable, $V$.

\begin{lem}\label{lem:trianglemapforalphabandmap1}Suppose $\cT=(\Sigma,\ve{\alpha}', \ve{\alpha},\ve{\beta},\ve{w},\ve{z})$ is a Heegaard triple with a distinguished point $z'\in \ve{z}$, and write $\cT^+=(\Sigma,\as'\cup \{\alpha_0'\},\as\cup \{\alpha_0\}, \bs\cup \{\beta_0\},\ws\cup \{w\}, \zs\cup \{z\})$ for the Heegaard triple which has been stabilized at $z'$, shown in Figure~\ref{fig::11}. For sufficiently stretched almost complex structure $J_s(T)$, we have the relations
\[
F_{\cT^+,\frs, J_s(T)}(\ve{x}\times \theta^{\ws}, \ys\times \theta^+)=F_{\cT,\frs,J_s}(\ve{x},\ve{y})\otimes \theta^{\ws}
\]
and
\[
F_{\cT^+,\frs, J_s(T)}(\ve{x}\times \theta^{\ws}, \ys\times \theta^-)=V\cdot F_{\cT,\frs,J_s}(\ve{x},\ve{y})\otimes \xi^{\ws}+G(\ve{x},\ve{y})\otimes \theta^{\ws},
\]
for some map
\[G\colon  \cCFL^-(\Sigma,\as',\as)\otimes \cCFL^-(\Sigma,\as,\bs)\to \cCFL^-(\Sigma,\as',\bs),\] which is not necessarily independent of $J_s$ or $T$.
\end{lem}

\begin{proof} If we ignore the $\zs$-basepoints, the triangle counts we wish to show are in fact just the ones used to show the well-definedness of the 1-handle map under isotopies and handles outside of the 1-handle region. Hence to prove the theorem, we will adapt the triangle counts used to show the well-definedness of the 1-handle maps from \cite{ZemGraphTQFT}*{Section~6}.

The proof is similar to the proof of \cite{ZemGraphTQFT}*{Theorem~6.17}. For a homology class of triangles $\psi\# \psi_0\in \pi_2(\ve{x}\times \theta^{\ws},\ve{y}\times y,\ve{z}\times z)$, it is straightforward to compute that
\begin{equation}
\mu(\psi\#\psi_0)=\mu(\psi)+2n_{w}(\psi_0)+\gr_{\ve{w}}(y,z), \label{eq:maslovindexneckstretch}
\end{equation} where $\gr_{\ve{w}}(y,z)$ denotes the drop in grading from $y$ to $z$, with the convention that both $\theta^+\in\alpha_0\cap \beta_0$ and $\theta^{\ws}\in \alpha'_0\cap \beta_0$ have $\gr_{\ws}$-grading 0. 

We note $n_w(\psi_0)\ge 0$. Also, if we stretch the neck sufficiently, we can assume that if $\psi\# \psi_0$ has a holomorphic representative, then both $\psi$ and $\psi_0$ have a broken holomorphic representative. In particular, $\mu(\psi)\ge 0$ by transversality considerations.

We consider first the case that $\gr_{\ws}(y,z)=1$, corresponding to when $y=\theta^+$ and $z=\xi^{\ws}$. In this case the expression in Equation~\eqref{eq:maslovindexneckstretch} must be at least 1, for sufficiently large $T$. Hence there are no solutions when $\mu(\psi\# \psi_0)=0$. Hence 
\begin{equation}
\langle F_{\cT^+, \frs,J_s(T)}(\ve{x}\times \theta^{\ws}, \ve{y}\times \theta^+),\xi^{\ws}\rangle=0,\label{eq:coefficienttheta^w}
\end{equation} 
where we are writing $\langle \ve{\tau}, \xi^{\ws}\rangle$ for the $\xi^{\ws}$-component of $\ve{\tau}$.

Next we consider the case that $\gr_{\ws}(y,z)=0$. In this case, we have that the pair $(y,z)$ is $(\theta^+,\theta^{\ws})$ or $(\theta^-,\xi^{\ws})$. In this case,  using Equation~\eqref{eq:maslovindexneckstretch}, we see that 
\[\mu(\psi)=0 \qquad \text{and} \qquad n_w(\psi_0)=0.\] According to \cite{ZemGraphTQFT}*{Theorem~6.17}, if $\psi$ is a homology class of triangles on $\cT$ with Maslov index 0, then 
\begin{equation}
\# \cM(\psi)\equiv \sum_{\substack{\psi_0\in \pi_2(\theta^{\ws}, \theta^+, \theta^{\ws})\\
n_w(\psi_0)=0}} \#\cM(\psi\# \psi_0) \pmod{2},\label{eq:equivalenceofcounts1}
\end{equation}
and
\begin{equation}
\# \cM(\psi)\equiv \sum_{\substack{\psi_0\in \pi_2(\theta^{\ws}, \theta^-, \xi^{\ws})\\
n_w(\psi_0)=0}} \#\cM(\psi\# \psi_0) \pmod{2}.\label{eq:equivalenceofcounts2}
\end{equation}
We note \cite{ZemGraphTQFT}*{Theorem~6.17} is based an analytic result of Lipshitz \cite{LipshitzCylindrical}*{Proposition~A.2}.

The next step is to relate the total multiplicity of a triangle $\psi\# \psi_0$ over $z$ and $z'$ to the multiplicity of $\psi$ over $z'$. We claim that if $\psi_0\in \pi_2(\theta^{\ws}, \theta^+,\theta^{\ws})$ is a homology class of triangles, then
\begin{equation}
m(\psi_0)+n_w(\psi_0)=n_{z}(\psi_0)+n_{z'}(\psi_0),\label{eq:relationsofmultiplicities1}
\end{equation} where $m(\psi_0)$ denotes the multiplicity in the connected sum region. To see this, we note that it holds for the small triangle shown on the top of Figure~\ref{fig::11}. The formula is preserved when we splice in a class on $(S^2,\alpha_0',\alpha_0,\beta_0)$ in any of the $\pi_2(t,t)$, and it is easy to see that on this diagram, any two homology classes with the same endpoints can be related by a sum of such classes.

Combining Equations~\eqref{eq:equivalenceofcounts1}~and~\eqref{eq:relationsofmultiplicities1}, we obtain
\begin{equation}
\langle F_{\cT^+,\frs,J_s(T)}(\ve{x}\times \theta^{\ws}, \ve{y}\times \theta^+),\theta^{\ws}\rangle=F_{\cT,\frs,J_s}(\ve{x},\ve{y})\label{eq:coefficienttheta^w2}
\end{equation} 

Similarly, if $\psi_0\in \pi_2(\theta^{\ws}, \theta^-,\theta^{\zs})$ is a homology class of triangles, then
\begin{equation}
m(\psi_0)+n_w(\psi_0)+1=n_{z}(\psi_0)+n_{z'}(\psi_0).\label{eq:relationsofmultiplicities2}
\end{equation} As before, this can be verified by checking it for a single class (for example the class shown on the bottom  of Figure~\ref{fig::11}), and then checking that the formula respects splicing in a class of disks in any $\pi_2(t,t)$. Combining Equations~\eqref{eq:equivalenceofcounts2}~and~\eqref{eq:relationsofmultiplicities2} we obtain the equality:
\begin{equation}
\langle F_{\cT^+,\frs,J_s(T)}(\ve{x}\times \theta^{\ws}, \ve{y}\times \theta^-),\xi^{\ws}\rangle=V\cdot F_{\cT,\frs,J_s}(\ve{x},\ve{y}).\label{eq:coefficientxi^w3}
\end{equation}

Combining Equations~\eqref{eq:coefficienttheta^w},~\eqref{eq:coefficienttheta^w2}, and \eqref{eq:coefficientxi^w3}, we obtain the lemma statement.

\end{proof}

\begin{figure}[ht!]
\centering
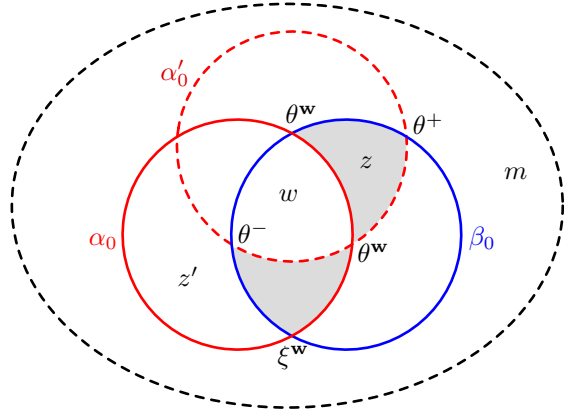
\caption{\textbf{The stabilized region of the Heegaard triple $\Hat{\cT}^+$ for splitting a doubly based unknot off of $\bL$.} This triple is considered in Lemma~\ref{lem:trianglemapforalphabandmap2}. The triple $\Hat{\cT}^+$ can be used to compute the $\alpha$-band map for a band which splits an unknot with two basepoints, $w$ and $z$, off of $\bL$. The two shaded regions are each examples of homology classes of triangles, which could be counted. The letter $m$ denotes the multiplicity in the connected sum region.\label{fig::37}}
\end{figure}

Adapting the proof of Lemma~\ref{lem:trianglemapforalphabandmap1} slightly, we can compute the effect of splitting off a doubly based unknot from a link $\bL$.

\begin{lem}\label{lem:trianglemapforalphabandmap2} Suppose $\cT=(\Sigma,\ve{\alpha}', \ve{\alpha},\ve{\beta},\ve{w},\ve{z})$ is a Heegaard triple with a distinguished point $z'\in \ve{z}$, and write $\Hat{\cT}^+=(\Sigma,\as'\cup \{\alpha_0'\},\as\cup \{\alpha_0\}, \bs\cup \{\beta_0\},\ws\cup \{w\}, \zs\cup \{z\})$ for the Heegaard triple which has been stabilized at $z'$, as shown in Figure~\ref{fig::37}. Then for sufficiently stretched almost complex structure $J_s(T)$, we have 
\[
F_{\Hat{\cT}^+,\frs, J_s(T)}(\ve{x}\times \theta^{\ws}, \ys\times \theta^{\ws})=V\cdot F_{\cT,\frs,J_s}(\ve{x},\ve{y})\otimes \theta^{+}
\]
and
\[
F_{\Hat{\cT}^+,\frs, J_s(T)}(\ve{x}\times \theta^{\ws}, \ys\times \xi^{\ws})= F_{\cT,\frs,J_s}(\ve{x},\ve{y})\otimes \xi^{\ws}+G'(\ve{x},\ve{y})\otimes \theta^-,
\]
for some map
\[G'\colon  \cCFL^-(\Sigma,\as',\as)\otimes \cCFL^-(\Sigma,\as,\bs)\to \cCFL^-(\Sigma,\as,\bs),\] which is not necessarily independent of $J_s$ or $T$.
\end{lem}

\begin{proof}The proof is a small adaptation of Lemma~\ref{lem:trianglemapforalphabandmap1}. The only difference is in the formulas which keep track of the total multiplicity over $z$ and $z'$, since Equations~\eqref{eq:relationsofmultiplicities1} and \eqref{eq:relationsofmultiplicities2} no longer hold. Instead, if $\psi_0\in \pi_2(\theta^{\ws}, \theta^{\ws},\theta^+)$, then adapting the argument from Lemma~\ref{lem:trianglemapforalphabandmap1}, we see that
\[
n_w(\psi_0)+m(\psi_0)+1=n_z(\psi_0)+n_{z'}(\psi_0).
\]
Similarly if $\psi_0\in \pi_2(\theta^{\ws},\xi^{\ws},\theta^-)$, then
\[
n_w(\psi_0)+m(\psi_0)=n_{z}(\psi_0)+n_{z'}(\psi_0).
\]  With these two formulas, the proof of Lemma~\ref{lem:trianglemapforalphabandmap1} applies with minimal change.
\end{proof}

\begin{proof}[Proof of Proposition~\ref{prop:bandsanddiskstab}] We first prove that $F_{B}^{\ve{z}} \cB_{\bU,D}^+\simeq S_{w,z}^+\phi_*$ in the case that $B$ is an $\alpha$-band. This will follow from the model triangle map computation in Lemma~\ref{lem:trianglemapforalphabandmap1}, as we now describe.

Let $(\Sigma,\as,\bs,\ws,\zs)$ be a diagram for $(Y,\bL)$. We let $\as'$ be small Hamiltonian isotopies of the $\as$ curves, and let $\cT^+$ be the stabilized triple shown in Figure~\ref{fig::11}. The birth map $\cB^+_{\bU,D}$ can be computed with using the subdiagram diagram $(\Sigma,\as\cup \{\alpha_0\}, \bs\cup \{\beta_0\})$ of the Heegaard triple $\cT^+$. In terms of this diagram, the map $\cB^+_{\bU,D}$ takes the form
\[
\cB^+_{\bU,D}(\ve{y})=\ve{y}\times \theta^+.
\]
The triple $(\Sigma, \ve{\alpha}'\cup \{\alpha_0'\},\ve{\alpha}\cup \{\alpha_0\},\ve{\beta}\cup \{\beta_0\},\ws\cup \{w\},\zs\cup \{z\})$ is a triple subordinate to the band $B$. By Lemma~\ref{lem:trianglemapforalphabandmap1}, we have
\[F_{B}^{\ve{z}}(\ve{y}\times \theta^+)=\Phi_{\as\to \as'}^{\bs}(\ys)\otimes \theta^{\ws},\] for an almost complex structure $J$ which has been stretched sufficiently along the dashed curve in Figure~\ref{fig::11}.

To compute $S_{w,z}^+$, we need to stretch the almost complex structure an a curve which encircles the $\alpha_0'$-curve, and intersects $\beta_0$ at two points. Using the change of almost complex structure computation from Lemma~\ref{lem:mapsforI+agree}, we know that we can ensure that this map takes the form 
\[\Phi_{J\to J'}(\ve{y}\times \theta^{\ws})=\ve{y}\times \theta^{\ws}.\] Hence we have that
\[\Phi_{J\to J'}F^{\ve{z}}_{B}\cB_{\bU,D}^+(\ys)=\Phi_{\as\to \as'}^{\bs}(\ys)\otimes \theta^{\ws},\] which is tautologically just $S_{w,z}^+\phi_*\Phi_{\as\to \as'}^{\bs}.$

To show $\cD_{\bU,D}^-F_{\bar{B}}^{\ve{z}}\simeq S_{w,z}^-\phi_*$, one follows similar reasoning as above, using the triangle map computation from Lemma~\ref{lem:trianglemapforalphabandmap2}. The statements about $\beta$-bands follow from an easy modification. The statements about $\ve{w}$-bands follow from symmetrical arguments.
\end{proof}

\section{Further relations involving the band maps}
\label{sec:furtherrelations}
In this section, we prove some important properties of the band maps. 

\subsection{Band maps and basepoint actions}

In this section, we describe when the maps $\Phi_w$, and $\Psi_z$ commute with the band maps $F_B^{\zs}$ and $F_{B}^{\ws}$:

\begin{lem}\label{lem:PhiPsicommutatorFB}Suppose $B$ is a band for $\bL=(L,\ws,\zs)$, of type-$\alpha$ or -$\beta$, and $w\in \ws$  and $z\in \zs$ are basepoints.  Then
\begin{enumerate} 
\item $F_B^{\zs}\Phi_w+\Phi_wF_{B}^{\zs}\simeq 0,$ regardless of the position of $w$,
\item $ F^{\zs}_B \Psi_z+\Psi_z F_B^{\zs}\simeq 0,$ if $z$ is not adjacent to an end of $B$.
\end{enumerate}  
\end{lem}

\begin{proof} Let us assume for concreteness that $B$ is an $\alpha$-band. Write $z_1$ and $z_2$ for the two $\zs$-basepoints adjacent to the ends of $B$. Write $\bmP_0=(\ws\cup \zs)/(z_1\sim z_2)$ and let $\sigma_0\colon \ws\cup \zs\to \bmP_0$ be the natural map. Let $(\Sigma,\as',\as,\bs,\ws,\zs)$ be a triple subordinate to $B$. By definition, the band map
\[F_B^{\zs}\colon  \cCFL^-(\Sigma,\as,\bs,\ws,\zs,\sigma_0,\frs)\to \cCFL^-(\Sigma,\as',\bs,\ws,\zs,\sigma_0,\frs)
\]
is defined by the formula
\[
F_B^{\zs}(-):=F_{\as',\as,\bs,\frs}(\Theta^{\ws},-).
\]
Gromov compactness combined with the fact that $\Theta^{\ws}$ is a cycle on the complex which is colored by $\sigma_0$ implies that
\begin{equation}
F_{B}^{\zs}\d_{\as,\bs}=\d_{\as',\bs}F_B^{\zs}.\label{eq:FBzchainmap}
\end{equation}
If $\d$ denotes the differential on either $\cCFL^-(\Sigma,\as',\bs,\sigma_0)$ or $\cCFL^-(\Sigma,\as,\bs,\sigma_0)$, we have 
\[\frac{d}{d U_w} (\d)=\Phi_w\qquad  \text{and}\qquad \frac{d}{d V_z} (\d) =\Psi_z,\] as long as $z\not \in \{z_1,z_2\}$. Hence, differentiating Equation~\eqref{eq:FBzchainmap} and using the Leibniz rule yields
\[F_B^{\zs} \Phi_w\simeq \Phi_w F_B^{\zs}\qquad \text{and}\qquad F_B^{\zs} \Psi_z\simeq \Psi_z F_B^{\zs},\] as long as $z\not\in \{z_1,z_2\}$.

Finally we note that these relations persist for an arbitrary coloring $\sigma\colon \ws\cup \zs\to \bmP$ such that $\sigma(z_1)=\sigma(z_2)$, by tensoring.
\end{proof}

Symmetrically, we have the following:
\begin{lem}Suppose $B$ is a band for $\bL=(L,\ws,\zs)$, of type-$\alpha$ or -$\beta$, and $w\in \ws$  and $z\in \zs$ are basepoints.  Then
\begin{enumerate} 
\item $F_B^{\ws}\Psi_z+\Psi_zF_{B}^{\ws}\simeq 0,$ regardless of the position of $z$,
\item $ F^{\ws}_B \Phi_w+\Phi_w F_B^{\ws}\simeq 0,$ if $w$ is not adjacent to an end of $B$.
\end{enumerate}  
\end{lem}

The previous two lemmas do not always tell us what the commutator of a basepoint map with a band map is when the basepoint is immediately adjacent to the end of the band. The following result will be useful for our purposes:

\begin{prop}\label{prop:FBzFBwrelatedbyPsiPhi}Suppose that $B$ is a band (either $\alpha$- or $\beta$-) for the link $\bL=(L,\ws,\zs)$ in $Y$, and write $w_1,$ $w_2$, $z_1$ and $z_2$ for the basepoints adjacent to the ends of $B$. If $\sigma\colon \ws\cup \zs\to \bmP$ is a coloring such that
\[\sigma(w_1)=\sigma(w_2)\qquad \text{and} \qquad \sigma(z_1)=\sigma(z_2),\] then
\[F_{B}^{\ve{w}}\simeq \Psi_{z_1} F_{B}^{\ve{z}}+F_B^{\ve{z}} \Psi_{z_1}\simeq \Psi_{z_2} F_B^{\ve{z}} +F_B^{\ve{z}} \Psi_{z_2} \] and
\[F_B^{\ve{z}}\simeq \Phi_{w_1} F_{B}^{\ve{w}}+F_B^{\ve{w}} \Phi_{w_1}\simeq \Phi_{w_2} F_B^{\ve{w}} +F_B^{\ve{w}} \Phi_{w_2}.\]
\end{prop}

Before we prove Proposition~\ref{prop:FBzFBwrelatedbyPsiPhi}, we prove the following useful lemma:

\begin{lem}\label{lem:PsiPhirelationbetweenThetaw,z} Suppose that $ \bU=(U,\ws,\zs)$ is an unlink in $(S^1\times S^2)^{\# k}$ such that all components of $\bU$ have 2 basepoints, except one which has 4. Write $w_1$, $w_2$, $z_1$ and $z_2$ for the basepoints on the component with four basepoints. If $\sigma(w_1)=\sigma(w_2)$ and $\sigma(z_1)=\sigma(z_2)$ then
\[\Phi_{w_i}(\Theta^{\zs})=\Theta^{\ws}\qquad \text{and} \qquad \Psi_{z_i}(\Theta^{\ws})=\Theta^{\zs},\] for $i\in \{1,2\}$.
\end{lem}
\begin{proof} We first check the claim  for an unknot $\bU_0$ with four basepoints in $S^3$. In that case, we can take a diagram $(S^2,\alpha_0,\beta_0,w_1,w_2,z_1,z_2)$ where $\alpha_0$ and $\beta_0$ intersect at two points, and $S^2\setminus (\alpha_0\cup \beta_0)$ consists of four bigons, each containing a single basepoint. In this case, the claim is an easy computation, since $\Phi_{w_i}$ counts only the bigon going over $w_i$, and similarly $\Psi_{z_i}$ counts the bigon going over $z_i$. Let us write $\Theta^{\ws}_0$ and $\Theta^{\zs}_0$ for the two distinguished elements of $\cHFL^-(S^3,\bU_0^\sigma,\frs_0)$.

We now verify the claim for an arbitrary unlink $\bU$ in $(S^1\times S^2)^{\# k}$, all of whose link components have exactly 2 basepoints, except for one component, which has 4. We note that there is a link cobordism from $(S^3,\bU_0)$ to $((S^1\times S^2)^{\# k},\bU)$ consisting of $|\bU|-1$ 0-handles (which each add a doubly based unknot) and $k$ 1-handles. Write $F$ for the composition of the corresponding 0-handle and 1-handle maps. By inspection we have
\begin{equation}
F(\Theta_0^{\ws})=\Theta^{\ws} \qquad \text{and} \qquad F(\Theta_0^{\zs})=\Theta^{\zs} \label{eq:FpreservesThetaw,z}
\end{equation}
Note that the 0-handle maps obviously commute with the maps $\Phi_{w_i}$ and $\Psi_{z_i}$. By Lemma~\ref{lem:PhiPsicommutewithhandles}, the 1-handle maps also commute with $\Phi_{w_i}$ and $\Psi_{z_i}$.  Hence, combining these observations we have
\[
\Phi_{w_i} (\Theta^{\zs})=\Phi_{w_i}(F(\Theta_0^{\zs}))=F(\Phi_{w_i}(\Theta_0^{\zs}))=F(\Theta_0^{\ws})=\Theta_0^{\ws}.
\] A similar argument shows that $\Psi_{z_i} (\Theta^{\ws})=\Theta^{\zs}$.
\end{proof}

We proceed with the proof of Proposition~\ref{prop:FBzFBwrelatedbyPsiPhi}:

\begin{proof}[Proof of Proposition~\ref{prop:FBzFBwrelatedbyPsiPhi}]Let us consider the case that $B$ is an $\alpha$-band; the case that $B$ is a $\beta$-band is analogous. We will show the first equivalence. The second follows from symmetry. Let $(\Sigma, \ve{\alpha}',\ve{\alpha},\ve{\beta},\ve{w},\ve{z})$ be a triple subordinate to the band $B$.

Note that by Lemma~\ref{lem:PsiPhirelationbetweenThetaw,z} we have that
\[\Psi_{z_i}(\Theta^{\ws})=\Theta^{\zs}\qquad \text{and} \qquad \Phi_{w_i}(\Theta^{\zs})=\Theta^{\ws},\] as elements of $\cHFL^-(\Sigma,\as',\as,\ws,\zs,\sigma,\frs_0)$. 

 Gromov compactness yields that
\[F_{\as',\as,\bs}(\d_{\as',\as}\otimes \id+\id\otimes \d_{\as,\bs})+\d_{\as',\bs} F_{\as',\as,\bs}=0,\] on the uncolored complexes. We view this as a matrix equation, where $F_{\as',\as,\bs}$ is a matrix with columns corresponding to the elements of $(\bT_{\as'}\cap \bT_{\as})\times (\bT_{\as}\cap \bT_{\bs} )$ and rows corresponding to the elements of $\bT_{\as'}\cap \bT_{\bs}$. Differentiating with respect to the variable $V_{z_i}$, for either $i=1$ or $i=2$, we obtain the relation
\begin{equation}
F_{\as',\as,\bs}(\Psi_{z_i}\otimes \id+\id\otimes \Psi_{z_i})+\Psi_{z_i} F_{\as',\as,\bs}+H(\d_{\as',\as}\otimes \id+\id\otimes \d_{\as,\bs})+\d_{\as',\bs}H=0,\label{eq:gromovcompactnessrelbetweenpsi}
\end{equation} where $H$ is obtained by differentiating the matrix for $F_{\as',\as,\bs}$ with respect to $V_{z_i}$. More concretely, $H$ can be thought of as the map which counts holomorphic triangles with the additional factor of $n_{z_i}(\psi)$ for each triangle, and an additional overall factor of $V_{z_i}^{-1}$.

Note that the relation from Equation~\eqref{eq:gromovcompactnessrelbetweenpsi} persists once we consider the complexes which are colored by $\sigma$. We now apply the relation from Equation~\eqref{eq:gromovcompactnessrelbetweenpsi} to an element of the form $\Theta^{\ve{w}}\otimes \ve{x}$. Noting that $\Psi_{z_i}(\Theta^{\ws})=\Theta^{\zs}$, we obtain
\[
F_{\as',\as,\bs}(\Theta^{\zs}\otimes \ve{x})+F_{\as',\as,\bs}(\Theta^{\ws}\otimes \Psi_{z_i}(\ve{x}))+\Psi_{z_i} F_{\as',\as,\bs}(\Theta^{\zs}\otimes \ve{x})=H'(\d_{\as,\bs}(\ve{x}))+\d_{\as',\bs}(H'(\ve{x})),
\]
 where $H':=H(\Theta^{\ws}\otimes \id)$. Using the definitions of the band maps, we obtain the relation
\[
F_B^{\ws}\simeq \Psi_{z_i}F_B^{\zs}+F_B^{\zs}\Psi_{z_i}.
\]

The relations $F_B^{\zs}\simeq \Phi_{w_i}F_B^{\ws}+F_B^{\ws} \Phi_{w_i}$ are proven analogously.
\end{proof}

\begin{figure}[ht!]
\centering
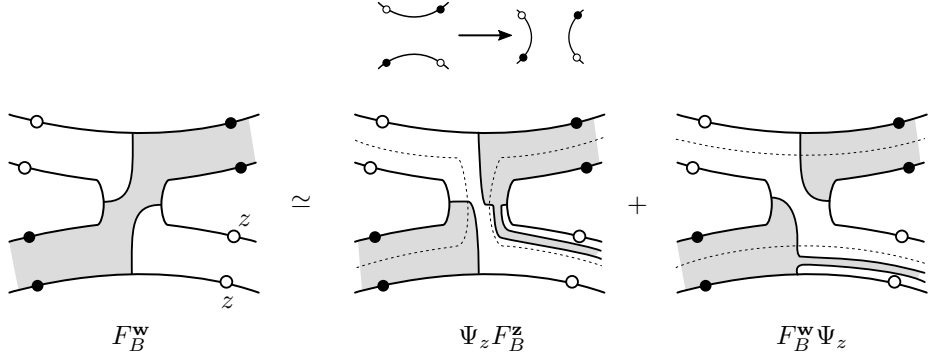
\caption{\textbf{Dividing sets for the relation from Proposition~\ref{prop:FBzFBwrelatedbyPsiPhi}.} The underlying link cobordism is a saddle. The dashed lines correspond to decomposing level sets of the surface, which induces the composition of maps shown. The three dividing sets form a bypass triple.\label{fig::105}}
\end{figure}

\subsection{Band maps and quasi-stabilization} We now consider the  commutator between a band map and a quasi-stabilization map. One approach would be to simply try to perform a holomorphic degeneration argument, as in Proposition~\ref{prop:singlealphaquasistabtriangle}. Unfortunately the result stated therein is not sufficient to show that the band maps commute with the quasi-stabilization maps in all of the cases that we are interested in. It is possible to compute the commutator of a quasi-stabilization and a band map by adapting \cite{ZemQuasi}*{Theorem~6.5}, however we will instead give a simpler, though somewhat less direct argument.

\begin{prop}\label{prop:zbandsandquasistab}Suppose that $B$ is a band for the link $\bL=(L,\ve{w},\ve{z})$, and $w$ and $z$ are two new basepoints which are contained in the same component of $L\setminus (\ve{w}\cup \ve{z}\cup B)$. Then, regardless of which component of $L\setminus (\ws\cup \zs\cup B)$ contains $w$ and $z$, one has 
\[F_B^{\ve{z}} S_{w,z}^{\circ}\simeq S_{w,z}^{\circ} F_B^{\ve{z}} \qquad \text{if $w$ follows $z$, and }\]
\[F_B^{\ve{z}} S_{z,w}^{\circ}\simeq S_{z,w}^{\circ} F_B^{\ve{z}}\qquad \text{if $z$ follows $w$},\] for $\circ\in \{+,-\}$.

If $w$ and $z$ do not separate an end of $B$ from either of the two $\ve{z}$-basepoints adjacent to $B$, then
\[F_B^{\ve{z}} T_{w,z}^{\circ}\simeq T_{w,z}^{\circ} F_B^{\ve{z}},\qquad \text{if $w$ follows $z$, and}\] 
\[F_B^{\ve{z}} T_{z,w}^{\circ} \simeq T_{z,w}^{\circ}F_B^{\ve{z}},\qquad \text{if $z$ follows $w$.}\]
\end{prop}
\begin{proof}We will only consider the case that $B$ is a $\beta$-band, and $\circ=+$, for notational simplicity. Let us call a component $C$ of $L\setminus (\ws\cup \zs)$ an \emph{$\alpha$-subarc} if $C\subset U_{\as}$ (or equivalently if the $\zs$ basepoint in $\d C$ follows the $\ws$ basepoint, with respect to the links orientation). Similarly, we will call a component $C\subset L\setminus (\ws\cup \zs)$ a \emph{$\beta$-subarc} if $C\subset U_{\bs}$.

There are two cases we need to consider:
\begin{enumerate}[leftmargin=25mm, label = \text{Case (\arabic*):}, ref=\arabic*]
\item\label{case:bandquasi2} $w$ and $z$ are in a component $C$ of $L\setminus (\ws\cup \zs)$ that is disjoint from $B$.
\item\label{case:bandquasi1} $w$ and $z$ are in a component $C$ of $L\setminus (\ws\cup \zs)$ which also contains an end of $B$.
\end{enumerate}
 There are two further subcases of Case~\eqref{case:bandquasi2}:
\begin{enumerate}[leftmargin=31.5mm, label = \text{Subcase (1\alph*):} , ref=1\alph*]
\item\label{case:alphasubarc} $C$ is an $\alpha$-subarc of $\bL$,
\item\label{case:betasubarc} $C$ is a $\beta$-subarc of $\bL$.
\end{enumerate}

Let us first consider Subcase~\eqref{case:alphasubarc}. Note that in this case, $w$ immediately follows $z$. On the unstabilized complexes, the map $F_{B}^{\zs}$ is defined by picking a Heegaard triple $\cT=(\Sigma,\as,\bs,\bs',\ws,\zs)$ subordinate to the $\beta$-band $B$. We note that a Heegaard triple $\cT^+$ subordinate to $B$ for the stabilized link can be constructed by quasi-stabilizing $\cT$ along a single curve $\alpha_s\subset \Sigma\setminus \as$, and then adding a curve $\beta_0$ to $\bs$ and a curve $\beta_0'$ to $\bs'$, such that $\beta_0$ and $\beta_0'$ are contained in a small ball on $\Sigma$, intersect each other twice, and intersect none of the other $\as$ curves. This is the configuration shown in Figure~\ref{fig::16quasi}. By stretching the almost complex structure in a circle bounding $\beta_0$ and $\beta_0'$, we can compute the triangle map $F_{\cT^+}$ in terms of the triangle map $F_{\cT}$, using Proposition~\ref{prop:singlealphaquasistabtriangle}. The holomorphic triangle counts from Proposition~\ref{prop:singlealphaquasistabtriangle} immediately imply that
\[
F_B^{\zs}S_{w,z}^+\simeq S_{w,z}^+F_{B}^{\zs}\qquad \text{and} \qquad F_B^{\zs} T_{w,z}^+\simeq T_{w,z}^+F_B^{\zs},
\] 
establishing the result in Subcase~\eqref{case:alphasubarc}.

We now analyze Subcase~\eqref{case:betasubarc}, corresponding to when the pair $(z,w)$ lies in a $\beta$-arc $C\subset L\setminus (\ws\cup \zs)$ and $z$ follows $w$. Using the basepoint moving relations from Lemmas~\ref{lem:mapsforI+agree}~and~\ref{lem:mapsforII+agree}, we will reduce the argument to Subcase~\eqref{case:alphasubarc}. Since the ends of $B$ are in $\beta$-subarcs of $\bL$, it follows that the $\alpha$-subarcs $C_1$ and $C_2$ of $\bL$ adjacent to $C$ are disjoint from $B$ (note that $C_1$ may be equal to $C_2$). Let $w'$ and $z'$ be the end points of $C$. Write $\zs_0$ for $\zs\setminus \{z'\}$ and let
 \[
 \tau^{z\from z'}\colon (L,\ws,\zs_0\cup \{z'\})\to (L,\ws,\zs_0\cup \{z\})
 \] 
 be the diffeomorphism which moves $z'$ to $z$. By Lemma~\ref{lem:mapsforI+agree} we know that
\[S_{z,w}^+\simeq S_{w,z'}^+\tau^{z\from z'}_*.\] Hence
\begin{equation}
\begin{alignedat}{3}
F_{B}^{\zs} S_{z,w}^+&\simeq F_B^{\zs} S_{w,z'}^+\tau^{ z\from z'}_*&&\qquad(\text{Lemma~\ref{lem:mapsforI+agree}})\\
&\simeq S_{w,z'}^+F_{B}^{\zs} \tau^{z\from z'}_*&&\qquad (\text{Subcase~\eqref{case:alphasubarc}})\\
&\simeq S_{w,z'}^+ \tau^{z\from z'}_*F_B^{\zs}&&\qquad(\text{Diffeomorphism invariance of $F_B^{\zs}$})\\
&\simeq S_{z,w}^+F_B^{\zs}&&\qquad(\text{Lemma~\ref{lem:mapsforI+agree}}).
\end{alignedat}
\label{eq:Fb,Scommutation}
\end{equation}
An entirely analogous argument, using Lemma~\ref{lem:mapsforII+agree}, establishes the relation $T_{z,w}^+F_B^{\zs}\simeq F_B^{\zs} T_{z,w}^+$.

It remains to consider Case~\eqref{case:bandquasi1}. The strategy is again to apply Lemmas~\ref{lem:mapsforI+agree}~and~\ref{lem:mapsforII+agree} to reduce to Case~\eqref{case:bandquasi2}, when possible. Write $w_1,$ $w_2,$ $z_1,$ and $z_2$ for the four basepoints adjacent to the ends of $B$. Note that in this case the pair $(z,w)$ is in a $\beta$-subarc of $\bL$, and $z$ follows $w$.

As a first step towards proving the claim in Case~\eqref{case:bandquasi1}, we will show that if $(z,w)$ are between an end of $B$ and either $z_1$ or $z_2$, then
\[F_B^{\zs} S_{z,w}^+\simeq S_{z,w}^+ F_{B}^{\zs}.\] As before, we use Lemma~\ref{lem:mapsforI+agree} to write $S_{z,w}^+\simeq S_{w,z_i}^+\tau^{z\from z_i}_*$. Analogously to Equation~\eqref{eq:Fb,Scommutation}, since $F_B^{\zs}$ commutes with both $S_{z,w}^+$ and $\tau^{ z\from z_i}_*$, we conclude that $F_{B}^{\zs}S_{z,w}^+\simeq S_{z,w}^+F_B^{\zs}$.

Noting that we have not yet used the fact that $F_B^{\zs}$ is the $\zs$-band map, it follows by a symmetrical argument that if $(z,w)$ are between an end of $B$ and either $w_1$ or $w_2$, then 
\begin{equation}F_B^{\zs}T_{z,w}^+\simeq T_{z,w}^+ F_{B}^{\zs}\label{eq:TF_B^zcommutation}\end{equation}

There is still one configuration we have left to analyze, which is when $(z,w)$ are between an end of $B$ and $w_1$ or $w_2$, and we wish to commute $F_B^{\zs}$ and $S_{z,w}^{+}$. Note that by Equation~\eqref{eq:TF_B^zcommutation}, in this configuration, we have that $F_B^{\zs}T_{z,w}^+\simeq T_{z,w}^+F_B^{\zs},$ where we are giving $w$ and $w_i$ the same color, and $z$, $z_1$ and $z_2$ the same color. From Lemma~\ref{lem:olddefsforTw,z}, we know $S_{z,w}^+\simeq \Phi_w T_{z,w}^+$. By Lemma~\ref{lem:PhiPsicommutatorFB}, we have $F_B^{\zs} \Phi_w\simeq \Phi_w F_B^{\zs}$. Combining these relations, we have
\[
F_B^{\zs}S_{z,w}^+\simeq F_B^{\zs} \Phi_w T_{z,w}^+\simeq \Phi_w F_B^{\zs} T_{z,w}^+\simeq \Phi_w T_{z,w}^+ F_B^{\zs} \simeq S_{z,w}^+ F_B^{\zs},\]
completing the proof.
\end{proof}

By a symmetrical argument, we also have the following:

\begin{prop}\label{prop:wbandsandquasistab}Suppose that $B$ is a band for the link $\bL=(L,\ve{w},\ve{z})$, and $w$ and $z$ are two new basepoints which are contained in the same component of $L\setminus (\ve{w}\cup \ve{z}\cup B)$.  Then, regardless of which component of $L\setminus (\ws\cup \zs\cup B)$ contains $w$ and $z$, one has
\[F_B^{\ve{w}} T_{w,z}^{\circ}\simeq T_{w,z}^{\circ} F_B^{\ve{w}} \qquad \text{if $w$ follows $z$, and }\]
\[F_B^{\ve{w}} T_{z,w}^{\circ}\simeq T_{z,w}^{\circ} F_B^{\ve{w}}\qquad \text{if $z$ follows $w$}\] for $\circ\in \{+,-\}$. If $w$ and $z$ do not separate an end of $B$ from either of the two $\ve{w}$-basepoints adjacent to $B$, then
\[F_B^{\ve{w}} S_{w,z}^{\circ}\simeq S_{w,z}^{\circ} F_B^{\ve{w}}\qquad \text{if $w$ follows $z$, and}\] 
\[F_B^{\ve{w}} S_{z,w}^{\circ} \simeq S_{z,w}^{\circ}F_B^{\ve{w}}\qquad \text{if $z$ follows $w$.}\]
\end{prop}

Proposition~\ref{prop:zbandsandquasistab} does not allow us to compute the commutator of $F_{B}^{\zs}$ and $T^\circ_{w,z}$ when $(w,z)$ is a new pair of basepoints which lie between an end of $B$ and one of the adjacent $\zs$-basepoints. In the following lemma, we give an example where $F_B^{\zs}$ and $T_{w,z}^+$ do not commute, and instead form a bypass triple. We will not use this result elsewhere in the paper.

\begin{lem}\label{lem:F_B^zandTcommutator} Suppose $B$ is an $\alpha$-band and $(w,z)$ is an adjacent pair of basepoints between an end of  $B$ and a $\ve{z}$-basepoint. Then
\[
F_B^{\ve{z}} T_{w,z}^++T_{w,z}^+ F_{B}^{\ve{z}}+S_{w,z}^+ F_{B}^{\ve{w}}\simeq 0.
\]
\end{lem}

\begin{proof}Let $w'$ be the $\ve{w}$-basepoint adjacent to an end of $B$, which is not adjacent to $z$ until after the band $B$ is added. We  compute
\begin{equation}
\begin{alignedat}{2}F_B^{\ve{z}} T_{w,z}^+ +T_{w,z}^+ F_B^{\ve{z}}&\simeq (\Phi_{w'} F_{B}^{\ve{w}}+F_{B}^{\ve{w}} \Phi_{w'}) T_{w,z}^++T_{w,z}^+ (\Phi_{w'} F_{B}^{\ve{w}}+F_{B}^{\ve{w}} \Phi_{w'})&& (\text{Proposition~\ref{prop:FBzFBwrelatedbyPsiPhi}})\\
&\simeq F_B^{\ve{w}}(\Phi_{w'} T_{w,z}^++T_{w,z}^+ \Phi_{w'})+(\Phi_{w'} T_{w,z}^++T_{w,z}^+ \Phi_{w'})F_B^{\ve{w}}\quad&& (\text{Proposition~\ref{prop:wbandsandquasistab}}).
\end{alignedat}
\label{eq:bypass1}
\end{equation}

Since $w'$ is not adjacent to $z$ until we add $B$, it follows that
\begin{equation}
 F_B^{\ve{w}}(\Phi_{w'} T_{w,z}^++T_{w,z}^+ \Phi_{w'})\simeq 0
 \label{eq:bypass2}
 \end{equation} by Lemma~\ref{lem:PhiPsiandS^pmcommutator}. However once we add the band $B$, the basepoints $z$ and $w'$ become adjacent, so by Lemma~\ref{lem:PhiPsiandS^pmcommutator}
\begin{equation}
(\Phi_{w'}T_{w,z}^++T_{w,z}^+ \Phi_{w'})F_B^{\ws}\simeq \Phi_w T_{w,z}^+F_B^{\ws}.
\label{eq:bypass3}
\end{equation} Noting that $\Phi_wT_{w,z}^+\simeq S_{w,z}^+$ by Lemma~\ref{lem:olddefsforTw,z}, we conclude that $\Phi_w T_{w,z}^+F_B^{\ws}\simeq S_{w,z}^+F_{B}^{\zs}$. Combining this with Equations~\eqref{eq:bypass1},~\eqref{eq:bypass2}~and~\eqref{eq:bypass3} yields the stated formula.
\end{proof}

  \begin{figure}[ht!]
\centering
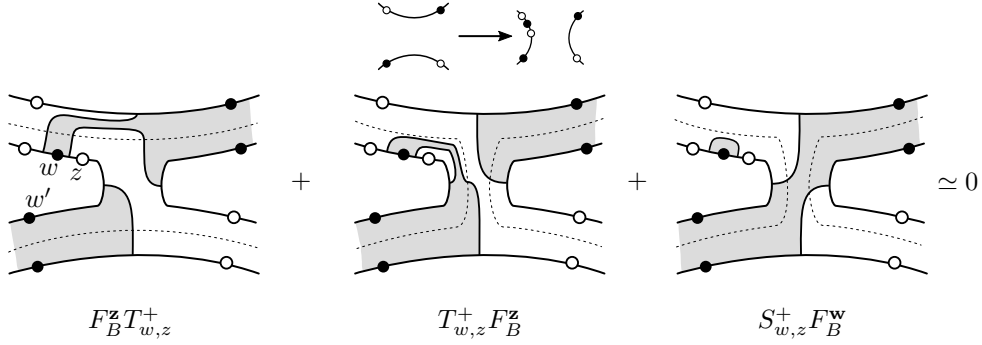
\caption{\textbf{Dividing sets for the maps appearing in Lemma~\ref{lem:F_B^zandTcommutator}}. The dashed lines correspond to decomposing level sets of the surface, which correspond to the composition of maps shown below each figure. The three dividing sets form a bypass triple.\label{fig::106}}
\end{figure}

\subsection{Commutation of band maps I}

In this subsection, we prove that band maps can always be commuted, as long as one is an $\alpha$-band, and the other is a $\beta$-band. We will later consider commuting two band maps when both bands are $\alpha$-bands or both are $\beta$-bands.

\begin{prop}\label{prop:commutealphaandbetabands}Suppose that $B_1^\alpha$ is an $\alpha$-band and $B^\beta_2$ is a $\beta$-band such that $\bL(B^\alpha_1),$ $\bL(B^\beta_2)$ and $\bL(B^{\alpha}_1,B^\beta_2)$ all have at least two basepoints on each component. Then, for any choice of $\ve{o}_1,\ve{o}_2\in \{\ve{w},\ve{z}\}$, we have
\[F_{B^\alpha_1}^{\ve{o}_1} F_{B^\beta_2}^{\ve{o}_2}\simeq F^{\ve{o}_2}_{B^\beta_2} F_{B^\alpha_1}^{\ve{o}_1}.\]
\end{prop}
\begin{proof}This is a simple application of the associativity relations for the holomorphic triangle maps. Let $(\Sigma,\as',\as,\bs,\bs',\ws,\zs)$ be a Heegaard quadruple such that the following hold:
\begin{enumerate}
\item $(\Sigma,\as,\bs,\ws,\zs)$ is a diagram for $(Y,\bL)$.
\item $(\Sigma,\as,\bs,\bs',\ws,\zs)$ and $(\Sigma,\as',\bs,\bs',\ws,\zs)$ are subordinate to the band $B^{\beta}_2$, attached to $\bL$ or $\bL(B^\alpha_1)$, respectively
\item $(\Sigma,\as',\as,\bs,\ws,\zs)$ and $(\Sigma,\as',\as,\bs',\ws,\zs)$  are subordinate to the band $B^{\alpha}_1$, attached to $\bL$ or $\bL(B^{\beta}_2)$, respectively.
\end{enumerate}
Let 
\[\Theta_{\as',\as}^{\ve{o}_1}\in \cHFL^-(\Sigma,\as',\as,\ws,\zs,\sigma,\frs_0)\qquad  \text{and} \qquad \Theta_{\bs',\bs}^{\ve{o}_1}\in \cHFL^-(\Sigma,\bs,\bs',\ws,\zs,\sigma,\frs_0)
\]
be the distinguished elements from Lemma~\ref{lem:topdegreeunlink2}. Using the associativity relations of the holomorphic triangle maps, we obtain that
\[F_{\as',\bs,\bs'}(F_{\as',\as,\bs}(\Theta_{\as',\as}^{\ve{o}_1},\ve{x}),\Theta_{\bs,\bs'}^{\ve{o}_2})+F_{\as',\as,\bs'}(\Theta_{\as',\as}^{\ve{o}_1}, F_{\as,\bs,\bs'}(\ve{x},\Theta_{\bs,\bs'}^{\ve{o}_2}))\]\[=\d_{\as',\bs'} (H_{\as',\as,\bs,\bs'}(\Theta_{\as',\as}^{\ve{o}_1},\ve{x},\Theta_{\bs,\bs'}^{\ve{o}_2}))+H_{\as',\as,\bs,\bs'}(\Theta_{\as',\as}^{\ve{o}_1},\d_{\as,\bs}(\ve{x}),\Theta_{\bs,\bs'}^{\ve{o}_2}).\] Using the definition of the band maps, we obtain the relation
\[F_{B^\alpha_1}^{\ve{o}_1} F_{B^\beta_2}^{\ve{o}_2}\simeq F_{B^\beta_2}^{\ve{o}_2} F_{B_1^\alpha}^{\ve{o}_1},\] concluding the proof.
\end{proof}

\subsection{Band maps and basepoint moving maps}

\label{sec:bandmapsandbasepointmovingmaps}

We now consider the relation between $\alpha$-bands and $\beta$-bands. Recall that we call a band an $\alpha$-band if the two ends of the band are in components of $L\setminus (\ws\cup \zs)$ which go from $\ws$ to $\zs$. Similarly, we say a band is a $\beta$-band if the two ends of the band lie in components of $L\setminus (\ws\cup \zs)$ which go from $\zs$ to $\ws$.

We claimed earlier that the distinction between $\alpha$- and $\beta$-bands was not especially important for the TQFT structure, in terms of which decorated link cobordism the maps represented. Instead, we claimed that the more important distinction is between the $\ws$- and $\zs$-versions of the band maps. In this section, we make that claim precise.

We  first consider the  type-$\zs$ band maps. Suppose that $B^{\alpha}$ is an $\alpha$-band. Let $z_1$ and $z_2$ be the $\zs$-basepoints adjacent to the ends of $B^{\alpha}$. Let $B^\beta$ be the $\beta$-band formed by isotoping the ends of $B^{\alpha}$ across $z_1$ and $z_2$. There is a diffeomorphism 
\[
\phi\colon (Y,\bL(B^\beta))\to (Y,\bL(B^\alpha))
\]
 which is fixed outside of a neighborhood of $B^{\alpha}\cup C_1\cup C_2$, where $C_i$ denotes the subarc of $L$ which goes between an end of $B^{\alpha}$ and $z_i$ and intersects none of the other basepoints. In particular, $\phi$ fixes all of the basepoints except $z_1$ and $z_2$, and switches $z_1$ and $z_2$.  The diffeomorphism $\phi$ is illustrated in Figure~\ref{fig::10}.

\begin{figure}[ht!]
\centering
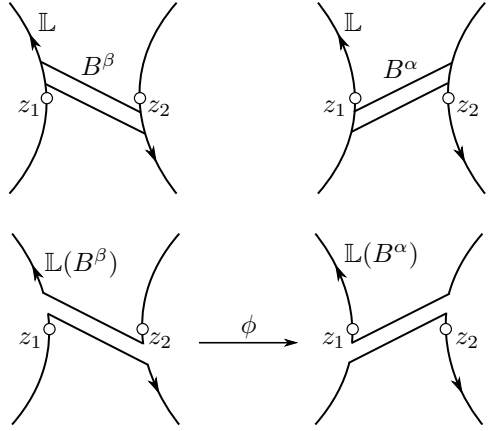
\caption{\textbf{The diffeomorphism $\phi$ which we use to related the type-$\zs$ band maps for $B^\alpha$ and $B^{\beta}.$} The diffeomorphism $\phi\colon (Y,\bL(B^\beta))\to (Y, \bL(B^\alpha))$ maps $z_1$ to $z_2$ and maps $z_2$ to $z_1$.\label{fig::10}}
\end{figure}

The diffeomorphism $\phi$ is in fact determined up to isotopy. To see this, and to simplify some of the arguments in this section, we prove the following elementary fact:
\begin{lem}\label{lem:MCGtrivialtangle}Let $T$ be a tangle in a 3-ball $\cB$ consisting of two strands, $S_1$ and $S_2$, such that there are two disjoint 2-cells $D_1, D_2\subset \cB$ such that $\d (D_i)=S_i\cup a_i$, where $a_i$ is an embedded arc in $\d \cB$. If $\Diff((\cB,T), \d \cB)$ denotes the set of self-diffeomorphisms of $(\cB,T)$ which are the identity on $\d \cB$, then
\[
\pi_0( \Diff((\cB,T), \d \cB))=\{1\}.
\]
\end{lem}
\begin{proof}Suppose that $\phi\in \Diff((\cB,T), \d \cB)$. Let $N(T)$ denote a regular neighborhood of $T$. Note that $\phi$ can be isotoped so that it restricts to a self-diffeomorphism of $N(T)$. Let $h_2$ denote the genus 2 handlebody $\cB\setminus \Int N(T)$. Let $A_1$ and $A_2$ denote the two annuli $\cl(\d N(T_i)\setminus\d \cB)$. Define $D_i':=D_i\cap h_2$, which we can assume are also disks. Furthermore, we assume that the disk $D_i'$ intersects the annulus $A_i$ along a single arc, $a_i'$. If we can show that $\phi$ preserves the arcs $a_1'$ and $a_2'$, up to isotopy relative to their endpoints, we will be done, since $\phi$ can then be isotoped to be the identity on $\d h_2$, allowing us to use the well known fact that $\MCG(h_g, \d h_g)=\{1\}$ for a genus $g$ handlebody $h_g$. Define the closed 2-chain  $c_1=a_1'-\phi(a_1')$ in $A_1$. Now $[c_1]=0\in H_1(A_1;\Z)$ if and only if $a_1'$ and $\phi(a_1')$ are isotopic in $A_1$ relative to their endpoints. Let $\iota: A_1\to h_2$ denote the inclusion map. Note that $\iota_*([c_1])=[\d D_1']-[\d \phi(D_1')]=0\in H_1(h_2;\Z)$. However it is easily verified that the map $\iota_*: H_1(A_1;\Z)\to H_1(h_2;\Z)$ is an injection, so we conclude that $[c_1]=0\in H_1(A_1;\Z)$.  It follows that $a_1'$ and $\phi(a_1)'$ are isotopic in $A_1$, relative to their endpoints. Similarly, $a_2'$ and $\phi(a_2')$ are isotopic in $A_2$, relative to their endpoints, completing the proof by our previous argument.
\end{proof}

\begin{prop}\label{prop:alphabandmapsarebetabandmaps}If $B^\alpha$ and  $B^\beta$ are the two bands described above, and $\phi\colon (Y,\bL(B^\beta))\to (Y,\bL(B^\alpha))$ is the diffeomorphism described above, then
\[ F_{B^\alpha}^{\ve{z}}\simeq \phi_* F_{B^\beta}^{\ve{z}}.\]
\end{prop}

\begin{proof} The overall idea of the proof is to consider a composition of three cobordisms: a birth cobordism, an $\alpha$-band attachment, and a $\beta$-band attachment. The birth cobordism can be used to cancel either the $\alpha$-band, or the $\beta$-band, leaving just one band attachment, of either type.

 Consider the configuration shown on the top of Figure~\ref{fig::108}. We start with a birth cobordism, which adds the doubly based unknot $(U,w,z)$, where $U=\d D$. There is also an $\alpha$-band $B_0^\alpha$ and a $\beta$-band $B_0^\beta$.  We assume that topologically the union of $B_0^\alpha$, $D$ and $B_0^{\beta}$ is isotopic to the band $B^\alpha$ (though not via an isotopy fixing the basepoints).
  
  Two diffeomorphisms $\phi^1$ and $\phi^2$ are also shown in Figure~\ref{fig::108}, which are the identity outside a neighborhood of the arcs and bands shown.  

\begin{figure}[ht!]
\centering
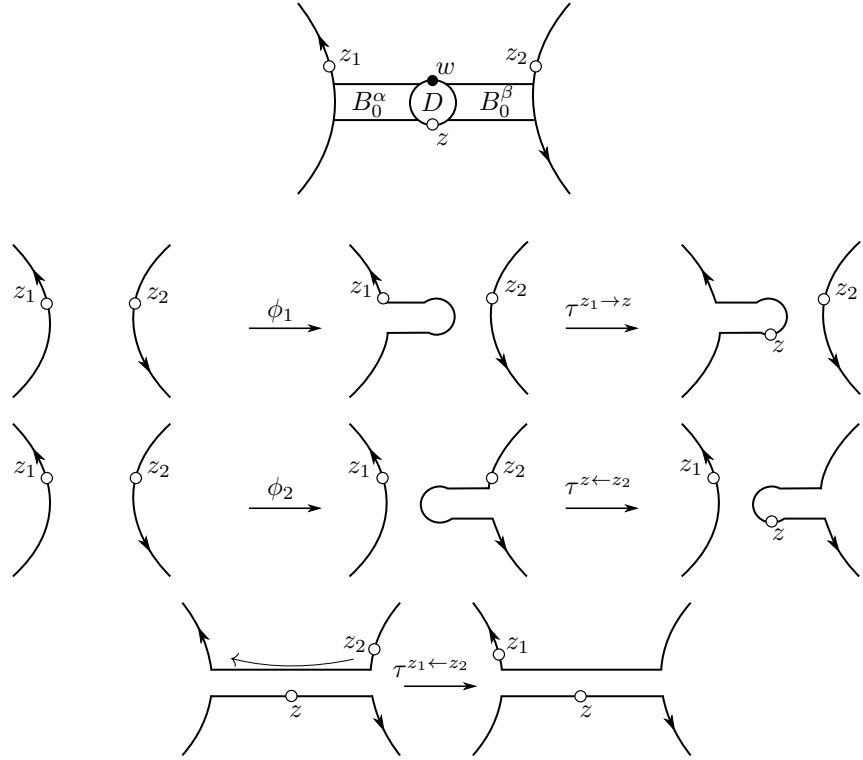
\caption{\textbf{Decomposing a band into a birth cobordism, followed by an $\alpha$-band $B_0^\alpha$ and a $\beta$-band $B_0^\beta$.} The bottom three rows show diffeomorphisms $\phi^1$ and $\phi^2$, which fix all of the basepoints but change the link, as well as basepoint moving maps $\tau^{z_1\to z},$ $\tau^{z\from z_2}$ and $\tau^{z_1\from z_2}$ which fix the link, setwise, but move the basepoints. The orientation of the link is shown.\label{fig::108}}
\end{figure}

First note that 
\begin{equation}
F_{B_0^\beta}^{\zs}\cB^+_{(U,w,z), D}\simeq S_{z,w}^+(\phi_2)_*,\label{eq:basepointmovingmapband12}
\end{equation} by Proposition~\ref{prop:bandsanddiskstab}. We post-compose Equation~\eqref{eq:basepointmovingmapband12} with $T_{w,z_2}^-$ to obtain the relation
\begin{equation}
T_{w,z_2}^-F_{B_0^\beta}^{\zs}\cB^+_{(U,w,z), D}\simeq T_{w,z_2}^-S_{z,w}^+\phi^2_*.
\label{eq:basepointmovingmapband10}
\end{equation}
 By Lemma~\ref{lem:movingbasepointsstep1}, $T_{w,z_2}^-S_{z,w}^+$ is chain homotopic to the basepoint moving map $\tau^{z\from z_2}_*$.  Hence Equation~\eqref{eq:basepointmovingmapband10} becomes
\begin{equation}
T_{w,z_2}^-F_{B_0^\beta}^{\zs}\cB^+_{(U,w,z), D}\simeq \tau^{z\from z_2}_*\phi^2_*. \label{eq:basepointmovingmapband11}
\end{equation}
We post-compose Equation~\eqref{eq:basepointmovingmapband11} with $F_{B_0^\alpha}^{\zs}$ to obtain the relation
\begin{equation}
F_{B_0^\alpha}^{\zs}T_{w,z_2}^-F_{B_0^\beta}^{\zs}\cB^+_{(U,w,z), D}\simeq F_{B_0^\alpha}^{\zs} \tau^{z\from z_2}_*\phi^2_*.
\label{eq:basepointmovingmapband1}
\end{equation}
We now perform the following manipulation:
\begin{equation}
\begin{alignedat}{3}
F_{B_0^\alpha}^{\zs}T_{w,z_2}^-F_{B_0^\beta}^{\zs}\cB^+_{(U,w,z), D}
&\simeq T_{w,z_2}^-F_{B_0^\alpha}^{\zs}F_{B_0^\beta}^{\zs}\cB^+_{(U,w,z), D}&&\qquad\text{(Proposition~\ref{prop:zbandsandquasistab})}\\
&\simeq T_{w,z_2}^-F_{B_0^\beta}^{\zs}F_{B_0^\alpha}^{\zs}\cB^+_{(U,w,z), D} &&  \qquad \text{(Proposition~\ref{prop:commutealphaandbetabands})}\\
&\simeq T_{w,z_2}^-F_{B_0^\beta}^{\zs}S_{w,z}^+ \phi^1_*  &&\qquad \text{(Proposition~\ref{prop:bandsanddiskstab})}\\
&\simeq T_{w,z_2}^-F_{B_0^\beta}^{\zs}S_{z_1,w}^+\tau^{z_1\to z}_* \phi^1_* &&\qquad \text{(Lemma~\ref{lem:mapsforI+agree})}\\
&\simeq T_{w,z_2}^-S_{z_1,w}^+ F_{B_0^\beta}^{\zs}\tau^{z_1\to z}_* \phi^1_* &&\qquad (\text{Proposition~\ref{prop:zbandsandquasistab}})\\
&\simeq \tau^{z_1\from z_2}_* F_{B_0^\beta}^{\zs}\tau^{z_1\to z}_* \phi^1_* && \qquad \text{(Lemma~\ref{lem:movingbasepointsstep1})}
\end{alignedat}
\label{eq:basepointmovingmapband2}
\end{equation}

Combining Equations \eqref{eq:basepointmovingmapband1} and \eqref{eq:basepointmovingmapband2}, we obtain the relation
\begin{equation}
F^{\zs}_{B_0^\alpha}\tau^{z\from z_2}_*\phi^2_*\simeq \tau_*^{z_1\from z_2} F_{B_0^\beta}^{\zs} \tau_*^{z_1\to z} \phi^1_*.
\label{eq:basepointmovingmapband3}
\end{equation}
Finally, we construct a diffeomorphism $\psi$ of $Y$ which moves the link $(L\cup U)(B_0^\alpha,B_0^{\beta})$ into the position of $L(B^\alpha)$, fixes $z_1$ and moves $z$ to the position of $z_2$. This is shown schematically in Figure~\ref{fig::109}. Post-composing Equation~\eqref{eq:basepointmovingmapband3} with $\psi_*$, we have
\begin{equation}
\psi_*F^{\zs}_{B_0^\alpha}\tau^{z\from z_2}_*\phi^2_*\simeq \psi_*\tau_*^{z_1\from z_2} F_{B_0^\beta}^{\zs} \tau_*^{z_1\to z} \phi^1_*.
\end{equation}
Using diffeomorphism invariance of the band maps, it is straightforward to rearrange the diffeomorphisms (Lemma~\ref{lem:MCGtrivialtangle} can be used to simplify things) to see that 
\[
\psi_*F^{\zs}_{B_0^\alpha}\tau^{z\from z_2}_*\phi^2_*\simeq F_{B^\alpha}^{\zs}\qquad \text{and} \qquad \psi_*\tau_*^{z_1\from z_2} F_{B_0^\beta}^{\zs} \tau_*^{z_1\to z} \phi^1_*\simeq \phi_* F_{B^\beta}^{\zs},
\] completing the proof.
\end{proof}

\begin{figure}[ht!]
\centering
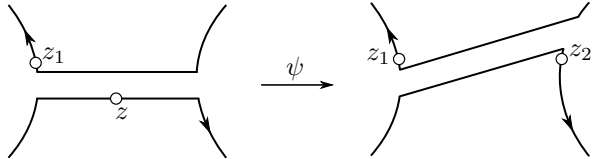
\caption{\textbf{A schematic of the diffeomorphism $\psi$.} The map $\psi$ sends $(L\cup U)(B_0^\alpha,B_0^\beta)$ to $L(B^\alpha)$, fixes $z_1$ and maps $z$ to $z_2$.\label{fig::109}}
\end{figure}

\begin{rem} Proposition~\ref{prop:alphabandmapsarebetabandmaps} can be interpreted similarly to \cite{OSTriangles}*{Lemma~5.2}, which says that one can compute maps induced by surgery on framed links using either left subordinate or right subordinate Heegaard triples. We sketched a proof of this result for link cobordism maps in Lemma~\ref{lem:alphasurgerymaps=betasurgerymaps}.  The argument uses the fact that the 3-handle maps cancel both the left subordinate, and right subordinate maps, for 0-surgery on an unlink. Our proof of Proposition~\ref{prop:alphabandmapsarebetabandmaps} follows from the analogous observation that the birth cobordism map $\cB_{(U,w,z),D}^+$ can be canceled by either an $\alpha$-band map, or a $\beta$-band map.
\end{rem}

The analogous statement also holds for the $\ve{w}$-band maps. Suppose that $B^\alpha$ is an $\alpha$-band and $B^\beta$ is a $\beta$-band resulting from sliding the ends of $B^\alpha$ across the two adjacent $\ve{w}$-basepoints,  $w_1$ and $w_2$. There is a well-defined diffeomorphism $\phi'\colon (Y,\bL(B^\beta))\to (Y,\bL(B^\alpha))$ which moves $w_1$ to $w_2$ along the link (similar to Figure~\ref{fig::10}, but with basepoints $w_1$ and $w_2$ replacing the $z_1$ and $z_2$ which appear in that picture.) We have the following:

\begin{prop}\label{prop:alphabandmapsarebetabandmapstypew}If $B^\alpha$ and $B^\beta$ are the two bands described above, and $\phi'\colon (Y,\bL(B^\beta))\to (Y,\bL(B^\alpha))$ is the diffeomorphism described above, then
\[
F_{B^\alpha}^{\ve{w}}=\phi_*'F_{B^\beta}^{\ve{w}}.
\]
\end{prop}

Proposition~\ref{prop:alphabandmapsarebetabandmapstypew} can be proven by rewriting the proof of Proposition~\ref{prop:alphabandmapsarebetabandmaps}, switching the $S$'s and $T$'s and switching the $\ve{w}$'s and $\ve{z}$'s.

\subsection{Commutation of band maps II}

We now revisit the question of commuting band maps. Our previous result, Proposition~\ref{prop:commutealphaandbetabands}, only addressed the case that one band was an $\alpha$-band, and the other was a $\beta$-band. Equipped with Proposition~\ref{prop:alphabandmapsarebetabandmaps}, we can now address the case that both bands are $\alpha$-bands, or both bands are $\beta$-bands:

\begin{prop}\label{prop:alphabandscommute}Suppose that $B_1$ and $B_2$ are two $\alpha$-bands attached to the link $\bL=(L,\ve{w},\ve{z})$ such that $\bL(B_1),$ $ \bL(B_2)$ and $\bL(B_1,B_2)$ all have at least two basepoints on each link component. Then 
\[F_{B_1}^{\ve{z}}F_{B_2}^{\ve{z}}\simeq F_{B_2}^{\ve{z}} F_{B_1}^{\ve{z}} \qquad \text{and} \qquad F_{B_1}^{\ve{w}}F_{B_2}^{\ve{w}}\simeq F_{B_2}^{\ve{w}} F_{B_1}^{\ve{w}}.\] Suppose further that there exist points $\ve{d}\subset L\setminus(\ve{w}\cup \ve{z})$ such that each component of $L\setminus (\ws\cup \zs)$ contains exactly one point of $\ve{d}$,  $B_2$ has its ends in regions of $L\setminus \ve{d}$ containing $\ve{w}$-basepoints, and $B_1$ has its ends in regions of $L\setminus \ve{d}$ containing $\ve{z}$-basepoints. Then
\[F^{\ve{z}}_{B_1} F_{B_2}^{\ve{w}}\simeq F_{B_2}^{\ve{w}} F_{B_1}^{\ve{z}}.\]
\end{prop}

\begin{proof} We will prove the claim when both bands are designated as type-$\zs$, or when one is type-$\ws$ and one is type-$\zs$. The claim when both are type-$\ws$ is a simple modification of the case that they are both type-$\zs$, so we will omit the proof in that case.

In the case that we wish to designate $B_1$ as type-$\zs$ and $B_2$ as type-$\ws$, the existence of the points $\ve{d}\subset L\setminus (\ws\cup \zs)$ ensures that if we write $C_1$ and $C_2$ for the two components of $L\setminus (\ws\cup \zs\cup B_1)$ whose closures intersect both $B_1$ and $\zs$, then $C_1$ and $C_2$ are disjoint from $B_2$.

In the case that we designate both bands as type-$\zs$, then after possibly reordering $B_1$ and $B_2$, we may assume that the two components, $C_1$ and $C_2$, of $L\setminus (\ws \cup \zs\cup B_1)$ whose boundaries contain a $\zs$-basepoint and intersect $B_1$ are disjoint from $B_2$.  The fact that one of $B_1$ and $B_2$ satisfies this condition follows from the fact that the bands are oriented and all components of $\bL(B_1,B_2)$ have at least two basepoints.

In both cases, if we write  $C_1$ and $C_2$ for the component sof $L\setminus (\ws\cup \zs\cup B_1)$ whose boundaries intersect $B_1$ and $\zs$, we have that $C_1$ and $C_2$ are disjoint from $B_2$. We will combine the proofs of both cases of interest, and show that under the assumption that such arcs $C_1$ and $C_2$ exist, we have
\[F_{B_1}^{\zs} F_{B_2}^{\os}\simeq F_{B_2}^{\os} F_{B_1}^{\zs},\] for either $\os\in \{\ws,\zs\}$, 

Let $B_1'$ be the $\beta$-band obtained by isotoping the ends of $B_1$ along the arcs $a_1$ and $a_2$ so that they cross the two $\zs$-basepoints adjacent to $B_1$. Let 
\[
\phi\colon (Y,L(B_1))\to (Y,L(B_1'))
\] be the specialization of the diffeomorphism pictured in Figure~\ref{fig::10} to our current situation. Note that $\phi$ fixes the band $B_2$, so the diffeomorphism $\phi$ induces a diffeomorphism of pairs 
\[
\phi^{B_2}\colon (Y,L(B_1,B_2))\to (Y,L(B_1',B_2)).
\] Note that as diffeomorphisms of $Y$, $\phi$ and $\phi^{B_2}$ are equal. We perform the following manipulation:
\begin{align*}F^{\ve{z}}_{B_1}F^{\ve{o}}_{B_2}&\simeq \phi_*^{B_2} F^{\ve{z}}_{B_1'} F^{\ve{o}}_{B_2}&&(\text{Proposition~\ref{prop:alphabandmapsarebetabandmaps}})\\ 
&\simeq \phi_*^{B_2}  F^{\ve{o}}_{B_2}F^{\ve{z}}_{B_1'}&&(\text{Proposition~\ref{prop:commutealphaandbetabands}})\\
&\simeq F^{\ve{o}}_{B_2}\phi_* F^{\ve{z}}_{B_1'}&&(\phi\text{ fixes } B_2)\\
&\simeq F_{B_2}^{\ve{o}} F^{\ve{z}}_{B_1}&&(\text{Proposition~\ref{prop:alphabandmapsarebetabandmaps}}).
\end{align*} The proof is complete.
\end{proof}

\section{Parametrized Kirby decompositions of link cobordisms}
\label{sec:modelcobs-parametrizedcobs}

In this section, we define a way of decomposing a link cobordism into elementary pieces, which we call a parametrized Kirby decomposition, borrowing the terminology of \cite{JCob}.

\subsection{Parametrized link cobordisms}

\begin{define} If $(W,\Sigma)\colon (Y,L)\to (Y',L')$ is an undecorated link cobordism, a \emph{parametrization} of $(W,\Sigma)$, is a collection of pairwise disjoint framed spheres $\bS$ in $Y$, together with a diffeomorphism of link cobordisms
\[
\Phi\colon \cW(Y,L,\bS)\to (W,\Sigma).
\]
such that $\Phi|_{\{0\}\times Y}=\id_Y$. Furthermore, we assume $\bS$ is either a collection of framed spheres in $Y\setminus L$, or is a single framed $0$-sphere along $L$. In the above expression, $\cW(Y,L,\bS)$ denotes the trace link cobordism obtained by attaching 4-dimensional handles along $\bS$ to $Y\setminus L$  (Section~\ref{subsec:modelhandlecobsawayfromL}), or a attaching a 4-dimensional 1-handle containing a  2-dimensional  1-handle if $\bS$ is a framed  0-sphere along $L$ (Section~\ref{sec:modelcompounthandlebandcobordisms}).
\end{define}

The following basic fact will be useful:

\begin{rem}\label{rem:canonicaldiffeomorphism} Suppose $\bS$ and $\bS'$ are disjoint collections of framed spheres in $Y$, and assume for notational convenience that $\bS$ and $\bS'$ are in the complement of $L$. Then the link cobordisms $\cW(Y,L,\bS\cup \bS')$, $\cW(Y(\bS),L,\bS')\circ \cW(Y,L,\bS),$ and $\cW(Y(\bS'),L,\bS)\circ \cW(Y,L,\bS')$ are all diffeomorphic. The diffeomorphism is canonically specified up to isotopy fixing $\{0\}\times Y$, and preserving $\Sigma$. To see this, it is sufficient to specify the diffeomorphism from the $\cW(Y,L,\bS\cup \bS')$ to $\cW(Y(\bS),L,\bS')\circ \cW(Y,L,\bS)$. Notice that we can think of the  second as being equal to the first with an extra copy of $[0,1]\times Y(\bS)$ inserted. A choice of bicollar neighborhood of $ \{0\}\times Y(\bS)$ allows us to specify such a diffeomorphism, up to isotopy. Since any two bicollar neighborhoods are isotopic, a diffeomorphism between the two cobordisms is specified uniquely up to isotopy fixing $\{0\}\times Y$ and preserving $\Sigma$.
\end{rem}

Now suppose $(W,\Sigma)\colon (Y_1,L_1)\to (Y_2,L_2)$ and $(W',\Sigma')\colon (Y_2,L_2)\to (Y_3,L_3)$ are two link cobordisms with parametrizations
\[
\Phi\colon \cW(Y_1,L_1,\bS_1)\to (W,\Sigma),\qquad \text{ and } \qquad \Phi'\colon \cW(Y_2,L_2,\bS_2)\to(W',\Sigma').
\]
 For notational simplicity, assume that no components of $\bS_1$ or $\bS_2$ are framed 0-spheres along $L_1$ or $L_2$. Let $\bS_2'\subset Y_1(\bS_1)$ denote $\Phi|_{Y_1(\bS_1)}^{-1}(\bS_2)$. The diffeomorphism $\Phi|_{Y_1(\bS_2)}\colon Y(\bS_2)\to Y_2$ induces a diffeomorphism $\tilde{\Phi}\colon W(Y_1(\bS_1), \bS_2')\to W(Y_2,\bS_2).$

 We define the \emph{concatenation} of the two parametrizations
\[
\Phi'*\Phi\colon  \cW(Y_1(\bS_1), L_1, \bS_2')\circ \cW(Y_1,L_1,\bS_1)\to (W',\Sigma')\circ (W,\Sigma),
\] 
by setting $\Phi'* \Phi$ to be $\Phi$ on $\cW(Y_1,L_1,\bS_1)$ and $\Phi'\circ \tilde{\Phi}$ on $\cW(Y_1(\bS_1), L_1, \bS_2')$. If $\bS_2'\subset Y_1\setminus \bS_1$, then in light of Remark~\ref{rem:canonicaldiffeomorphism}, the concatenated diffeomorphisms induce a parametrization
\[
\Phi'*\Phi\colon \cW(Y_1,L_1,\bS_1\cup \bS_2')\to (W',\Sigma')\circ (W,\Sigma)
\] 
which is well-defined up to isotopies (of link cobordisms) fixing $Y_1\times \{0\}$. The concatenation of the two parametrizations can also be defined in the case that $\bS_1$ or $\bS_2$ is a framed link along $L_1$ or $L_2$, with only the obvious notational changes to the domain of $\Phi'* \Phi$.

Analogous to \cite{JClassTQFT}, we need a notion of elementary parametrized cobordisms, corresponding to link cobordisms which have a Morse function with at most one critical value of $f|_{W\setminus \Sigma},f|_{\Sigma}$ or $f|_{\cA}$:
 
\begin{define}\label{def:parametrizedelementarycob}A tuple $\cW=(W,\cF^\sigma,\Phi,\bS)$ is called an \emph{elementary, parametrized link cobordism} if $(W,\cF^\sigma)\colon (Y_1,\bL_1)\to (Y_2,\bL_2)$ is a decorated link cobordism with $\d W=-Y_1\sqcup Y_2$, $\cF=(\Sigma, \cA)$ is a surface with divides, and $\bS$ is a (possibly empty) collection of framed spheres (0-, 1-, or 2-spheres away from $L$, or 0-spheres along $L$) in $Y_1$ which are pairwise disjoint. Furthermore, we assume that exactly one of the following holds:

\begin{enumerate}[leftmargin=20mm, ref= 
\textrm{\arabic*}, label = \textrm{($\mathcal{EPC}$-\arabic*)}:]
\item \label{elemencobtype1}  $\bS=\bS_{\varnothing}$ and $\Phi\colon \cW(Y_1,L_1,\bS_{\varnothing})\to (W,\Sigma)$ is a parametrization which is the identity on $Y_1\times \{0\}$ and is defined up to isotopies of link cobordisms which fix $Y_1\times \{0\}$. Furthermore, each arc of  $\cA$ goes from the incoming boundary to the outgoing boundary.

\item \label{elemencobtype3} $\bS=\bS_{\varnothing}$ and $\Phi$ is a parametrization
\[\Phi\colon \cW(Y_1,L_1,\bS_{\varnothing})\to (W,\Sigma),\] which is the identity on $Y_1\times \{0\}$, and is defined up to isotopies fixing $Y_1\times \{0\}$. Furthermore, all divides of $\cA$, except for exactly one arc $A$, go from the incoming boundary to the outgoing boundary. There are four subtypes:
\begin{enumerate}
\item[($\mathcal{EPC}$-$2_{S^+}$):] The arc $A$ has two ends on the outgoing boundary, and the arc bounds a bigon of type-$\ws$.
\item[($\mathcal{EPC}$-$2_{S^-}$):] The arc $A$ has two ends on the incoming boundary,  and the arc bounds a bigon of type-$\ws$.
\item[($\mathcal{EPC}$-$2_{T^+}$):] The arc $A$ has two ends on the outgoing boundary,  and the arc bounds a bigon of type-$\zs$.
\item[($\mathcal{EPC}$-$2_{T^-}$):] The arc $A$ has two ends on the incoming boundary,  and the arc bounds a bigon of type-$\zs$.
\end{enumerate}

\item \label{elemencobtype4}$\bS$ is a framed 0-sphere, embedded along $L_1$, and $\Phi$ is a parametrization
\[\Phi\colon \cW(Y_1,L_1,\bS)\to (W,\Sigma),\] which is the identity on $Y_1\times \{0\}$, and such that the dividing set  $\Phi^{-1}(\cA)\subset \Sigma(L_1,\bS)$ is disjoint from $B\subset \Sigma(L_1,\bS)$. The parametrization is well-defined up to isotopies which fix $Y_1\times \{0\}$ and the point $(0,0)\in B$. Furthermore, all of the arcs in $\cA$ go from the incoming boundary to the outgoing boundary.

The type ($\ve{w}$ or $\ve{z}$) of each region of $\Sigma(L_1,\bS)\setminus \Phi^{-1}(\cA)$ is well-defined. Hence there are two subtypes:

\begin{enumerate}
\item[($\mathcal{EPC}$-$3_{\ws}$):] The region containing the band is a component of $\Sigma_{\ws}$.
\item[($\mathcal{EPC}$-$3_{\zs}$):\hspace{.8mm}] The region containing the band is a component of $\Sigma_{\zs}$.
\end{enumerate}

\item \label{elemencobtype2} $\bS\neq \bS_{\varnothing}$ and $\bS$ is a framed 0-sphere, 1-dimensional link or 2-sphere in $Y_1\setminus L_1$ and $\Phi$ is a parametrization
\[\Phi\colon \cW(Y_1,L_1,\bS)\to (W,\Sigma)\] which is the identity on $Y_1\times \{0\}$, and which is defined up to isotopies of link cobordisms which fix $Y_1\times \{0\}$. Furthermore, each arc of the divides $\cA$ goes from the incoming boundary to the outgoing boundary.
\end{enumerate} 

\end{define}

The dividing sets for the 4-subtypes of ($\mathcal{EPC}$-3) cobordisms are illustrated in Figure~\ref{fig::62}.

\begin{rem}\label{rem:typewelldefined}In Definition~\eqref{def:parametrizedelementarycob}, the diffeomorphism $\Phi$  is of undecorated link cobordisms, and does not necessarily respect the divides beyond what is stated in ($\mathcal{EPC}$-\ref{elemencobtype4}) link cobordisms. Furthermore,  for elementary parametrized cobordisms of type  ($\mathcal{EPC}$-\ref{elemencobtype1}),  ($\mathcal{EPC}$-\ref{elemencobtype3}) and  ($\mathcal{EPC}$-\ref{elemencobtype2}), the parametrization $\Phi$ is only defined up to isotopies which fix $Y_1\times \{0\}$. For elementary parametrized cobordisms of type  ($\mathcal{EPC}$-\ref{elemencobtype4}), it is defined up to isotopies which fix $Y_1\times \{0\}$ and the point $(0,0)\in B$.
\end{rem}

We make the following definition, borrowing the terminology from \cite{JCob} and \cite{JClassTQFT}:

\begin{define}\label{def:parametrizedKirbydecomp}If $(W,\cF)\colon (Y,\bL)\to (Y',\bL')$ is a decorated link cobordism, we define a \emph{parametrized Kirby decomposition} $\cK$ of $(W,\cF)$  to be the following collection of data:

\begin{enumerate}
\item A decomposition of $(W,\cF)$  as 
\[(W,\cF)= (W_n,\cF_n) \cup_{Y_n} \cdots   \cup_{Y_1} (W_0,\cF_0)\] where $(W_i,\cF_i)$ is a decorated link cobordism with $\d W_i=-Y_i\sqcup Y_{i+1}$ and $\d \cF_i=-L_i\sqcup L_{i+1}$, for 3-manifolds with links $(Y_i,L_i)$ and surface with divides $\cF_i$ (all embedded in $W$), such that $Y_0=Y$ and $Y_{n+1}=Y'$. Furthermore $\cA$ intersects each component of each $L_i$ non-trivially and transversally.
\item A choice of basepoints $\ve{w}_i$ and $\ve{z}_i$ on each $L_i$, with one basepoint per component of $L_i\setminus \cA$, of the same type as the subsurface of $\Sigma\setminus \cA$ which contains it.
\item A framed sphere or link, $\bS_i\subset Y_i$ and a parameterizing diffeomorphism  $\Phi_i$, making each $(W_i,\cF_i,\Phi_i,\bS_i)$ an elementary, parametrized link cobordism.

\end{enumerate}
We assume further that there is only one term of $\cK$ of type ($\mathcal{EPC}$-\ref{elemencobtype2}) which contains a 1-dimensional framed link. Additionally, for convenience, we assume that all terms of the decomposition of type ($\mathcal{EPC}$-\ref{elemencobtype2}) occur after all of the terms of type ($\mathcal{EPC}$-\ref{elemencobtype3}) or ($\mathcal{EPC}$-\ref{elemencobtype4}), and that terms of the decomposition of type ($\mathcal{EPC}$-\ref{elemencobtype2}) appear with non-decreasing index (i.e. the 0-spheres occur before the 1-dimensional links, which occur before the 2-spheres).
\end{define}

\section{Moves between parametrized Kirby decompositions}
\label{sec:morsetheoyr-weakequivalences}
 In this section, we define a notion of equivalence between two parametrized Kirby decompositions, which we call \emph{Cerf equivalence}. The main theorem of this section is Theorem~\ref{prop:allPKDSigmasweaklyequivalent}, which states that any two parametrized Kirby decompositions of a decorated link cobordism $(W,\cF)$ are Cerf equivalent. Most of this section is devoted to proving Theorem~\ref{prop:allPKDSigmasweaklyequivalent}.

\subsection{Cerf equivalences of parametrized Kirby decompositions}
\label{subsec:defweakequiv}

In this section we provide the definition of Cerf equivalence (Definition~\ref{def:weaklyequivalent}). Before we can state the definition, we describe several constructions which will feature in the definition.

We first describe a special notion of isotopy between two parametrized Kirby decompositions, which will be one of the moves. Suppose that $\cK$ is a parametrized Kirby decomposition of the decorated link cobordism $(W,\cF)\colon (Y_1,\bL_1)\to (Y_2,\bL_2)$. Write $\cF=(\Sigma,\cA)$ and suppose that  $\phi\colon (W,\Sigma)\to (W,\Sigma)$ is a diffeomorphism which is the identity on $Y_1$. We now describe a decomposition $\phi_*\cK$ of $(W,\cF)$ into parametrized cobordisms, which we note may fail to be a parametrized Kirby decomposition in the sense of Definition~\ref{def:parametrizedKirbydecomp}. If 
\[
(W,\Sigma)=(W_n,\Sigma_n)\cup_{Y_n} \cdots \cup_{Y_1}  (W_0,\Sigma_0)
\] is the decomposition of undecorated link cobordisms associated to $\cK$, then we define the decomposition associated to $\phi_* \cK$ to be
\begin{equation}
 (\phi(W_n),\phi(\Sigma_n)) \cup_{\phi(Y_n)} \cdots \cup_{\phi(Y_1)}(\phi(W_0),\phi(\Sigma_0)).\label{eq:pushforwarddecomposition}
\end{equation} The framed spheres $\bS_i$ and  parametrizations $\Phi_i$ of $\cK$ can also be pushed forward along $\phi$ to give diffeomorphisms
\[\phi_*(\Phi_i):=\phi\circ \Phi_i\circ \phi^{-1}\colon \cW(\phi(Y_i),\phi(L_i),\phi(\bS_i))\to(\phi(W_i), \phi(\Sigma_i)).\] Note that $\phi_* \cK$ may fail to be a parametrized Kirby decomposition, because in general there is no reason to expect $(\phi(W_i), \phi(\Sigma_i), \phi(\Sigma_i)\cap \cA)$ to be an elementary parametrized cobordism in the sense of  Definition~\ref{def:parametrizedelementarycob}. With this in mind, we make the following definition:

\begin{define}\label{def:strongisotopy}
Suppose that $(W,\cF)$ is a link cobordism with $\cF=(\Sigma, \cA)$. If $\cK$ and $\cK'$ are two parametrized Kirby decompositions of $(W,\cF)$, we say that $\cK$ and $\cK'$ are related by an \emph{$\cA$-adapted isotopy} if there is a 1-parameter family of diffeomorphisms $\phi_t\colon (W,\Sigma)\to (W,\Sigma)$ such that following are satisfied:
\begin{enumerate}
\item $\phi_t|_{Y_1}=\id$ for all $t$.
\item $(\phi_t)_* \cK$ is a valid parametrized Kirby decomposition for all $t$.
\item $(\phi_1)_* \cK=\cK'$.
\end{enumerate}
\end{define}

\begin{rem}\label{rem:strongisotopy}Suppose $\phi_t$ is an $\cA$-adapted isotopy of $(W,\cF)$ and $(W_i,\cF_i,\Phi_i,\bS_i^0)$ is a term of $\cK$ corresponding to a framed 0-sphere along $L_i$. Writing $p$ for $\Phi_i(0,0)$, where $(0,0)$ center of the band in the cobordism $\cW(Y_i,L_i,\bS_i^0)$, then $\phi_t(p)$ is a path on $\Sigma$. For $\phi_t$ to induce an $\cA$-adapted isotopy, this path cannot intersect an arc of $\cA$, since at a point of time when $\phi_t(p)\in \cA$, the diffeomorphism $(\phi_t)_* \Phi_i$ fails to be a valid parametrization for an elementary parametrized cobordism, as it maps the center of the band onto an arc of $\cA$.
\end{rem}

Additionally, we need the following definition related to birth-death singularities of Morse functions (compare \cite{JClassTQFT}*{Section~1.3}):

\begin{define}\label{def:handlecanceldiffeo}Suppose $\bS\subset Y$ is a framed $k$-sphere, and $\bS'\subset Y(\bS)$ is a framed $(k+1)$-sphere which intersects the belt sphere of $\bS$, transversely at a single point. We say a diffeomorphism
\[
\Psi\colon \cW(Y(\bS),L_i,\bS')\circ \cW(Y,L, \bS)\to \cW(Y, L,\bS_{\varnothing}),
\] is \emph{formed by a handle cancellation} if it is constructed by the following procedure: take two Morse functions and gradient-like vector fields on $\cW(Y(\bS),L,\bS')$ and  $\cW(Y,L, \bS)$  which each have exactly one critical point, and induce the identity parametrization of the two cobordisms. In Section~\ref{subsec:MorsedataKirbydecomp} we will describe exactly how a Morse function and gradient-like vector field induce a parametrizing diffeomorphism. Furthermore, we will see that such Morse functions and gradient-like vector fields exist in Lemma~\ref{lem:findMorsedata}. The stable manifold of the higher critical point is transverse to the unstable manifold of the lower critical point and there is a single flowline connecting the two critical points. Write $(f,v)$ for the Morse function and gradient-like vector field obtained by stacking the Morse function and gradient-like vector fields on the two cobordisms. Using the argument from \cite{JClassTQFT}*{Section 1.3}, there is a homotopically unique path of functions with gradient-like vector fields $(f_t,v_t)_{t\in [0,1]}$ with $f_0=f$, which are constant outside a neighborhood of the flowline between the two critical points, which realizes a birth-death singularity. The pair $(f_1,v_1)$ on $\cW(Y(\bS),L,\bS')\circ \cW(Y,L, \bS)$ determines a diffeomorphism $\Psi\colon \cW(Y,L,\bS_{\varnothing})\to \cW(Y(\bS),L,\bS')\circ \cW(Y,L,\bS),$ which is the identity on the incoming boundary, $\{0\}\times Y$, and is well-defined up to isotopy relative $\{0\}\times Y$, since $(f_1,v_1)$ is determined up to homotopy (through paths of Morse functions and gradient-like vector fields, with no intermediate singularities).
\end{define}

We can now give our definition of equivalence between two parametrized Kirby decompositions.

\begin{define}\label{def:weaklyequivalent}We say two parametrized Kirby decompositions  are \emph{Cerf equivalent} if they are related by a sequence of the following set of moves and their inverses:
\begin{enumerate}
\item \label{def:weakequiv:strongisotopy}($\cA$-adapted isotopy): See Definition~\ref{def:strongisotopy}.
\item \label{def:weakequiv:isotopyawayfromL}(Isotopy of an attaching sphere $\bS_i$ away from $L_i$): Suppose $\cW_i=(W_i,{\cF}_i,\Phi_i,\bS_i)$ is a term in $\cK$ and suppose $\bS_i^t$ is a 1-parameter family of framed $k$-spheres in $Y_i\setminus L_i$ starting at $\bS_i$ and ending at a framed sphere $\bS_i'$. Let $d_i^t$ denote an extension of this isotopy to $Y_i$ which fixes $L_i$ pointwise. Define
\[
D\colon \cW(Y_i,L_i,\bS_i)\to \cW(Y_i,L_i,\bS_i')
\] 
via the formula
\[
D(t,x)=(t,d_i^t(x))
\] 
for $(t,x)\in [0,1]\times Y$. We define $D$ to map the handle regions on the two trace cobordisms diffeomorphically onto each other, under their identifications as $D^{k+1}\times D^{3-k}$. There is an induced diffeomorphism $d_i^{\bS_i}$ from $Y_i(\bS_i)$ to $Y_i(\bS_i')$, which $D$ restricts to. The elementary parametrized cobordism $\cW_i$ is replaced with an elementary cobordism $\cW_i'=(W_i, {\cF}_i, \Phi_i',\bS_i')$ where $\Phi_i'\circ D=\Phi_i$.

\item \label{def:weakequiv:isotopyalongL}(Isotopy of a framed 0-sphere $\bS_i$ along $L_i$): Suppose $\cW_i=(W_i,{\cF}_i,\Phi_i,\bS_i)$ is a term of $\cK$ and that $\bS_i^t$ is a 1-parameter family of framed 0-spheres along $L_i$ from $\bS_i$ to $\bS_i'$. The diffeomorphism sending $\bS_i$ to $\bS_i^t$ extends to an isotopy $d_i^t$ of $Y_i$ which preserves $L_i$. The isotopy $d_i^t$ defines a diffeomorphism 
\[
D\colon \cW(Y_i,L_i,\bS_i)\to \cW(Y_i,L_i,\bS_i')
\] 
via the formula
\[
D(t,x)=\begin{cases}(t,d_i^t(x)) &\text{if } (t,x)\in  [0,1]\times Y\\
 (t,d_i^1(x))& \text{if } (t,x)\in [1,2]\times  Y(\bS),
 \end{cases}
 \]
 and $D(z)=z$ if $z\in D^1\times D^3$ (the 1-handle).
 
  Note that since $d_i^t$ preserves $L_i$  setwise, the map $D$ sends $\Sigma(L_i,\bS)$ to $\Sigma(L_i,\bS')$, diffeomorphically, and maps the bands to each other as the identity, under the identification of both as $[-1,1]\times [-1,1]$.

The parametrized cobordism $(W_i,{\cF}_i,\Phi_i,\bS_i)$ is replaced with $\cW_i'=(W_i,{\cF}_i, \Phi_i',\bS_i')$, where $ \Phi_i'\circ D=\Phi_i$.
 
\item \label{def:weakequiv:addtrivialcylinder}(Addition of a trivial cylinder): If $\cW_i=(W_i,{\cF}_i,\Phi_i,\bS_i)$ and $\cW_{i+1}=(W_{i+1},{\cF}_{i+1},\Phi_{i+1},\bS_{i+1})$ are adjacent terms of $\cK$, with at least one of $\cW_i$ and $\cW_{i+1}$ being of type ($\mathcal{EPC}$-\ref{elemencobtype1}), then the pair $\cW_i$ and $\cW_{i+1}$ is replaced with a single $\cW_i'=(W_{i+1}\cup W_{i},{\cF}_{i+1}\cup {\cF}_i,\Phi_i',\bS_i')$. Under the canonical identification of 
\[
\cW(Y_{i+1},L_{i+1}, \bS_{i+1})\cup_{\Phi_i} \cW(Y_i, L_i, \bS_i) \iso \cW(Y_i, L_i, \bS_i\cup \pi_{Y_i}(\Phi_i^{-1}(\bS_{i+1}))),
\] where 
\[
\pi_{Y_i}\colon [0,1]\times Y_i\to Y_i
\] 
denotes projection.  The maps $\Phi_i'$ and the concatenation $\Phi_{i+1}*\Phi_i$  are isotopic through diffeomorphisms fixing $\{0\}\times Y_i$ (or $(\{0\}\times Y_i)\cup \{(0,0)\}$, where $(0,0)$ denotes the center of the band).

\item \label{def:weakequiv:changeorientationofS} (Reversal of the orientation of a framed sphere): Suppose $\cW_i=(W_i,{\cF}_i,\Phi_i,\bS_i)$ is an elementary parametrized cobordism with $\bS_i$ a non-empty framed $0$-sphere, 2-sphere, or 1-dimensional link, in $Y_i\setminus L_i$. We view $\bS_i$ as a map from a $S^{k}\times D^{3-k}$ to $Y_i$, and give $S^k\times D^{3-k}\subset D^{k+1}\times D^{3-k}\subset \R^4$ the coordinates $(x,y,w,z)$ induced by $\R^4$. We replace the map $\bS_i$ with the map $\bar{\bS}_i$ formed by pre-composing $\bS_i$ with the map
\[
\sigma(x,y,w,z)= (-x,y,w,-z).
\]

If $\bS_i$ is a framed 0-sphere along $L_i$, we can replace $\bS_i$ with the map $\bar{\bS}_i$ induced by pre-composing $\bS_i$ with the map
\[
\sigma(x,y,w,z)= (-x,-y,w,z).
\]

The parameterizing diffeomorphism $\Phi$ is also changed by precomposition with $\sigma$ on the handle region, $D^{k+1}\times D^{3-k}$.

\item \label{def:weakequiv:index12birthdeath}(Index 1/2 birth-death singularity away from $\Sigma$): Suppose $\cW_{i}=(W_i,{\cF}_i,\Phi_i,\bS_i)$ and $\cW_{i+1}=(W_{i+1},{\cF}_{i+1},\Phi_{i+1},\bS_{i+1})$ are adjacent terms of the decomposition, both of type ($\mathcal{EPC}$-\ref{elemencobtype2}), such that $\cW_i$ has framed 0-sphere $\bS_i$ away from $L_i$ and parametrization 
\[
\Phi_i\colon \cW(Y_i,L_i,\bS_i)\to (W_i,\Sigma_i),
\]
 and $\cW_{i+1}$ has framed 1-dimensional link $\bS_{i+1}$, and parametrization 
 \[
 \Phi_{i+1}\colon \cW(Y_{i+1},L_{i+1},\bS_{i+1})\to (W_{i+1},\Sigma_{i+1}).
 \]
  Suppose the image of the belt sphere of $\bS_{i}$ in $Y_{i+1}$ (under the parametrization) intersects a single component, $\bK$, of $\bS_{i+1}$ exactly once. The move is a composition of following two moves:
\begin{itemize}
\item We split the component $\bK$ of $\bS_{i+1}$ off from the rest of the components of $\bS_{i+1}$ by replacing $\cW_{i+1}\circ \cW_i$ with $\cW_{i+1}'\circ \cW_{i+1}''\circ \cW_i$ where 
 \[
 \cW_{i+1}'=(W_{i+1}', {\cF}_{i+1}',\Phi_{i+1}', \bS_{i+1}') \quad \text{and} \quad \cW_{i+1}''=(W_{i+1}'', {\cF}_{i+1}'',\Phi_{i+1}'', \bK)
 \] 
 and $\pi_{Y_{i+1}}((\Phi_{i+1}'')^{-1}(\bS_{i+1}'))=\bS_{i+1}\setminus \bK.$ 
\item We cancel $\bK$ and $\bS_i$ by replacing $\cW_{i+1}'\circ \cW_{i+1}''\circ \cW_i$ with $\cW_{i+1}'\circ \cW_{i}'$ where $\cW_i'=(W_i', {\cF}_i', \Phi_i', \bS_{\varnothing})$ is an elementary parametrized cobordism of type ($\mathcal{EPC}$- \ref{elemencobtype1}). Furthermore,
\[
\Phi_i'\circ \Psi\simeq \Phi_{i+1}''* \Phi_i,
\] 
and
\[
\Psi\colon \cW(Y_i(\bS_i),L_i,\Phi_i^{-1}(\bK))\circ \cW(Y_i,L_i,\bS_i)\to \cW(Y_i,L_i, \bS_{\varnothing})
\] 
is a diffeomorphism obtained by handle cancellation (Definition~\ref{def:handlecanceldiffeo}).
\end{itemize}

\item \label{def:weakequiv:index23birthdeath}(Index 2/3 birth-death singularity away from $\Sigma$): This move is defined analogously to Move \eqref{def:weakequiv:index12birthdeath}, but instead we start with a pair of adjacent terms $\cW_i=(W_{i},{\cF}_i,\Phi_i,\bS_i)$ and $\cW_{i+1}=(W_{i+1},{\cF}_{i+1},\Phi_{i+1},\bS_{i+1})$ where $\bS_i$ is a framed 1-dimensional link, and $\bS_{i+1}$ is a framed 2-sphere, which cancels a component $\bK$ of $\bS_i$. The pair is replaced with two elementary parametrized cobordisms, $\cW_{i+1}'$ and $\cW_i'$, with $\cW_{i+1}'$ being a trivial cobordism and $\cW_{i}'$ corresponding to surgery on $\bS_i\setminus \bK$.
\item \label{def:weakequiv:criticalpointswitchawayfromSigma} (Critical value switch between framed 0- or 2-spheres, away from $\Sigma$): If $\cW_i=(W_i,{\cF}_i,\Phi_i,\bS_i)$ and $\cW_{i+1}=(W_{i+1},{\cF}_{i+1},\Phi_{i+1},\bS_{i+1})$ are both of type ($\mathcal{EPC}$-\ref{elemencobtype2}), for 0- or 2-spheres $\bS_i\subset Y_{i}$ and $\bS_{i+1}\subset Y_{i+1}$ such that under the images of $\bS_i$ and $\pi_{Y_{i}}(\Phi_i^{-1}(\bS_{i+1}))$ in $Y_{i}$ are disjoint, then we replace the two composition $\cW_{i+1}\circ \cW_i$ with a composition $\cW_{i+1}'\circ \cW_i'$ where 
\[
\cW_{i}'=(W_i',{\cF}_i',\Phi_{i}', \bS_i')\qquad \text{ and }\qquad \cW_{i+1}'=(W_{i+1}',{\cF}_{i+1}',\Phi_{i+1}', \bS_{i+1}')
\]
 are both  of type ($\mathcal{EPC}$-\ref{elemencobtype2}), and satisfy $ W_{i+1}\cup W_i=W_{i+1}'\cup W_i'$ and $\cF_{i+1}\cup \cF_{i} ={\cF}_{i+1}'\cup  {\cF}_{i}'$. Furthermore
\[
\Phi_i(\bS_i')=\bS_{i+1}\qquad \text{and} \qquad \Phi_i'(\bS_i)=\bS_{i+1}'.
\] Furthermore, under the canonical identifications
\[
\cW(Y_i(\bS_i), L_i, \bS_{i}')\circ \cW(Y_i, L_i, \bS_i)\iso \cW(Y_{i}(\bS_i'), L_i, \bS_i)\circ \cW(Y_i, L_i, \bS_{i}'),
\]
 we have that the concatenations
$\Phi_{i+1}* \Phi_i$ and $\Phi_{i+1}'* \Phi_i'$ are isotopic through diffeomorphisms which fix $Y_i\times \{0\}$.

\item \label{def:weakequiv:handleslide}(Handleslide amongst the framed 1-dimensional spheres): If $\cW_i=(W_i,{\cF}_i,\Phi_i,\bS_i)$ is of type ($\mathcal{EPC}$-\ref{elemencobtype2}), and $\bS_i$ is a framed 1-dimensional link, then handlesliding a component of $\bS_i$ across another component yields a framed link $\bS_i'$, and a diffeomorphism $D\colon \cW(Y_i,L_i,\bS'_i)\to \cW(Y_i,L_i,\bS_i)$ fixing $Y_1\times \{0\}$, well-defined up to isotopy\footnote{An isotopy of an attaching map yields a diffeomorphism of the spaces obtained by attaching a handle, well-defined up to isotopy. Handlesliding a handle across another corresponds to an isotopy of the attaching sphere of a handle on the boundary of the 4-manifold obtained by attaching another handle. As such, the diffeomorphism $D$ resulting from handlesliding is well-defined up to isotopy.}. The parametrized cobordism $\cW_i$ is replaced by
\[
\cW_i'=(W_i,{\cF}_i, \Phi_i\circ D, \bS_i').
\]

\item \label{def:weakequiv:birthdeathalongA} (Birth-death singularity along $\cA$): If $\cW_i=(W_i,{\cF}_i,\Phi_i,\bS_{\varnothing})$ and $\cW_{i+1}=(W_{i+1},{\cF}_{i+1},\Phi_{i+1},\bS_{\varnothing})$ are elementary parametrized link cobordisms, both of type ($\mathcal{EPC}$-\ref{elemencobtype3}), such that the two critical arcs of $\cF_{i}$ and $\cF_{i+1}$  intersect at exactly one point on the common boundary $W_i\cap W_{i+1}$, then we replace the pair $\cW_i$ and $\cW_{i+1}$ with 
\[
\cW_i'=(W_{i+1}\cup W_i,\cF_{i+1}\cup \cF_i,\Phi_{i+1}* \Phi_i,\bS_{\varnothing}),
\] 
an elementary, parametrized link cobordism of type ($\mathcal{EPC}$-\ref{elemencobtype1}).

\item \label{def:weakequiv:switchalongA} (Critical point switch along $\cA$): If  $\cW_i=(W_i,{\cF}_i,\Phi_i,\bS_{\varnothing})$ and $\cW_{i+1}=(W_{i+1},{\cF}_{i+1},\Phi_{i+1},\bS_{\varnothing})$  are both of type~($\mathcal{EPC}$-\ref{elemencobtype3}), such that the critical arcs of $\bF_i$ and $\cF_{i+1}$ are disjoint, then $\cW_{i+1}\circ \cW_{i}$ is replaced by $\cW_{i+1}'\circ \cW_{i}'$ where 
\[
\cW_i'=(W_i',{\cF}_i',\Phi_i',\bS_{\varnothing})\qquad  \text{and} \qquad \cW_{i+1}'=(W_{i+1}', {\cF}_{i+1}',\Phi_{i+1}',\bS_{\varnothing})
\] are both of type~($\mathcal{EPC}$-\ref{elemencobtype3}), and $W_{i+1}\cup W_{i}=W_{i+1}'\cup W_{i}'$ and $\cF_{i+1}\cup \cF_i=\cF_{i+1}'\cup \cF_{i}'$. If $a_i$ and $a_{i+1}$ are the dividing arcs of $\cF_{i+1}\cup \cF_i$ which contain the critical arcs of $\cF_i$ and $\cF_{i+1}$, respectively, then $a_i\neq   a_{i+1}$. Furthermore, $a_i$ contains the critical arc of $\cF_{i+1}'$ and $a_{i+1}$ contains the critical arc of $\cF_{i}'$. Lastly, $\Phi_{i+1}* \Phi_i$ and $\Phi_{i+1}'* \Phi_i'$ are isotopic through diffeomorphisms which fix $\{0\}\times Y_i$.

\item \label{def:weakequiv:switchalongSigma} (Critical point switch along $\Sigma$): If $\cW_i=(W_i,{\cF}_i,\Phi_i,\bS_i)$ and $\cW_{i+1}=(W_{i+1},{\cF}_{i+1},\Phi_{i+1},\bS_{i+1})$ are both of type~($\mathcal{EPC}$-\ref{elemencobtype4}), then the pair is replaced by a pair
\[
\cW_i'=(W_i',{\cF}_{i}',\Phi_i',\bS_i')\qquad \text{and}\qquad \cW_{i+1}'=(W_{i+1}',{\cF}_{i+1}',\Phi_{i+1}',\bS_{i+1}'),
\]both of type ($\mathcal{EPC}$-\ref{elemencobtype4}), with $W_{i+1}\cup W_i=W_{i+1}'\cup W_i'$ and ${\cF}_{i+1}\cup {\cF}_{i}={\cF}_{i+1}'\cup {\cF}_{i}'$, and
\[
\pi_{Y_i}(\Phi_i^{-1}(\bS_{i+1}))=\bS_i'\qquad \text{and} \qquad \pi_{Y_i}((\Phi_i')^{-1}(\bS_{i+1}')=\bS_i.
\]
Furthermore, under the canonical identifications
\[
\cW(Y_i(\bS_i), L_i(\bS_i), \bS_{i}')\circ \cW(Y_i, L_i, \bS_i)\iso \cW(Y_{i}(\bS_i'), L_i(\bS_i'), \bS_i)\circ \cW(Y_i, L_i, \bS_{i}'),
\] 
the maps
$\Phi_{i+1}* \Phi_i$ and $\Phi_{i+1}'* \Phi_i'$ are isotopic through diffeomorphisms of link cobordisms which fix $\{0\}\times Y_i,$ as well as the center points of the two bands.

\item \label{def:weakequiv:switchbetweenAandSigma}(Critical point switch between critical points of $\cA$ and $\Sigma$): If $\cW_i=(W_i,{\cF}_i,\Phi_i,\bS_{\varnothing})$ is of type ($\mathcal{EPC}$-\ref{elemencobtype3}) and $\cW_{i+1}=(W_{i+1},{\cF}_{i+1},\Phi_{i+1},\bS_{i+1})$ is of type ($\mathcal{EPC}$-\ref{elemencobtype4}), then we replace the expression $\cW_{i+1}\circ \cW_i$  with $\cW_{i+1}'\circ \cW_{i}'$ where
\[
\cW_i'=(W_i',{\cF}_{i}',\Phi_i',\bS_i'),\qquad \text{and}\qquad \cW_{i+1}'=(W_{i+1}',{\cF}_{i+1}',\Phi_{i+1}',\bS_{\varnothing}).
\] 
Furthermore $\cW_{i}'$ is of type ($\mathcal{EPC}$-\ref{elemencobtype4}) and $\cW_{i+1}'$ is of type ($\mathcal{EPC}$-\ref{elemencobtype3}). Furthermore, $W_{i+1}\cup W_{i}=W_{i+1}'\cup W_{i}'$ and ${\cF}_{i+1}\cup {\cF}_{i}={\cF}_{i+1}'\cup {\cF}_{i}'$, and $\Phi_i(\{1\}\times \bS_{i}')=\bS_{i+1}$. Furthermore, under the canonical identification (well-defined up to isotopy)
\[
\cW(Y_i, L_i,\bS_i')\circ \cW(Y_i, L_i, \bS_{\varnothing})\iso \cW(Y_{i}(\bS_i'), L_i, \bS_{\varnothing})\circ \cW(Y_i, L_i, \bS_{i}'),
\]
 we have that
$\Phi_{i+1}* \Phi_i$ and $\Phi_{i+1}'* \Phi_i'$ are isotopic through diffeomorphisms which fix $(\{0\}\times Y_i)\cup \{(0,0)\}$.

\item \label{def:weakequiv:critcrossesA}(A critical point of $\Sigma$ crossing $\cA$): If $\cW_i=(W_i,{\cF}_i,\Phi_i,\bS_{\varnothing})$ is of type ($\mathcal{EPC}$-$\ref{elemencobtype3}_{S^+}$) and $\cW_{i+1}=(W_{i+1},{\cF}_{i+1},\Phi_{i+1},\bS_{i+1})$  is of type ($\mathcal{EPC}$-$\ref{elemencobtype4}_{\ve{w}}$), then $\cW_{i+1}\circ \cW_{i}$ is replaced by $\cW_{i+1}'\circ \cW_{i}'$, where 
\[
\cW_i'=(W_i',{\cF}_{i}',\Phi_i',\bS_{\varnothing})\qquad \text{and}\qquad \cW_{i+1}'=(W_{i+1}',{\cF}_{i+1}',\Phi_{i+1}',\bS'_{i+1}),
\]
 and $\cW_i'$ is of type ($\mathcal{EPC}$-$\ref{elemencobtype3}_{T^+}$) and $\cW_{i+1}'$ is of type ($\mathcal{EPC}$-$\ref{elemencobtype4}_{\ve{z}}$).  Furthermore $W_{i+1}\cup W_{i}=W_{i+1}'\cup W_{i}'$ and ${\cF}_{i+1}\cup {\cF}_{i}={\cF}_{i+1}'\cup {\cF}_{i}'$. Additionally, $\pi_{Y_i}(\Phi_i^{-1}(\bS_{i+1}))$ and $\pi_{Y_i}((\Phi_i')^{-1}(\bS_{i+1}'))$ are equal. The maps $\Phi_{i+1}* \Phi_i$ and $\Phi_{i+1}'* \Phi_i'$ are isotopic diffeomorphisms of link cobordisms, relative to $\{0\}\times Y_i$, however the isotopy between $\Phi_{i+1}* \Phi_i$ and $\Phi_{i+1}'* \Phi_i'$ does not fix the point $(0,0)$ of the band. Instead the image of $(0,0)\in \Sigma(L_i,\bS)$ (where $\bS= \pi_{Y_i}(\Phi_i^{-1}(\bS_{i+1}))$) crosses  a single arc of $\cA$, transversely, at a single point. See Figure~\ref{fig::57}.

\end{enumerate}
\end{define}

The remainder of this section is devoted to proving the following theorem, which will be used to prove invariance of the link cobordism maps:

\begin{thm}\label{prop:allPKDSigmasweaklyequivalent}If $(W,\cF)\colon (Y_1,\bL_1)\to (Y_2,\bL_2)$ is a decorated link cobordism  and each component of $W$ and each component of $\cF$ intersect $Y_1$ and $Y_2$ non-trivially, then any two parametrized Kirby decompositions of $(W,\cF)$ are Cerf equivalent.
\end{thm}

\subsection{Morse functions and adapted coordinates}
\label{subsec:spacesofadaptedcoords}

Suppose that $X$ is a smooth manifold with a Morse function $f\colon X\to \R$ and $p\in X$ is an index $k$ critical point of $f$. If $U\subset X$ is a precompact set, let $\Coor_p(U)$ denote the set of oriented local coordinates $\vec{x}=(x_1,\dots, x_n)$ defined on $U$, centered at $p$, which extend smoothly to a neighborhood of $\cl(U)\subset X$. We will write $\Coor_p(U,f)$ for the subset of $\Coor_p(U)$ of coordinates $\vec{x}=(x_1,\dots, x_n)$ where
\[f=f(p)-(x_1^2+\cdots +x_k^2)+(x_{k+1}^2+\cdots+ x_n^2).\]  The sets $\Coor_p(U)$ and $\Coor_p(U,f)$ can be given the $C^\infty$ topology.

If $X$ and $Y$ are open subsets of $\R^n$, both containing $\vec{0}$, and $X$ is precompact, we will write $\Diff'((X,\vec{0}), (Y,\vec{0}))$ for the space of maps from $X$ to $Y$ which fix $\vec{0}$, are diffeomorphisms onto their image, and which extend smoothly to a neighborhood of $\cl (X)$ in $\R^n$.

We will write $F_{k,n-k}^y:\R^n\to \R$ for the function
\[
F_{k,n-k}^{y}(\vec{u}):=y-(u_1^2+\cdots u_k^2)+(u_{k+1}^2+\cdots +u_{n}^2),
\]
where $\vec{u}=(u_1,\dots, u_n)$ denotes a point in $\R^n$. Note that if $\vec{x}\in \Coor_p(U)$, then $F_{k,n-k}^{f(p)}\circ \vec{x}=f$ if and only if $\vec{x}\in \Coor_p(U,f)$.

The map $F_{k,n-k}^y$ satisfies the equation 
\[
F_{k,n-k}^y(\vec{u})=y+\vec{u}^{\top}\cdot I_{k,n-k}\cdot \vec{u},
\]
 where $I_{k,n-k}$ is the diagonal matrix whose first $k$ entries are $-1$ and the last $n-k$ are $+1$. The indefinite orthogonal group $O(k,n-k)$ is defined as the set of $n\times n$ matrices $M$ over $\R$ which satisfy $M^\top\cdot  I_{k,n-k}\cdot  M=I_{k,n-k}$. If $0<k<n$, then $O(k,n-k)$ has 4 components, which can be related by reversing the orientation of one or both of the subspaces $\R^k\times \{\vec{0}\}$, $\{\vec{0}\}\times \R^{n-k}\subset \R^k\times \R^{n-k}$. If $0<k<n$, then $\SO(k,n-k)$, the set of matrices in $O(k,n-k)$ with determinant $+1$ (equivalently the matrices in $O(k,n-k)$ whose action preserves the orientation of $\R^n$), has two components. The two components are related by reversing the orientations of $\R^k\times \{\vec{0}\}$ and $\{\vec{0}\}\times \R^{n-k}$ simultaneously.

 The following fact is essentially well known (cf. \cite{CerfTheory}*{pg. 168}), and will be useful to us.

\begin{lem}\label{lem:linearization}Suppose that there are coordinates $\vec{x}_0\in \Coor_p(U,f)$ such that $\vec{x}_0(U)$ is a convex subset of $\R^n$. Then the set $\Coor_p(U,f)$ is homotopy equivalent to $\SO(k,n-k)$.
\end{lem}
\begin{proof}Fixing a base set of coordinates $\vec{x}_0$, we can define a homeomorphism
\[
I_{\vec{x}_0}\colon \Coor_p(U)\to \Diff'((\vec{x}_0(U),\vec{0}), (\R^n,\vec{0}))
\] via the formula
\[
I_{\vec{x}_0}(\vec{x}):=\vec{x}\circ \vec{x}_0^{-1}.
\]
Let us write $\Diff'((\vec{x}_0(U),\vec{0}), (\R^n, \vec{0}); F_{k,n-k}^{f(p)})$ for the subset of $\Diff'((\vec{x}_0(U),\vec{0}), (\R^n, \vec{0}))$ consisting of maps which preserve $F_{k,n-k}^{f(p)}$. Note that $I_{\vec{x}_0}$ restricts to a homeomorphism from $\Coor_p(U,f)$ to $\Diff'((\vec{x}_0(U),\vec{0}), (\R^n, \vec{0}); F_{k,n-k}^y)$.
We define the map
\[
H\colon [0,1]\times \Diff'((\vec{x}_0(U),\vec{0}), (\R^n,\vec{0}))\to \Diff'((\vec{x}_0(U),\vec{0}), (\R^n,\vec{0}))
\]
via the formula
\[
H(t,\phi)(\vec{u})=\begin{cases} \tfrac{1}{t} \phi(t\cdot \vec{u}), & t\neq 0\\
d_{\vec{0}}(\phi)(\vec{u}),& t=0,
\end{cases}
\]
where $\vec{u}$ denotes a point in $\R^n$. Note we are using the fact that $\vec{x}_0(U)$ is convex for $H$ to be well-defined.

 To see that the map $H$ is continuous in the $C^\infty$ topology, we note that if $I$ is a multi-index, then
\[\left(\frac{\d^I}{\d \vec{u}^I}H(t,\phi)\right)(\vec{u})=t^{|I|-1} \left(\frac{\d^I}{\d \vec{u}^I} \phi\right)(t\cdot \vec{u}),\] from which continuity of $H$ in the $C^\infty$ topology is straightforward.

We claim that $H$ maps $[0,1]\times \Diff'((\vec{x}_0(U),\vec{0}), (\R^n, \vec{0}); F_{k,n-k}^{f(p)})$ into $\Diff'((\vec{x}_0(U),\vec{0}), (\R^n, \vec{0}); F_{k,n-k}^{f(p)})$. To see this, if $\phi$ preserves $F_{k,n-k}^{f(p)}$ we compute that
\[
F_{k,n-k}^{f(p)}(\tfrac{1}{t} \phi(t\cdot \vec{u}))=f(p)+F_{k,n-k}^{0}( \tfrac{1}{t} \phi(t\cdot \vec{u}))=f(p)+\tfrac{1}{t^2} F_{k,n-k}^{0}(\phi(t\cdot \vec{u}))
\]
\[
=f(p)+\tfrac{1}{t^2} F_{k,n-k}^0(t\cdot \vec{u})=F_{k,n-k}^{f(p)}(\vec{u}).
\] 
Furthermore, $H(t,\phi)=\phi$ for any linear diffeomorphism of $\R^n$. Hence $H$ determines a deformation retraction of $\Diff'((\vec{x}_0(U),\vec{0}), (\R^n, \vec{0}); F_{k,n-k}^{f(p)})$ onto the subspace  consisting of linear diffeomorphisms which are oriented and preserve $F_{k,n-k}^{f(p)}$, which is clearly homeomorphic to $\SO(k,n-k)$.

\end{proof}

We now prove several refinements of Lemma~\ref{lem:linearization}, concerning spaces of coordinates adapted to Morse functions and surfaces in a 4-manifold. Suppose $\Sigma^2$ is an oriented, embedded surface in a smooth manifold $X^4$. Suppose $U\subset X$ is a precompact open subset, $f\colon X\to \R$ is a Morse function such that $f|_{\Sigma}$ is Morse, and $p\in \Sigma$ is a critical point of $f$ and $f|_{\Sigma}$, which is index 1 for both. Write $\Coor_p(U,\Sigma,f)$ to be the set of coordinates $\vec{x}=(x,y,w,z)$ for $X$, which extend smoothly to a neighborhood of $\cl(U)$, such that $\Sigma\cap U=\im(\vec{x})\cap \{(x,y,0,0)\}$ and 
\[f=f(p)-x^2+y^2+w^2+z^2\] on $U$.

\begin{lem}\label{lem:linearization3}Suppose that $\Sigma$ is an oriented, embedded surface in $X$, an oriented 4-manifold, and that $f$ is a Morse function on $X$, as above. If there is an $\vec{x}_0\in \Coor_p(U,\Sigma,f)$ such that $\vec{x}_0(U)$ is a convex subset of $\R^4$, then
$\Coor_p(U,\Sigma,f)$ is homotopy equivalent to $S(O(1,1)\times O(2))$, which has 4 components. The subset of these coordinates which preserves a given orientation of $\Sigma$ is homotopy equivalent to $\SO(1,1)\times \SO(2)$, which has two components. If we write $v_0$ for the vector field \[
v_0:=\left(-x\frac{\d}{\d x}+y\frac{\d}{\d y}+w\frac{\d}{\d w}+z\frac{\d}{\d z}\right)
\] defined on $\R^4$, then the set of vector fields
 \[\cV:=\{\vec{x}^*(v_0): \vec{x}\in \Coor_p(U,\Sigma,f)\}\]
is connected.
\end{lem}

\begin{proof}The deformation retraction from Lemma~\ref{lem:linearization} is easily seen to respect $\Sigma$, since $\vec{x}_0(\Sigma\cap U)$ is a linear subspace of $\R^4$. Hence $\Coor_p(U, \Sigma,f)$ is homotopy equivalent to the subset of orientation preserving, linear diffeomorphisms of $\R^4$ which  preserve the function $F_{1,3}^{f(p)}$ and also preserve the subspace $\{(x,y,0,0)\}\subset \R^4$. Let $\cL$ denote the set of linear transformations of $\R^4$, which preserve $f$ and preserve $\{(x,y,0,0)\}$ setwise. If $M\in \cL$, we can view $M$ as a $2\times 2$ block matrix, with the first block corresponding to $(x,y)$, and the second block corresponding to $(w,z)$. The requirement that $M$ preserves $\{(x,y,0,0)\}$ is equivalent to $M$ taking the form
\[
M=\begin{pmatrix} A& B\\
0& D
\end{pmatrix}.
\]
The requirement that $M$ preserves the function $F_{1,3}^{f(p)}$ is equivalent to the requirement that the block matrix for $M$ satisfy
\[
A\in O(1,1),\qquad B=0 \qquad \text{and} \qquad D\in O(2)
.\]

Finally, for the claim about $\cV$, we note that the two components of $\SO(1,1)\times \SO(2)$ are related by replacing $(x,y)$ with $(-x,-y)$, which does not change the vector field $v_0$.
\end{proof}

We prove a final result  concerning Morse functions and coordinate systems. Suppose that $\Sigma^2$ is an embedded submanifold of $X^4$ and both are oriented. Suppose that $h$ is a Morse function on $\Sigma$, and $p\in \Sigma$ is an index 1 critical point of $h$. If $U\subset X$ is a precompact open set containing $p$, let $\Coor_p(U,\Sigma, h)$ denote the set of coordinates $\vec{x}\colon U\to \R^4$, centered at $p$, which extend smoothly to a neighborhood of $\cl(U)$, such that $\vec{x}(U\cap \Sigma)$ is embedded in $\{(w,z,0,0)\}$ in an orientation preserving manner, and such that
\[
h=h(p)-x^2+y^2
\]
 on $\Sigma\cap U$.

\begin{lem}\label{lem:linearization4} Let $\Sigma^2$ be an oriented submanifold of $X^4$, and let $h$ be a Morse function on $\Sigma$ with a critical point $p$ of index 1. If $\Coor_p(U,\Sigma,h)$ contains an element $\vec{x}_0$ such that $\vec{x}_0(U)\subset \R^4$ is convex, then the space $\Coor_p(U,\Sigma,h)$ has 2 components, which are related by replacing $(x,y)$ with $(-x,-y)$.
\end{lem}
\begin{proof}The technique from Lemma~\ref{lem:linearization}  shows that if $\Coor_p(U,\Sigma,h)$ contains such an $\vec{x}_0$, then $\Coor_p(U,\Sigma,h)$ is homotopy equivalent to the space of linear transformations of $\R^4$ which preserve $\{(x,y,0,0)\}$ setwise,  preserve the orientations of $\R^4$ and $\{(x,y,0,0)\}$, and preserve the function $h(p)-x^2+y^2$ on $\Sigma$. Such a linear transformation can be written in block matrix form as
\[M=\begin{pmatrix}A& B\\
0& D
\end{pmatrix}\] where $A\in O(1,1),$ $B\in \Mat_{2\times 2}(\R)$ and $D\in GL(2, \R)$. The set of such matrices deformation retracts onto the space with $B=0$. Furthermore, requiring that $M$ preserve the orientations of $X$ and $\Sigma$ implies that $A$ and $D$ have positive determinant, so we see that $\Coor_p(U,\Sigma,h)$ is homotopy equivalent to $\SO(1,1)\times \GL^+(2,\R)$, which has two components, which are related by replacing $(x,y)$ with $(-x,-y)$. 
\end{proof}

\begin{define}\label{def:adaptedMorsegradient}If $(W,\Sigma)$ is an undecorated link cobordism, we define $\ve{M}(W,\Sigma)$ to be the space of Morse functions $f\colon W\to [0,2]$ such that the following hold:
\begin{enumerate}
\item $\Crit(f|_\Sigma)=\Crit(f)\cap \Sigma$.
\item For each $p\in \Crit(f)\cap \Sigma$ there  are coordinates $\vec{x}=(x,y,w,z)$, defined on a neighborhood $U$ of $p$, which are oriented for both $W$ and $\Sigma$,  such that $f=f(p)-x^2+y^2+w^2+z^2$, and $\vec{x}(\Sigma\cap U)=\vec{x}(U)\cap \{(x,y,0,0)\}$.
\end{enumerate}
If $h$ is a fixed Morse function on $\Sigma$ which has only index 1 critical points, we will write $\ve{M}(W,\Sigma;h)$ for the set of Morse functions in $\ve{M}(W,\Sigma)$ which also restrict to $h$ on $\Sigma$.
\end{define}
\begin{rem}If $f\in \ve{M}(W,\Sigma)$, then by property (2) of Definition~\ref{def:adaptedMorsegradient}, it follows that $f|_\Sigma$ has only index 1 critical points.
\end{rem}

\begin{define}If $(W,\Sigma)$ is a link cobordism, we say that a smooth function $f\colon W\to [0,2]$ is a \emph{nice} Morse function  if the following are satisfied:
\begin{enumerate}
\item $f\in \ve{M}(W,\Sigma)$.
\item All critical values of $f|_{\Sigma}$ are in $(0,1)$, and all critical values of $f|_{W\setminus \Sigma}$ are in $(1,2)$.
\item $f|_{W\setminus \Sigma}$ has only index 1, 2 or 3 critical points.
\item All index 1 critical points of $f|_{W\setminus \Sigma}$ occur before the index 2 critical points, which occur before the index 3 critical points.
\item All critical points of $f$ have distinct critical values.
\end{enumerate}
We write ${\ve{M}}^{\nice}(W,\Sigma)$ for the set of nice Morse functions. If $h$ is a fixed Morse function on $\Sigma$, we write $\ve{M}^{\nice}(W,\Sigma;h)$ for the subset of $\ve{M}^{\nice}(W,\Sigma)$ of Morse functions which restrict to $h$ on $\Sigma$.
\end{define}

\subsection{Connecting Morse functions which agree on $\Sigma$}
\label{subsec:changeMorsefunction-fixedonSigma}

In this subsection, we construct appropriately nice paths between Morse functions which agree on $\Sigma$.

\begin{lem}\label{lem:genericpath}Suppose that $h$ is a Morse function on $\Sigma$ with only index 1 critical points. If $f_0, f_1\in \ve{M}(W,\Sigma,h)$, then there is a 1-parameter family $(f_t)_{t\in [0,1]}$ of smooth functions such that  $f_t\in {\ve{M}}(W,\Sigma,h)$ for all but finitely many $t$, and at the finitely many points of time when $f_t$ fails to be in ${\ve{M}}(W,\Sigma,h),$ exactly one of the following occurs:
\begin{enumerate}
\item A critical point birth-death occurs in $W\setminus \Sigma$.
\item Two critical points exchange relative value.
\end{enumerate} Furthermore the path $f_t$ can be chosen so that for each $p\in \Crit(h)$, there is a neighborhood $U$ of $p$ in $W$ and a 1-parameter family of coordinates $\vec{x}_t=(x_t,y_t,w_t,z_t)\colon U\to \R^4$ such that $\vec{x}_t(U\cap\Sigma)\subset \{(x,y,0,0)\}\subset \R^4$, and
 \[f_t=h(p)-x_t^2+y_t^2+w_t^2+z_t^2.\]
\end{lem}

\begin{proof} First, consider the case that $f_0$ and $f_1$ agree in an open set $N$ containing $\Crit(h)$. In this case, we start with the 1-parameter family $f_t=t\cdot f_1+(1-t)\cdot f_0$. Notice that $f_t$ has no critical points along $\Sigma$ except for those in $\Crit(h)$ since it is constant in $t$ on $\Sigma$. A generic perturbation of $f_t$ near the set of critical points in $W\setminus \Sigma$ will be Morse at all but finitely many $t\in [0,1]$, where a birth-death singularity, or a critical value switch occurs. Since $f_t$ is fixed on $N$, one does not need to perturb near $\Sigma$.

Using the previous observation, it is sufficient to show that if $f_0,f_1\in \ve{M}(W,\Sigma;h)$ are arbitrary Morse functions, then there is a 1-parameter family $(\hat{f}_t)_{t\in [0,1]}$ such that  $\hat{f}_t\in \ve{M}(W, \Sigma;h)$ for all $t$, and such that  $\hat{f}_0=f_0$ and $\hat{f}_1$ agrees with $f_1$ in a neighborhood of $\Crit(h)$.

For notational simplicity, let's consider only the case that $\Crit(h)$ contains a single point $p$. By the assumption that $f_0,f_1\in \ve{M}(W,\Sigma;h)$, there is an open neighborhood $U$, with coordinates maps $\vec{x}_0=(x_0,y_0,w_0,z_0)$ and $\vec{x}_1=(x_1,y_1,w_1,z_1)$, centered at $p$, such that $f_i=F_{1,3}^{h(p)}\circ \vec{x}_i$. We can assume, without loss of generality, that $\vec{x}_0(U)\subset \R^4$ is convex. Lemma~\ref{lem:linearization4} implies that after possibly replacing $(x_0,y_0)$ with $(-x_0,-y_0)$, there is a 1-parameter family of coordinates $\vec{x}_t=(x_t,y_t,w_t,z_t)$, starting at $\vec{x}_0$ and ending at $\vec{x}_1$, such that $h=h(p)-x_t^2+y_t^2$ for all $t$, and such that $\vec{x}_t$ maps $\Sigma$ to $\{(x,y,0,0)\}\subset \R^4$ for all $t$.

Note that $\phi_t:=\vec{x}_t^{-1}\circ \vec{x}_0$ determines a 1-parameter family of embeddings of $U$ into $W$, which fix $p$ and preserve $\Sigma$, setwise. Furthermore $h\circ (\phi_t)|_{\Sigma}\circ \iota_\Sigma=h$, where $\iota_\Sigma\colon \Sigma\to W$ denotes inclusion. 

We now extend $\phi_t$ to an isotopy of all of $W$, as follows. Pick a compactly supported bump function $\rho$ which is supported in $U$ and is $1$ in a neighborhood of $p$. Define the time dependent vector field \[V_t(x)=\rho(x)\cdot \frac{d}{ds}\bigg|_{s=0} \phi_{t+s}(\phi_t^{-1}(x)).\] We integrate $V_t(x)$ to define a 1-parameter family of diffeomorphisms $(\Phi_t)_{t\in [0,1]}$, i.e., we define  $\Phi_t$ to be a family of diffeomorphisms which satisfies
\[V_t(\Phi_t(x))=\frac{d}{d s}\bigg|_{s=0} \Phi_{s+t}(x).\] It is straightforward to check the following properties:
\begin{enumerate}
\item $\Phi_t(p)=p$ and $\Phi_t(\Sigma)=\Sigma$.
\item $h\circ \Phi_t\circ \iota_\Sigma=h$.
\item $\Phi_t=\phi_t$ in a neighborhood of $p$.
\end{enumerate}
We define $\hat{f}_t:=f_0\circ \Phi_t$, and note that $\hat{f}_0=f_0$, $\hat{f}_t\in \ve{M}(W,\Sigma;h)$ for all $t$, and $\hat{f}_1$ and $f_1$ agree in a neighborhood of $p$. Using our previous argument, the proof is complete.
\end{proof}

\begin{lem}\label{lem:existsnicepath} Suppose $h\colon \Sigma\to [0,2]$ is a Morse function, all of whose critical points have index 1  and all of whose critical values are in $(0,1)$. If $f_0,f_1\in {\ve{M}}^{\nice}(W,\Sigma;h)$, then there is a smooth path $f_t$ from $f_0$ to $f_1$ such that $f_t$ is a nice Morse function except at finitely many $t$, where one of the following occurs:
\begin{enumerate}
\item Two critical points in $W\setminus \Sigma$ exchange relative value.
\item An index 1/2 or index 2/3 birth-death singularity occurs in $W\setminus \Sigma$.
\end{enumerate}
\end{lem}

\begin{proof}Using Lemma~\ref{lem:genericpath}, we can find a generic path $f_t$ from $f_0$ to $f_1$, which is in ${\ve{M}}(W,\Sigma;h)$ for all but finitely many $t$, where a critical value switch occurs, or a birth-death singularity occurs in $W\setminus \Sigma$. We pick a path of gradient-like vector fields $v_t$ for $f_t$, such that $v_t|_{\Sigma}$ is tangent to $\Sigma$ for all $t$ (it is straightforward to show that such vector fields exist).  We note that by assumption, the descending manifolds of $\Crit(f|_{\Sigma})$ are contained in $\Sigma$, and hence are disjoint from the ascending manifolds of $\Crit(f_t|_{W\setminus \Sigma})$. Hence we can modify the path $f_t$ so that all critical points of any $f_t|_{W\setminus \Sigma}$ occur in the interval $(1,2)$, using the ``independent trajectories principle'' (\cite{KirbyCalculus}*{pg. 40}).

Now one can use the techniques of Cerf graphics (see \cite{KirbyCalculus}*{Section~3}, \cite{CerfTheory}, \cite{HatcherWagoner}) to modify the path $f_t$ so that all index 1 critical points occur before the index 2 critical points, which occur before the index 3 critical points, and so that there are also no index 0 or 4 critical points. Generically, the path $f_t$ will be Morse and have distinct critical values for all but finitely many $t\in [0,1]$, where a critical point birth-death may occur, or a critical value switch may occur.
\end{proof}

\subsection{Connecting Morse functions which disagree on $\Sigma$}
\label{subsec:changingtheMorsefunctiononSigma}

In this section, we describe how to connect two Morse functions which may not agree on $\Sigma$. As a first step, we prove a version of \cite{GWW}*{Lemma 3.1} for link cobordisms:

\begin{lem}\label{lem:nicepath->isotopy} Suppose that $f_t\colon W\to [0,2]$ is a path of Morse functions on $(W,\Sigma)$ such that $f_t\in {\ve{M}}^{\nice}(W,\Sigma;f_t|_\Sigma)$ for all $t$ (in particular all critical values are distinct, and there are no birth-death singularities of $f_t$). Then there is a smooth isotopy $\psi_t\colon (W,\Sigma)\to (W,\Sigma)$ such that $\psi_t|_{\d W}=\id$ for all $t$, and an isotopy $\phi_t\colon \R\to \R$ such that
\[f_t=\phi_{t}\circ f_0 \circ \psi_t^{-1}.\]
\end{lem}
\begin{proof}If we do not require $\psi_t$ to map $\Sigma$ to $\Sigma$, then \cite{GWW}*{Lemma 3.1} gives us the result immediately. We briefly sketch the modification of their proof, to the setting of link cobordisms.  First pick an isotopy $\phi_t\colon \R\to \R$ such that $ f_t$ and $\phi_t\circ f_0$ have the same critical values. Such a $\phi_t$ exists because by assumption all critical values of $f_t$ are distinct, and there are no critical point birth-death singularities. Write $g_t$ for $\phi_t^{-1}\circ f_t$. Pick a path of gradient-like vector fields $V_t$ for $g_t$ such that $V_t|_{\Sigma}$ is tangent to $\Sigma$. Away from $\Sigma$ and the critical points of $g_t$, we define a vector field
\[\hat{V}_t=-\frac{\d_t g_t}{V_t(g_t)}\cdot V_t.\] Note that
\[dg_t(\hat{V}_t)+\d_t g_t=0.\] Extend $\hat{V}_t$ over a neighborhood of a critical point $p_t$ of $g_t$ by picking a smooth path of embeddings $\theta_t\colon B^4\to W$, centered at $p_t$, such that if $p_t\in \Sigma$, then
\[
(g_t\circ \theta_t)(x,y,w,z)=c-x^2+y^2+w^2+z^2
\] (with $\Sigma$ locally given by $\{(x,y,0,0)\}$),  and  defining
\[\hat{V}_t(y)=\frac{d}{d s}\bigg|_{s=0} \theta_{s+t}(\theta_t^{-1}(y)).\] We use a similar formula near critical points which are not on $\Sigma$. We patch together the various definition of $\hat{V}_t$ using partitions of unity, and note that they satisfy
\begin{equation}d g_t(\hat{V}_t)+\d_t g_t=0.\label{eq:diffeq}\end{equation} We define $\psi_t$ to be the flow of the time dependent vector field $\hat{V}_t$, and note that Equation~\eqref{eq:diffeq} implies $g_t\circ \psi_t$ is constant, and hence equal to $f_0$. Hence $f_t=\phi_t\circ f_0\circ \psi_t^{-1}$. Finally, we note that $\psi_t$ preserves $\Sigma$, since $\hat{V}_t|_{\Sigma}$ is tangent to $\Sigma$.
\end{proof} 

\begin{lem}\label{lem:changeMfonSigmanocriticalvalueswaps}Suppose that $f\in \ve{M}^{\nice}(W,\Sigma)$ and $h_t:\Sigma\to [0,2]$ is a smooth path of Morse functions on $\Sigma$ such that $h_0=f|_\Sigma$ and each $h_t$ has distinct critical values in the interval $(0,1)$. Then there is a path $f_t$ such that each $f_t\in \ve{M}^{\nice}(W, \Sigma;h_t)$ and $f_0=f$.
\end{lem}

\begin{proof}The proof is similar to the proof of Lemma~\ref{lem:nicepath->isotopy}. Using the techniques of that lemma, we can find smooth isotopies $\phi_t\colon\R\to \R$ and $\psi_t\colon\Sigma\to \Sigma$, such that $\psi_t$ is fixed on $\d \Sigma$ such that $h_t=\phi_t\circ h_0\circ \psi_t^{-1}$. Furthermore, since $h_t$ has no critical values in $[1,2]$, we can take $\phi_t$ to be the identity on $[1,2]$. One simply extends $\psi_t$ to a smooth isotopy  $\Psi_t\colon W\to W$, in such a way that $\Psi_t$ is fixed on $\d W$, and is the identity on all of the critical points of $f|_{W\setminus \Sigma}$. We define the path $f_t:=\phi_t\circ f_0\circ \Psi_t^{-1}$.
\end{proof}

\begin{lem}\label{lem:changeMorsefunctionsalongSigma}Suppose that $(W,\Sigma)$ is a link cobordism, $(h_t)_{t\in [0,1]}\colon\Sigma\to \R$ is a smooth path of Morse functions which have only index 1 critical points, and $f\in \ve{M}^{\nice}(W,\Sigma; h_0)$ is a nice Morse function. After modifying the path $h_t$ slightly near each critical value switch, there is a smooth family of Morse functions $f_t$ such that $f_t\in \ve{M}^{\nice}(W,\Sigma;h_t),$ except at finitely many $t$, where a critical value switch between two critical points in $\Crit(h_t)$ occurs.
\end{lem}

\begin{proof} Since the path $h_t$ is Morse for all $t$, and has only index 1 critical points, we can perturb $h_t$ slightly so that there are only finitely many times $t$ when the critical points of $h_t$ do not have distinct critical values. By modifying $h_t$ slightly, we can subdivide $[0,1]$ by picking finitely many points of time $0=t_0<t_1<\dots < t_n=1$ such that on each interval $[t_i,t_{i+1}]$ there is at most one critical value switch. Note that Lemma~\ref{lem:changeMfonSigmanocriticalvalueswaps} handled the case that there are no critical value switches of $h_t$ on $[t_i,t_{i+1}]$. Hence it is sufficient to prove the lemma statement for each subinterval $[t_i,t_{i+1}]$ individually. 

We now consider the claim in the case that $[t_i,t_{i+1}]$ contains a single critical value switch. On the interval $[t_i,t_{i+1}]$, we will describe how to modify the path $h_t$ so that it satisfies a simple standard model. Let us renormalize the interval $[t_i,t_{i-1}]$ to be $[-1,1]$. Cerf \cite{CerfTheory}*{pg. 41} describes standard models for critical value switches which he calls \emph{chemin descendant standard} and \emph{chemin ascendant standard}. We describe a slight modification of Cerf's construction. Suppose $h_t$ is a path of Morse functions on $\Sigma$ which has distinct critical values except at $t=0$, where two critical points, $p$ and $p'$, are involved in a critical value switch. Assume that  $h_{-1}(p)>h_{-1}(p')$ while $h_{1}(p)<h_1(p')$. Let us pick a bump function $\hat{\omega}$ which is supported in a small neighborhood of $p$ and $p'$ such that $\hat{\omega}\equiv+1$ near $p$ and $\hat{\omega}\equiv-1$ near $p'$.

We now describe how to modify the path $h_t$ so that for $t$ in a neighborhood of $0$, the path $h_t$ has a very simple local form. Fix a small $\epsilon>0$, and replace $(h_t)_{t\in [-1,1]}$ with the concatenation of the following  five paths of Morse functions:
\begin{enumerate}
\item The path $(h_{-1} +t \hat{\omega})_{t\in [0,\epsilon]}$ from $h_{-1}$ to $h_{-1}+\epsilon\hat{\omega}$.
\item The path $(h_{t}+\epsilon\hat{\omega})_{t\in [-1,0]}$ from $h_{-1}+\epsilon\hat{\omega}$ to $h_0+\epsilon\hat{\omega}$.
\item The path $(h_0-t\epsilon \hat{\omega})_{t\in [-1,1]}$ from $h_0+\epsilon\hat{\omega}$ to $h_0-\epsilon\hat{\omega}$.
\item The path $(h_t-\epsilon\hat{\omega})_{t\in [0,1]}$ from $h_0-\epsilon\hat{\omega}$ to $h_{-1}-\epsilon\hat{\omega}$.
\item The path $(h_1+t\hat{\omega})_{t\in [-\epsilon,0]}$ from $h_1-\epsilon\hat{\omega}$ to $h_1$.
\end{enumerate}

Note that paths (1), (2), (4) and (5) have distinct critical values, and hence Lemma~\ref{lem:changeMfonSigmanocriticalvalueswaps} applies to those paths. It remains to prove the claim for path (3). To extend the path $(h_0-t\epsilon \hat{\omega})_{t\in [-1,1]}$ to all of $W$, given an extension $f_{-1}$ for $t=-1$ we simply extend the map $\hat{\omega}$ to a small neighborhood of $p$ and $p'$ in $W$, and then use the formula $f_{t}=f_{-1}-(t+1)\epsilon \hat{\omega}.$ We note that except for $p$ and $p'$, the critical points of $f_{-1}$ and their critical values are unchanged by this procedure. Hence if $f_{-1}\in \ve{M}^{\nice}(W,\Sigma; h_{-1})$ then $f_t\in \ve{M}^{\nice}(W,\Sigma;h_t)$, except at $t=0$, where a critical value switch occurs.
\end{proof}

\subsection{Morse functions which respect the dividing set}

Since we cannot define link Floer homology when a link component has no basepoints, we need to consider Morse functions which satisfy some extra compatibility requirements requirements with respect to the dividing set $\cA$.

\begin{define}Suppose that $(W,\cF)$ is a decorated link cobordism, with $\cF=(\Sigma,\cA)$. If $f\colon W\to [0,2]$ is a Morse function in $\ve{M}^{\nice}(W,\Sigma)$, we say $f$ is \emph{$\cA$-compatible} if the following hold:
\begin{enumerate}
\item Each component of $f|_{\Sigma}^{-1}(t)$ intersects $\cA$ non-trivially for each $t$.
\item $f|_{\cA}$ is Morse.
\item All critical points of $f$ and $f|_{\cA}$ have distinct values.
\item All of the critical values of $f|_{\cA}$ are contained in the interval $[0,1]$.
\end{enumerate}
We  write $\ve{M}^{\nice}_{\cA}(W,\Sigma)$ for the set of $\cA$-compatible, nice Morse functions. If $h$ is a fixed Morse function on $\Sigma$ with only index 1 critical points, we write $\ve{M}^{\nice}_{\cA}(W,\Sigma;h)$ for the elements of $\ve{M}^{\nice}_{\cA}(W,\Sigma)$ which restrict to $h$ on $\Sigma$.
\end{define}

\begin{lem}\label{lem:sufficientlynicepath}Suppose that $f_0,f_1\in \ve{M}^{\nice}_{\cA}(W,\Sigma)$. Then there is a path $f_t\colon W\to [0,2]$ of smooth functions which are in $\ve{M}^{\nice}_{\cA}(W,\Sigma)$ except for finitely many $t$, where one of the following occurs:
\begin{enumerate}
\item An index 1/2 or 2/3 birth-death singularity occurs between two critical points in $W\setminus \Sigma$, with critical values contained in the interval $(1,2)$.
\item A critical value switch occurs between two critical points of $f|_{W\setminus \Sigma}$ of the same index, or between two critical points of $f|_{\Sigma}$.
\item An index 0/1 birth-death singularity of $f_t|_{\cA}$ occurs along $\cA$.
\item A critical value switch occurs between two critical points of $f_t|_{\cA}$.
\item A critical value switch occurs between a critical point of $f_t|_{\cA}$ and a critical point of $f_t|_{\Sigma}$.
\item A critical point of $f_t|_{\Sigma}$ crosses an arc of $\cA$, transversely.
\end{enumerate}
\end{lem}

\begin{proof}First we claim that is sufficient to construct a path of Morse functions $h_t\colon\Sigma\to [0,2]$ between $f_0|_{\Sigma}$ and $f_1|_{\Sigma}$, with critical values in $[0,1]$, such that $h_t$ and $h_t|_{\cA}$ have distinct critical values for all but finitely many $t$, where one of the singularities (3), (4), (5) and (6) can occur, and such that each component of $h_t^{-1}(s)$ intersects $\cA$ non-trivially. If we have such a path, Lemma~\ref{lem:changeMfonSigmanocriticalvalueswaps} allows us to extend this to a path $\hat{f}_t$ from $f_0$ to some other Morse function $\hat{f}_1\in \ve{M}^{\nice}(W,\Sigma; f_1|_{\Sigma})$. Lemma~\ref{lem:genericpath}, then allows us to connect $\hat{f}_1$ and $f_1$ via a path of functions which is fixed on $\Sigma$, and which are nice at all but finitely $t$, where one of singularities (1) and (2) may occur.

Hence, it is sufficient to construct a suitable path $h_t\colon \Sigma\to [0,2]$. We can construct a path $h_t$ which potentially has index 0/1 or index 1/2 birth-death singularities, however standard Morse theory techniques allow us to reduce to the case that there are no index 0 or index 2 critical points throughout (note that by assumption $f_0|_\Sigma$ and $f_1|_{\Sigma}$ have none). It follows that we can construct a path $h_t$ which is Morse for all $t$. Via a generic perturbation, we can assume that all of the critical values of $h_t$ and $h_t|_{\cA}$ are all distinct, except at finitely many $t$ where one of the singularities (3), (4), (5) or (6) occurs. However, we still need to reason that we can pick a path $h_t$ such that  each connected component of $h_t^{-1}(s)$ intersects $\cA$ non-trivially, for all $t$ and $s$.

To demonstrate that we can construct $h_t$ so that each component of $h_t^{-1}(s)$ intersects $\cA$ non-trivially, we argue as follows. First, pick any path $h_t$, which has only singularities of type (3), (4), (5) and (6). Since each $h_t$ has only index 1 critical points, each component of each $h_t^{-1}(s)$ represents a non-zero class in $H_1(\Sigma;\Z)$. Hence we can pick a collection of arcs $\ve{E}$ in $\Sigma$, with boundary on $\d \Sigma$, such that each component of $h_t^{-1}(s)$ intersects an arc in $\ve{E}$ non-trivially. We isotope the boundaries of the curves in $\ve{E}$ so that each arc in $\ve{E}$ intersects the interior of an arc in $\cA$. Let $\ve{E}_0$ denote the collection of arcs in $\Int(\Sigma)$ obtained by removing a small portion of the boundary of each arc in $\ve{E}$. Let $\phi_t$ be an isotopy of the surface $\Sigma$, which is fixed on $\d \Sigma$ and is supported in a neighborhood of $\ve{E}_0$. By picking $\phi_t$ appropriately (e.g. viewing a neighborhood of each component of $\ve{E}_0$ as a disk, and having $\phi_t$ rotate these disks), we can arrange that $\phi_t(\cA)$ intersects each component of each level set of $h_0$ and $h_1$, and also that $\phi_1(\cA)$ intersects each component of each level set of each $h_t$. We then consider the path $h_t'$ defined by concatenating the three paths $(h_0\circ \phi_t)_{t\in [0,1]}$, $(h_t\circ \phi_1)_{t\in [0,1]}$ and $(h_1\circ \phi_{1-t})_{t\in [0,1]}$. 
\end{proof}

\subsection{Morse functions and parametrized Kirby decompositions}
\label{subsec:MorsedataKirbydecomp} We now describe how a Morse function on a decorated link cobordism gives a parametrized Kirby decomposition, well-defined up to Cerf equivalence.

\begin{define}\label{def:niceMorsedata} Suppose that $(W,\cF)$ is a decorated link cobordism with $\cF=(\Sigma,\cA)$ and  $f\colon W\to [0,2]$ is a Morse function with a collection of regular values $\ve{b}\subset [0,2]$. Let us write $\ve{b}=\{b_0,\dots, b_{n+1}\}$ with $0=b_0<b_1<\cdots< b_n<b_{n+1}=2$. We define 
\[
 W_i:=f^{-1}([b_i,b_{i+1}])\qquad \text{and}\qquad  \Sigma_i:=\Sigma\cap W_i.
 \]
  We say that $(f,\ve{b})$ is \emph{Kirby-type Morse data} if $f\in \ve{M}^{\nice}_{\cA}(W,\Sigma)$, and on each $(W_i,\Sigma_i)$, exactly one of the following holds:
\begin{enumerate}
\item $f$ and $f|_{\cA}$ have no critical points.
\item  $f$ has a single critical point, which is of index 1 or 3 and occurs in $W_i\setminus \Sigma_i$. Furthermore $f|_{\cA}$ has no critical points.
\item $f$ has (arbitrarily many) critical points, all of which are index 2 and are in $W_i\setminus \Sigma_i$. Furthermore $f|_{\cA}$ has no critical points.
\item $f$ has a single critical point, which is of index 1 and occurs in $\Sigma_i$. Furthermore $f|_{\cA}$ has no critical points.
\item $f$ has no critical points. Furthermore $f|_{\cA}$ has a single critical point. 
\end{enumerate} 
Furthermore, we assume only one $W_i$ contains any index 2 critical points.
\end{define}

If $(W,\cF)$ is a decorated link cobordism and $(f,\ve{b})$ is a collection of Kirby-type Morse data we now describe how to construct a parametrized Kirby decomposition
\[\cK(f,\ve{b}).\] In Lemma~\ref{lem:differentchoicesvectorfield=>weakequiv}, below, we show that $\cK(f,\ve{b})$ is well-defined up to Cerf equivalence. Let us write
\[
Y_i=f^{-1}(b_i),\qquad L_i=Y_i\cap \Sigma, \qquad \cA_i=\Sigma_i\cap \cA\qquad \text{and} \qquad {\cF}_i=(\Sigma_i,\cA_i).\]

  We note that, except for the lack of parametrization, each $(W_i,\cF_i)$ satisfies the definition of an elementary link cobordism  (Definition~\ref{def:parametrizedelementarycob}). Furthermore, each $(W_i,\cF_i)$ is equipped with a Morse function $f_i:=f|_{W_i}$, whose restriction to $\Sigma_i$ and $\cA_i$ is also Morse. To construct a parametrized Kirby decomposition, we need to construct framed spheres $\bS_i\subset Y_i$ as well as parametrizing diffeomorphisms 
\[
\Phi_i\colon \cW(Y_i,L_i,\bS_i)\to (W_i,\Sigma_i).
\]

Suppose first that the Morse function $f_i$ on $(W_i,\Sigma_i)$ has no critical points. We set $\bS_i=\bS_{\varnothing}$,  we pick a vector field $v_i$ on $W_i$ such that $v_i(f_i)>0$ and $v_i|_{\Sigma}$ is tangent to $\Sigma$.  The flow of $v_i$ (suitably normalized) gives a parameterizing diffeomorphism 
\[
\Phi_i\colon ([0,1]\times Y_i, [0,1]\times L_i)\to (W_i, \Sigma_i).
\]
 Also note that the space of such vector fields is connected, so any two parametrizations $\Phi_i$, made in this way, are isotopic relative to $\{0\}\times Y_i$. The parametrized cobordism $(W_i, {\cF}_i, \Phi_i,\bS_{\varnothing} )$ is then an elementary parametrized cobordism of type ($\mathcal{EPC}$-\ref{elemencobtype1}) or ($\mathcal{EPC}$-\ref{elemencobtype3}), depending on whether $f_i|_{\cA}$ has a critical point or not.

Now suppose that $f_i|_{W_i\setminus \Sigma_i}$ has a single index 1 or index 3 critical point, or arbitrarily many index 2 critical points. We will construct a framed sphere or link $\bS_i\subset Y_i$ and a parameterizing diffeomorphism $\Phi_i$  which turns $(W_i,\cF_i)$  into an elementary parametrized cobordism of type ($\mathcal{EPC}$-\ref{elemencobtype2}). Suppose first that $f_i$ has a single critical point of index 1, 2 or 3. A gradient-like vector field $v_i$ such that $v_i|_{\Sigma}$ is tangent to $\Sigma$, as well as a choice of local coordinates in a neighborhood of the critical point give a well-defined parametrization, as we now describe. Our description is based on \cite{Palais}*{Section~11} and \cite{JClassTQFT}*{Section~2.3.1}.

Write $p$ for the critical point of $f_i$, and write for $k\in \{1,2,3\}$ for the index of $p$. We assume $p\in W_i\setminus \Sigma_i$. We first consider the case that $p$ is the only critical point of $f_i$. For notational convenience, we will describe the handle attachment procedure for arbitrary ambient dimension $n$ (though we only need the case that $n=2$ or $n=4$). Pick coordinates $\vec{x}=(x_1,\dots, x_n)$ near $p$ such that $f=f(p) -(x_1^2+\cdots+ x_k^2)+(x_{k+1}^2+\cdots+ x_n^2)$. We assume that the gradient-like vector field $v_i$ satisfies
\[
v_i=2\left(-x_1\frac{\d}{\d x_1}-\cdots -x_k\frac{\d}{\d x_{k}}+x_{k+1}\frac{\d}{\d x_{k+1}}+\cdots +x_n\frac{\d}{\d x_n}\right)
\] 
near $p$. The descending manifold is equal to $\{(x_1,\dots x_k,0,\dots,0)\}$ and the ascending manifold is $\{(0,\dots, 0,x_{k+1},\dots,x_n)\}$. Fix a small $\epsilon>0$, such that the ball of radius $2\epsilon$ is contained in the image of the coordinates $\vec{x}$. 

We pick a smooth function $\lambda\colon \R\to \R$ which has $\lambda'(t)\le 0$ and satisfies $\lambda(t)=1$ for $t\le \tfrac{1}{2}$, $\lambda(t)>0$ if $t<1$ and $\lambda(t)=0$ if $t\ge 1$. Notice that the space of such functions $\lambda$ is convex. Consider the function $g\colon W_i\to [b_i,b_{i+1}]$ defined by
\[
g:=f_i-\frac{3\epsilon}{2} \lambda\left(\frac{x_{k+1}^2+\cdots +x_n^2}{\epsilon}\right).
\] Note that 
\[
f_i^{-1}([b_i,c-\epsilon])\subset g^{-1}([b_i, c-\epsilon]).
\]
 Define the subset
\[
H:=g^{-1}([b_i, c-\epsilon])\setminus f_i^{-1}([b_i,c-\epsilon])\subset W_i.
\] 
According to \cite{Palais}*{pg. 316}, $H$ is homeomorphic to $D^k\times D^{n-k}$, via a homeomorphism which can be canonically specified in terms of $\lambda,$ $\epsilon$ and $\vec{x}$ (we omit the explicit formula for the homeomorphism, since it is rather complicated and not important for our purposes). Furthermore this homeomorphism restricts to a diffeomorphism on $(\Int D^k)\times D^{n-k}$, and also restricts to a smooth embedding of $S^{k-1}\times D^{n-k}$ into $f_i^{-1}(c-\epsilon)$, which is a framed $k$-sphere in $f_i^{-1}(c-\epsilon)$. The normalized gradient-like vector field, $v_i/v_i(f_i)$, induces a well-defined diffeomorphism between $[0,1]\times Y_i$ and $f_i^{-1}([b_i,c-\epsilon])$. The framed $k$-sphere in $f_i^{-1}([b_i,c-\epsilon])$ thus naturally induces a framed $k$-sphere $\bS_i$ in $\{0\}\times Y_i $, by using this identification.  The homeomorphism of $H$ with $D^k\times D^{n-k}$ described above thus yields a diffeomorphism of $([0,1]\times Y_i)\cup_{\bS_i} (D^k\times D^{n-k})$ with $g^{-1}([b_i, c-\epsilon])$, where the smooth structure on the former space is obtained by attaching the handle and then smoothing corners.  By \cite{Palais}*{pg. 319}, the gradient-like vector field $v_i$ is transverse to the level sets of $g$, except at the single critical point $p$ of $g$ (which is also the critical point of $f_i$). Hence by flowing along $v_i$, the subset $f_i^{-1}([b_i,b_{i+1}])\setminus g^{-1}([b_i, c-\epsilon])$ is diffeomorphic to $[0,1]\times g^{-1}(c-\epsilon) $, via a diffeomorphism which is uniquely specified up to isotopy, relative to $g^{-1}(c-\epsilon)$. By picking a collar neighborhood of $g^{-1}(c-\epsilon)$ in $g^{-1}([b_i, c-\epsilon])$ induced by the flow of $v_i$, we obtain a diffeomorphism $\Phi_i$ between  $([0,1]\times Y_i )\cup_{\bS_i} (D^k\times D^{n-k})$ and $W_i$ which is determined up to isotopies fixing $\{0\}\times Y_i $. Since $v_i$ is tangent to $\Sigma_i$, the map $\Phi_i$ sends $[0,1]\times L_i$ to $\Sigma_i$, and is hence a parametrization of link cobordisms.

Note that even for a fixed gradient-like vector field and coordinates near $p$, the construction still relies on a choice of $\epsilon$ and $\lambda$, though the space of such $\lambda$ is convex, so any two choices yield framed spheres which are isotopic away from $L_i$, which are related by Move \eqref{def:weakequiv:isotopyawayfromL}.

In the case that $(W_i,{\cF}_i)$ is an elementary cobordism with multiple  index 2 critical points, we can still use the above strategy to construct a parametrization $\Phi_i\colon \cW(Y_i,L_i,\bS_i)\to (W_i,\Sigma_i)$. The space of gradient-like vector fields for $f_i$ is still connected. A generic gradient-like vector field will have ascending and descending manifolds of the critical points which are transverse, and hence generically the intersection of the ascending and descending manifolds of two index 2 critical points will be empty, allowing us to perform the above construction whenever the choice of gradient-like vector field is generic (paths of gradient-like vector fields may of course result in handleslides, as we discuss in the following lemma).

We lastly consider cobordisms of type ($\mathcal{EPC}$-\ref{elemencobtype4}), where the critical point $p$ is of index 1, and $p\in \Sigma_i$. In this case, we pick a gradient-like vector field $v_i$ and oriented coordinates $\vec{x}=(x,y,w,z)$, defined in a neighborhood $U$ containing the critical point $p$, such that 
\[
f_i=f_i(p)-x^2+y^2+w^2+z^2,\quad \vec{x}(U\cap \Sigma)\subset \{(x,y,0,0)\}\quad \text{and} \quad v_i=2\left(-x\tfrac{\d}{\d x}+y\tfrac{\d}{\d y}+w\tfrac{\d}{\d w}+z\tfrac{\d}{\d z}\right).
\]
 The space of such coordinates, which are oriented for both $W$ and $\Sigma$, has two connected components by Lemma~\ref{lem:linearization3}. Using one of these two sets of coordinates, we can now adapt the construction defined above for critical points in $W_i\setminus \Sigma_i$. In this case, using the function $g$ constructed above, we obtain a diffeomorphism $\Phi_i$ between $W_i$ and $([0,1]\times Y_i)\cup H\cup ( [1,2]\times g^{-1}(c+\epsilon))$. Furthermore $H$ is canonically identified as $D^1\times D^3$ by our choice of local coordinates centered at $p$. Notice that the local coordinates also identify $\Sigma_i$ with a subset of $\{(x,y,0,0)\}\subset \R^4$. It is straightforward to check that using this identification, the band region $B\subset D^1\times D^3$ defined by $B=\{(x,y,0,0):-1\le x\le 1, -1\le y\le 1\}$ is mapped into $\Sigma_i$, and the identification also sends $(0,0)\in B$ to the critical point of $f$. Hence the map $\Phi_i$ is in fact a parametrization of the link cobordism $(W_i,\Sigma_i)$. 

The above construction of $\cK(f,\ve{b})$ several choices of auxiliary data, such as a choice of local coordinates and gradient-like vector field. Nevertheless, we have the following:

\begin{lem}\label{lem:differentchoicesvectorfield=>weakequiv}If $(W,\cF)$ is a decorated link cobordism and $(f,\ve{b})$ is a collection of Kirby-type Morse data, then any two choices of gradient-like vector fields and local coordinates yield Cerf equivalent parametrized Kirby decompositions of $(W,\cF)$.
\end{lem}
\begin{proof}It is sufficient to show the claim when $|\ve{b}|=2$, and $(W,\cF)$ is an elementary cobordism, except for the fact that it lacks a parametrization. We first consider the case that $(W,\cF)$ is of type ($\mathcal{EPC}$-\ref{elemencobtype1}) or ($\mathcal{EPC}$-\ref{elemencobtype2}). We further restrict to the case that the Morse function $f$ has at most one critical point. Recall that the construction of $\cK(f,\ve{b})$ depended on a choice of $\epsilon>0$, bump function $\lambda$, local coordinates $\vec{x}$ and gradient-like vector field $v$.

First, note that changing  $\epsilon$ or the bump function $\lambda$ only affects the parametrization up to an $\cA$-adapted isotopy or an isotopy of $\bS_i$ away from $L_i$ (Moves~\eqref{def:weakequiv:strongisotopy} and \eqref{def:weakequiv:isotopyawayfromL}). 

 If $\vec{x}=(x_1,x_2,x_3,x_4)$ is a choice of coordinates in $\Coor_{p}(U,f)$, defined in a neighborhood $U$ of $p$, let $\cV(W,\Sigma,\vec{x}, f)$ denote the space of vector fields $v$ which have the prescribed form in a neighborhood of $p$, with respect to $\vec{x}$, such that $v|_{\Sigma}$ is tangent to $\Sigma$, and $v(f)>0$ away from $p$. The space $\cV(W,\Sigma, \vec{x},f)$ is convex, and hence connected. It is straightforward  to also see that $\cV(W,\Sigma,\vec{x},f)$ is non-empty. If $v_1,v_2\in \cV(W,\Sigma,\vec{x},f)$, we can connect them by a path in $\cV(W,\Sigma,\vec{x},f)$, which results in changing the parametrization by Moves~\eqref{def:weakequiv:strongisotopy} and \eqref{def:weakequiv:isotopyawayfromL}.

Next, suppose $\vec{x}_0$ and $\vec{x}_1$ are two choices of coordinates. By examining the construction, we see that restricting the domain of a set of coordinates has no effect, provided $\epsilon$ is chosen sufficiently small. Hence, we may assume that $\vec{x}_0$ and $\vec{x}_1$ are defined on the same neighborhood of $p$. By Lemma~\ref{lem:linearization}, the space of local coordinates which put $f$ into standard form has two components, which can be related by simultaneously reversing the signs of the first and last coordinates. Switching signs of these coordinates results in Move~\eqref{def:weakequiv:changeorientationofS}. If $\vec{x}_0$ and $\vec{x}_1$ are homotopic via a path $(\vec{x}_t)_{t\in [0,1]}$ in $\Coor_p(U,f)$, we first pick a $v_0\in \cV(W,\Sigma, \vec{x}_0,f)$. It is straightforward to extend $v_0$ to a path $(v_t)_{t\in [0,1]}$ of vector fields such that $v_t\in \cV(W,\Sigma, \vec{x}_t, f)$ for all $t\in [0,1]$. Hence, changing from $\vec{x}_0$ to $\vec{x}_1$ results in an $\cA$-admissible isotopy and an isotopy of the framed sphere away from $L_i$ (Moves~\eqref{def:weakequiv:strongisotopy} and \eqref{def:weakequiv:isotopyawayfromL}).

We now consider the case that $(W,\cF)$ is of type~($\mathcal{EPC}$-\ref{elemencobtype2}) but $f$ has multiple index 2 critical points. In this case, if $\vec{x}$ is a choice of coordinates on a neighborhood of all of the critical points, then the descending and ascending manifolds of a generic $v\in \cV(W,\Sigma,\vec{x},f)$ will be transverse, and hence disjoint. For a generic path of vector fields in $\cV(W,\Sigma, \vec{x},f)$, there will be finitely many points of time when there is a single flow line between two critical points of $f$. At such a point of time, a handleslide of one of the framed 1-spheres occurs, which results in Move \eqref{def:weakequiv:handleslide}. Over the course of time when the path of gradient-like vector fields have no flow lines amongst each other, the associated parametrized Kirby decompositions are related to each other by a sequence of Moves~\eqref{def:weakequiv:strongisotopy} and \eqref{def:weakequiv:isotopyawayfromL}.

Finally, we note that the case when $(W,\cF)$ is of type~($\mathcal{EPC}$-\ref{elemencobtype4}) (a framed $0$-spheres along $L$) can be handled similarly to the case of framed $0$-spheres away from $L$, except we use Lemma~\ref{lem:linearization3} to connect two coordinate charts defined in a neighborhood of  $p$.
\end{proof}

\begin{lem}\label{lem:findMorsedata}Given a parametrized Kirby decomposition $\cK$ of $(W,{\cF})$, one can find Kirby-type Morse data $(f,\ve{b})$ for $(W,{\cF})$ such that
\[\cK=\cK(f,\ve{b}).\] 
\end{lem}
\begin{proof} One can modify the construction from \cite{JClassTQFT}*{Lemma 2.14} to get a Morse function on each elementary cobordism which induces the correct parametrization for a choice of gradient-like vector field. By modifying the Morse functions near the boundary, one can ensure that they glue together to form a smooth function on all of $W$.
\end{proof}

\subsection{Proof of Theorem~\ref{prop:allPKDSigmasweaklyequivalent}}
\label{subsec:ProofofallPKDsequivalent}
We now have the tools to prove that all parametrized Kirby decompositions of a decorated link cobordism $(W,{\cF})$ are Cerf equivalent:
\begin{proof}[Proof of Theorem~\ref{prop:allPKDSigmasweaklyequivalent}] Suppose $\cK_0$ and $\cK_1$ are two parametrized Kirby decompositions of $(W,\cF)$. We use Lemma~\ref{lem:findMorsedata} to find Morse functions $f_0, f_1\in \ve{M}^{\nice}_{\cA}(W,\Sigma)$, together with collections of regular values, $\ve{b}_0, \ve{b}_1\subset [0,2]$,  choices of local coordinates centered at the critical points and gradient-like vector fields so that
\[
\cK_0=\cK(f_0,\ve{b}_0), \qquad \text{ and } \qquad \cK_1=\cK(f_1,\ve{b}_1).
\] 
Using Lemma~\ref{lem:sufficientlynicepath}, we can find a path $(f_t)_{t\in [0,1]}$ between $f_0$ and $f_1$ such that $f_t\in \ve{M}^{\nice}_{\cA}(W,\Sigma)$ for all but finitely many $t$, where one of the six singularities listed in Lemma~\ref{lem:sufficientlynicepath} occurs. Let us write $0<t_1<\cdots< t_n<1$ for the points in $[0,1]$ where a listed singularity occurs.

 For all $t\in [0,1]\setminus \{t_1,\dots, t_n\}$, the function $f_t$ is in $\ve{M}_{\cA}^{\nice}(W,\Sigma)$, and hence on any interval $[a,b]$ not containing a point in $\{t_1,\dots, t_n\}$ we can construct a set of regular values $\ve{b}_t$ so that $(f_t,\ve{b}_t)$ is Kirby-type Morse data. Furthermore, by picking local coordinates and gradient-like vector fields, we can construct parametrized Kirby decomposition $\cK(f_t,\ve{b}_t)$. Outside of small neighborhoods of the $t_i$, we can assume that the values $\ve{b}_t$ vary continuously in $t$.
 
Suppose first that $[a,b]$ is an interval such that $f_t\in \ve{M}_{\cA}^{\nice}(W,\Sigma)$ for all $t\in [a,b]$. Since in particular $f_t$ has no 4-dimensional critical point birth-deaths or critical point value switches,  then by Lemma~\ref{lem:nicepath->isotopy}, we can find  isotopies $\phi_t\colon [0,2]\to [0,2]$ and $\psi_t\colon (W,\Sigma)\to (W,\Sigma)$, both ranging over $t\in [a,b]$, with  $\phi_a=\id$, $\psi_a=\id$ and $\psi|_{\d W}=\id_{\d W}$ for all $t$, such that 
 \begin{equation}
 f_t=\phi_t\circ f_a \circ \psi_t^{-1}.\label{eq:isotopyofMorsefunction}
\end{equation} 
  In particular, given a choice of $\ve{b}_a$, we can take the parametrized Kirby decomposition $\cK(f_a,\ve{b}_a)$, and push it forward along the diffeomorphism $\psi_t$ to get a parametrized decomposition (of undecorated link cobordisms) $(\psi_b)_*\cK(f_a,\ve{b}_a)$, which is the same as $\cK(f_b,\phi_t(\ve{b}_a))$ by Equation~\eqref{eq:isotopyofMorsefunction}. Since $f_t\in \ve{M}^{\nice}_{\cA}(W,\Sigma)$ for $t\in [a,b]$, it follows that all of the critical values of $f_t|_{\cA}$ and $f_t$ are distinct for each $t\in [a,b]$, and hence each term in the decomposition $(\psi_t)_* (\cK(f_a, \ve{b}_a))=\cK(f_a\circ \psi_t^{-1}, \ve{b}_a)$ will be an elementary parametrized cobordism, and hence $\cK(f_a, \ve{b}_a)$ and $\cK(\phi_b^{-1}\circ f_b, \ve{b}_a)$ differ by an $\cA$-adapted isotopy (Move~\eqref{def:weakequiv:strongisotopy}). Note that $\cK(\phi_b^{-1}\circ f_b,\ve{b}_a)$ is the same as $\cK(f_b, \phi_b(\ve{b}_a))$. The regular values $\phi_b(\ve{b}_a)$ can be changed to $\ve{b}_b$ by a sequence of applications  of Move \eqref{def:weakequiv:addtrivialcylinder}, to add and remove trivial cylinders. Hence $\cK(f_a,\ve{b}_a)$ and $\cK(f_b,\ve{b}_b)$ are Cerf equivalent.

Next, we suppose that over $[a,b]$ the 4-dimensional critical points of $f_t$ all have distinct critical values, but that there is a single point of time $t_i\in (a,b)$ where one of the following occurs: a birth-death singularity of $f_t|_{\cA}$ occurs; a critical value switch occurs between two critical points of $f_t|_{\cA}$; a critical value switch occurs between a critical point of $f_t|_{\cA}$ and a critical point of $f_t|_{\Sigma}$; or a critical point of $f_t|_{\Sigma}$ crosses an arc of $\cA$, transversely. In all three cases, we can still use Lemma~\ref{lem:nicepath->isotopy} to find isotopies $\phi_t$ and $\psi_t$ of $\R$ and $(W,\Sigma)$ (resp.) such that $f_t=\phi_t\circ f_a\circ \psi_t^{-1}$. Using this fact, it is straightforward to see that (after perhaps a small perturbation of the path $f_t$) the parametrized Kirby decompositions $\cK(f_{a}, \ve{b}_a)$ and $\cK(f_b, \ve{b}_b)$ differ by Move~\eqref{def:weakequiv:birthdeathalongA}, Move~\eqref{def:weakequiv:switchalongA}, Move~\eqref{def:weakequiv:switchbetweenAandSigma} or Move~\eqref{def:weakequiv:critcrossesA}.

Finally, we consider the case that $[a,b]$ is an interval of time where $f_t\in \ve{M}^{\nice}_{\cA}(W,\Sigma)$ except for a single point of time, where an index 1/2 or 2/3 birth-death singularity occurs in $W\setminus \Sigma$, or a critical value switch occurs between two critical points of $f_t$ of the same index. Consider first the case of a critical point birth-death singularity.  Near the point of time $t_i$ when the singularity occurs, by perturbing the path of functions slightly, we can assume that for a small subinterval of time around the singularity, the path of functions $f_t$ is supported in a small ball containing the critical points, and takes on a simple normal form. The existence of such normal forms is established in \cite{CerfTheory} (compare \cite{JClassTQFT}*{pg. 17}). As such the parametrized Kirby decompositions immediately preceding and following the singularity may be related by Move \eqref{def:weakequiv:index12birthdeath} or \eqref{def:weakequiv:index23birthdeath}. Next we consider the case of a critical value switch of critical points of the same index. Note that since the $f_t$ are nice Morse functions, if a critical value switch occurs, it either involves critical points which are both in $W\setminus \Sigma$, or critical points which are both on $\Sigma$. As with critical point cancellations, by perturbing the path $f_t$ slightly, we can assume that in a small interval of time around the singularity that $f_t$ is constant outside of a neighborhood of the two critical points, and inside of the neighborhood, it takes on a simple normal form (see the proof of Lemma~\ref{lem:changeMorsefunctionsalongSigma}). If the two critical points are in $W\setminus \Sigma$, then $\cK(f_a,\ve{b}_a)$ and $\cK(f_b, \ve{b}_b)$ can be taken to be equal if the two critical points have index 2, and are otherwise related by Move~\eqref{def:weakequiv:criticalpointswitchawayfromSigma} if the two critical points are of index 1 or 3. Similarly if the two critical points involved in the switch are contained in $\Sigma$, then they are both of index 1, and $\cK(f_a,\ve{b}_a)$ and $\cK(f_b, \ve{b}_b)$ are related by Move~\eqref{def:weakequiv:criticalpointswitchawayfromSigma}.
\end{proof}

We make an additional remark about Move~\eqref{def:weakequiv:critcrossesA}, where a critical point of $\Sigma$ crosses $\cA$. It is easier to visualize the move in terms of the dividing set changing, with a fixed Morse function on $(W,\Sigma)$ (these two viewpoints are related by Lemma~\ref{lem:nicepath->isotopy}).  In this case, we can view a subarc of $\cA$ (containing a critical point of $f_t|_{\cA}$) as being pushed through the critical point of $f_t|_{\Sigma}$. We demonstrate this pictorially in Figure~\ref{fig::57}. See also Figure~\ref{fig::59}, where we prove invariance of the move.

 \begin{figure}[ht!]
\centering
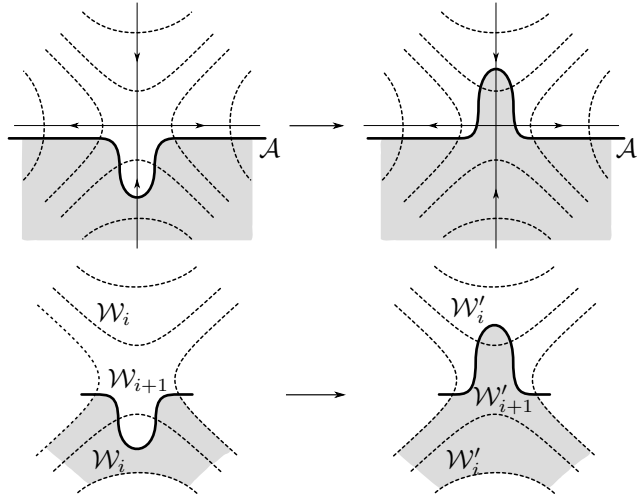
\caption{\textbf{A local model for Move \eqref{def:weakequiv:critcrossesA}.} In the picture, we view the Morse function as being fixed, while the dividing set is changed by an isotopy. The thick solid lines indicate $\cA$, while the dashed lines indicate the level sets of $f_t$, and the thin solid lines with arrows are upward gradient flowlines of $f_t$. On the bottom, we indicate the change in associated parametrized Kirby decomposition described in the definition of the move. The dotted lines correspond to level sets of the Morse function.\label{fig::57}}
\end{figure}

\section{Construction and invariance of the cobordism maps}
\label{sec:construction-and-invariance}

In this section, we provide a precise construction of the link cobordism maps and prove Theorem~\ref{thm:A}, invariance.

In Sections~\ref{sec:defmapscob1-2}--\ref{subsec:mapsfortype4} we define the maps for elementary parametrized link cobordisms. In Section~\ref{subsec:invariance} we define the maps in the case that each component of $W$ and $\cF$ intersects a component of $Y_1$ and $Y_2$ non-trivially, and prove invariance for the maps associated to such cobordisms. In Section~\ref{sec:removeballsforgeneralmaps}, we describe how to puncture a link cobordism to introduce additional boundary components in order to define the cobordism maps more generally. Finally, we show invariance of the link cobordism maps under this puncturing procedure, completing the proof of Theorem~\ref{thm:A}.

\subsection{Maps induced by elementary parametrized link cobordisms of type ($\mathcal{EPC}$-\ref{elemencobtype1}) and ($\mathcal{EPC}$-\ref{elemencobtype2})}
\label{sec:defmapscob1-2}

We now consider an elementary parametrized link cobordism $\cW=(W,\cF^\sigma,\Phi,\bS)\colon (Y_1,\bL_1)\to (Y_2,\bL_2)$ where $\bS=\bS_{\varnothing}$ (in which case $\cW$ is of type ($\mathcal{EPC}$-\ref{elemencobtype1})), or where $\bS$ is a framed 0-sphere,  1-dimensional link or 2-sphere which does not intersect $L_1$  (in which case $\cW$ is of type ($\mathcal{EPC}$-\ref{elemencobtype2})). In both cases, the dividing set is a collection of arcs from $L_1$ to $L_2$ and the underlying surface of $\cF$ is a disjoint union of annuli.

Define a dividing set $\cA_0\subset [0,1]\times L_1$ by the formula
\[
\cA_0:=\Phi^{-1}(\cA).
\] 
Note that $\cA_0$ is isotopic relative to $\{0,1\}\times L_1 $ to a dividing set, $\cA_0'$, which is transverse to $\{t\}\times L_1$, for each $t\in [0,1]$. The dividing set $\cA_0'$ on $[0,1]\times L_1$ induces an isotopy $\psi_t^0\colon [0,1]\times L_1\to L_1$ such that $\{t\}\times \psi_t^0(\ws_1\cup \zs_1)$ is disjoint from $\cA_0'$ for all $t$, and $\psi_1^0(\ws_1\cup \zs_1)=\Phi^{-1}(\ws_2\cup \zs_2)$. Extend $\psi_t^0$ to an isotopy $\psi_t$ of the pair $(Y_1,L_1)$ and define $\psi:=\psi_1$. We define
\[
F_{W,\cF^\sigma,\Phi,\bS,\frs}:=(\Phi|_{Y_1(\bS)})_*F_{Y_1,\bL_1,\bS,\Phi^{-1}(\frs)}\psi_*.
\]
We need to show the above expression is invariant under the choice of isotopic $\cA_0'$, as well as invariant under isotopies of $\Phi$, relative to $\{0\}\times Y_1$.  The following definition will be useful:

\begin{define}\label{def:admissibleisotopy1}Suppose that $\cA$ is a dividing set on $[0,1]\times L$. We say that $\cA$ is \emph{admissible} if each arc of $\cA$ which goes from $\{0\}\times L$ to $\{1\}\times L$ is transverse to $\{s\}\times L$ for each $s\in [0,1]$. If $\cA_t$ is a 1-parameter family of dividing sets on $[0,1]\times L$, which is fixed pointwise on $\{0,1\}\times L$, we say that $\cA_t$ is an \emph{admissible isotopy} if each $\cA_t$ is admissible.
\end{define}

We state the following simple lemma about admissible isotopies, which will also be helpful later in Section~\ref{subsec:cylindercobordismmaps}.

 \begin{lem}\label{lem:admissibleisotopiesofcylinders} Suppose $\cA_1$ and $\cA_2$ are two admissible dividing sets on $[0,1]\times L$, which are isotopic relative to $\{0,1\}\times L$. Then $\cA_1$ and $\cA_2$ are admissibly isotopic relative to $\{0,1\}\times L$.
 \end{lem}
 \begin{proof} For notational simplicity, assume that $L$ has only one component. Note that if an arc $A$ of $\cA_1$ goes from $\{0\}\times L$ to $\{0\}\times L$, then it bounds a bigon, which can only contain other arcs of $\cA_1$ which go from $\{0\}\times L$ to $\{0\}\times L$. Hence such arcs may be isotoped very close to $\{0\}\times L$ without moving any arcs which go from $\{0\}\times L$ to $\{1\}\times L$. Hence we may reduce to the case that no arcs go from $\{0\}\times L$ to $\{0\}\times L$. We can similarly reduce to the case that there are also no arcs which go from $\{1\}\times L$ to $\{1\}\times L$.
 
In the case that $\cA_1$ and $\cA_2$ have the same endpoints, and both only have arcs going from $\{0\}\times L$ to $\{1\}\times L$, it is straightforward to see that $\cA_1$ and $\cA_2$ are admissibly isotopic if and only if whenever $A_1$ and $A_2$ are arcs in $\cA_1$ and $\cA_2$ (resp.) which have the same endpoints, we have
 \[ \#(A_1\cap ([0,1]\times \{p\}))= \# (A_2\cap ([0,1]\times \{p\} ))\] for any $p\in L$ which is not the endpoint of an arc on either end. The algebraic intersection number of an arc with $ [0,1]\times \{p\}$ is unchanged by arbitrary isotopies which fix $\{0,1\}\times L$, and hence two dividing sets are isotopic relative to $\{0,1\} \times L $ if and only if they are admissibly isotopic relative to $\{0,1\}\times L$.
 \end{proof}

\begin{lem}\label{lem:invariantunderisotopiestype1}Isotopies of the parameterizing diffeomorphism $\Phi$ through diffeomorphisms fixing $\{0\}\times Y_1$ do not affect $F_{W,\cF^\sigma,\Phi,\bS,\frs}$ for type ($\mathcal{EPC}$-\ref{elemencobtype1}) and ($\mathcal{EPC}$-\ref{elemencobtype2}) elementary cobordisms.
\end{lem}
\begin{proof}Isotopies of $\Phi$ which fix $\Sigma$ pointwise have no effect on the cobordism map. Similarly isotopies of $\Phi$ which are supported only in a neighborhood of $\Sigma$, but which fix $\{1\}\times L_1$ only affect the composition by isotoping the pullback of the dividing set, relative to $\{0,1\}\times L_1$, and hence possibly changing the map $\psi_*$. By Lemma~\ref{lem:admissibleisotopiesofcylinders}, such an isotopy of the dividing set can be taken to be an admissible isotopy, which clearly has no effect (since the resulting $\psi$ is unchanged). It remains to just check the effect of isotopies supported in a neighborhood of $\{1\}\times L_1$. If $\tau$ denotes a diffeomorphism which twists a neighborhood of $\{1\}\times L_1$, then precomposing $\Phi$ with $\tau$ changes $(\Phi|_{\{1\}\times Y_1 })_*$ to $(\Phi|_{\{1\}\times Y_1})_*(\tau_*)$ but also changes $\psi_*$ to $(\tau^{-1})_* (\psi_*)$ (since the pullback under $\Phi$ of the dividing set changes). Since $(\tau^{-1})_*$ commutes with $F_{Y_1,\bL_1,\bS,\Phi^{-1}(\frs)}$, as $\tau$ does not move the framed sphere $\bS$, the factors of $(\tau)_*$ and $(\tau^{-1})_*$ cancel and the composition is unchanged.
\end{proof}

\subsection{Maps induced by elementary parametrized link cobordisms of type ($\mathcal{EPC}$-\ref{elemencobtype3})}
\label{subsec:cylindercobordismmaps}

Suppose $(W,\cF^\sigma,\Phi,\bS)\colon(Y_1,\bL_1)\to (Y_2,\bL_2)$ is a parametrized link cobordism of type ($\mathcal{EPC}$-\ref{elemencobtype3}). As before, we write 
\[
\cA_0=\Phi^{-1}(\cA).
\]

Recall that there are four subtypes of elementary cobordisms of type ($\mathcal{EPC}$-\ref{elemencobtype3}), which we denote by $S^+$, $S^-,$ $T^+,$ and $T^-$. These are shown in Figure~\ref{fig::62}.

\begin{figure}[ht!]
\centering
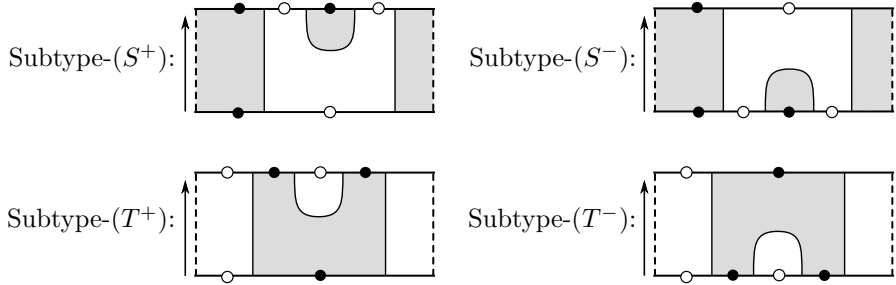
\caption{\textbf{The four subtypes of elementary cobordisms of type~($\mathcal{EPC}$-\ref{elemencobtype3})}.\label{fig::62}}
\end{figure}

\subsubsection{Type $S^+$ elementary cobordisms}
 Suppose $(W,\cF^\sigma,\Phi,\bS_{\varnothing})\colon (Y_1,\bL_1)\to (Y_2,\bL_2)$ is an elementary parametrized cobordism of type ($\mathcal{EPC}$-$\ref{elemencobtype3}_{S^+}$). 
 
There is a single $\ve{z}$-region of $([0,1]\times L_1)\setminus \cA_0$ which has three basepoints on its boundary. In this region, let $z_0\in \{0\}\times L_1$ be the basepoint on the incoming boundary, and let $z,z'\in \{1\}\times L_1$ be the two basepoints on the outgoing boundary. Let $w\in \{1\}\times L_1$ be the basepoint between $z$ and $z'$.  Suppose the basepoints are ordered $(z',w,z)$, read right to left.  By a slight abuse of notation, let us also write $z'$, $w$ and $z$ for images of $z'$, $w$ and $z$ under the projection map 
\[\pi\colon [0,1]\times Y_1\to Y_1.\]
 
We perform an isotopy of the dividing set $\cA_0$, relative to $\{0,1\}\times L_1$ to obtain an admissible dividing set $\cA_0'$ on $[0,1]\times L_1$ (i.e. all arcs going from $\{0\}\times L_1$ to $\{1\}\times L_1$ are transverse to the sets $\{t\}\times L_1$). There are two self-diffeomorphisms $\psi$ and $\psi'$ of the pair $(Y_1,L_1)$ obtained by twisting the link $L_1$ according to the dividing set $\cA_0'$, and requiring $\ws_1\cup \zs_1$ be mapped to a subset of $\pi(\Phi^{-1}(\ws_2\cup \zs_2))$.

 We pick $\psi$ and $\psi'$ to satisfy
\[\psi(z_0)=z'\qquad \text{and}\qquad \psi'(z_0)=z.\] Using Lemma~\ref{lem:admissibleisotopiesofcylinders}, we see that the diffeomorphisms $\psi$ and $\psi'$ are each independent (up to post-composition by an isotopy fixing the basepoints) of the choice of admissible $\cA_0'$ which is isotopic to $\cA_0$.
 
 We now define two maps:
\[F_{W,\cF^\sigma,\Phi,\bS_{\varnothing},\frs}:=(\Phi|_{\{1\}\times Y_1 })_* S_{w,z}^+ \psi_*\qquad \text{and} \qquad F'_{W,\cF^\sigma,\Phi,\bS_{\varnothing}, \frs}:=(\Phi|_{\{1\}\times Y_1 })_*S_{z',w}^+  \psi_*'.\] Note that there  is no canonical reason to choose one map over the other, since we have to choose which of $z$ and $z'$ we want to identify $z_0$ with. However an application of  Lemma~\ref{lem:mapsforI+agree} shows that the two maps are chain homotopic, so we denote the common map by $F_{W,\cF^\sigma,\Phi,\bS_{\varnothing},\frs}$.

\subsubsection{Type $S^-$ elementary cobordisms}
Suppose $(W,\cF^\sigma,\Phi,\bS_{\varnothing})$ is of type ($\mathcal{EPC}$-$\ref{elemencobtype3}_{S^-})$. There is a single region with three $\zs$-basepoints, two of which are on the incoming end, and one of which is on the outgoing end.  Suppose that $z$ and $z'$ are the basepoints on $(L_1,\ve{w}_1,\ve{z}_1)$ which are in the same region of $([0,1]\times L_1)\setminus \cA_0$, and let $w$ be the basepoint between them. There is another basepoint, $z_0\in \Phi^{-1}(\ve{w}_2\cup \ve{z}_2)$, which is in the same region of $([0,1]\times L_1)\setminus \cA_0$ as $z$ and $z'$. Assume that the ordering of these basepoints is $(z',w,z)$, read right to left.  Let us write $z_0$ also for the image of $z_0$ under $\pi$.

 As with elementary cylindrical cobordisms of type ($\mathcal{EPC}$-$\ref{elemencobtype3}_{S^+}$), the dividing set $\cA_0=\Phi^{-1}(\cA)$ determines two well-defined diffeomorphisms $\psi$ and $\psi'$ of $(Y_1,L_1)$, each coming from twisting $L_1$ according to the dividing set, and requiring that either $(\ws_1\cup \zs_1)\setminus \{w,z\}$ or $(\ws_1\cup\zs_1)\setminus \{z',w\}$ be mapped to $\ws_2\cup \zs_2$. We pick $\psi$ and $\psi'$ so that
\[
\psi(z')=z_0 \qquad \text{and} \qquad\psi'(z)=z_0.
\] 
We then define two maps:
\[
F_{W,\cF^\sigma,\Phi,\bS_{\varnothing},\frs}=(\Phi|_{ \{1\}\times Y_1})_* \psi_*S_{w,z}^- , \qquad \text{and}\qquad F'_{W,\cF^\sigma,\Phi,\bS_{\varnothing},\frs}=(\Phi|_{\{1\}\times Y_1})_* \psi'_*S_{z',w}^- .
\] By Lemma~\ref{lem:mapsforI+agree}, these two maps are filtered chain homotopic, so we denote the common map by $F_{W,\cF^\sigma,\Phi,\bS_{\varnothing},\frs}$.

\subsubsection{Type $T^+$ elementary cobordisms}
Suppose $(W,\cF^\sigma,\Phi,\bS_{\varnothing})$ is of type ($\mathcal{EPC}$-$\ref{elemencobtype3}_{T^+}$).  In the region of $([0,1]\times L_1)\setminus \cA_0$ which has three $\ve{w}$-basepoints, let $w_0$ be the basepoint on $\{0\}\times L_1 $, and let $w$ and $w'$ be the basepoints on $\{1\}\times L_1$. Let $z$ denote the basepoint between $w$ and $w'$. Assume that on $\{0\}\times L_1$, the basepoints are ordered $(w,z,w')$, read right to left. By abuse of notation, we also write $w$, $z$ and $w'$ for the images of $w$, $z$ and $w'$ under the map $\pi\colon [0,1]\times Y_1\to Y_1$.

Using Lemma~\ref{lem:admissibleisotopiesofcylinders}, the dividing set $\cA_0$ determines two diffeomorphisms of $(Y_1,L_1)$, $\psi$ and $\psi'$, corresponding to twisting a link component as prescribed by the divides, and matching up the basepoints. Assume that 
\[\psi(w_0)=w'\qquad \text{and}\qquad  \psi'(w_0)=w.\] We define two maps:
\[F_{W,\cF^\sigma,\Phi,\bS_{\varnothing},\frs}:=(\Phi|_{\{1\}\times Y_1})_*T_{w,z}^+ \psi_*\qquad  \text{and}\qquad 
F'_{W,\cF^\sigma,\Phi,\bS_{\varnothing},\frs}:=(\Phi|_{\{1\}\times Y_1})_*T_{z,w'}^+ \psi'_*.\]  By Lemma~\ref{lem:mapsforII+agree}, these two maps are filtered chain homotopic, so we denote the common map by $F_{W,\cF^\sigma,\Phi,\bS_{\varnothing},\frs}$.

\subsubsection{Type $T^-$ elementary cobordisms}
Finally, suppose $(W,\cF^\sigma,\Phi,\bS_{\varnothing})$ is of type ($\mathcal{EPC}$-$\ref{elemencobtype3}_{T^-}$).  In the region of $([0,1]\times L_1)\setminus \cA_0$ which has three $\ve{w}$-basepoints, let $w_0$ be the basepoint on $\{1\}\times L_1$, and let $w$ and $w'$ be the basepoints on $ \{0\}\times L_1$, in this region. Let $z$ denote the basepoint between $w$ and $w'$, and suppose that the basepoints are ordered $(w,z,w')$, read right to left. Using Lemma~\ref{lem:admissibleisotopiesofcylinders}, the divides $\cA_0$ determine two diffeomorphisms, $\psi$ and $\psi'$, of $(Y_1,L_1)$ corresponding to twisting $L_1$, as prescribed by the divides. Assume that \[\psi(w')=w_0\qquad \text{and}\qquad  \psi'(w)=w_0.\] We define two maps:
\[
F_{W,\cF^\sigma,\Phi,\bS_{\varnothing},\frs}:=(\Phi|_{\{1\}\times Y_1})_* \psi_* T_{w,z}^- \qquad \text{and}\qquad F'_{W,\cF^\sigma,\Phi,\bS_{\varnothing},\frs}:=(\Phi|_{\{1\}\times Y_1})_* \psi'_*T_{z,w'}^- .
\]
By Lemma~\ref{lem:mapsforII+agree}, these two maps are filtered chain homotopic, so we denote the common map by $F_{W,\cF^\sigma,\Phi,\bS_{\varnothing},\frs}$.

Adapting Lemma~\ref{lem:invariantunderisotopiestype1} immediately yields the following:

\begin{lem}\label{lem:invariantunderisotopiestype3}Isotopies of the parameterizing diffeomorphism $\Phi$ through diffeomorphisms fixing $\{0\}\times Y_1$,  do not affect the maps $F_{W,\cF^\sigma,\Phi,\bS_{\varnothing}, \frs}$ for elementary parametrized cobordisms of type ($\mathcal{EPC}$-\ref{elemencobtype3}).
\end{lem}
\begin{proof}This follows similarly to Lemma~\ref{lem:invariantunderisotopiestype1}, using Lemma~\ref{lem:admissibleisotopiesofcylinders}.
\end{proof}

\subsection{Maps induced by elementary parametrized link cobordisms of type ($\mathcal{EPC}$-\ref{elemencobtype4})}\label{subsec:mapsfortype4}

Suppose that $(W,\cF^\sigma, \Phi, \bS^0)$ is an elementary cobordism of type ($\mathcal{EPC}$-\ref{elemencobtype4}) from $(Y_1,\bL_1)$ to $(Y_2,\bL_2)$. Here $\bS^0$ is a framed 0-sphere in $Y_1$, which intersects $L_1$ along two arcs.

Recall that we constructed a surface
\[
\Sigma(L_1,\bS^0)=([0,1]\times L_1)\cup B\cup ([1,2]\times L_1(\bS^0))\subset W(Y_1,\bS^0)
,\]  which is diffeomorphic to a disjoint union of a pair-of-pants together with some cylinders. The subset $B$ is identified with $[-1,1]\times [-1,1]$. Define the embedded 1-complex 
\[
\cL=(\{1\}\times L)\cap (\{1\}\times L(\bS^0))\subset \Sigma(L_1,\bS^0),
\] 
which can think of as a subset the critical level set of a Morse function on $\Sigma(L_1,\bS^0)$, with a single index 1 critical point on the band.

 The parameterizing diffeomorphism
\[\Phi\colon \cW(Y_1,L_1,\bS^0) \to (W,\cF^\sigma)\] is assumed to have the property that $\Phi^{-1}(\cA)$ does not intersect the band region $B\subset \Sigma(L_1,\bS^0)$.  Furthermore, the dividing set $\cA$ consists entirely of arcs going from the incoming boundary to the outgoing boundary. We note that the dividing set $\Phi^{-1}(\cA)$ could  wind  many times around the annular subsets of $\Sigma(L_1,\bS^0)$ identified with $[0,1]\times L_1$ or $[1,2]\times L_1(\bS^0)$.

In analogy with Definition~\ref{def:admissibleisotopy1}, we make the following definition:
\begin{define}Suppose $\cA$ is a dividing set on $\Sigma(L,\bS^0)$, such that all arcs go from the incoming boundary to the outgoing boundary. We say $\cA$ is \emph{admissible} if no arcs intersect $B$, and all arcs intersect $\cL$ transversely at a single point. If $\cA_t$ is a 1-parameter family of dividing sets on $\Sigma(L_1,\bS^0)$ which are fixed on the $\{0\}\times L$ and $\{2\}\times L(\bS^0)$, then we say $\cA_t$ is an \emph{admissible isotopy} if $\cA_t$ is admissible for all $t$.
\end{define}

\begin{lem}\label{lem:allisotopiesrespectL}Suppose that $\cA_1$ and $\cA_2$ are admissible dividing sets on $\Sigma(L_1,\bS^0)$. If $\cA_1$ and $\cA_2$ are isotopic through dividing sets fixed on $\{0\}\times L_1$ and $\{2\}\times L_1(\bS^0)$ and not intersecting the point $(0,0)\in B$, then $\cA_1$ and $\cA_2$ are admissibly isotopic.
\end{lem}
\begin{proof}On any component of $\Sigma(L_1,\bS^0)$ which does not intersect $B$, the claim is obvious, so suppose without loss of generality that $L_1$ is a one (resp. two) component link which is separated (resp. joined) by the band $B$. Suppose for definiteness, that $L_1$ is a two component link which is joined by $B$. The case that $L_1$ is a one component link which is separated by the band follows by turning the picture upside down.

Let $p_-,p_+\in  L_1$ denote the two center points of the framed sphere $\bS^0$. Let $C_-$ and $C_+$ be curves on $([0,1]\times L_1)\cup B$, each consisting of a path on $[0,1]\times L_1$ from a point on $(\{0\}\times L_1)\setminus \cA_i$, to $(p_-,1)$ or $(p_+,1)$, respectively, concatenated with one of the arcs  $[-1,0]\times \{0\}$ or $[0,1]\times \{0\}$, respectively, on $B$. These are shown in Figure~\ref{fig::69}. 

 \begin{figure}[ht!]
\centering
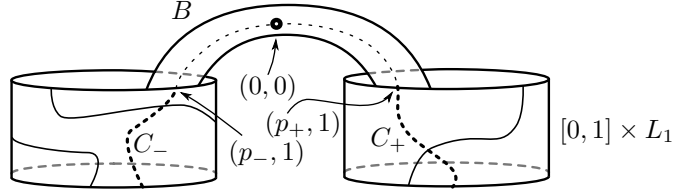
\caption{\textbf{The region $([0,1]\times L_1)\cup B\subset \Sigma(L_1,\bS^0)$.} The arc $C_-$ is the concatenation of the arc (boldly dotted) on the left component of $[0,1]\times L_1$, together with the arc $[-1,0]\times \{0\}$ (lightly dotted) on $B$. Similarly $C_+$ is the concatenation of the arc (boldly dotted) on the right component of $[0,1]\times L_1$, together with the arc $[0,1]\times \{0\}$ (lightly dotted) on $B$. The point $(0,0)\in B$, is marked as well. \label{fig::69}}
\end{figure}

Similarly let $C_+'$ and $C_-'$ be two similarly defined curves on $B \cup ([1,2]\times L_1(\bS^0))$ from $(0,0)$ to two points on $(\{2\}\times L_1(\bS^0))\setminus \cA_i$.  

The admissible isotopy class of the divides $\cA_1$ and $\cA_2$ relative $\{0\}\times L_1\cup \{2\}\times L_1(\bS^0)$ is determined by the ordering that the arcs appear on $\cL$, and the intersection numbers of arcs with $C_-,$ $C_+$, $C_-',$ and $C_+'$.

The configuration of endpoints of arcs on $\{1\}\times L_1$ is determined by the region of $\Sigma\setminus \cA$ that the band is in and the order that the arcs appear on $\{0\}\times L_1 $. This is unchanged by isotopies of $\Sigma(L_1,\bS^0)$ which may not be admissible, but which do not map any arcs to $(0,0)$. Similarly the algebraic intersection numbers above are not changed by isotopies of the arcs which fix the ends of the arcs on $\{0\}\times L_1$ and $\{2\}\times L_1(\bS^0)$, and do not map any arcs to $(0,0)$. The claim follows.
\end{proof}

We now pick  basepoints $\ve{w}_{1.5}\cup \ve{z}_{1.5}$ on $\{1\}\times L_1$ which make $(\{1\}\times L_1,\ve{w}_{1.5},\ve{z}_{1.5})$ a multi-based link, and which make the band $B$ either an $\alpha$-band or a $\beta$-band. Note that the type of the band, either type-$\ve{w}$ or type-$\ve{z}$, is determined by which connected component of $\Sigma(L_1,\bS^1)\setminus \cA_0$ contains four basepoints. For a fixed choice of $\ve{w}_{1.5}$ and $\ve{z}_{1.5}$ there is no ambiguity in whether $B$ is an $\alpha$-band or a $\beta$-band, but there are two valid choices of basepoints $\ve{w}_{1.5}$ and $\ve{z}_{1.5}$ which respect the dividing set, one of which makes $B$ an $\alpha$-band, and one of which makes $B$ a $\beta$-band. For a choice of $\xi\in \{\alpha,\beta\}$, write $\ve{w}_{1.5}^\xi$ and $\ve{z}_{1.5}^\xi$ for these basepoints.

Adapting Lemma~\ref{lem:admissibleisotopiesofcylinders}, the dividing set $\cA_0$ restricted to $[0,1]\times L_1$ induces  an isotopy $\psi_1^\xi$ of $(Y_1,L_1)$ which sends $(\ve{w}_1,\ve{z}_1)$ to $(\ve{w}_{1.5}^\xi,\ve{z}_{1.5}^\xi)$. Similarly the dividing set $\cA_0$ restricted to $[1,2]\times L_1(\bS^0)$ induces an  isotopy $\psi_2^\xi$ of $(Y_1(\bS^0),L_1(\bS^0))$ which sends $(\ve{w}_{1.5}^\xi,\ve{z}_{1.5}^\xi)$ to $(\Phi^{-1}(\ve{w}_2),\Phi^{-1}(\ve{z}_2))$. We now define
\[F_{W,\cF^\sigma,\Phi,\bS^0,\frs}^\xi:=(\Phi|_{\{2\}\times Y_1(\bS^0)})_* (\psi_2^\xi)_* F_{Y_1,\bL_1,\bS^0,\Phi^{-1}(\frs)}^{\ve{o}} (\psi_1^\xi)_*.\] There is no ambiguity in the choice of $\ve{o}\in \{\ve{w},\ve{z}\}$, though $\xi$ can be either $\alpha$ or $\beta$, as determined by the choice of basepoints $\ve{w}_{1.5}^\xi$ and $\ve{z}_{1.5}^\xi$. By Propositions~\ref{prop:alphabandmapsarebetabandmaps} and~\ref{prop:alphabandmapsarebetabandmapstypew} we know that the $\alpha$- and $\beta$-band maps are related by the diffeomorphism resulting from moving basepoints, which is already encoded into $\psi_1^\xi$ and $\psi_2^\xi$. From these considerations we conclude that
\[F_{W,\cF^\sigma,\Phi,\bS^0,\frs}^\alpha\simeq F_{W,\cF^\sigma,\Phi,\bS^0,\frs}^\beta,\] and we denote the common map
\[F_{W,\cF^\sigma,\Phi,\bS^0,\frs}.\]

We have the following:

\begin{lem}\label{lem:invariantunderisotopiestype4} For an elementary parametrized cobordism $(W,\cF^\sigma,\Phi,\bS^0)$ of type ($\mathcal{EPC}$-\ref{elemencobtype4}), isotopies of the parametrization $\Phi$ which fix $\{0\}\times Y_1$ and do not map the point $(0,0)\in B$ to $\cA$ do not affect the map $F_{W,\cF^\sigma,\Phi,\bS^0,\frs}$.
\end{lem}

\begin{proof} Isotopies which are fixed pointwise on $\Sigma(L_1,\bS^0)$ clearly have no effect, so we can just consider isotopies supported in a neighborhood of $\Sigma(L_1,\bS^0)$. By the same argument as in Lemma~\ref{lem:invariantunderisotopiestype1}, the maps are independent of isotopies of $\Phi$ which are supported in a neighborhood of $\{2\}\times L_1(\bS^0)$. Hence it remains to show that the maps are invariant under isotopies which are supported in a neighborhood of $\Sigma(L_1,\bS^0)$ and are fixed on $\{0\}\times Y_1$ and $\{2\}\times Y_1(\bS^0)$. Suppose $\Phi_t$ is such an isotopy. Since $\Phi_t$ is fixed on $\{2\}\times Y_1(\bS^0)$, the map $F_{W,\cF,\Phi_t,\bS^0,\frs}$ only depends on $\Phi_t$ through the dividing set $\cA_t:=\Phi_t^{-1}(\cA)$ on $\Sigma(L_1,\bS^0)$. By Lemma~\ref{lem:allisotopiesrespectL}, we can assume that $\cA_0$ and $\cA_1$ are admissibly isotopic, i.e.,  the dividing sets $\cA_t$ never pass through the band $B$, and the curves $\cA_t$ remain transverse to $\cL$ throughout.

Such  a 1-parameter family $\cA_t$ induces an isotopy $\tau_t\colon [0,1]\times \cL\to \cL$ which is fixed on $\d \cL$. We can extend this to a diffeomorphism of  both $(Y_1,L_1)$ and $(Y_1(\bS^0),L_1(\bS^0))$ which fixes $\bS^0\subset Y_1$ and $D^1\times S^2\subset Y_1(\bS^0)$, respectively. We let $Z$ denote an extension of $\tau_1$ to all of $(Y_1,L_1)$, and let $Z^{\bS^0}$ denote the induced diffeomorphism of $(Y_1(\bS^0),L_1(\bS^0))$. The effect of the isotopy $\Phi_t$ is to replace $\psi_1$ with $\psi_1'=Z\circ \psi_1$ and to replace $\psi_2$ with $\psi_2'=\psi_2\circ(Z^{\bS^0})^{-1}$.  Since $Z$ fixes $\im(\bS^0)\subset Y_1$, it follows that $Z^{\bS^0}$ and $Z$ commute with the compound 1-handle/band map $F_{Y_1,\bL_1,\bS^0,\frs}^{\ve{o}}$. Hence we have
\[\Phi|_{\{2\}\times Y_1(\bS^0)} (\psi_2')_* F_{Y_1,\bL_1,\bS^0,\frs}^{\ve{o}} (\psi_1')_*\simeq \Phi|_{\{2\}\times Y_1(\bS)} (\psi_2)_*(Z^{\bS_i})^{-1}_* F_{Y_1,\bL_1,\bS^0,\frs}^{\ve{o}}  Z_*(\psi_1)_*\]
\[\simeq \Phi|_{ \{2\}\times Y_1(\bS)} (\psi_2)_* F_{Y_1,\bL_1,\bS^0,\frs}^{\ve{o}}  (Z^{-1})_* Z_*(\psi_1)_*\simeq \Phi|_{ \{2\}\times Y_1(\bS^0)} (\psi_2)_* F_{Y_1,\bL_1,\bS^0,\frs}^{\ve{o}}  (\psi_1)_*,\] showing that the maps are unchanged.
\end{proof}

\subsection{Invariance of the cobordism maps from a parametrized Kirby decomposition}
\label{subsec:invariance}
If $(W,\cF^\sigma)$ is a link cobordism such that each component of $W$ has a nonempty incoming and outgoing end, and each component of $\Sigma$ has a nonempty incoming and outgoing end, we take a parametrized Kirby decomposition,  $\cK$, with elementary parametrized cobordisms $\cW_i=(W_i,\cF_i^{\sigma_i},\Phi_i,\bS_i)$, and define the maps
\[F_{W,\cF^\sigma,\frs,\cK}= \prod_{i} F_{W_i,\cF_i^{\sigma_i},\Phi_i,\bS_i, \frs|_{W_i}},\] where the maps $F_{W_i,\cF_i^{\sigma_i},\Phi_i,\bS_i, \frs|_{W_i}}$ are the maps defined in the previous subsections for elementary parametrized link cobordisms.

We now prove a slightly weaker version of Theorem~\ref{thm:A}, where each component of $W$ and $\cF$ intersects a component of $Y_1$ and $Y_2$ non-trivially. In Section~\ref{sec:removeballsforgeneralmaps}, where we may have to puncture the cobordism to introduce new ends, finishing the proof of Theorem~\ref{thm:A}.

\begin{customthm}{A$'$}\label{thm:A'}Suppose that $(W,\cF^\sigma)\colon (Y_1,\bL_1)\to (Y_2,\bL_2)$ is a decorated link cobordism and each component of $W$ and $\cF$ intersects a component of both $Y_1$ and $Y_2$ non-trivially. The maps $F_{W,\cF^\sigma,\frs,\cK}$  are independent of the parametrized Kirby decomposition, $\cK$, up to filtered, equivariant chain homotopy.
\end{customthm}
\begin{proof} By Theorem~\ref{prop:allPKDSigmasweaklyequivalent}, it is sufficient to show that the maps $F_{W,\cF^\sigma,\frs,\cK}$ are equal for Cerf equivalent parametrized Kirby decompositions. Hence it is sufficient to check invariance from Moves~\eqref{def:weakequiv:strongisotopy}--\eqref{def:weakequiv:critcrossesA}  from Definition~\ref{def:weaklyequivalent}. 

\textbf{Move~\eqref{def:weakequiv:strongisotopy}:} We consider $\cA$-adapted isotopies (Definition~\ref{def:strongisotopy}). First note that if 
\[
\phi\colon (W,\Sigma)\to (W,\Sigma),
\] is a diffeomorphism such that $\phi|_{Y_1}=\id$, then we can push forward $\cK$ under $\phi$ using Equation~\eqref{eq:pushforwarddecomposition} to get a parametrized Kirby decomposition $\phi_*(\cK)$, which is tautologically a parametrized Kirby decomposition of $(W,\phi(\cF^\sigma))$. Tautologically, we have the relation
\begin{equation}
 F_{W,(\Sigma,\cA)^\sigma, \frs,\cK}\simeq (\phi|_{Y_2}^{-1})_* F_{W,(\Sigma,\phi(\cA))^\sigma,\phi_*(\frs),\phi_*(\cK)}.\label{eq:strongisotopiestautological}
\end{equation}
We now consider the case that $\phi=\phi_1$ for an $\cA$-adapted isotopy $\phi_t:(W,\Sigma)\to (W,\Sigma)$. Suppose first that $\phi_t$ is supported in a neighborhood of $L_2$ (i.e. $\phi_t$ just twists $\Sigma$ near $L_2$). Using the definitions of the maps for elementary parametrized cobordisms from Sections~\ref{sec:defmapscob1-2}--\ref{subsec:mapsfortype4}, one easily verifies the relation
\[
F_{W,(\Sigma,\phi(\cA))^\sigma,\frs,\phi_*(\cK)}\simeq (\phi|_{Y_2})_* F_{W,(\Sigma,\cA)^\sigma,\frs,\cK}.
\]
Combining with Equation~\eqref{eq:strongisotopiestautological}, we obtain invariance under isotopies $\phi_t$ which are supported in a neighborhood of $L_2$. Since isotopies of $(Y_2,L_2)$ which are fixed on $L_2$ induce the identity map, by naturality, we can reduce to the case that $\phi_1|_{Y_2}=\id$.

Assuming now that $\phi_1|_{Y_2}=\id$, Equation~\eqref{eq:strongisotopiestautological} reads
\[
F_{W,(\Sigma,\cA)^\sigma,\frs,\cK}\simeq F_{W,(\Sigma,\phi(\cA))^\sigma,\frs,\phi_*(\cK)}.
\]
Recall that our goal is to show that the left side of the above equation is $F_{W,(\Sigma,\cA)^\sigma,\frs,\phi_*(\cK)}$. We note that by assumption, the 1-parameter family of dividing sets $\cA_t:=\phi_{1-t}(\cA)$ connects $\phi(\cA)$ and $\cA$. Furthermore, since $\phi_t$ is $\cA$-adapted, each of the dividing sets $\cA_t$ makes $\phi_*(\cK)$ into a valid parametrized Kirby decomposition (i.e. each level of $\cK$ is an elementary parametrized cobordism with the dividing set induced by $\cA_t$). Hence it is sufficient to show invariance of the cobordism maps $F_{W,(\Sigma,\cA)^{\sigma},\frs,\cK}$ under changing $\cA$ through a 1-parameter family which is fixed on $\d \Sigma$ and has the property that $\cK$ is a parametrized Kirby decomposition of $(W,(\Sigma,\cA_t)^\sigma)$ for all $t$. To establish this fact, note that we can modify the isotopy $\cA_t$ so that it decomposes into a sequence of isotopies, each of which satisfies one of the following:
\begin{enumerate}
\item $\cA_t$ is supported in a neighborhood of $L= \d \Sigma_i\cap \d \Sigma_{i+1}$, where $\Sigma_i$ denotes the subset of $\Sigma$ inside of the level $\cW_i$ of $\cK$.
\item $\cA_t$ is fixed on $\d \Sigma_i$, for all $i$.
\end{enumerate}
Isotopies of the first type are easily seen to have no effect, since they change the map associated to $\cW_i$ by post-composition by a diffeomorphism map $\tau_*$ for twisting along $L$, and they change the map for $\cW_{i+1}$ by pre-composition by $(\tau^{-1})_*$. These two factors clearly cancel and the overall map is unchanged. Similarly isotopies of the second kind have no effect, by Lemmas~\ref{lem:invariantunderisotopiestype1},~\ref{lem:invariantunderisotopiestype3}, and~\ref{lem:invariantunderisotopiestype4}.

\textbf{Move~\eqref{def:weakequiv:isotopyawayfromL}:} We now consider replacing an   elementary cobordism $\cW_i=(Y_i,L_i,\bS_i)$ of  type ($\mathcal{EPC}$-\ref{elemencobtype3}) in $\cK$ with an elementary parametrized cobordism $\cW_i'=(Y_i,L_i,\bS_i')$, also of  type ($\mathcal{EPC}$-\ref{elemencobtype3}), such that $\bS_i'$ is related to $\bS_i$ by an isotopy in $Y_i\setminus L_i$.  We will omit the subscript $i$, and the $\Spin^c$ structures, for notational simplicity.

Suppose that $\bS^t$ is an isotopy of framed sphere in $Y\setminus L$, starting at $\bS$ and ending at $\bS'$. We let $d^t:(Y,L)\to (Y,L)$ denote an extension of the isotopy $\bS^t$, which is fixed on $L$, and  denote $d:=d^1\colon (Y,L)\to (Y,L)$. Let $d^{\bS}\colon Y(\bS)\to Y(\bS')$ denote the induced map. We pull back the dividing set under $\Phi$ to the trace link cobordism $\cW(Y,L,\bS)$, and define $\psi:(Y,L)\to (Y,L)$ to be the diffeomorphism induced by twisting along $L$, according to the dividing set. Recall that $\Phi=\Phi'\circ D$, where $D\colon \cW(Y,L,\bS)\to \cW(Y,L,\bS')$ is the diffeomorphism defined in Move~\eqref{def:weakequiv:isotopyawayfromL}. Notice that $\Phi$ and $\Phi'$ can be taken to be equal along $[0,1]\times L \subset \cW(Y,L,\bS)$, since the isotopy of framed spheres occurs away from $L$. By definition we have that
\[F_{W,\bL,\bS,\Phi}:=(\Phi|_{Y(\bS)})_*F_{Y,\bL,\bS}\psi_*\qquad \text{and}\qquad 
F_{W,\bL,\bS',\Phi'}:=(\Phi'|_{Y(\bS')})_*F_{Y,\bL,\bS'}\psi_*.\] Since $\Phi=\Phi'\circ D$, by definition, we have that
\[(\Phi'|_{Y(\bS')})_*d^{\bS}_*\simeq (\Phi|_{Y(\bS)})_*.\] By naturality and the well-definedness of the maps for framed spheres, we have
\[F_{Y,\bL,\bS'}d_*\simeq d^{\bS}_* F_{Y,\bL,\bS}.\] We note that $\psi_*$ and $d_*$ commute, since they are induced by isotopies supported on disjoint subsets of $Y$. As a consequence we have that
\begin{align*}F_{W,\bL,\bS',\Phi'} d_*&\simeq (\Phi'|_{Y(\bS')})_*F_{Y,\bL,\bS'}\psi_*d_*
\\&\simeq (\Phi'|_{Y(\bS')})_*F_{Y,\bL,\bS'}d_*\psi_*\\
&\simeq (\Phi'|_{Y(\bS')})_*d_*^{\bS}F_{Y,\bL,\bS}\psi_*\\
&\simeq (\Phi|_{Y(\bS)})_*F_{Y,\bL,\bS}\psi_*\\
&\simeq F_{W,\bL,\bS,\Phi}.\end{align*}

Noting that $d$ is a diffeomorphism of $Y$ which is isotopic to the identity relative $L$, and hence we conclude that $d_*\simeq \id$. Hence
\[F_{W,\bL,\bS',\Phi'}\simeq F_{W,\bL,\bS,\Phi},\] showing the maps to be invariant under this move.

\textbf{Move~\eqref{def:weakequiv:isotopyalongL}:} We now prove invariance of isotopies of a framed 0-sphere along a link $L_i$. This follows similarly to Move \eqref{def:weakequiv:isotopyawayfromL}, and is essentially a tautology. Suppose that $\cW=(W,\cF^\sigma,\bS,\Phi)$ is an elementary parametrized link cobordism of type ($\mathcal{EPC}$-\ref{elemencobtype4}). Note that the subtype (either $\ws$ or $\zs$) is unchanged by this move.  Let $d_t$ be an isotopy of $\bS$ along $L$, define $d=d_1$  and let $\bS'$ be the image of $\bS$ and let $\Phi'$ be the parametrization described in Move \eqref{def:weakequiv:isotopyalongL}. Let 
\[\psi_1\colon (Y,L)\to (Y,L)\qquad \text{and} \qquad \psi_2\colon (Y(\bS),L(\bS))\to (Y(\bS),L(\bS))\]  be the two twisting diffeomorphisms associated to $\Phi$ (cf. Subsection~\ref{subsec:mapsfortype4}), and let $\psi_1'$ and $\psi_2'$ be two twisting diffeomorphisms associated to $\Phi'$. By definition
\[F_{W,\cF^\sigma,\bS,\Phi}\simeq (\Phi|_{Y(\bS)\times \{2\}})_*(\psi_2)_* F^{\ve{o}}_{Y, \psi_1(\bL),\bS} (\psi_1)_*\] and
\[F_{W,\cF^\sigma,\bS',\Phi'}\simeq (\Phi'|_{Y(\bS)\times \{2\}})_*(\psi_2')_* F^{\ve{o}}_{Y, \psi_1'(\bL),\bS'} (\psi_1')_*.\] It is easy to see that by construction we have that \[(\psi_1')_*=d_*(\psi_1)_*.\] Tautologically we have
\[(d^{\bS})^{-1}_* F_{Y,\psi'_1(\bL),\bS'} d_*\simeq F_{Y,\psi_1(\bL),\bS}.\] Now also by construction, we have
\[\psi_2'\simeq (d^{\bS})_* (\psi_2)_*(d^{\bS})^{-1}_*.\] Finally we note that, by construction
\[(\Phi|_{Y(\bS)\times \{2\}})_*\simeq (\Phi'|_{Y(\bS')\times \{2\}})_* d^{\bS}_*.\] As a consequence, we see that 
\[F_{W,\cF^\sigma,\bS',\Phi'}\simeq (\Phi'|_{Y(\bS)\times \{2\}})_*(\psi_2')_* F^{\ve{o}}_{Y, \psi_1'(\bL),\bS'} (\psi_1')_*\simeq (\Phi'|_{Y(\bS)\times \{2\}})_*(d^{\bS})_* (\psi_2)_*(d^{\bS})^{-1}_* F^{\ve{o}}_{Y, \psi_1'(\bL),\bS'} d_*(\psi_1)_*\]
\[\simeq (\Phi|_{Y(\bS)\times \{2\}})_*(\psi_2)_* F^{\ve{o}}_{Y, \psi_1(\bL),\bS} (\psi_1)_*\simeq F_{W,\cF^\sigma,\bS,\Phi},\] completing the proof of invariance under Move \eqref{def:weakequiv:isotopyalongL}.

\textbf{Move~\eqref{def:weakequiv:addtrivialcylinder}:} Invariance under adding or removing a trivial cylinder is essentially a tautology. The map induced by a trivial cylinder is a diffeomorphism map, and all of the other maps involve pre- or post-composition with an analogous diffeomorphism. It is easy to check that combining a trivial cylinder with another elementary cobordism (from the left or the right) does not affect the maps. 

\textbf{Move~\eqref{def:weakequiv:changeorientationofS}:} We now consider invariance under changing the orientation of the framed sphere $\bS$. We argue in the case of a framed sphere $\bS$, attached away from $L$. The case that $\bS$ is attached along $L$ follows from similar considerations. View $\bS$ as an embedding of $S^k\times D^{3-k}$ into $Y$. Let $\sigma\colon \R^4\to \R^4$ denote the map $\sigma(x,y,w,z)=(-x,y,w,-z)$. Viewing $S^k\times D^{3-k}$ as a subset of $\R^{k+1}\times \R^{3-k}$, we define $\bar{\bS}:=\bS\circ \sigma$. Note that $\sigma$ induces a map $\sigma^Y\colon Y(\bS)\to Y(\bar{\bS})$, as well as $\sigma^W\colon \cW(Y,L,\bS)\to \cW(Y,L,\bar{\bS})$. We have, 
\[\sigma^Y_* F_{Y,\bL,\bS,\frs}=F_{Y,\bL,\bar{\bS},\sigma_*^W\frs},\] which follows from the fact that the maps do not depend on a choice of orientation of the $S^k$ factor of the framed spheres $\bS$ if $ k\in \{0,1,2\}$. Since the parametrizations $\Phi$ and $\Phi'$ also differ by composition with $\sigma^W$, the composed maps are unchanged.

\textbf{Move~\eqref{def:weakequiv:index12birthdeath}:} Invariance under index 1/2 critical point birth-deaths of $f|_{W\setminus \Sigma}$ follows from the same triangle map computation used to prove the analogous result about the closed 3-manifold invariants \cite{ZemGraphTQFT}*{Theorem~9.7}, which is a slight adaptation of the original argument by Ozsv\'{a}th and Szab\'{o} in \cite{OSTriangles}*{Lemma~4.16}, to handle our slightly different definition of the 1-and 3-handle maps.

\textbf{Move~\eqref{def:weakequiv:index23birthdeath}:} Invariance under index 2/3 critical point birth-deaths of $f|_{W\setminus \Sigma}$ follows from \cite{ZemGraphTQFT}*{Theorem~9.11}. Note that although generically a framed 2-sphere in $Y$ could intersect the link at a collection of points, by our construction, any framed 2-sphere appearing in a parametrized Kirby decomposition does not intersect the link $L$, which allows us to use the same model computation as in the case of \cite{ZemGraphTQFT}*{Theorem~9.11}. Again note that an analogous count of holomorphic triangles was proven by Ozsv\'{a}th and Szab\'{o} in \cite{OSTriangles}*{Lemma 4.17}.

\textbf{Move~\eqref{def:weakequiv:criticalpointswitchawayfromSigma}:} Invariance under critical value switches between two index 1 or two index 3 critical points of $f|_{W\setminus \Sigma}$ follows from Lemma~\ref{lem:1-3-handlemapscommute}, where we showed that pairs of 1-handle maps or pairs of 3-handle maps commute with each other.

\textbf{Move~\eqref{def:weakequiv:handleslide}:} Invariance from handleslides amongst the components of a framed 1-dimensional link follows identically to \cite{OSTriangles}*{Lemma 4.14}. If $(\Sigma, \as',\as,\bs)$ is subordinate to an $\alpha$-bouquet for $\bS^1$, and $(\bS^1)'$ is obtained from $\bS^1$ via a handleslide, then a triple subordinate to a bouquet for $(\bS^1)'$ can be obtained by handlesliding the associated curves in $\as'$ over each other. Invariance from such handleslide follows using the associativity relations, and a model computation (performed in \cite{OSTriangles}).

\textbf{Move~\eqref{def:weakequiv:birthdeathalongA}:} Invariance from a birth-death singularity of $f|_\cA$ follows from the relations 
\[S_{w,z}^-T_{w,z}^+\simeq  T_{w,z}^-S_{w,z}^+\simeq \id\] proven in Lemma~\ref{lem:addtrivialstrandII}. This is demonstrated in Figure~\ref{fig::61}.
 \begin{figure}[ht!]
\centering
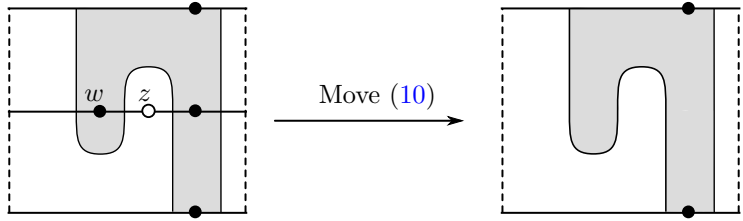
\caption{\textbf{Move~\eqref{def:weakequiv:birthdeathalongA}.} Invariance from this move can be proven using the relation $T_{w,z}^-S_{w,z}^+\simeq \id$, if we view the cobordism as going from bottom to top. If we view the cobordism as going from top to bottom, we use the relation $S_{w,z}^-T_{w,z}^+\simeq \id$. The map induced by the dividing set on the right is the identity map (in either direction). \label{fig::61}}
\end{figure}

\textbf{Move~\eqref{def:weakequiv:switchalongA}:} For a critical value switch of critical points of $f|_{\cA}$, there are several possible configurations, depending on whether the two type ($\mathcal{EPC}$-\ref{elemencobtype3}) elementary parametrized cobordisms are of subtype $S^+,$ $S^-,$ $T^+,$ or $T^-$. The cases where both are of type $S^+$ or $S^-$ are handled by Proposition~\ref{thm:quasistabcommute}. The cases where both are of type $T^+$ or $T^-$ are covered by Proposition~\ref{lem:z-quasistabcommute}, and the cases that one is of type $S^{+}$ or $S^-$ and the other is of type $T^{+}$ or $T^-$ is covered by Proposition~\ref{prop:TScommute}.

It is important to note that in this last case, one does not expect $T_{w,z}^{\circ}$ and $S_{w',z'}^{\circ'}$ to always commute. Indeed we gave an example in Lemma~\ref{lem:bypasstriplefromquasistabilization} where they did not commute. Proposition~\ref{prop:TScommute} proved that they commute under some additional hypotheses on the configuration of the basepoints $w$, $z$, $w'$ and $z'$. It is easy to check that the requirements on the configuration of basepoints is always satisfied when the basepoints arise in critical value switch of critical points of $f|_{\cA}$ where one induces a type $T^+$ or $T^-$ map, and the other induces a type $S^{+}$ or $S^-$ map.

\textbf{Move~\eqref{def:weakequiv:switchalongSigma}:} We now consider the effect of switching the order of two index 1 critical points on $\Sigma$. The maps associated to critical points of $f$ along $\Sigma$ are the compound 1-handle/band maps from Section~\ref{subsec:constructioncompound1-handlemaps}. The triangle map computation from Proposition~\ref{prop:1-handletrianglecount} implies that the 1-handle maps commute with the band maps. By Lemma~\ref{lem:1-3-handlemapscommute} the 1-handle maps commute with each other.  Thus it remains just to show that the band maps can be commuted amongst themselves. This now follows from Proposition~\ref{prop:alphabandscommute}. Note that if both band maps are $\ve{z}$-band maps or both are $\ve{w}$-band maps, then the maps always commute, though if one band map is $\ws$-band while the other
is a $\ve{z}$-band, we must be careful about the configurations of the ends of the of the bands, because Proposition~\ref{prop:alphabandscommute} has some non-trivial requirements on configuration of the ends of the bands. We note that the hypotheses of Proposition~\ref{prop:alphabandscommute} are  always satisfied when the two bands are induced from a surface with divides with a Morse function $f$ such that $f$ has two index one critical points, and $f|_{\cA}$ has no critical points.

\textbf{Move~\eqref{def:weakequiv:switchbetweenAandSigma}:} We now consider a critical value switch between critical points of $f|_{\Sigma}$ and $f|_{\cA}$. A critical point of $f|_{\Sigma}$ induces compound 1-handle/band map, and a critical point of $f|_{\cA}$ induces a quasi-stabilization map. The 1-handle maps commute with quasi-stabilization maps by Lemma~\ref{lem:1-handlesquasistabcommute} and the band maps by the triangle map computation of Proposition~\ref{prop:1-handletrianglecount}. Thus it is sufficient to analyze when the band maps commute with quasi-stabilization maps.

Commutation of the quasi-stabilization maps $S_{w,z}^{\circ}$ and $T_{w,z}^{\circ}$ with the band maps $F_{B}^{\ve{w}}$ and $F_{B}^{\ve{z}}$ now follows from Propositions~\ref{prop:zbandsandquasistab} and~\ref{prop:wbandsandquasistab}. We note that the hypotheses of those two propositions regarding the configuration of basepoints involved are always satisfied when the configuration of basepoints and bands are induced by a surface with divides with a Morse function $f$ which has a single index one critical point, and such that $f|_{\cA}$ also has a single index 1 critical point.

\textbf{Move~\eqref{def:weakequiv:critcrossesA}:} We now consider the move corresponding to a critical point of $\Sigma$ crossing $\cA$. The 1-handle maps commute with other 1-handle maps (Lemma~\ref{lem:1-3-handlemapscommute}) and the band maps (using the triangle counts from Proposition~\ref{prop:1-handletrianglecount}). Hence we can bring the 1-handle maps beneath the quasi-stabilization maps and the band maps. It follows that it is sufficient to consider the relation between the quasi-stabilization maps and the band maps. In this move, we have arranged for the crossing to locally be as in Figure~\ref{fig::57}. With basepoints, we can redraw this as in Figure~\ref{fig::59}.

 \begin{figure}[ht!]
\centering
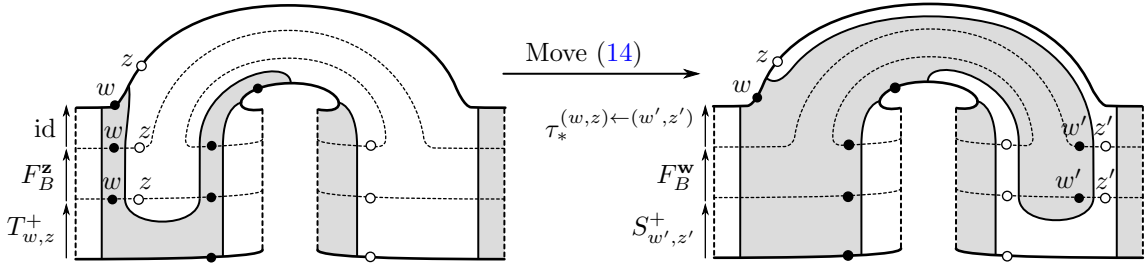
\caption{\textbf{A schematic of Move~\eqref{def:weakequiv:critcrossesA}}. \label{fig::59}}
\end{figure}

 The decomposition on the right side of Figure~\ref{fig::59} induces the composition
\[ 
\tau_*^{(w,z)\from (w',z')}F^{\ve{w}}_B S_{w',z'}^+,
\] 
where $\tau^{(w,z)\from (w',z')}$ is the diffeomorphism which fixes the link $L(B)$ setwise, but moves $(w',z')$ to $(w,z)$ while fixing all the other basepoints. The decomposition on the left side of Figure~\ref{fig::59} induces the composition
\[
F_B^{\ve{z}} T_{w,z}^+.
\] 
The diffeomorphism map $\tau_*^{ (w,z)\from (w',z')}$ is chain homotopic to  $S^-_{w',z'}T_{w,z}^+$ by Lemma~\ref{lem:newversionbasepointmovingmaps}. Thus we wish to show
\[
S^-_{w',z'}T_{w,z}^+ F^{\ve{w}}_B S_{w',z'}^+\simeq F_B^{\ve{z}} T_{w,z}^+.
\]
 This will be an algebraic computation. We first note that
\begin{equation}
\Psi_z T_{w,z}^+\simeq 0, \label{eq:PsiT=0}
\end{equation}
which follows from Lemmas~\ref{lem:T^+T^-=Psi}~and~\ref{lem:addtrivialstrandII}. We thus compute:
\begin{align*}S^-_{w',z'}T_{w,z}^+ F^{\ve{w}}_B S_{w',z'}^+&\simeq S^-_{w',z'} F^{\ve{w}}_B T_{w,z}^+ S_{w',z'}^+&&\text{(Proposition~\ref{prop:wbandsandquasistab})}\\
&\simeq S^-_{w',z'} (F_B^{\ve{z}}\Psi_z+\Psi_z F_B^{\ve{z}}) T_{w,z}^+ S_{w',z'}^+&&\text{(Proposition~\ref{prop:FBzFBwrelatedbyPsiPhi})}\\
&\simeq S^-_{w',z'} \Psi_z F_B^{\ve{z}} T_{w,z}^+ S_{w',z'}^+&& \text{(Equation \eqref{eq:PsiT=0})}\\
&\simeq S^-_{w',z'} \Psi_z F_B^{\ve{z}}S_{w',z'}^+ T_{w,z}^+&& \text{(Proposition~\ref{prop:TScommute})}\\
&\simeq S^-_{w',z'} \Psi_z S_{w',z'}^+ F_B^{\ve{z}} T_{w,z}^+&&\text{(Proposition~\ref{prop:zbandsandquasistab})}\\
&\simeq  F_B^{\ve{z}} T_{w,z}^+,&& \text{(Lemma~\ref{lem:oldaddtrivialstrand})}
\end{align*}
completing the proof.
\end{proof}

\subsection{Constructing the cobordism maps in the presence of empty ends}
\label{sec:removeballsforgeneralmaps}

In this section we construct cobordism maps when a link cobordism does not have enough incoming or outgoing ends. To define the maps for such link cobordisms, we construct the cobordism maps by puncturing the link cobordism (i.e. removing a 4-ball) at a collection of points along the dividing set, and then pre- and post-composing with the 0-handle and 4-handle maps from Section~\ref{sec:0/4-handlemaps}.

\begin{define}\label{def:standard4-ball} Suppose that $(W,\cF^\sigma)$ is a decorated link cobordism, with $\cF=(\Sigma,\cA)$. Suppose $B^4\subset \Int W$ is a closed 4-ball.  We say $B$ \emph{intersects $\cF$ standardly}  if the following hold:
\begin{enumerate}
\item $B\cap \Sigma$ is a 2-disk.
\item $B\cap \cA$ is a connected arc.
\item There are smooth coordinates $(x,y,w,z)$ on a neighborhood of $B\subset \Int W$ such that $B=\{(x,y,w,z): x^2+y^2+w^2+z^2\le 1\}$ and in these coordinates $\Sigma\cap B=\{(x,y,0,0)\}\cap B$.
\end{enumerate}
\end{define}

Now suppose that $(W,\cF^\sigma)\colon (Y_1,\bL_1)\to (Y_2,\bL_2)$ is a decorated link cobordism, possibly without enough incoming or outgoing ends. To define the cobordism map $F_{W,\cF^\sigma,\frs}$, we pick two collections of 4-balls, $B_1,\dots, B_n$, and  $B_1',\dots, B_{m}'$, which intersect $\cF$ standardly. We assume that all $n+m$ 4-balls are pairwise disjoint. Write $(W_0,\cF_0^{\sigma})$ for the link cobordism obtained by removing the interiors of the 4-balls $B_1,\dots, B_n,$  $B_1',\dots, B_m$. We view $(W_0,\cF_0^\sigma)$ as a cobordism from $(Y_1,\bL_1)\sqcup \bigsqcup_{i=1}^n (S^3, \bU)$ to $(Y_2,\bL_2)\sqcup \bigsqcup_{i=1}^m (S^3,\bU)$, where $\bU$ denotes a double based unknot in $S^3$. We can now define $F_{W,\cF^\sigma,\frs}$ as the composition
\begin{equation}
F_{W,\cF^\sigma,\frs}:=F_4F_{W_0,\cF_0^\sigma,\frs|_{W_0}}F_0,\label{eq:definitionpuncturedcobordismmap}
\end{equation}
where $F_4$ denotes the composition of all the 4-handle maps associated to the 4-balls $B_1',\dots, B_m'$, and $F_1$ denotes the composition of all the 0-handle maps associated to the 4-balls $B_1,\dots, B_n$.

We now prove that the above construction does not depend on which 4-balls $B_1,\dots, B_n$ and $B_1',\dots, B_m'$.

\begin{lem}\label{lem:canperformadditionalpunctures}Suppose $(W,\cF^\sigma)\colon (Y_1,\bL_1)\to (Y_2,\bL_2)$ is a decorated link cobordism, such that each component of $W$ and $\cF$ intersects a nonempty incoming end and outgoing end of $W$, and $B\subset \Int W$ is a 4-ball which intersects $\cF$ standardly. Write $(W_0,\cF_0)$ for the decorated link cobordism obtained by removing the interior of $B$ from $W_0$ and $\cF$, and declaring the new boundary copy of $(S^3,\bU)$ to be an incoming end. Write $(W_0',\cF_0')$ for the decorated link cobordism obtained by viewing the new copy of $(S^3,\bU)$ in the boundary as an outgoing end.  Then
\[
F_{W_0,{\cF}_0^\sigma,\frs|_{W_0}} F_{Y_1,\bL_1, \bS^{-1}}\simeq F_{W,\cF,\frs}\simeq  F_{Y_2\sqcup S^3, \bL_2\sqcup \bU, S^3} F_{W_0',({\cF_0}')^\sigma,\frs|_{W_0'}}.
\] 
\end{lem}

\begin{proof} Let us consider the first relation, 
\[F_{W_0,{\cF}_0^\sigma,\frs|_{W_0}} F_{Y_1,\bL_1, \bS^{-1}}\simeq F_{W,\cF,\frs}.
\]  Pick a parametrized Kirby decomposition $\cK$ for $(W_0,\cF_0)$. By performing Move \eqref{def:weakequiv:birthdeathalongA} (replacing an identity elementary cobordism with a pair of quasi-stabilizations), we can assume that $\cK$ contains an elementary link cobordism $\cC_i$ of type ($\mathcal{EPC}$-$\ref{elemencobtype3}_{S^+}$). By definition, the induced map is $S_{w,z}^+$, for two new basepoints $w$ and $z$. We now use Proposition~\ref{prop:bandsanddiskstab}, which shows that $S_{w,z}^+\simeq F_B^{\zs}\cB^+_{\bU,D} \phi_*$, for a diffeomorphism $\phi$ which is supported in a neighborhood of a subarc of the link between $w$ and $z$. We note that by definition $\cB^+_{\bU,D}$ is a 0-handle map followed by a 1-handle map, so 
\begin{equation} S_{w,z}^+\simeq F_{B}^{\zs}F_{Y\sqcup S^3,\bL\cup \bU,\bS^0} F_{Y,\bL,\bS^{-1}}\phi_*,\label{eq:birthdeathtocobordismmap}\end{equation} for a framed 0-sphere $\bS^0$ in $Y\sqcup S^3$. Hence, we create a parametrized Kirby decomposition $\cK_0$ of $(W_0,\cF_0)$ by replacing the elementary cobordism $\cC_i$ (corresponding to the quasi-stabilization map $S_{w,z}^+$), with a composition of two elementary cobordisms $\cC_{2}\circ \cC_1$, where $\cC_1$ corresponds to the diffeomorphism $\phi$ (i.e. $\cC_1$ is an elementary cobordism of type ($\mathcal{EPC}$-\ref{elemencobtype1})) and $\cC_2$ corresponds to the composition of the band map $F_{B}^{\zs}$ and the 1-handle map $F_{Y\sqcup S^3,\bL\cup \bU,\bS^0}$ (i.e. $\cC_2$ is an elementary cobordism of type ($\mathcal{EPC}$-\ref{elemencobtype4})).

Since we can commute the 0-handle map $F_{Y,\bL,\bS^{-1}}$ past all of the maps for elementary cobordisms $\cC_j$ with $j<i$, we see that
\[F_{W,\cF,\frs,\cK}\simeq F_{W_0,\cF_0,\frs|_{W_0}, \cK_0}F_{Y_1,\bL_1,\bS^{-1}}. \]
The second relation, involving $(W_0',\cF_0')$ is proven analogously.
\end{proof}

Using the above Lemma, we can finish the proof of Theorem~\ref{thm:A}, invariance of the cobordism maps:

\begin{proof}[Proof of Theorem~\ref{thm:A}] Theorem~\ref{thm:A'}, shows that the link cobordism maps are well-defined for a link cobordism $(W,\cF)$ whenever each component of $W$ and $\cF$ intersect the incoming and outgoing boundaries non-trivially. For a general cobordism we define the cobordism map by puncturing $(W,\cF)$ (removing 4-balls which intersect $\cF$ standardly) and then using Equation~\eqref{eq:definitionpuncturedcobordismmap}. To see that this is independent of the choice of 4-balls, we use Lemma~\ref{lem:canperformadditionalpunctures}, which implies that the map defined in Equation~\eqref{eq:definitionpuncturedcobordismmap} cobordism maps are invariant under removing additional 4-balls which intersect $\cF$ standardly. Given any two sets of 4-balls intersecting $\cF$ standardly, we can show that the induced maps are equal by moving between the two collections of 4-balls by sequentially removing additional 4-balls and filling in boundary spheres, using Lemma~\ref{lem:canperformadditionalpunctures} at each step.
\end{proof}

\section{The composition law}
\label{sec:compositionlaw}
In this Section we prove the composition law:

\begin{customthm}{B}Suppose that $(W,\cF^\sigma)\colon (Y_1,\bL_1)\to (Y_2,\bL_2)$ is decorated link cobordism which decomposes as the composition
\[(W,\cF^\sigma)=(W_2,\cF_2^{\sigma_2})\circ (W_1,\cF_1^{\sigma_1}).\]  If $\frs_1$ and $\frs_2$ are $\Spin^c$ structures on $W_1$ and $W_2$ respectively, then
\[F_{W_2,\cF_2^{\sigma_2},\frs_2}F_{W_1,\cF_1^{\sigma_1},\frs_1}\simeq \sum_{\substack{\frs\in \Spin^c(W)\\ \frs|_{W_i}=\frs_i}} F_{W,\cF^\sigma,\frs}.\]
\end{customthm}

\begin{proof} Suppose first that each component of $W_i$ and $\cF_i$ intersects an incoming and outgoing end of the cobordism $(W_i,\cF_i^\sigma)$ non-trivially. Take parametrized Kirby decompositions $\cK_1=\cK(f_1,\ve{b}_1)$ of $(W_1,\cF_1)$ and $\cK_2=\cK(f_2,\ve{b}_2)$ of $(W_2,\cF_2)$, induced by very nice Morse functions $f_1$ and $f_2$ with collections of regular values $\ve{b}_1$ and $\ve{b}_2$. Also pick gradient-like vector fields $v_1$ and $v_2$, inducing the decompositions. We can assume that $f_1$ and $f_2$ patch together, and $v_1$ and $v_2$ patch together to give Morse functions and gradient-like vector fields on $(W,\cF)$. The collections of regular values $\ve{b}_1$ and $\ve{b}_2$ induce a collection $\ve{b}$ for $f$. This yields a decomposition, $\cK(f,\ve{b})$, of $(W,\cF)$ into elementary parametrized cobordisms, which is the one obtained by just stacking $\cK_1$ and $\cK_2$. We note that $\cK(f,\ve{b})$ may not be a parametrized Kirby decomposition since the terms are in the wrong order, and there may be two 2-handle cobordisms. Nonetheless, the composition of the cobordism maps for each elementary parametrized cobordism in the decomposition is $F_{W_2,\cF_2^{\sigma_2},\frs_2}F_{W_1,\cF_1^{\sigma_1},\frs_1}$.

 Notice that the index 1, 2 and 3 dimensional critical points of $f_1$ which are away from $\Sigma_1$ have unstable manifolds which do not intersect $\Sigma$, and similarly the critical points of $f_2$ away from $\Sigma$ have stable manifolds which do not intersect $\Sigma$. Hence, assuming that $(f,v)$ is generic, we can modify the Morse functions to pull all of the critical points of $f_1|_{W_1\setminus \Sigma_1}$ above any critical point of $f_2|_{\Sigma_2}$ along $\Sigma$, and then pull the index 3 critical points of $f_1$ above the index 1 and 2 critical points of $f_2$, and pull the index 2 critical points of $f_1$ above the index 1 critical points of $f_1$.

The composition of the map for each piece of the stacked decomposition is invariant under pulling the 3-handles of $f_1$ above the 1-handles of $f_2$ by Lemma~\ref{lem:1-3-handlemapscommute}. The composition is also invariant of  pulling the 3-handles of $f_1$ past the 2-handles of $f_2$ by the triangle map computation of Proposition~\ref{prop:1-handletrianglecount}.

Hence it remains to show that the composition is unchanged when we pull the 1-,2- and 3-handles of $f_1|_{W_1\setminus \Sigma_1}$ past the elementary cobordisms associated to critical points of $f_2|_{\Sigma}$ and $f_2|_{\cA}$. The critical points of $f_2|_{\cA}$ induce maps which are a composition of the basepoint moving maps along the link components, and the quasi-stabilization maps. The maps for surgery on framed spheres away from the link commute with the basepoint moving maps by diffeomorphism invariance of the maps. The 1- and 3-handles of $f_1$ can be pulled past the critical points of $f_2|_{\cA}$ without changing the composition of the maps by Lemma~\ref{lem:1-handlesquasistabcommute} (showing that the 1-handle and 3-handle maps commute with quasi-stabilization). The 1- and 3-handle maps can be commuted past critical points of $f_2|_{\Sigma}$ using Lemma~\ref{lem:1-3-handlemapscommute} (to commute a 1- or 3-handle past the 1-handle map from the critical point along $\Sigma_2$) as well as the triangle map computation from Proposition~\ref{prop:1-handletrianglecount} (to commute a 1- or 3-handle past the band map for the critical point along $\Sigma_2$). The 2-handles of $f_1$ can be commuted past the critical points of $f_2|_{\cA}$ by Lemma~\ref{lem:2-handlemapsquasistabcommute}. Similarly the 2-handles of $f_1$ can be commuted past the critical points of $f_2|_{\Sigma}$ by Lemma~\ref{lem:bandmapsandsurgerymapscommute} (to commute the 2-handle maps past the band maps), and Proposition~\ref{prop:1-handletrianglecount} (to commute the 2-handle maps past the 1-handle map).

Hence $F_{W_2,\cF_2^{\sigma_2},\frs_2}F_{W_1,\cF_1^{\sigma_1},\frs_1}$ is equal to the composition of maps induced by a decomposition of $(W,\cF)$ into parametrized elementary cobordisms, which has elementary cobordisms satisfying the requirements of a parametrized Kirby decomposition, except that there are exactly two adjacent terms corresponding to surgery on framed 1-dimensional links. As $W$ is obtained from the union of the two 2-handle pieces by adding 1-handles and 3-handles,  a $\Spin^c$ structure defined on the union of the two 2-handle cobordisms extends uniquely to a $\Spin^c$ structure on all of $W$. Hence  Lemma~\ref{lem:compositionlawforlinks} (the composition law for the 2-handle maps) the theorem in this case.

Finally, we need to consider the case not all of the components of $W_i$ or $\cF_i$ intersect both an incoming and outgoing ends of $W_i$. The cobordism maps, in this case, are defined using Equation~\eqref{eq:definitionpuncturedcobordismmap} by removing a collection of 4-balls which intersect $\cF_i$ standardly. We note that the 4-handle maps for $(W_1,\cF_1)$ can always be commuted past each of the elementary cobordism maps in the composition for $(W_2,\cF_2)$. Similarly all of the 0-handle maps can be commuted past all of the elementary cobordism maps in the composition for $(W_1,\cF_1)$. Using the version of the composition law for link cobordisms $(W_i,\cF_i)$ where each component of $W_i$ and $\cF_i$ intersects both an incoming and outgoing boundary component of $W_i$ non-trivially, we obtain the general version of Theorem~\ref{thm:B}.
\end{proof}

\section{Algebraic reduction to the graph TQFT}
\label{sec:graphTQFTs}

 In \cite{ZemGraphTQFT}, the author constructed a graph TQFT for Heegaard Floer homology. In this section, we relate the reductions of the link cobordism maps from this paper to the graph TQFT maps.

Throughout this section, we  restrict to surfaces with divides which are colored using exactly two colors, one for the $\ws$-basepoints and regions, and one for the $\zs$-basepoints and regions. We will write $U$ for the variable corresponding to the $\ws$-basepoints and regions, and we will write $V$ for the variable corresponding to the $\zs$-basepoints and regions. It is straightforward to generalize the results of this section to more general colorings, as long as no $\ws$-basepoint or region shares the same color as a $\zs$-basepoint or region.

If $(Y,\bL)$ is a 3-manifold with a multi-based link, we note that
\[\cCFL^-(Y,\bL,\frs)\otimes_{\bF_2[U,V]} \bF_2[U,V]/(V-1)\iso \CF^-(Y,\ve{w},\frs)\] and 
\[\cCFL^-(Y,\bL,\frs)\otimes_{\bF_2[U,V]} \bF_2[U,V]/(U-1)\iso \CF^-(Y,\ve{z},\frs-\PD[L]).\] The change in $\Spin^c$ structure of the $U=1$ reduction is a consequence of Lemma~\ref{lem:changeSpincstructure}.

The link cobordism maps $F_{W,{\cF},\frs}$ thus naturally induce two maps on $\CF^-$, which we call the \textit{algebraic reductions} of $F_{W,{\cF},\frs}$. We will write $F_{W,{\cF},\frs}|_{U=1}$ or $F_{W,{\cF},\frs}|_{V=1}$, for these reductions.

We now state the main results of this section (note that the precise definitions of ribbon-equivalence and ribbon 1-skeletons will be provided in Section~\ref{subsec:graphdefinitionsandtopologicalprelims}):

\begin{customthm}{C}\label{thm:C'} Suppose that $(W,{\cF})\colon (Y_1,\bL_1)\to (Y_2,\bL_2)$ is a decorated link cobordism with ${\cF}=(\Sigma,\cA)$. If $\Gamma(\Sigma_{\ve{w}})$ and $\Gamma(\Sigma_{\ve{z}})$ are choices of ribbon 1-skeletons  of $\Sigma_{\ve{w}}$ and $\Sigma_{\ve{z}}$, then
\[F_{W,{\cF},\frs}|_{V=1}\simeq F_{W,\Gamma(\Sigma_{\ve{w}}),\frs}^B\qquad \text{and} \qquad F_{W,{\cF},\frs}|_{U=1}\simeq F_{W,\Gamma(\Sigma_{\ve{z}}),\frs-\PD[\Sigma]}^A.\]
\end{customthm}

See Definition~\ref{def:ribbongraphcore} for the precise definition of a ribbon 1-skeleton.

The following two results are immediate corollaries of the above theorem:

\begin{customcor}{D} If $(W,\Gamma)$ and $(W,\Gamma')$ are ribbon-equivalent ribbon graph cobordisms (Definition~\ref{def:ribbonequivalent}), then
\[F_{W,\Gamma,\frs}^A\simeq F_{W,\Gamma',\frs}^A \qquad \text{and} \qquad F_{W,\Gamma,\frs}^B \simeq F_{W,\Gamma',\frs}^B.\]
\end{customcor}

\begin{customcor}{E}
 The $V=1$ reduction of $F_{W,\cF,\frs}$,
\[
F_{W,\cF,\frs}|_{V=1}\colon \CF^-(Y_1,\ve{w}_1,\frs|_{Y_1})\to\CF^-(Y_2,\ve{w}_2,\frs|_{Y_2}),
\] depends only on the ribbon surface $\Sigma_{\ve{w}}$ and the $\Spin^c$ structure $\frs$. The $U=1$ reduction of $F_{W,\cF,\frs}$, 
\[
F_{W,\cF,\frs}|_{U=1}\colon \CF^-(Y_1,\ve{z}_1,\frs-\PD[L_1])\to\CF^-(Y_2,\ve{z}_2,\frs-\PD[L_2]),
\]
 depends only on the ribbon surface $\Sigma_{\ve{z}}$, and the $\Spin^c$ structure $\frs-\PD[\Sigma]$.
\end{customcor}

\subsection{Definitions and topological preliminaries}

In this section, we provide the topological definitions and some preliminary results used in the statement of Theorem~\ref{thm:C'} and the subsequent corollaries.

\label{subsec:graphdefinitionsandtopologicalprelims}
\begin{define}A \emph{ribbon graph cobordism}     between two multi-based 3-manifolds $(Y_1,\ve{w}_1)$ and  $(Y_2,\ve{w}_2)$ is a pair $(W,\Gamma)$ such that
\begin{enumerate}
\item $W$ is a cobordism from $Y_1$ to $Y_2$.
\item $\Gamma$ is an embedded graph in $W$ such that $\Gamma\cap \d W=\ve{w}_1\cup \ve{w}_2$ and each basepoint in $\ve{w}_1\cup\ve{w}_2$ has valence 1 in $\Gamma$.
\item Each vertex of $\Gamma$ has valance at least 1.
\item $\Gamma$ is decorated with  a choice of cyclic ordering of the edges adjacent to each vertex.
\end{enumerate}
\end{define}

The above definition naturally inspires the following definition:

\begin{define}\label{def:ribbonsurface} A \emph{ribbon surface cobordism}  between two multi-based 3-manifolds $(Y_1,\ve{w}_1)$ and $(Y_2,\ve{w}_2)$ is a pair $(W,R)$ such that $W$ is a cobordism from $Y_1$ to $Y_2$, and $R\subset W$ is an oriented surface with boundary and corners such that the following holds:
\begin{enumerate}
\item $R\cap Y_i$ is a finite collection of closed intervals, each containing exactly one of the basepoints in $\ve{w}_i$. Furthermore, each basepoint in $\ve{w}_i$ is contained in a component of $R\cap Y_i$.
\item The corners of $R$ correspond exactly to the boundaries of the closed intervals forming $R\cap Y_i$.
\item $R$ contains no closed components.
\end{enumerate}
\end{define}

\begin{lem}\label{lem:ribbongraphtoribbonsurface}Given a ribbon graph cobordism $(W,\Gamma)\colon (Y_1,\ve{w}_1)\to (Y_2,\ve{w}_2)$, there is an induced ribbon surface $R_\Gamma\subset W$, which is well-defined up to 1-parameter families of ribbon surfaces $(R_\Gamma)_t$ in $W$, each making $(W,(R_{\Gamma})_t)$ into a ribbon surface cobordism.
\end{lem}

\begin{proof}We first show how to construct the surface $R_\Gamma$. Pick an orientation of each edge of $\Gamma$, and also pick an oriented 2-plane in $W$ at each interior vertex of $\Gamma$. Isotope the edges of $\Gamma$ inside of $W$, while fixing the vertices,  so that the edges meet in the chosen 2-plane, and such that the ordering of the edges in the chosen oriented 2-plane field agrees with the cyclic ordering from the ribbon structure on $\Gamma$. Call this isotoped graph $\tilde{\Gamma}$. To each vertex in $\tilde{\Gamma}\cap \Int W$, we fix an oriented disk centered at the vertex, whose oriented tangent space agrees with the chosen oriented 2-plane field. Along each edge, we pick a nonzero vector field $w_e$ which is tangent to $e$, and positively oriented with respect to the orientation of  $e$. We also pick a non-vanishing vector field $v_e$ along $e$, which is orthogonal to $e$. We require that the pair $(v_e,w_e)$ restricts to an oriented basis on each of the disks which we centered at the vertices. At the vertices of $\Gamma\cap \d W$, we have no requirement on $w_e$. 

We can then build the surface $R_\Gamma$ by attaching strips to the disks, along the edges of $\tilde{\Gamma}$, such that $(v_e,w_e)$ gives an oriented basis of the tangent space along each edge.

To show uniqueness up to isotopy of the surface $R_\Gamma$, we note that the construction involved several steps: choosing an orientation of each edge, choosing an oriented 2-plane field at each vertex, choosing an isotopic $\tilde{\Gamma}$ whose edges meet in the chosen 2-plane fields, and choosing the vector fields $w_e$. Firstly, the construction is clearly independent of the orientations of the edges of $\Gamma$. Having fixed $v_e$, the space of non-vanishing vector fields $w_e$ satisfying the stated assumptions is homotopy equivalent to the space of maps from $(e,\d e)$ to $(S^2,\{pt\})$, and is hence connected. Lastly, we show independence of the choice of isotoped graph $\tilde{\Gamma}$. Suppose $\tilde{\Gamma}$ and $\tilde{\Gamma}'$ are two graphs, constructed by isotoping $\Gamma$ so that the edges meet in the chosen 2-plane fields at the vertices. Clearly the corresponding edges are homotopic relative to the vertices. Since $W$ is a 4-manifold, a homotopy of a 1-manifold can be taken to be an isotopy, so $\tilde{\Gamma}$ and $\tilde{\Gamma}'$ are isotopic through  graphs which meet in the chosen 2-plane field with the correct cyclic ordering at the vertices of $\Gamma$.
\end{proof}

\begin{define}\label{def:ribbonequivalent}We say that two graph cobordisms $(W,\Gamma_1),(W,\Gamma_2)\colon (Y_1,\ve{w}_1)\to (Y_2,\ve{w}_2)$ are \emph{ribbon-equivalent} if the two ribbon surfaces $R_{\Gamma_1}$ and $R_{\Gamma_2}$ are isotopic in $W$ through ribbon surfaces.
\end{define}

Given a ribbon surface cobordism $(W,R)$, we can go in the other direction, and construct a ribbon graph cobordism $(W,\Gamma_R)$:

\begin{define}\label{def:ribbongraphcore}If $(W,R)$ is a ribbon surface cobordism from $(Y_1,\ve{w}_1)$ to $(Y_2,\ve{w}_2)$, we say that a ribbon graph $\Gamma\subset R$ is a \emph{ribbon 1-skeleton} of $R$ if the cyclic orders of $\Gamma$ correspond to the orientation of $R$, and if the surface $R_\Gamma$ (constructed in Lemma~\ref{lem:ribbongraphtoribbonsurface}) can be constructed to lie in $R$, and such that the inclusion $R_\Gamma \hookrightarrow R$ is isotopic to a diffeomorphism between $R_\Gamma$ and $R$.
\end{define}

Note that there are many non-isotopic ribbon 1-skeletons for a fixed ribbon surface (in fact there are many which are not even isomorphic as graphs). We will not endeavor to create a set of moves between two ribbon 1-skeletons of a ribbon surface, though presumably a small set of moves could be established. In Figure~\ref{fig::79} we show a few examples of ribbon 1-skeletons of surfaces.

 \begin{figure}[ht!]
 \centering
 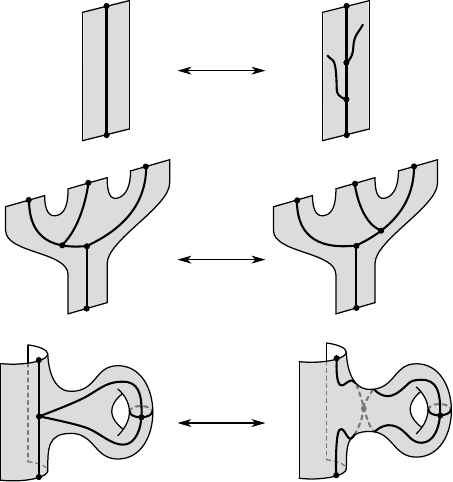
 \caption{\textbf{Examples of different ribbon 1-skeletons for ribbon surfaces.} The ribbon 1-skeletons are, by definition, ribbon-equivalent. The pairs of ribbon graphs are not isomorphic as ribbon graphs. The top pair of ribbon 1-skeletons are related by adding trivial strands. The middle pair are related by sliding an edge across a vertex, in a way which is compatible with the cyclic orderings. The bottom pair can be related by a sequence of such moves. Note that the bottom graphs are isomorphic as graphs, but not as ribbon graphs.\label{fig::79}}
 \end{figure}

\subsection{Construction and basic relations of the graph TQFT}
\label{subsec:graphTQFTmapsandrelations}

In this section, we recall the construction of the graph cobordism maps from \cite{ZemGraphTQFT}, and prove some basic results, which will be helpful for our purposes.

Analogous to the link cobordism maps defined in this paper, the graph cobordism maps are constructed by decomposing a graph cobordism into elementary pieces, and defining a map for each piece. There are 1-handle, 2-handle and 3-handle maps, modeled on the maps defined by Ozsv\'{a}th and Szab\'{o} \cite{OSTriangles}, and similar to the ones we defined for link cobordisms earlier in this paper. There are also 0-handle and 4-handle maps, corresponding to adding or removing a copy of $(S^3,w)$, similar to the maps defined in Section~\ref{sec:0/4-handlemaps}.  A key component of the construction of the graph cobordism maps is a maps for graphs embedded in a fixed 3-manifold. We make the following definition:

\begin{define}A \emph{ribbon flow-graph} $\cG=(\Gamma,V_1, V_2)$ in $Y^3$ is a tuple consisting of a ribbon graph $\Gamma\subset Y$, with two disjoint subset of vertices $V_1$ and $V_2$ in $V(\Gamma)$ such that the following hold:
\begin{enumerate}
\item $V_1$ and $V_2$ both intersect each component of $Y$ non-trivially.
\item Each vertex in $V_1$ and $V_2$ has valence 1 in $\Gamma$.
\item Each vertex in $\Gamma$ has valence at least 1.
\end{enumerate}
\end{define}
If $\cG=(\Gamma, V_1,V_2)$ and $\cG'=(\Gamma',V_2,V_3)$ are two flow-graphs with $\Gamma\cap \Gamma'=V_2$, then the composition of $\cG'$ and $\cG$ is defined by concatenation: $\cG'\circ \cG=(\Gamma'\cup \Gamma, V_1, V_3)$.

In \cite{ZemGraphTQFT}, a \textit{graph action map} was defined defined. If $\cG=(\Gamma,V_1,V_2)$ is a ribbon flow-graph in a fixed 3-manifold $Y$, the graph action map is a chain map
\[
A_{\cG,\frs}\colon \CF^-(Y,V_1,\frs)\to \CF^-(Y,V_2,\frs).
\]
 There is a natural variation, $B_{\cG,\frs}$, with the same domain and range, which will discuss later.

If $\cG=(\Gamma,V_1,V_2)$ is a flow-graph in $Y$, then we can construct a graph cobordism $([0,1]\times Y,\hat{\Gamma})\colon ( Y, V_1)\to (Y,V_2)$, such that the projection of $\hat{\Gamma}$ onto $Y$ is the graph $\Gamma$. The graph cobordism maps $F_{[0,1]\times  Y, \hat{\Gamma},\frs}^A$ and $F_{[0,1]\times Y,\hat{\Gamma},\frs}^B$ are equal to the graph action maps $A_{\cG,\frs}$ and $B_{\cG,\frs}$, respectively.

There are two key constituent maps which feature in the construction of the graph action map: the \emph{relative homology maps}, and the \emph{free-stabilization maps}. 

We now describe the relative homology maps. Given a path $\lambda$ between two basepoints $w_1$ and $w_2$ in $Y$, there is a $-1$ graded endomorphism
\[A_\lambda\colon \CF^-(Y,\ve{w},\frs)\to \CF^-(Y,\ve{w},\frs),\] which satisfies
\[\d  A_\lambda+A_\lambda \d=U_{w_1}+U_{w_2}.\] On the level of Heegaard diagrams, the map $A_\lambda$ is defined by homotoping the path $\lambda$ so that it lies in the Heegaard surface, and then defining
\[
A_\lambda(\ve{x}):=\sum_{\substack{\phi\in \pi_2(\ve{x},\ve{y})\\ \mu(\phi)=1}} a(\lambda,\phi) \# \Hat{\cM}(\phi) U_{\ve{w}}^{n_{\ve{w}}(\phi)}\cdot \ve{y}.
\]
 The quantity $a(\lambda,\phi)\in \bF_2$ is defined to be the sum of changes of the multiplicities of the class $\phi$ over the $\ve{\alpha}$ curves, as one traverses $\lambda$.
 
 Naturally, there is a map $B_\lambda$ which is constructed similarly, but instead weighting by the quantity $b(\lambda,\phi)\in \bF_2$ obtained by counting the total sum of changes of $\phi$ across the $\ve{\beta}$ curves as one travels along $\lambda$. The sum $A_\lambda+B_{\lambda}$ counts disks with an extra factor equal to the sum of changes across both the $\as$ and $\bs$ curves, which is just the difference between the multiplicities at $w_1$ and $w_2$. Hence
\[
A_{\lambda}+B_{\lambda}=U_{w_1}\Phi_{w_1}+U_{w_2} \Phi_{w_2},
\] 
where
\[
\Phi_{w_i}(\ve{x}):=U_{w_i}^{-1} \sum_{\substack{\phi\in \pi_2(\xs,\ys)\\ \mu(\phi)=1}} n_{w_i}(\phi) \# \Hat{\cM}(\phi) U_{\ws}^{n_{\ws}(\phi)} \cdot \ys.
\]

The following relations are obtained in \cite{ZemGraphTQFT}*{Lemma~5.10, 5.12} by counting the ends of index 2 moduli spaces of holomorphic curves:

\begin{equation}
A_\lambda^2\simeq U_{w_1}\simeq U_{w_2},\label{eq:graphTQFT1}
\end{equation}
\begin{equation}
A_{\lambda_1}  A_{\lambda_2}+A_{\lambda_2}A_{\lambda_1}\simeq \sum_{w\in (\d \lambda_1)\cap (\d \lambda_2)} U_{w}.\label{eq:graphTQFT2}\end{equation}

We now describe the second type of map featuring in the construction of the graph action maps: the free-stabilization maps. The free-stabilization maps have domain and range equal to two different Heegaard Floer complexes. If $(Y,\ve{w})$ is a multi-based 3-manifold and $w\not\in \ve{w}$, in \cite{ZemGraphTQFT} a map
\[
S_{w}^+\colon  \CF^-(Y,\ve{w},\frs)\to \CF^-(Y,\ve{w}\cup \{w\},\frs)
\]
 is defined. There is also a map, $S_{w}^-$, defined in the opposite direction. If $\cH=(\Sigma,\as,\bs,\ws)$ is a diagram for $(Y,\ws)$ such that $w\in \Sigma\setminus (\as\cup \bs)$, then a diagram for $(Y,\ws\cup \{w\})$ can be constructed by adding two new curves, $\alpha_0$ and $\beta_0$, which are contained in a small ball bounding $w\in \Sigma$. Furthermore, we pick $\alpha_0$ and $\beta_0$ so that $\alpha_0\cap \beta_0$ consists of two points, which are distinguished by the relative grading. The free-stabilization maps are defined by the equations
\[
S_{w}^+(\xs)=\xs\otimes \theta^+,
\]
\[
S_{w}^-(\xs\otimes \theta^+)=0\qquad \text{and} \qquad S_{w}^-(\xs\otimes \theta^-)=\xs.
\]
An alternate description of the free-stabilization map $S_w^+$ is the composition of a 0-handle map, which adds a new copy of $(S^3,w_0)$, and then a 1-handle map for a 1-handle with a foot near $w\in Y$ and a foot near $w_0\in S^3$. As such the formula for $S_w^+$ is similar to the 1-handle map defined in Section~\ref{sec:1--handlemaps}. 

In \cite{ZemGraphTQFT}*{Proposition~6.23}, it is shown that
\begin{equation}
S_{w_1}^{\circ} S_{w_2}^{\circ'}\simeq S_{w_2}^{\circ'} S_{w_1}^{\circ}\label{eq:graphTQFT3}
\end{equation} 
for any distinct $w_1$ and $w_2$ and $\circ,\circ'\in \{+,-\}$ (compare Lemma~\ref{lem:1-3-handlemapscommute}). If $\{w_1,\dots, w_n\}$ is a collection of distinct basepoints, we will write $S_{w_n,\dots, w_1}^\circ$ for the composition
\[
S_{w_n,\dots ,w_1}^{\circ}:=S_{w_n}^\circ\cdots S_{w_1}^\circ,\] noting that the composition is independent of the ordering of $w_1,\dots, w_n$ by Equation~\eqref{eq:graphTQFT3}.

 Also, 
\begin{equation}
S_{w}^{\circ} A_\lambda\simeq A_{\lambda} S_{w}^{\circ}\label{eq:graphTQFT4}
\end{equation} 
for $\circ\in \{+,-\}$,  as long as $w\not\in \d \lambda$ \cite{ZemGraphTQFT}*{Lemma~6.6}. Additionally, according to \cite{ZemGraphTQFT}*{Lemma~14.16}, one has
\begin{equation}
S_w^+ S_w^-\simeq \Phi_w\label{eq:graphTQFT5}.
\end{equation} 
By examining the formulas for the free-stabilization maps, one immediately has
 \begin{equation}
 S_w^- S_w^+\simeq 0.\label{eq:graphTQFT6}
 \end{equation} 
 If $\lambda$ is a path from $w_1$ to $w_2$, and $w_1\in \ve{w}$ but $w_2\not \in \ve{w}$, then 
 \begin{equation}
 S_{w_2}^- A_\lambda S_{w_2}^+\simeq \id_{\CF^-(Y,\ws,\frs)}.
 \label{eq:graphTQFT7}
 \end{equation}
 We refer to Equation~\eqref{eq:graphTQFT7} as the \textit{trivial strand relation}, since it corresponds to invariance from the manipulation of graphs shown on the top of Figure~\ref{fig::79}.
 
 Finally, there is the \textit{basepoint moving relation}. If $\lambda$ is a path from $w_1$ to $w_2$ and $w_1,w_2\not \in \ve{w}$, then the diffeomorphism $\phi_\lambda$ which moves $w_1$ to $w_2$ along $\lambda$ induces a chain homotopy equivalence
 \[
 (\phi_\lambda)_*\colon \CF^-(Y,\ve{w}\cup \{w_1\},\frs)\to \CF^-(Y,\ve{w}\cup \{w_2\},\frs).
 \]
  The basepoint moving relation is that
 \begin{equation}(\phi_\lambda)_*\simeq S_{w_1}^- A_\lambda S_{w_2}^+.\label{eq:graphTQFT8}\end{equation}
 This is proven in \cite{ZemGraphTQFT}*{Section~14}. It is instructive to compare Equation~\eqref{eq:graphTQFT8} to the basepoint moving relations from Section~\ref{sec:quasi-stab-and-identifications}, using the reduction techniques from later in this section (in particular Lemma~\ref{lem:Alambda=Psi+UPhi}).

\subsection{Constructing the graph action map}
To define the graph action map $A_{\cG,\frs}$, one subdivides the graph $\Gamma$ by adding extra vertices, and writes $\cG$ (now subdivided) as a composition of \textit{elementary flow-graphs}. Such a decomposition is termed a \textit{Cerf decomposition} of the graph. By definition, there are three types of elementary flow-graphs $\cG=(\Gamma,V_1,V_2)$:

\begin{enumerate}[leftmargin=20mm, ref= 
\textrm{\arabic*}, label = \textrm{($\mathcal{E\!FG}$-\arabic*)}:]
\item \label{def:elemengraphtype1} $V(\Gamma)=V_1\sqcup V_2$ and all edges go from $V_1$ to $V_2$;
\item \label{def:elemengraphtype2} $V(\Gamma)=V_1\sqcup V_2\sqcup \{v_0\}$ and all edges either connect a vertex in $V_1$ to $V_2$ or connect a vertex in $V_1\cup V_2$ to $\{v_0\}$;
\item \label{def:elemengraphtype3} $V(\Gamma)=V_1\sqcup V_2$ and there is exactly one edge which connects two vertices in $V_1$ or two vertices in $V_2$.
\end{enumerate}

Examples of the three subtypes of elementary ribbon flow-graphs are shown in Figure~\ref{fig::84}. Note than any flow-graph can be subdivided so that it is a concatenation of elementary flow-graphs. In \cite{ZemGraphTQFT}*{Theorem~7.6} it is shown that the graph action map is invariant under the choice of decomposition into elementary flow-graphs. Furthermore, adapting the proof of \cite{ZemGraphTQFT}*{Theorem~7.6} shows that the graph action map is also invariant under subdivision.

 \begin{figure}[ht!]
 \centering
 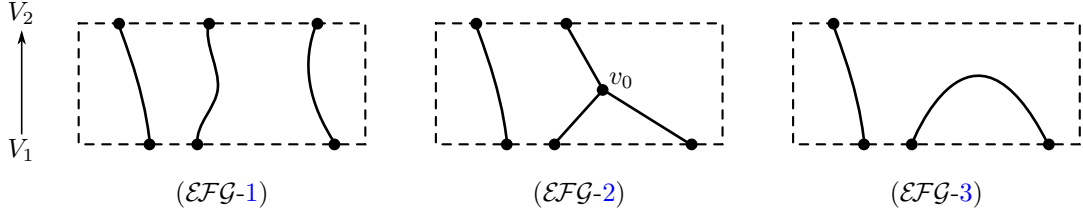
 \caption{\textbf{Examples of the three types elementary ribbon flow-graphs.}\label{fig::84}}
 \end{figure}

If $\cG=(\Gamma,V_1,V_2)$ is an elementary flow-graph satisfying ~($\mathcal{E\!FG}$-\ref{def:elemengraphtype1}), the graph action map is defined to be
\[
A_{\cG,\frs}:=\bigg(\prod_{v\in V_{1}} S_v^-\bigg) \bigg(\prod_{e\in E(\Gamma)} A_{e}\bigg) \bigg( \prod_{v\in V_2} S_{v}^+\bigg).\tag{$\mathcal{E\!FG}$\text{-\ref{def:elemengraphtype1}}}
\] 
The map does not depend on the ordering of the terms in the three factors, by Equations \eqref{eq:graphTQFT3} and \eqref{eq:graphTQFT2}.

If $\cG$ is an elementary flow-graph satisfying ~($\mathcal{E\!FG}$-\ref{def:elemengraphtype2}), the graph action map is defined to be
\[
A_{\cG,\frs}:=\bigg(\prod_{v\in V_1\cup \{v_0\}}S_v^-\bigg)\bigg(\prod_{\substack{e\in E(\Gamma)\\ v_0\not \in \d e}} A_e\bigg)(A_{e_n} \cdots  A_{e_1})\bigg(\prod_{v\in V_2\cup \{v_0\}}S_v^+\bigg),\tag{$\mathcal{E\!FG}$\text{-\ref{def:elemengraphtype2}}}
\] 
where $e_1,\dots, e_n$ are the edges adjacent to $v_0$, in the order that they appear according to the cyclic ordering. In \cite{ZemGraphTQFT}*{Lemma~6.12} it is shown that this expression is invariant under cyclic permutation of the edges $e_1,\dots, e_n$.

Finally, if $\cG=(\Gamma,V_1,V_2)$ is an elementary flow-graph satisfying~($\mathcal{E\!FG}$-\ref{def:elemengraphtype3}), the graph action map is defined to be
\[
A_{\cG,\frs}:=\bigg(\prod_{v\in V_{1}} S_v^-\bigg) \bigg(\prod_{e\in E(\Gamma)} A_{e}\bigg) \bigg( \prod_{v\in V_2} S_{v}^+\bigg).\tag{$\mathcal{E\!FG}$\text{-\ref{def:elemengraphtype3}}}
\]

The map $B_{\cG,\frs}$ is defined similarly, by simply replacing each instance of $A_{\lambda}$ with $B_{\lambda}$, in the above formulas. The maps $A_{\cG,\frs}$ and $B_{\cG,\frs}$ are related by the following result:

\begin{prop}If $\cG=(\Gamma,V_1,V_2)$ is a ribbon flow-graph, and $\bar{\cG}=(\bar{\Gamma},V_1,V_2)$ denotes the ribbon flow-graph with all cyclic orders reversed, then
\[B_{\bar{\cG},\frs}=A_{\cG,\frs}.\]
\end{prop}

We will not have a need for the above relation. A proof for trivalent graphs can be found in \cite{HMZConnectedSums}*{Lemma~5.9}. For general graphs, one can use Lemma~\ref{lem:replacewithtrivalent}, below, to reduce to the case of trivalent graphs.

\subsection{Relating the graph action map to link Floer homology}

In this section, we prove several useful computations of graph action maps, and compute the algebraic reductions of the basepoint actions on the link Floer complexes. In Section~\ref{sec:proofofThmC}, we will use these results to prove Theorem~\ref{thm:C}.

A first relation, which will be useful for us, is related to subdivision invariance of the graph action map:

\begin{lem}\label{lem:invariantundersubdivision} Suppose that $e_1$ and $e_2$ are two paths which share a single endpoint $w$.  Then, writing $e_1*e_2$ for the concatenation, one has
\[A_{e_2* e_1}\simeq S_{w}^- A_{e_2}A_{e_1} S_w^+.\]
\end{lem}
\begin{proof}On a diagram which has $w$ as a basepoint, additivity of the quantity $a(\lambda,\phi)$ under concatenation of arcs implies that $A_{e_2*e_1}\simeq A_{e_2}+A_{e_1}$. Hence $A_{e_2}\simeq A_{e_2*e_1}+A_{e_1}$. Hence
\[
S_{w}^- A_{e_2}A_{e_1} S_w^+\simeq S_{w}^- (A_{e_2*e_1}+A_{e_1})A_{e_1} S_w^+\simeq A_{e_2*e_1} S_w^- A_{e_1} S_{w}^+ +U S_w^- S_w^+\simeq A_{e_2* e_1},
\]
 as we wanted.
\end{proof}

Any ribbon flow-graph in $Y$ is ribbon-equivalent to a trivalent ribbon graph. At any vertex of valence greater than three, one can split off a consecutive pair of edges, reducing the number of edges adjacent to the vertex, while adding a new trivalent vertex to the graph. We show that the graph action map is invariant under this move:

\begin{lem}\label{lem:replacewithtrivalent} If $\cG=(\Gamma,V_1,V_2)$ is a ribbon flow-graph in $Y$, and $\cG'=(\Gamma',V_1,V_2)$ is related to $\cG$ by splitting off a consecutive pair of edges at a given vertex of $\Gamma$, and forming a new trivalent vertex, then
\[A_{\cG,\frs}\simeq A_{\cG',\frs}.\]
\end{lem}
\begin{proof}Since the graph action map is constructed by decomposing a subdivision of the graph into a composition of elementary flow-graphs, it is sufficient to prove the claim for elementary flow-graphs satisfying~($\mathcal{E\!FG}$-\ref{def:elemengraphtype2}), the only type of elementary ribbon flow-graph which can have a vertex of valence greater than three. Suppose that $v_0$ is the middle vertex and suppose that $e_1,\dots, e_n$  are the edges adjacent to $v_0$, indexed according to the cyclic ordering assigned to $v_0$. Let $e_0',$ $e_1',$ and $e_2'$, denote three new edges adjacent to a new vertex $v_0'$, as in Figure~\ref{fig::81}. By definition
\[A_{\cG,\frs}=\bigg(\prod_{v\in V_1\cup \{v_0\}}S_v^-\bigg)\bigg(\prod_{\substack{e\in E(\Gamma)\\ v_0\not \in \d e}} A_e\bigg)(A_{e_n} \cdots  A_{e_1})\bigg(\prod_{v\in V_2\cup \{v_0\}}S_v^+\bigg).\]

By commuting the various $S_{v}^+$ amongst themselves, we can isolate a factor of $A_{e_2}A_{e_1} S_{v_0}^+$ before the other relative homology maps but after all of the other positive free-stabilization maps. We now claim that
\begin{equation}
S_{v_0'}^- A_{e_0'} A_{e_2'} A_{e_1'} S_{v_0',v_0}^+\simeq A_{e_2}A_{e_1} S_{v_0}^+.
\label{eq:graphTQFTslideedge}\end{equation} 
To this end, the our strategy will be to identify both expressions with a third expression, namely 
\begin{equation}
S_{v_0'}^- A_{e_0'} (A_{e_2'}+A_{e_0'})(A_{e_1'}+A_{e_0'}) S_{v_0',v_0}^+.\label{eq:graphTQFTslideedge2}
 \end{equation} 
 To show the left side of Equation \eqref{eq:graphTQFTslideedge} is equal to the Equation \eqref{eq:graphTQFTslideedge2}, we compute by multiplying out Equation~\eqref{eq:graphTQFTslideedge2}, and rearranging terms:
\begin{equation*}
\begin{split}&S_{v_0'}^- A_{e_0'} (A_{e_2'}+A_{e_0'})(A_{e_1'}+A_{e_0'}) S_{v_0',v_0}^+\\
\simeq &S_{v_0'}^-(A_{e_0'}A_{e_2'}A_{e_1'}+A_{e_0'}^2A_{e_1'}+A_{e_0'}A_{e_2'}A_{e_0'}+ A_{e_0'}^3) S_{v_0',v_0}^+\\
\simeq  &S_{v_0'}^-(A_{e_0'}A_{e_2'}A_{e_1'}+A_{e_0'}^2A_{e_1'}+A_{e_0'}^2A_{e_2'}+U A_{e_0'}+ A_{e_0'}^3) S_{v_0',v_0}^+\\
\simeq& S_{v_0'}^-(A_{e_0'}A_{e_2'}A_{e_1'}+UA_{e_1'}+UA_{e_2'}+U A_{e_0'}+ UA_{e_0'}) S_{v_0',v_0}^+\\
\simeq &S_{v_0'}^-A_{e_0'}A_{e_2'}A_{e_1'}S_{v_0',v_0}^++US_{v_0'}^-(A_{e_1'}+A_{e_2'}) S_{v_0',v_0}^+\\
\simeq &S_{v_0'}^-A_{e_0'}A_{e_2'}A_{e_1'}S_{v_0',v_0}^++U(S_{v_0'}^-A_{e_1'} S_{v_0'}^++S_{v_0'}^-A_{e_2'} S_{v_0'}^+)S_{v_0}^+\\
\simeq &S_{v_0'}^-A_{e_0'}A_{e_2'}A_{e_1'}S_{v_0',v_0}^++2 \cdot U S_{v_0}^+\\
 \simeq &S_{v_0'}^-A_{e_0'}A_{e_2'}A_{e_1'}S_{v_0',v_0}^+.
\end{split}
\end{equation*}

To see that the right side of Equation \eqref{eq:graphTQFTslideedge} is equal to the expression from Equation \eqref{eq:graphTQFTslideedge2}, we compute that 
\[S_{v_0'}^- A_{e_0'} (A_{e_2'}+A_{e_0'})(A_{e_1'}+A_{e_0'}) S_{v_0',v_0}^+\simeq S_{v_0'}^- A_{e_0'} (A_{e_2'*e_0'})(A_{e_1'*e_0'}) S_{v_0',v_0}^+\]
\[\simeq S_{v_0'}^- A_{e_0'}S_{v_0'}^+ (A_{e_2'*e_0'})(A_{e_1'*e_0'}) S_{v_0}^+\simeq A_{e_2}A_{e_1} S_{v_0}^+.\] Hence we conclude that
\[A_{\cG,\frs}\simeq \bigg(\prod_{v\in V_1\cup \{v_0\}}S_v^-\bigg)\bigg(\prod_{\substack{e\in E(\Gamma)\\ v_0\not \in \d e}} A_e\bigg)(A_{e_n} \cdots A_{e_3} ) (S_{v_0'}^-A_{e_0'}A_{e_2'}A_{e_1'}S_{v_0',v_0}^+)\bigg(\prod_{v\in V_2}S_v^+\bigg).\] This expression is almost, but not quite, the expression for a Cerf decomposition of $\cG'$. To obtain the map induced by a genuine Cerf decomposition of $\cG'$, we use Lemma~\ref{lem:invariantundersubdivision} to subdivide the edge $e_0'$, and also all the other edges which are adjacent to a vertex in $V_1$. We leave this last step to the reader, since the manipulation is straightforward. The process is shown schematically in Figure~\ref{fig::81}.
\end{proof}

 \begin{figure}[ht!]
 \centering
 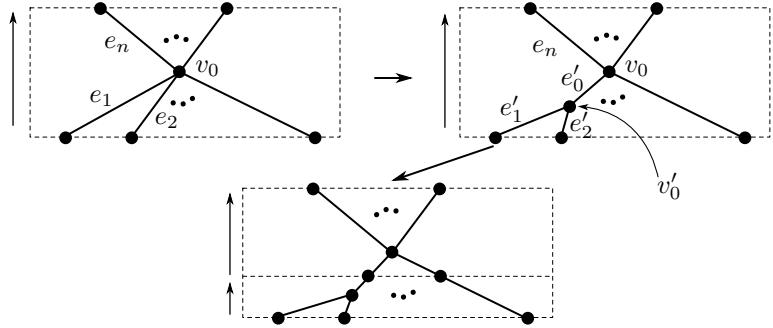
 \caption{\textbf{Replacing a pair of consecutive edges at a higher valence vertex with three edges and a trivalent vertex.} On the top, we show the manipulation performed from Lemma~\ref{lem:replacewithtrivalent}. In order to turn this into a Cerf decomposition, one must additionally also subdivide several edges of the graph, as shown on the bottom. Cyclic orders are counterclockwise with respect to the page.\label{fig::81}}
 \end{figure}

\begin{rem} Lemma~\ref{lem:replacewithtrivalent} can be viewed as a first step toward proving invariance of the graph cobordism maps under ribbon-equivalence, Corollary~\ref{cor:D}. For example, it implies invariance from equivalences like the one shown in the second row of Figure~\ref{fig::79}.
\end{rem}

For the purposes of comparing the graph TQFT to the reductions of the link Floer TQFT, we will need to compute the graph cobordism map for $W=[0,1]\times Y$ containing the $H$-shaped graph shown in Figure~\ref{fig::81}.

 \begin{figure}[ht!]
 \centering
 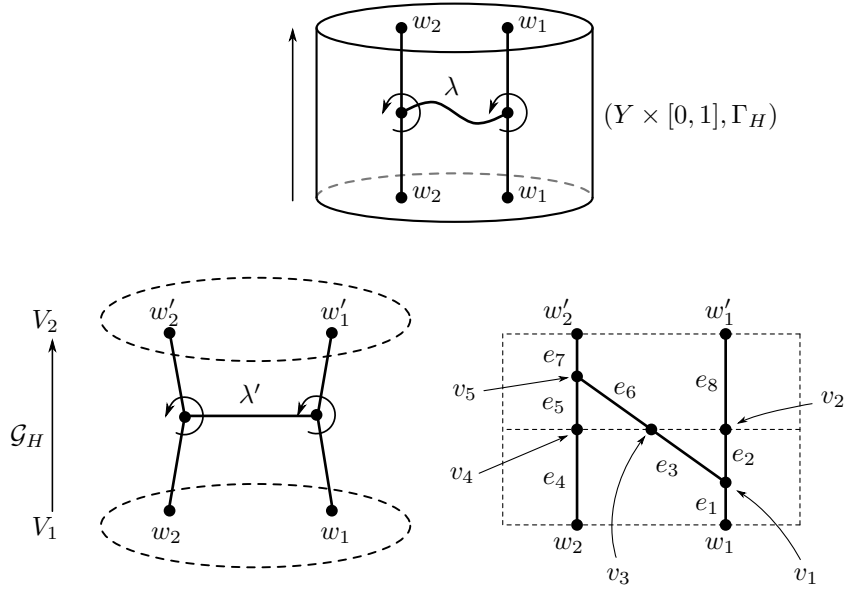
 \caption{\textbf{The graph cobordism $([0,1]\times Y,\Gamma_H)$ and the ribbon flow-graph $\cG_H=(H,\{w_1,w_2\},\{w_1',w_2'\})$, obtained by pushing $\Gamma_H$ into $\{0\}\times Y$.} In Lemma~\ref{lem:Hgraphaction}, we compute cobordism map for the graph cobordism $([0,1]\times Y,\Gamma_H)$, by computing the graph action map for $\cG_H$. A Cerf decomposition of the flow-graph is shown on the bottom right. \label{fig::80}}
 \end{figure}

\begin{lem}\label{lem:Hgraphaction}Suppose $\lambda$ is a path in $Y$ from $w_1$ to $w_2$. Let $\Gamma_H$ denote the graph $([0,1]\times \{w_1,w_2\} )\cup \{\tfrac{1}{2}\}\times \lambda$, with cyclic orders as shown in Figure~\ref{fig::80}. Then 
\[F_{[0,1]\times Y,\Gamma_H,\frs}^A\simeq  A_\lambda+U\Phi_{w_1},\qquad \text{and} \qquad F_{[0,1]\times Y,\Gamma_H,\frs}^B\simeq B_\lambda+U\Phi_{w_1}.\] The same formula holds in the case that $Y$ has additional basepoints $\ve{w}$, and we add $[0,1]\times \ve{w}$ to $\Gamma_H$.
\end{lem}

\begin{proof}Let $\cG_{H}$ denote the ribbon flow-graph in $Y$ which is shown in Figure~\ref{fig::80}. In Figure~\ref{fig::80}, we show a Cerf decomposition for $\cG_H$ into two elementary ribbon flow-graphs, both satisfying ~($\mathcal{E\!FG}$-\ref{def:elemengraphtype2}). The flow-graph $\cG_H=(H,\{w_1,w_2\},\{w_1',w_2'\})$ is obtained by pushing $\Gamma_H$ into $\{0\}\times Y$.  The graph cobordism map $F_{[0,1]\times Y,\Gamma_H,\frs}^A$ is equal to the graph action map $A_{\cG_H,\frs}$, composed with basepoint moving maps to move $w_1'$ to $w_1$ and $w_2'$ to $w_2$.

We will use  the relations from Section~\ref{subsec:graphTQFTmapsandrelations} to compute the maps $A_{\cG_H}$ and $B_{\cG_H}$. If $e$ is a path between two basepoints in $Y$, we will write $\phi_e$ for the diffeomorphism of $Y$ obtained by pushing a basepoint along $e$. Using the Cerf decomposition shown in Figure~\ref{fig::80}, we see that the graph action map takes the form
\begin{align*}A_{\cH,\frs}&=S_{w_1w_2v_1v_2v_3v_5}^-  A_{e_8}A_{e_7}A_{e_6}A_{e_5}A_{e_4}A_{e_3}A_{e_2}A_{e_1}S^+_{v_1v_2v_3v_4v_5w_1'w_2'}\\
&\simeq (S_{v_5}^- A_{e_7} S_{w_2'}^+) S_{v_1v_2v_4}^-  A_{e_8}(S_{v_3}^-A_{e_6}A_{e_3}S_{v_3}^+)A_{e_2} S^+_{v_2w_1'}(S_{w_1}^- A_{e_1} S_{v_1}^+)(S_{w_2}^-(S_{v_4}^-A_{e_5}A_{e_4}S_{v_4}^+)S_{v_5}^+)\\
&\simeq(\phi_{e_7})_* S_{v_1v_2}^- A_{e_8} A_{e_6*e_3}  A_{e_2}S^+_{v_2w_1'} (\phi_{e_1})_*(\phi_{e_4*e_5})_*\\
&\simeq U(\phi_{e_7})_* S_{v_1v_2}^- A_{e_8} S^+_{v_2w_1'} (\phi_{e_1})_*(\phi_{e_4*e_5})_*+ (\phi_{e_7})_* S_{v_1v_2}^- A_{e_8}A_{e_2} A_{e_6*e_3}  S^+_{v_2w_1'} (\phi_{e_1})_*(\phi_{e_4*e_5})_*\\
&\simeq U(\phi_{e_7})_* S_{v_1}^- S^+_{w_1'} (\phi_{e_1})_*(\phi_{e_4*e_5})_*+(\phi_{e_7})_* S_{v_1}^-S_{v_2}^- A_{e_8}A_{e_2}S^+_{v_2}S_{w_1'}^+ A_{e_6*e_3}   (\phi_{e_1})_*(\phi_{e_4*e_5})_*\\
&\simeq U(\phi_{e_7})_* (S_{v_1}^- A_{e_2* e_8} S_{v_1}^+)S_{v_1}^- S^+_{w_1'} (\phi_{e_1})_*(\phi_{e_4*e_5})_*+(\phi_{e_7})_* (\phi_{e_8*e_2})_* A_{e_6*e_3}   (\phi_{e_1})_*(\phi_{e_4*e_5})_*
\\&\simeq U(\phi_{e_7})_* S_{v_1}^- A_{e_2* e_8} \Phi_{v_1} S^+_{w_1'} (\phi_{e_1})_*(\phi_{e_4*e_5})_*+(\phi_{e_7})_* (\phi_{e_8*e_2})_* A_{e_6*e_3}   (\phi_{e_1})_*(\phi_{e_4*e_5})_*
\\ &\simeq U(\phi_{e_7})_* S_{v_1}^- A_{e_2* e_8}S^+_{w_1'} \Phi_{v_1}  (\phi_{e_1})_*(\phi_{e_4*e_5})_*+(\phi_{e_7})_* (\phi_{e_8*e_2})_* A_{e_6*e_3}   (\phi_{e_1})_*(\phi_{e_4*e_5})_*\\
&\simeq U(\phi_{e_7})_* (\phi_{e_8*e_2})_* \Phi_{v_1}  (\phi_{e_1})_*(\phi_{e_4*e_5})_*+(\phi_{e_7})_* (\phi_{e_8*e_2})_* A_{e_6*e_3}   (\phi_{e_1})_*(\phi_{e_4*e_5})_*\\
&=\phi_*'(U \Phi_{v_1}+ A_{\lambda'})\phi_*,\end{align*} where $\phi':=(\phi_{e_7})_*(\phi_{e_8* e_2})_*$ and $\phi:=(\phi_{e_1})_*(\phi_{e_4*e_5})_*$ are the basepoint moving maps for moving along the ``vertical'' edges of $\cG_H$, and $\lambda':=e_6*e_3$ is the ``horizontal'' edge of $\cG_H$. Noting that the diffeomorphism $\phi$ moves $w_1$ to $v_1$,  the formula 
\[F_{[0,1]\times Y,\Gamma_H,\frs}^A\simeq A_\lambda+U \Phi_{w_1}\] follows. The formula for $F_{[0,1]\times Y ,\Gamma_H,\frs}^B$ follows by replacing the $A$'s with $B$'s in the above argument.
\end{proof}

The graph in the previous lemma resembles a ribbon 1-skeleton of the $\Sigma_{\ve{w}}$ subsurface of the decorated surface for $\Psi_z$. This is no accident:

\begin{lem}\label{lem:Alambda=Psi+UPhi}Suppose that $w_1,z,$ and $w_2$ are consecutive basepoints on a link $\bL$ (in that order, though $w_1$ and $w_2$ need not be distinct), and $\lambda$ is the subarc of $L$ from $w_1$ to $w_2$. Then the reduction of $\Psi_z$ satisfies
 \[
 \Psi_{z}|_{V=1}\simeq A_\lambda +U_{w_2}\Phi_{w_2}\simeq B_\lambda+U_{w_1}\Phi_{w_1}.
 \]
 \end{lem}
 
 \begin{proof}Write $\lambda$ as the concatenation of two paths, $\lambda_1$ and $\lambda_2$, which meet at $z$. The configuration is shown in Figure~\ref{fig::67}.
 
 \begin{figure}[ht!]
 \centering
 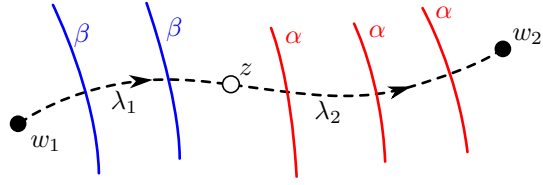
 \caption{\textbf{A portion of the Heegaard diagram containing the subarc $\lambda$ of $L$.} The arc $\lambda$ is the concatenation of $\lambda_1$ and $\lambda_2$.\label{fig::67}}
 \end{figure} 
 
 Using Condition~\eqref{HD:cond7} of Definition~\ref{def:Heegaarddiagramlink} (the definition of a Heegaard diagram for a link), we can assume that $\lambda$ is immersed in the Heegaard diagram, and furthermore, we can assume that $\lambda_1$ does not intersect any of the $\ve{\alpha}$-curves, and that $\lambda_2$ does not intersect any of the $\ve{\beta}$-curves. Hence the quantity $a(\lambda,\phi)$ is equal to the total change of the multiplicities of $\phi$ along $\lambda_2$, i.e.,
 \[
 a(\lambda,\phi)=n_{w_2}(\phi)-n_{z}(\phi).
 \]
  Multiplying $a(\lambda,\phi)$ by $\# \Hat{\cM}(\phi)U_{\ve{w}}^{n_{\ve{w}}(\phi)} \cdot \ve{y}$  and summing over all index 1 classes $\phi$, we arrive at the stated relation
 \[
 A_\lambda(\ve{x})=U_{w_2}\Phi_{w_2}(\ve{x})+\Psi_z(\ve{x})|_{V=1}.
 \] 
  Analogously, by adapting the same line of reasoning yields for the map $B_{\lambda}$, which uses the weights $b(\lambda,\phi)$ instead of $a(\lambda,\phi)$, we arrive at the relation
 \[
 B_\lambda(\ve{x})=U_{w_1}\Phi_{w_1}(\ve{x})+\Psi_z(\ve{x})|_{V=1},
 \] 
completing the proof.
 \end{proof}

Analogously, we can compute the reduction of the maps $\Phi_w$, after setting $U=1$.

\begin{lem}\label{lem:Alambda=Phi+VPsi}If $w\in \ve{w}$ is a basepoint of $\bL=(L,\ve{w},\ve{z})$, let $\lambda$ denote the component of $L\setminus \ve{z}$ which contains $w$. Suppose that $\lambda$, oriented as a subset of $L$, goes from $z_1$ to $z_2$. Then
\[\Phi_w|_{U=1}\simeq A_\lambda+ V_{z_1} \Psi_{z_1}\simeq B_\lambda +V_{z_2} \Psi_{z_2}.\]
\end{lem} 
\begin{proof}The proof is analogous to the proof of Lemma~\ref{lem:Alambda=Psi+UPhi}.
\end{proof}
\subsection{Proof of Theorem~C}
\label{sec:proofofThmC}
We can now prove our stated theorem about the algebraic reductions of $F_{W,{\cF},\frs}$.

\begin{proof}[Proof of Theorem~\ref{thm:C'}] 
 Let $(W,{\cF})$ be a decorated link cobordism with ${\cF}=(\Sigma,\cA)$. We first consider the $V=1$ reduction of the link cobordism maps. First note that the composition law holds for both the graph cobordism maps, and the link cobordism maps. Noting also that the reduction formula is obviously true for the 0-handle and 4-handle cobordism maps, we can assume that each component of $\Sigma$ intersects a component of the incoming boundary and the outgoing boundary non-trivially.

Given a ribbon 1-skeleton $\Gamma(\Sigma_{\ve{w}})$ of $\Sigma_{\ve{w}}$, we note that by using Lemma~\ref{lem:replacewithtrivalent}, we can reduce to the case that $\Gamma(\Sigma_{\ve{w}})$ has no vertices of valence more than three. Using our assumption that each component of $\Sigma$  intersects a component of the incoming and outgoing boundaries of $W$, we can decompose $(W,{\cF})$ into a composition $(W_n,{\cF}_n)\circ \cdots \circ (W_1,{\cF}_1)$ where each $(W_i,{\cF}_i)$ satisfies one of the following:

\begin{enumerate}[leftmargin=20mm, ref= 
\textrm{\arabic*}, label = \textrm{($\cR\cG$-\arabic*)}:]
\item\label{def:surfacegraphdecomp1} $(W_i,{\cF}_i)$ is diffeomorphic to the trace of an isotopy $\phi_t$ of a link in a  3-manifold, and $\Gamma(\Sigma_{\ve{w}})\cap W_i$ consists of a collection of paths from the incoming boundary to the outgoing boundary.
\item\label{def:surfacegraphdecomp2} $(W_i,{\cF}_i)$ is a diffeomorphic to the handle attachment cobordism formed by attaching a 4-dimensional handle or collection of handles along a framed $0$-sphere, 2-sphere or collection of framed 1-spheres, in the complement of the link. Furthermore, the graph $\Gamma(\Sigma_{\ve{w}})\cap {\cF}_i$ consists of a collection of paths from the incoming boundary to the outgoing boundary.
\item\label{def:surfacegraphdecomp3} $(W_i,{\cF}_i)$ is diffeomorphic to a cobordism obtained by attaching a type-$\ve{z}$ band to the link. Furthermore $\Gamma(\Sigma_{\ve{w}})$ consists of a collection of paths from the incoming boundary to the outgoing boundary.
\item \label{def:surfacegraphdecomp4} $(W_i,\Sigma_i)$ is diffeomorphic (as an undecorated link cobordism) to a cylindrical link cobordism. The graph $\Gamma(\Sigma_{\ve{w}})\cap W_i$ is an elementary flow-graph, satisfying one of~($\mathcal{E\!FG}$-\ref{def:elemengraphtype1}), ($\mathcal{E\!FG}$-\ref{def:elemengraphtype2}) or~($\mathcal{E\!FG}$-\ref{def:elemengraphtype3}). The subsurface $\Sigma_{\ve{w}}\cap W_i$ is a regular neighborhood of the graph $\Gamma(\Sigma_{\ve{w}})\cap W_i$.
\end{enumerate}

We now show the claim for a link cobordism satisfying condition ($\cR\cG$-\ref{def:surfacegraphdecomp1}). In this case, the induced link cobordism is the diffeomorphism map induced by the isotopy $\phi_1$. It is clear that the reduction of the diffeomorphism map on link Floer homology is equal to the map induced by the same diffeomorphism on Heegaard Floer homology.

We now consider the claim for link cobordisms satisfying condition ($\cR\cG$-\ref{def:surfacegraphdecomp2}). In this case, the link cobordism maps are defined by using the 4-dimensional handle attachment maps. The 1-handle and 3-handle maps on link Floer homology obviously reduce to the 1-handle and 3-handle maps used to construct the graph cobordism maps. The 2-handle maps on link Floer homology count holomorphic triangles representing homology classes of triangles with $\frs_{\ve{w}}(\psi)=\frs$. The graph cobordism maps for $\Gamma(\Sigma_{\ve{w}})$ count the same triangles. Hence the $V=1$ reduction of the link cobordism map $F_{W_i,{\cF}_i,\frs|_{W_i}}$ coincides with $F_{W_i,\Gamma(\Sigma_{\ve{w}})\cap W_i, \frs|_{W_i}}^B$, if $(W_i,{\cF}_i)$ satisfies condition ($\cR\cG$-\ref{def:surfacegraphdecomp2}). Note that the $U=1$ reduction of the  2-handle maps is not covered by the above argument (as we discuss below).

We now consider the claim for link cobordisms satisfying condition ($\cR\cG$-\ref{def:surfacegraphdecomp3}). In this case, the link cobordism map is a type-$\ve{z}$ band map. The band map is defined by counting holomorphic triangles via the formula 
\[
F_{B}^{\ve{z}}(\ve{x})=F_{\as',\as,\bs}(\Theta^{\ve{w}}_{\as',\as},\ve{x}),
\]
 where $\Theta^{\ve{w}}_{\as',\as}$ denotes the top $\gr_{\ve{w}}$-degree generator. After setting $V=1$, the generator $\Theta^{\ws}_{\as',\as}$ reduces to the top degree generator $\Theta^+_{\as',\as}\in \HF^-(\Sigma,\as',\as,\ws,\frs_0)$, and hence $F_{B}^{\zs}$ reduces to the change of diagrams map associated to moving the $\ve{\alpha}$ curves, which induces the identity morphism on the level of transitive systems of chain complexes. On the other hand, the graph cobordism $([0,1]\times Y,\Gamma(\Sigma_{\ve{w}}))$ is the identity graph cobordism, so the induced graph cobordism map is also the identity map.  

We now consider link cobordisms satisfying condition ($\cR\cG$-\ref{def:surfacegraphdecomp4}). Given a cobordism $(W_i,{\cF}_i)$ satisfying condition ($\cR\cG$-\ref{def:surfacegraphdecomp4}),  we can subdivide the cobordism and add trivial strands to the graph, and replace $(W_i,{\cF}_i)$ with a sequence of cobordisms satisfying the following:
\begin{enumerate}[leftmargin=15mm, ref= 
\textrm{\arabic*}, label = \textrm{($\cR\cG$-\arabic*)}:] \setcounter{enumi}{4}
\item \label{def:surfacegraphdecomp5} $(W_i,{\cF}_i)$ is equivalent (as an undecorated link cobordism) to a cylindrical link cobordism. The subsurface $\Sigma_{\ve{w}}\cap W_i$ is a regular neighborhood of $\Gamma(\Sigma_{\ve{w}})\cap W_i$. The graph $\Gamma(\Sigma_{\ve{w}})\cap W_i$ satisfies one of the following:
\begin{enumerate}[leftmargin=15mm, ref= 
5\textrm{\alph*}, label = \textrm{($\cR\cG$-5\alph*)}:]
\item  \label{def:elementgraph5a} $\Gamma(\Sigma_{\ve{w}})$ is an elementary flow-graph, of type~($\mathcal{E\!FG}$-\ref{def:elemengraphtype2}), and the vertex $v_0$ in the definition of such an elementary flow-graph, has valence 1.
\item \label{def:elementgraph5b} $\Gamma(\Sigma_{\ve{w}})$ is the $H$-shaped graph from Figure~\ref{fig::80}.
\end{enumerate}

\end{enumerate}

The process of manipulating and decomposing link cobordisms $(W_i,{\cF}_i)$ satisfying condition {($\cR\cG$-\ref{def:surfacegraphdecomp4})} into a sequence of cobordisms satisfying condition ($\cR\cG$-\ref{def:surfacegraphdecomp5}) is illustrated in Figure~\ref{fig::83}.

 \begin{figure}[ht!]
 \centering
 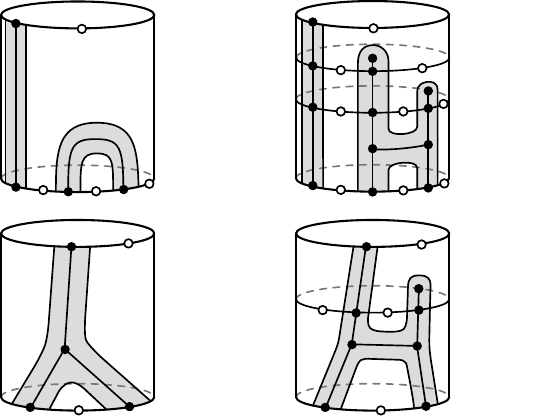
 \caption{\textbf{Decomposing a link cobordism and ribbon 1-skeleton which satisfies ($\cR\cG$-\ref{def:surfacegraphdecomp4}) into a sequence of link cobordisms with ribbon 1-skeletons which satisfy ($\cR\cG$-\ref{def:surfacegraphdecomp5}).} Dividing sets on surfaces of the form $ [0,1]\times L$ are shown, as well as the graph $\Gamma(\Sigma_{\ve{w}})\subset \Sigma_{\ws}$. These are inside of the link cobordism $([0,1]\times Y)$ (not shown).\label{fig::83}}
 \end{figure} 

Thus, to show that the $V=1$ reduction of the link cobordism maps agree with the graph cobordism maps for a ribbon 1-skeleton of $\Sigma_{\ve{w}}$, it is sufficient instead to show the claim for link cobordisms with embedded ribbon 1-skeletons satisfying one of conditions ($\cR\cG$-\ref{def:elementgraph5a}) or ($\cR\cG$-\ref{def:elementgraph5b}).

If $(W_i,{\cF}_i)$ satisfies condition ($\cR\cG$-\ref{def:elementgraph5a}), then the link cobordism map is equal to $S_{w,z}^+$ or $S_{w,z}^-$, where $w$ is the basepoint on the single component of $(\Sigma_{\ve{w}})\cap W_i$ which is a bigon. Similarly, the graph cobordism map is easily seen to be $S_w^+$ or $S_w^-$. From the definition of the quasi-stabilization maps, it is clear that
\[S_{w,z}^{\circ}|_{V=1}\simeq S_w^{\circ},\] proving the reduction formula if $(W_i,{\cF}_i)$ satisfies condition ($\cR\cG$-\ref{def:elementgraph5a}).

If $(W_i,{\cF}_i)$ satisfies condition ($\cR\cG$-\ref{def:elementgraph5b}), then the graph $\Gamma(\Sigma_{\ve{w}})$ is the $H$ graph from Figure~\ref{fig::80} and $\Sigma_{\ve{w}}$ is a regular neighborhood of $\Gamma(\Sigma_{\ve{w}})$. The link cobordism $(W_i,{\cF}_i)$ and the graph $\Gamma(\Sigma_{\ve{w}})\cap {\cF}_i$ is shown with orientations and cyclic orders in Figure~\ref{fig::85}.

 \begin{figure}[ht!]
 \centering
 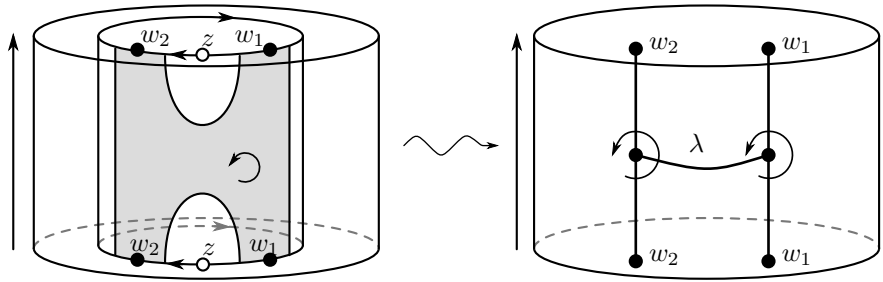
 \caption{\textbf{A cylindrical link cobordism with $H$-shaped subsurface $\Sigma_{\ve{w}}$ and the corresponding graph cobordism.}  Orientations and cyclic orders are shown.\label{fig::85}}
 \end{figure}

We can write the link cobordism maps as a composition of two quasi-stabilizations:
\[
F_{W,{\cF},\frs}\simeq T_{w_2,z}^+T_{w_2,z}^- \simeq \Psi_z,
\]
 where $z$ is the basepoint corresponding to either of the two bigons of type $\Sigma_{\ve{z}}$. By Lemma~\ref{lem:Alambda=Psi+UPhi}, the $V=1$ reduction of the link cobordism map is
\[
\Psi_{z}|_{V=1}\simeq B_\lambda+U_{w_1}\Phi_{w_1}.
\] By Lemma~\ref{lem:Hgraphaction}, this is equal to the type-$B$ graph cobordism map $F_{[0,1]\times Y,\Gamma(\Sigma_{\ve{w}}),\frs}^B$, completing the proof of the claim if $(W_i,{\cF}_i)$ satisfies condition ($\cR\cG$-\ref{def:elementgraph5b}).

Having proven the claim for link cobordisms satisfying one  ($\cR\cG$-\ref{def:surfacegraphdecomp1})--($\cR\cG$-\ref{def:surfacegraphdecomp4}), it now follows from the composition law that $F_{W,{\cF},\frs}|_{V=1}\simeq F_{W,\Gamma(\Sigma_{\ve{w}}),\frs}^B$ for a general decorated link cobordism $(W,{\cF}).$

We now consider the $U=1$ reduction of the link cobordism maps. The above strategy is easily adapted to consider these maps, however there are several slight differences. Firstly, the 2-handle maps for $F_{W,{\cF},\frs}$ count holomorphic triangles with underlying homology class $\psi$ satisfying $\frs_{\ve{w}}(\psi)=\frs$, while the 2-handle maps for $F_{W,\Gamma(\Sigma_{\ve{z}}),\frs}^B$ count holomorphic triangles which satisfy $\frs_{\ve{z}}(\psi)=\frs$. However, by \cite{ZemAbsoluteGradings}*{Lemma~3.9}, one has
\[
\frs_{\ve{w}}(\psi)-\frs_{\ve{z}}(\psi)=\PD[\Sigma],
\]
 where $[\Sigma]\in H_2(W,\d W;\Z)$ denotes the homology class of the surface in the link cobordism. Hence, for a 2-handle cobordism, one has
\[
F_{W_i,{\cF}_i,\frs|_{W_i}}|_{U=1}\simeq F_{W_i,\Gamma(\Sigma_{\ve{z}})\cap W_i,(\frs-\PD[\Sigma])|_{W_i}}^A.
\]

Finally for link cobordisms satisfying ($\cR\cG$-\ref{def:elementgraph5b}) (but now with $\Sigma_{\ve{z}}$ having an $H$-shaped component, instead of $\Sigma_{\ve{w}}$), we note that the induced link cobordism map is $\Phi_w$, for a $w$ basepoint in the boundary of one of the bigon components of $\Sigma_{\ve{w}}$. By Lemma~\ref{lem:Alambda=Phi+VPsi}, we have
\[
\Phi_w|_{U=1}\simeq A_{\lambda}+V_{z_1} \Psi_{z_1},
\] which is the type-$A$ graph cobordism map $F^A_{W_i,\Gamma(\Sigma_{\ve{z}})\cap W_i,(\frs-\PD[\Sigma])|_{W_i}}$ by Lemma~\ref{lem:Hgraphaction} (note that the change in $\Spin^c$ structures, in this case, is simply due to the change in $\Spin^c$ structures of the reductions of the link Floer complexes, from Lemma~\ref{lem:changeSpincstructure}). Hence we conclude that, in general
\[
F_{W,{\cF},\frs}|_{U=1}\simeq F_{W,\Gamma(\Sigma_{\ve{z}}),\frs-\PD[\Sigma]}^A.\]
\end{proof}

It is clear from the previous argument that the $F_{W,{\cF},\frs}|_{V=1}$ reduction depends only on the graph $\Gamma(\Sigma_{\ve{w}})$ and $\frs$. Similarly it is clear that the reduction $F_{W,{\cF},\frs}|_{U=1}$ depends only on $\Gamma(\Sigma_{\ve{z}})$ and $\frs-\PD[\Sigma]$. Hence Corollary~\ref{cor:E} follows.

\subsection{Constructing the graph TQFT as the reduction of the link Floer TQFT}
\label{sec:alternateconstructionofgraphs}
In the previous section, we showed that the link cobordism maps reduce to the graph cobordism maps, which were constructed using other techniques in \cite{ZemGraphTQFT}. In this section, we show that the link Floer TQFT can be used directly to construct a TQFT for cobordisms with embedded ribbon graphs. To see this, we show that given a ribbon graph cobordism $(W,\Gamma)$ one can always associate a decorated link cobordism $\cF(\Gamma)$, which is well-defined up to a small ambiguity. We will show that the ambiguity does not effect the maps. It will then follow that we get an alternate construction and proof of invariance of the graph cobordism maps from \cite{ZemGraphTQFT}.

By Lemma~\ref{lem:ribbongraphtoribbonsurface}, to a ribbon graph cobordism $(W,\Gamma)$, we can associate a ribbon surface cobordism $(W,R_\Gamma)$, which is well-defined up to isotopy. The normal bundle to the surface $R_\Gamma$ is an oriented 2-plane bundle, over a space homotopy equivalent to a graph, so it is trivializable. We can pick a section $v$ of the normal bundle, and use it to push off a copy $R_{\Gamma}'$ of $R_{\Gamma}$ (in such a way which keeps $\d R_{\Gamma}\cap \Int W$ fixed). We define a decorated link cobordism ${\cF}(\Gamma)$ to be the union of $R_{\Gamma}$ and $R_{\Gamma}'$ (giving $R_{\Gamma}'$ the opposite orientation). By Theorem~\ref{thm:C}, we have that
\[F_{W,{\cF}(\Gamma),\frs}|_{V=1}\simeq F^A_{W,\Gamma,\frs}.\]

The surface ${\cF}(\Gamma)$ is not quite well-defined,  since the vector $v$ is not uniquely specified, even up to isotopy, since a choice of such a vector field $v$ is equivalent to a choice of map from $\Gamma$ to $S^1$. The ambiguity thus lies in twisting along the edges of the graph. However the induced  link cobordism map for different choices of the vector field $v$ can differ by the diffeomorphism map induced by twisting along doubly based unknot, but the map induced by such a twist is clearly the identity map (one simply considers a diagram where the two basepoints of the knot are right next to each other). Hence the map $F_{W,{\cF}(\Gamma),\frs}$ is independent of the choice of vector field $v$, normal to $R_\Gamma$.

\bibliographystyle{custom}
\bibliography{biblio}
\end{document}